\documentclass[a4paper,10pt]{amsart}
\pdfoutput=1
\usepackage{etex}

\usepackage[T1]{fontenc}
\usepackage{lmodern}
\usepackage[utf8]{inputenc}
\usepackage{amssymb}
\usepackage{enumitem}
\usepackage{mathrsfs,dsfont,mathtools}
\usepackage{nicefrac}
\usepackage{microtype}

\usepackage[lite]{amsrefs}

\renewcommand{\PrintDOI}[1]{\href{http://dx.doi.org/\detokenize{#1}}{doi: \detokenize{#1}}}

\BibSpec{book}{%
    +{}  {\PrintPrimary}                {transition}
    +{,} { \textit}                     {title}
    +{.} { }                            {part}
    +{:} { \textit}                     {subtitle}
    +{,} { \PrintEdition}               {edition}
    +{}  { \PrintEditorsB}              {editor}
    +{,} { \PrintTranslatorsC}          {translator}
    +{,} { \PrintContributions}         {contribution}
    +{,} { }                            {series}
    +{,} { \voltext}                    {volume}
    +{,} { }                            {publisher}
    +{,} { }                            {organization}
    +{,} { }                            {address}
    +{,} { \PrintDateB}                 {date}
    +{,} { }                            {status}
    +{}  { \parenthesize}               {language}
    +{}  { \PrintTranslation}           {translation}
    +{;} { \PrintReprint}               {reprint}
    +{.} { }                            {note}
    +{.} {}                             {transition}
    +{} { \PrintDOI}                   {doi}
    +{}  {\SentenceSpace \PrintReviews} {review}
}

\usepackage{enumitem} 
\setlist[enumerate,1]{label=\textup{(\arabic*)}}

\usepackage[pdftitle={Representations of *-algebras by unbounded operators:
  C*-hulls, local--global principles, and induction},
  pdfauthor={Ralf Meyer},
  pdfsubject={Mathematics}
]{hyperref}

\numberwithin{equation}{section}
\theoremstyle{plain}
\newtheorem{theorem}[equation]{Theorem}
\newtheorem{lemma}[equation]{Lemma}
\newtheorem{proposition}[equation]{Proposition}
\newtheorem{corollary}[equation]{Corollary}
\theoremstyle{definition}
\newtheorem{definition}[equation]{Definition}

\theoremstyle{remark}
\newtheorem{remark}[equation]{Remark}
\newtheorem{example}[equation]{Example}

\newcommand*{\C}{\mathbb C}
\newcommand*{\Z}{\mathbb Z}
\newcommand*{\N}{\mathbb N}
\newcommand*{\R}{\mathbb R}
\newcommand*{\T}{\mathbb T}
\newcommand*{\Comp}{\mathbb K}
\newcommand*{\Bound}{\mathbb B}
\newcommand*{\Mat}{\mathbb M}

\newcommand*{\affil}{\mathrel\eta}

\newcommand*{\id}{\textup{id}}

\newcommand*{\ima}{\textup i}
\newcommand*{\dd}{\textup d}
\newcommand*{\Cst}{\texorpdfstring{\textup C^*}{C*}}
\newcommand*{\Cstar}{\texorpdfstring{$\textup C^*$}{C*}}
\newcommand*{\Star}{\texorpdfstring{$^*$\nb-}{*-}}

\newcommand*{\Lg}{\mathfrak g}
\newcommand*{\Hilm}[1][E]{\mathcal{#1}}
\newcommand*{\Hilms}[1][E]{\mathfrak{#1}}
\newcommand*{\Mult}{\mathcal M}

\newcommand*{\Toep}{\mathcal T}

\newcommand*{\braket}[2]{\langle#1,#2\rangle}

\newcommand*{\cl}[1]{\overline{#1}}
\newcommand*{\conj}[1]{\overline{#1}}

\newcommand*{\Cont}{\textup C} 
\newcommand*{\Contc}{\textup C_\textup c} 
\newcommand*{\Contb}{\textup C_\textup b} 

\newcommand{\idealin}{\mathrel{\triangleleft}} 
\newcommand*{\nb}{\nobreakdash}  
\newcommand*{\alb}{\hspace{0pt}} 

\newcommand*{\blank}{\text{\textvisiblespace}}

\mathtoolsset{mathic}
\DeclarePairedDelimiter{\abs}{\lvert}{\rvert}
\DeclarePairedDelimiter{\norm}{\lVert}{\rVert}

\DeclareMathOperator{\sign}{sign}

\DeclareMathOperator{\Prim}{Prim}
\DeclareMathOperator{\dom}{dom}
\DeclareMathOperator{\Endo}{End}
\DeclareMathOperator{\Rep}{Rep}
\DeclareMathOperator{\Repi}{Rep_{int}}

\newcommand*{\defeq}{\mathrel{\vcentcolon=}}
\newcommand*{\into}{\rightarrowtail}
\newcommand*{\prto}{\twoheadrightarrow}
\newcommand*{\injto}{\hookrightarrow}
\newcommand*{\congto}{\xrightarrow\sim}

\begin{document}

\title[Representations by unbounded operators]{Representations of *-algebras\\
  by unbounded operators:\\
  C*-hulls, local--global principle, and induction}

\author{Ralf Meyer}
\email{rmeyer2@uni-goettingen.de}

\address{Mathematisches Institut\\
  Georg-August Universit\"at G\"ottingen\\
  Bunsenstra\ss{}e 3--5\\
  37073 G\"ottingen\\
  Germany}

\begin{abstract}
  We define a \Cstar\nb-hull for a \Star{}algebra, given a notion of
  integrability for its representations on Hilbert modules.  We
  establish a local--global principle which, in many cases,
  characterises integrable representations on Hilbert modules
  through the integrable representations on Hilbert spaces.  The
  induction theorem constructs a \Cstar\nb-hull for a certain class
  of integrable representations of a graded \Star{}algebra, given a
  \Cstar\nb-hull for its unit fibre.
\end{abstract}


\maketitle

\setcounter{tocdepth}{1}
\tableofcontents

\section{Introduction}
\label{sec:intro}

Savchuk and Schmüdgen~\cite{Savchuk-Schmudgen:Unbounded_induced}
have introduced a method to define and classify the integrable
representations of certain \Star{}algebras by an inductive
construction.  The original goal of this article was to clarify this
method and thus make it apply to more situations.  This has led me
to reconsider some foundational aspects of the theory of
representations of \Star{}algebras by unbounded operators.  This is
best explained by formulating an induction theorem
that is inspired by~\cite{Savchuk-Schmudgen:Unbounded_induced}.

Let~\(G\) be a discrete group with unit element \(e\in G\).  Let
\(A=\bigoplus_{g\in G} A_g\) be a \(G\)\nb-graded unital
\Star{}algebra.  That is, \(A_g\cdot A_h\subseteq A_{gh}\),
\(A_g^*=A_{g^{-1}}\), and \(1\in A_e\).  In particular, the
\emph{unit fibre}~\(A_e\) is a unital \Star{}algebra.  Many
interesting examples of this situation are studied in
\cites{Savchuk-Schmudgen:Unbounded_induced, Dowerk-Savchuk:Induced}.
A \emph{Fell bundle} over~\(G\) is a family of
subspaces~\((B_g)_{g\in G}\) of a \Cstar\nb-algebra~\(B\)
(which is not part of the data) such that \(B_g\cdot B_h\subseteq
B_{gh}\) and \(B_g^*=B_{g^{-1}}\).  The universal choice for~\(B\)
is the \emph{section \Cstar\nb-algebra} of the Fell bundle.

Briefly, our main result says the following.  Let~\(B_e\) be a
\Cstar\nb-algebra such that ``integrable'' ``representations''
of~\(A_e\) are ``equivalent'' to ``representations'' of~\(B_e\).
Under some technical conditions, we construct a Fell
bundle~\((B_g^+)_{g\in G}\) over~\(G\) such that ``integrable''
``representations'' of~\(A\) are ``equivalent'' to
``representations'' of its section \Cstar\nb-algebra.  Here the
words in quotation marks must be interpreted
carefully to make this true.

A \emph{representation} of a \Star{}algebra~\(A\) on a Hilbert
\(D\)\nb-module~\(\Hilm\) is an algebra homomorphism~\(\pi\)
from~\(A\) to the algebra of \(D\)\nb-module endomorphisms of a
dense \(D\)\nb-submodule \(\Hilms\subseteq \Hilm\) with
\(\braket{\xi}{\pi(a)\eta} = \braket{\pi(a^*)\xi}{\eta}\) for all
\(\xi,\eta\in\Hilms\), \(a\in A\).  The representation induces a
graph topology on~\(\Hilms\).  We restrict to \emph{closed}
representations most of the time, that is, we require~\(\Hilms\) to
be complete in the graph topology.  The difference from usual
practice is that we consider representations on \emph{Hilbert
  modules} over \Cstar\nb-algebras.  A \emph{representation of a
  \Cstar\nb-algebra}~\(B\) on a Hilbert module~\(\Hilm\) is a
nondegenerate \Star{}homomorphism \(B\to\Bound(\Hilm)\),
where~\(\Bound(\Hilm)\) denotes the \Cstar\nb-algebra of adjointable
operators on~\(\Hilm\).

The notion of ``integrability'' for representations is a choice.
The class of all Hilbert space representations of a \Star{}algebra
may be quite wild.  Hence it is customary to limit the study to some
class of ``nice'' or ``integrable'' representations.  For instance,
for the universal enveloping algebra of the Lie algebra of a Lie
group~\(G\), we may call those representations ``integrable'' that
come from a unitary representation of~\(G\).  This example suggests
the name ``integrable'' representations.

In our theorem, a notion of integrability for representations of
\(A_e\subseteq A\) on all Hilbert modules over all
\Cstar\nb-algebras is fixed.  A representation of~\(A\) is called
integrable if its restriction to~\(A_e\) is integrable.  The
induction theorem describes the integrable representations of~\(A\)
in terms of integrable representations of~\(A_e\).  For instance,
if~\(A_e\) is finitely generated and commutative, then we may call a
representation~\(\pi\) on a Hilbert module integrable if the
closure~\(\cl{\pi(a)}\) is a regular, self-adjoint operator for each
\(a\in A_e\) with \(a=a^*\); all examples in
\cites{Savchuk-Schmudgen:Unbounded_induced, Dowerk-Savchuk:Induced}
are of this type.

An ``equivalence'' between the integrable representations of a
unital \Star{}algebra~\(A\) and the representations of a
\Cstar\nb-algebra~\(B\) is a family of bijections -- one for each
Hilbert module~\(\Hilm\) over each \Cstar\nb-algebra~\(D\) --
between the sets of integrable representations of~\(A\) and of
representations of~\(B\) on~\(\Hilm\); these bijections must be
compatible with isometric intertwiners and interior tensor products.
These properties require some more definitions.

First, an \emph{isometric intertwiner} between two representations
is a Hilbert module isometry -- not necessarily adjointable --
between the underlying Hilbert modules that restricts to a left
module map between the domains of the representations.  For an
equivalence between integrable representations of~\(A\) and
representations of~\(B\) we require an isometry to intertwine two
representations of~\(B\) if and only if it intertwines the
corresponding integrable representations of~\(A\).

Secondly, a \emph{\Cstar\nb-correspondence} from~\(D_1\) to~\(D_2\)
is a Hilbert \(D_2\)\nb-module~\(\Hilm[F]\) with a representation
of~\(D_1\).  Given such a correspondence and a Hilbert
\(D_1\)\nb-module~\(\Hilm\), the
interior tensor product \(\Hilm\otimes_{D_1} \Hilm[F]\) is a Hilbert
\(D_2\)\nb-module.  A representation of \(A\) or~\(B\) on~\(\Hilm\)
induces a representation on \(\Hilm\otimes_{D_1} \Hilm[F]\).  We
require our bijections between integrable representations of \(A\)
and representations of~\(B\) to be compatible with this interior
tensor product construction on representations.

We call~\(B\) a \emph{\Cstar\nb-hull for the integrable
  representations of~\(A\)} if the integrable representations
of~\(A\) are equivalent to the representations of~\(B\) as explained
above, that is, through a family of bijections compatible with
isometric intertwiners and interior tensor products.  The Induction
Theorem builds a \Cstar\nb-hull for the integrable representations
of~\(A\) using a \Cstar\nb-hull for the integrable representations
of~\(A_e\) and assuming a further mild technical condition, which we
explain below.

Many results of the general theory remain true if we only require
the equivalence of representations to be compatible with interior
tensor products and \emph{unitary} \Star{}intertwiners, that is,
isomorphisms of representations; we speak of a \emph{weak
  \Cstar\nb-hull} in this case.  The Induction Theorem, however,
fails for weak \Cstar\nb-hulls.  We show this by a counterexample.
Some results only need the class of integrable representations to
have some properties that are clearly necessary for the existence of
a \Cstar\nb-hull or weak \Cstar\nb-hull, but they do not need the
(weak) \Cstar\nb-hull itself.  This is formalised in our notions of
\emph{admissible} and \emph{weakly admissible} classes of
representations.

For example, let~\(A\) be commutative.  Let~\(\hat{A}\) be the space
of characters of~\(A\) with the topology of pointwise convergence.
If~\(\hat{A}\) is locally compact and~\(A\) is countably generated,
then \(\Cont_0(\hat{A})\) is a \Cstar\nb-hull for the integrable
representations of~\(A\) as defined above, that is, those
representations where each~\(\cl{\pi(a)}\) for \(a\in A\) with
\(a=a^*\) is regular and self-adjoint.  If, say, \(A= \C[x]\) with
\(x=x^*\), then the \Cstar\nb-hull is \(\Cont_0(\R)\).  Here the
equivalence of representations maps an integrable representation~\(\pi\)
of~\(\C[x]\) to the functional calculus homomorphism for the
regular, self-adjoint operator~\(\cl{\pi(x)}\).

If~\(\hat{A}\) is not locally compact, then the integrable
representations of~\(A\) defined above still form an admissible
class, but they have no \Cstar\nb-hull.  If, say, \(A\) is the
algebra of polynomials in countably many variables, then \(\hat{A} =
\R^\infty\), which is not locally compact.  The problem of
associating \Cstar\nb-algebras to this \Star{}algebra has recently
been studied by Grundling and
Neeb~\cite{Grundling-Neeb:Infinite_tensor}.  From our point of view,
this amounts to choosing a smaller class of ``integrable''
representations that does admit a \Cstar\nb-hull.

We have now explained the terms in quotation marks in our Induction
Theorem and how we approach the representation theory of
\Star{}algebras.  Most previous work focused either on
representations on Hilbert spaces or on single unbounded operators
on Hilbert modules.  Hilbert module representations occur both in
the assumptions and in the conclusions of the Induction Theorem, and
hence we cannot prove it without considering representations on
Hilbert modules throughout.  In addition, taking into account
Hilbert modules makes our \Cstar\nb-hulls unique.

Besides the Induction Theorem, the other main strand of this article
are Local--Global Principles, which aim at reducing the study of
integrability for representations on general Hilbert modules to
representations on Hilbert space.  We may use a state~\(\omega\) on
the coefficient \Cstar\nb-algebra~\(D\) of a Hilbert module~\(\Hilm\)
to complete~\(\Hilm\) to a Hilbert space.  Thus
a representation of~\(A\) on~\(\Hilm\) induces Hilbert space
representations for all states on~\(D\).  The \emph{Local--Global
  Principle} says that a representation of~\(A\) on~\(\Hilm\) is
integrable if and only if these induced Hilbert space
representations are integrable for all states; the \emph{Strong}
Local--Global Principle says the same with all states replaced by
all \emph{pure} states.  We took these names
from~\cite{Kaad-Lesch:Local_global}.  Earlier results of
Pierrot~\cite{Pierrot:Reguliers} show that the Strong Local--Global
Principle holds for any class of integrable representations that is
defined by certain types of conditions, such as the regularity and
self-adjointness of~\(\cl{\pi(a)}\) for certain \(a\in A\) with
\(a=a^*\).  For instance, this covers the integrable representations
of commutative \Star{}algebras and of universal enveloping algebras
of Lie algebras.

In all examples that we treat, the regularity of~\(\cl{\pi(a)}\)
for certain \(a\in A\)
is part of the definition of an integrable representation.  Other
elements of~\(A\)
may, however, act by irregular operators in some integrable
representations.  Thus affiliation and regularity are important to
study the integrable representations in concrete examples, but cannot
play a foundational role for the \emph{general} representation theory
of \Star{}algebras.

If~\(B\) is generated in the sense of
Woronowicz~\cite{Woronowicz:Cstar_generated} by some self-adjoint,
affiliated multipliers that belong to~\(A\), then it is a
\Cstar\nb-hull and the Strong Local--Global Principle holds (see
Theorem~\ref{the:Woronowicz_local-global}).  A counterexample shows
that this theorem breaks down if the generating affiliated
multipliers are not self-adjoint: both the Local--Global Principle
and compatibility with isometric intertwiners fail in the
counterexample.  So regularity without self-adjointness seems to be
too weak for many purposes.  The combination of regularity and
self-adjointness is an easier notion than regularity alone.  A
closed operator~\(T\) is regular and self-adjoint if and only if
\(T-\lambda\) is surjective for all \(\lambda\in\C\setminus\R\), if
and only if the Cayley transform of~\(T\) is unitary, if and only
if~\(T\) has a functional calculus homomorphism on~\(\Cont_0(\R)\).

\smallskip

Now we describe the Fell bundle in the Induction Theorem and, along
the way, the further condition besides compatibility with isometric
intertwiners that it needs.  Our input
data is a graded \Star{}algebra \(A=\bigoplus_{g\in G} A_g\) and a
\Cstar\nb-hull~\(B_e\) for~\(A_e\).  A
representation of~\(A\) is integrable if its restriction to~\(A_e\)
is integrable.  We seek a \Cstar\nb-hull for the integrable
representations of~\(A\).

As in~\cite{Savchuk-Schmudgen:Unbounded_induced}, we induce
representations from~\(A_e\) to~\(A\), and this requires a
positivity condition.  We call representations of~\(A_e\) that may
be induced to~\(A\) \emph{inducible}.  We describe a quotient
\Cstar\nb-algebra~\(B_e^+\) of~\(B_e\) that is a \Cstar\nb-hull for
the inducible, integrable representations of~\(A\).  It is the unit
fibre of our Fell bundle.

If a representation~\(\pi\) of~\(A\) is integrable, then its
restriction to~\(A_e\) is integrable and inducible.  Thus it
corresponds to a representation~\(\bar\pi_e^+\) of~\(B_e^+\).  The
identity correspondence on~\(B_e^+\) corresponds to a particular
(``universal'') inducible, integrable representation of~\(A_e\)
on~\(B_e^+\).  Its domain is a dense right ideal
\(\Hilms[B]_e^+\subseteq B_e^+\).  The operators
\(\pi(a)\bar\pi_e^+(b)\) on~\(\Hilm\) for \(a\in A_g\),
\(b\in\Hilms[B]_e^+\) are adjointable.  Their closed linear span is
the fibre~\(B_g^+\) of our Fell bundle at~\(g\) provided~\(\pi_e^+\)
is faithful.  The most difficult point is to prove \(B_e^+\cdot B_g^+
= B_g^+\) for all \(g\in G\); this easily implies \(B_g^+\cdot
B_h^+\subseteq B_{gh}^+\) and \((B_g^+)^* = B_{g^{-1}}^+\), so that
the subspaces \(B_g^+\subseteq \Bound(\Hilm)\) form a Fell bundle.

To prove \(B_e^+\cdot B_g^+ = B_g^+\), we need compatibility with
isometric intertwiners and that induction maps inducible, integrable
representations of~\(A_e\) to \emph{integrable} representations
of~\(A\).  Two counterexamples show that both
assumptions are necessary for the Induction Theorem.

Fell bundles are noncommutative partial dynamical systems.  More
precisely, a Fell bundle~\((B_g^+)_{g\in G}\)
over~\(G\)
is equivalent to an action of~\(G\)
on~\(B_e^+\)
by partial Morita--Rieffel equivalences; this is made precise
in~\cite{Buss-Meyer:Actions_groupoids}.  In the examples in
\cites{Savchuk-Schmudgen:Unbounded_induced, Dowerk-Savchuk:Induced},
the group~\(G\)
is almost always~\(\Z\);
the \Cstar\nb-algebras \(B_e\)
and hence~\(B_e^+\)
are commutative; and the resulting Fell bundle comes from a partial
action of~\(G\)
on the spectrum of~\(B_e^+\).
In these examples, the section \Cstar\nb-algebra is a partial crossed
product.  This may also be viewed as the groupoid \Cstar\nb-algebra of
the transformation groupoid for the partial action of~\(G\)
on the spectrum of~\(B_e^+\).
We show that the \Cstar\nb-hull~\(B\)
for the integrable representations of~\(A\)
is a \emph{twisted} groupoid \Cstar\nb-algebra of this transformation
groupoid whenever~\(B_e\)
is commutative.  We give some criteria when the twist is absent, and
examples where the twist occurs.  One way to insert such twists is by
Rieffel deformation, using a \(2\)\nb-cocycle
on the group~\(G\).
We show that Rieffel deformation is compatible with the construction
of \Cstar\nb-hulls.

We describe commutative and noncommutative \Cstar\nb-hulls for the
polynomial algebra~\(\C[x]\) in §\ref{sec:polynomials1}
and~§\ref{sec:polynomials2}; the noncommutative \Cstar\nb-hulls
for~\(\C[x]\) make
very good counterexamples.  We classify and study commutative
\Cstar\nb-hulls in~§\ref{sec:commutative_hulls}.  Many results about
them generalise easily to \emph{locally bounded} representations.
Roughly speaking, these are representations where the vectors on
which the representation acts by bounded operators form a core.  The
only \Star{}algebras for which we treat locally bounded
representations in some detail are the commutative ones.

Through the Induction Theorem, the representation theory of
commutative \Star{}algebras is important even for noncommutative
algebras because they may admit a grading by some group with
commutative unit fibre.  Many examples of this are treated in detail
in \cites{Savchuk-Schmudgen:Unbounded_induced,
  Dowerk-Savchuk:Induced}.  We discuss untwisted and twisted Weyl
algebras in finitely and infinitely many generators
in~§\ref{sec:Weyl_twisted}.  The twists involved are Rieffel
deformations.  Since these examples have commutative unit fibres,
the resulting \Cstar\nb-hulls are twisted groupoid
\Cstar\nb-algebras.  As it turns out, all twists of the relevant
groupoids are trivial, so that the twists do not change the
representation theory of the Weyl algebras up to equivalence.

I am grateful to Yuriy Savchuk for several discussions, which led me
to pursue this project and eliminated mistakes from early versions of
this article.  And I am grateful to the referee as well for several
useful suggestions.

\section{Representations by unbounded operators on Hilbert modules}
\label{sec:rep_Hilbert_module}

Let~\(A\) be a unital \Star{}algebra, \(D\) a \Cstar\nb-algebra,
and~\(\Hilm\) a Hilbert \(D\)\nb-module.  Our convention is that
inner products on Hilbert spaces and Hilbert modules are linear in the
second and conjugate-linear in the first variable.

\begin{definition}
  \label{def:rep_Hilbert_module}
  A \emph{representation} of~\(A\) on~\(\Hilm\) is a
  pair~\((\Hilms,\pi)\), where \(\Hilms\subseteq\Hilm\) is a dense
  \(D\)\nb-submodule and \(\pi\colon A\to\Endo_D(\Hilms)\) is a unital
  algebra homomorphism to the algebra of \(D\)\nb-module endomorphisms
  of~\(\Hilms\), such that
  \[
  \braket{\pi(a)\xi}{\eta}_D = \braket{\xi}{\pi(a^*)\eta}_D
  \qquad
  \text{for all }a\in A,\ \xi,\eta\in\Hilms.
  \]

  We call~\(\Hilms\) the \emph{domain} of the representation.  We
  may drop~\(\pi\) from our notation by saying that~\(\Hilms\) is an
  \(A,D\)-bimodule with the right module structure inherited
  from~\(\Hilm\), or we may drop~\(\Hilms\) because it is the common
  domain of the partial linear maps~\(\pi(a)\) on~\(\Hilm\) for all
  \(a\in A\).

  We equip~\(\Hilms\) with the \emph{graph topology}, which is
  generated by the \emph{graph norms}
  \[
  \norm{\xi}_a \defeq \norm{(\xi,\pi(a)\xi)}
  \defeq \norm{\braket{\xi}{\xi} + \braket{\pi(a)\xi}{\pi(a)\xi}}^{\nicefrac12}
  = \norm{\braket{\xi}{\pi(1+a^*a)\xi}}^{\nicefrac12}
  \]
  for \(a\in A\).  The representation is \emph{closed} if~\(\Hilms\)
  is complete in this topology.  A \emph{core} for~\((\Hilms,\pi)\) is
  an \(A,D\)-subbimodule of~\(\Hilms\) that is dense in~\(\Hilms\) in
  the graph topology.
\end{definition}

Definition~\ref{def:rep_Hilbert_module} for \(D=\C\) is the usual
definition of a representation of a \Star{}algebra on a Hilbert
space by unbounded operators.  This situation has been
studied extensively (see, for instance,
\cite{Schmudgen:Unbounded_book}).  For \(\Hilm=D\) with the
canonical Hilbert \(D\)\nb-module structure, we get representations
of~\(A\) by \emph{densely defined unbounded multipliers}.  The
domain of such a representation is a dense right ideal
\(\Hilms[D]\subseteq D\).  This situation is a special case of the
``compatible pairs'' defined by
Schmüdgen~\cite{Schmuedgen:Well-behaved}.

Given two norms \(p,q\), we write \(p \preceq q\) if there is a
scalar \(c>0\) with \(p \le c q\).

\begin{lemma}
  \label{lem:graph_norms_directed}
  The set of graph norms partially ordered by~\(\preceq\) is
  directed: for all \(a_1,\dotsc,a_n\in A\) there are \(b\in A\) and
  \(c\in\R_{>0}\) so that \(\norm{\xi}_{a_i} \le c\norm{\xi}_b\) for
  any representation \((\Hilms,\pi)\), any \(\xi\in\Hilms\), and
  \(i=1,\dotsc,n\).
\end{lemma}

\begin{proof}
  Let \(b = \sum_{j=1}^n a_j^* a_j\).  The following computation
  implies \(\norm{\xi}_{a_i} \le \nicefrac54 \norm{\xi}_b\):
  \begin{align*}
    0&\le
    \braket{\xi}{\pi(1+a_i^* a_i)\xi}
    \\&\le \braket{\xi}{\pi(1+a_i^* a_i)\xi}
    + \sum_{i\neq j} \braket{\pi(a_j)\xi}{\pi(a_j)\xi}
    + \braket{\pi(b-\nicefrac12)\xi}{\pi(b-\nicefrac12) \xi}
    \\&= \braket{\xi}{\pi(1+b + (b-\nicefrac12)^2) \xi}
    = \braket{\xi}{\pi(\nicefrac54 + b^2) \xi}
    \le \nicefrac54 \braket{\xi}{\pi(1 + b^* b) \xi}.\qedhere
  \end{align*}
\end{proof}

\begin{definition}[\cite{Pal:Regular},
  \cite{Lance:Hilbert_modules}*{Chapter 9}]
  \label{def:regular}
  A densely defined operator~\(t\) on a Hilbert module~\(\Hilm\) is
  \emph{semiregular} if its adjoint is also densely defined; it is
  \emph{regular} if it is closed, semiregular and \(1+t^*t\) has dense
  range.  An \emph{affiliated multiplier} of a
    \Cstar\nb-algebra~\(D\) is a regular operator on~\(D\) viewed as
  a Hilbert \(D\)\nb-module.
\end{definition}

The closability assumption in \cite{Pal:Regular}*{Definition 2.1.(ii)}
is redundant by \cite{Kaad-Lesch:Local_global}*{Lemma 2.1}.
Regularity was introduced by Baaj and
Julg~\cite{Baaj-Julg:non_bornes}, affiliation by
Woronowicz~\cite{Woronowicz:Unbounded_affiliated}.

\begin{remark}
  \label{rem:affiliated}
  Let~\((\Hilms,\pi)\) be a representation of~\(A\) on~\(\Hilm\) and
  let \(a\in A\).  The operator~\(\pi(a)\) is automatically
  semiregular because~\(\pi(a)^*\) is defined on~\(\Hilms\).  The
  closure~\(\overline{\pi(a)}\) of~\(\pi(a)\) need not be regular.
  The regularity of~\(\cl{\pi(a)}\) for \emph{some} \(a\in A\) is
  often assumed in the definition of \emph{integrable}
  representations.  For non-commutative~\(A\), we should expect
  that~\(\overline{\pi(a)}\) is irregular for some \(a\in A\) even
  if~\(\pi\) is integrable.  For instance, a remark after Corollaire
  1.27 in~\cite{Pierrot:Reguliers} says that this happens for
  certain symmetric elements in the universal enveloping
  algebra~\(U(\Lg)\) for a simply connected Lie group~\(G\): they
  act by irregular operators in certain representations that
  integrate to unitary representations of~\(G\).
\end{remark}

The usual norm on~\(\Hilm\) is the graph norm for \(0\in A\).  Hence
the inclusion map \(\Hilms\injto\Hilm\) is continuous for the graph
topology on~\(\Hilms\) and extends continuously to the
completion~\(\cl{\Hilms}\) of~\(\Hilms\) in the graph topology.

\begin{proposition}
  \label{pro:closure_rep}
  The canonical map \(\cl{\Hilms}\to\Hilm\) is injective, and its
  image is
  \begin{equation}
    \label{eq:domain_closure}
    \cl{\Hilms} = \bigcap_{a\in A} \dom \cl{\pi(a)}.
  \end{equation}
  Thus~\((\Hilms,\pi)\) is closed if and only if \(\Hilms =
  \bigcap_{a\in A} \dom \cl{\pi(a)}\).  Each \(\pi(a)\) extends
  uniquely to a continuous operator~\(\cl{\pi(a)}\)
  on~\(\cl{\Hilms}\).  This defines a closed
  representation~\((\cl{\Hilms},\cl{\pi})\) of~\(A\), called the
  \emph{closure} of~\((\Hilms,\pi)\).
\end{proposition}

\begin{proof}
  The operator~\(\pi(a)\) for \(a\in A\) is semiregular and hence
  closable by \cite{Kaad-Lesch:Local_global}*{Lemma 2.1}.
  Equivalently, the canonical map from the completion of~\(\Hilms\)
  in the graph norm for~\(a\) to~\(\Hilm\) is injective.  Its image
  is \(\dom \cl{\pi(a)}\), the domain of the closure of~\(\pi(a)\).
  The graph norms for \(a\in A\) form a directed set that defines
  the graph topology on~\(\Hilms\).  So the completion of~\(\Hilms\)
  in the graph topology is the projective limit of the graph norm
  completions for \(a\in A\).  Since each of these graph norm
  completions embeds into~\(\Hilm\), the projective limit in
  question is just an intersection in~\(\Hilm\),
  giving~\eqref{eq:domain_closure}.  For Hilbert space representations,
  this is \cite{Schmudgen:Unbounded_book}*{Proposition 2.2.12}.

  The operators \(\pi(a)\in\Endo_D(\Hilms)\) for \(a\in A\) are
  continuous in the graph topology.  Thus they extend uniquely to
  continuous linear operators \(\cl{\pi}(a)\in\Endo_D(\cl{\Hilms})\).
  These are again \(D\)\nb-linear and the map~\(\cl{\pi}\) is linear
  and multiplicative because extending operators to a completion is
  additive and functorial.  The set of
  \((\xi,\eta)\in\cl{\Hilms}\times\cl{\Hilms}\) with
  \(\braket{\xi}{\cl{\pi}(a)\eta} =
  \braket{\cl{\pi}(a^*)\xi}{\eta}\) for all \(a\in A\) is closed in
  the graph topology and contains \(\Hilms\times\Hilms\), which is
  dense in \(\cl{\Hilms}\times\cl{\Hilms}\).  Hence this equation
  holds for all \(\xi,\eta\in\cl{\Hilms}\).
  So~\((\cl{\Hilms},\cl{\pi})\) is a representation of~\(A\)
  on~\(\Hilm\).  The graph topology on~\(\cl{\Hilms}\)
  for~\(\cl{\pi}\) extends the graph topology on~\(\Hilms\)
  for~\(\pi\) and hence is complete.  So \((\cl{\Hilms},\cl{\pi})\)
  is a closed representation.
\end{proof}

We shall need a generalisation of~\eqref{eq:domain_closure} that
replaces~\(A\) by a sufficiently large subset.

\begin{definition}
  \label{def:strong_generating_set}
  A subset \(S\subseteq A\)
  is called a \emph{strong generating} set if it generates~\(A\)
  as an algebra and the graph norms for \(a\in S\)
  generate the graph topology in any representation.  That is, for any
  representation on a Hilbert module, any vector~\(\xi\)
  in its domain and any \(a\in A\),
  there are \(c\ge1\)
  in~\(\R\)
  and \(b_1,\dotsc,b_n\in S\)
  with \(\norm{\xi}_a \le c \sum_{i=1}^n \norm{\xi}_{b_i}\).
\end{definition}

An estimate \(\norm{\xi}_a \le c \sum_{i=1}^n \norm{\xi}_{b_i}\)
is usually shown by finding \(d_1,\dotsc,d_m\in A\)
with
\(a^* a + \sum_{j=1}^m d_j^* d_j = c\cdot \sum_{i=1}^n b_i^* b_i\),
compare the proof of Lemma~\ref{lem:graph_norms_directed}.

\begin{example}
  \label{exa:strong_generating}
  Let \(A_h \defeq \{a\in A\mid a=a^*\}\) be the set of
  \emph{symmetric} elements.  Call an element of~\(A\)
  \emph{positive} if it is a sum of elements of the form~\(a^* a\).
  The positive elements and, \emph{a fortiori}, the symmetric
  elements form strong generating sets for~\(A\).  Any element is of
  the form \(a_1+\ima a_2\) with \(a_1,a_2\in A_h\), and
  \[
  a = \left(\frac{a+1}{2}\right)^2 - \left(\frac{a-1}{2}\right)^2
  \]
  for \(a\in A_h\).  Thus the positive elements generate~\(A\) as an
  algebra.  The graph norms for positive elements generate the graph
  topology by the proof of Lemma~\ref{lem:graph_norms_directed}.
\end{example}

\begin{proposition}
  \label{pro:equality_if_closure_equal}
  Let \(S\subseteq A\) be a strong generating set.  Two closed
  representations \((\Hilms_1,\pi_1)\) and~\((\Hilms_2,\pi_2)\)
  of~\(A\) on the same Hilbert module~\(\Hilm\) are equal if and
  only if \(\cl{\pi_1(a)} = \cl{\pi_2(a)}\) for all \(a\in S\).
\end{proposition}

\begin{proof}
  One direction is trivial.  To prove the non-trivial direction,
  assume \(\cl{\pi_1(a)} = \cl{\pi_2(a)}\) for all \(a\in S\).
  Let \((\Hilms,\pi)= (\Hilms_i,\pi_i)\) for \(i=1,2\).  The
  completion of~\(\Hilms\) for the graph norm of~\(a\) is \(\dom
  \cl{\pi(a)}\), compare the proof of Proposition~\ref{pro:closure_rep}.
  Hence the completion of~\(\Hilms\) in the sum of graph norms
  \(\sum_{k=1}^n \norm{\xi}_{b_k}\) for \(b_1,\dotsc,b_n\in S\) is
  \(\bigcap_{k=1}^n \dom \cl{\pi(b_k)}\).  These sums of graph norms
  for \(b_1,\dotsc,b_n\in S\) form a directed set that generates the
  graph topology on~\(\Hilms\).  Hence
  \begin{equation}
    \label{eq:domain_strong_generating_set}
    \cl{\Hilms} = \bigcap_{a\in S} \dom \cl{\pi(a)},
  \end{equation}
  compare the proof of~\eqref{eq:domain_closure}.  So
  \(\Hilms_1=\Hilms_2\).  Moreover, \(\pi_1(a) =
  \cl{\pi_1(a)}|_{\Hilms_1} = \cl{\pi_2(a)}|_{\Hilms_2} = \pi_2(a)\)
  for all \(a\in S\).  Since~\(S\) generates~\(A\) as an algebra and
  \(\pi_i(A)\Hilms_i\subseteq \Hilms_i\), this implies
  \(\pi_1=\pi_2\).
\end{proof}

Proposition~\ref{pro:equality_if_closure_equal} may fail for
generating sets that are not strong, see
Example~\ref{exa:x_not_strong_generator}.

\begin{corollary}
  \label{cor:bounded_rep}
  Let~\(S\) be a strong generating set of~\(A\) and
  let~\((\Hilms,\pi)\) be a closed
  representation of~\(A\) with \(\dom \cl{\pi(a)} = \Hilm\) for each
  \(a\in S\).  Then \(\Hilms=\Hilm\) and~\(\pi\) is a
  \Star{}homomorphism to the \Cstar\nb-algebra \(\Bound(\Hilm)\) of
  adjointable operators on~\(\Hilm\).
\end{corollary}

\begin{proof}
  Equation~\eqref{eq:domain_strong_generating_set} gives
  \(\Hilms=\Hilm\).  Since \(\pi(a^*)\subseteq \pi(a)^*\)
  and~\(\pi(a^*)\) is defined everywhere, it is adjoint
  to~\(\pi(a)\).  So \(\pi(a)\in\Bound(\Hilm)\) and~\(\pi\) is a
  \Star{}homomorphism to~\(\Bound(\Hilm)\).
\end{proof}

\begin{lemma}
  \label{lem:Cstar-rep_bounded}
  Let~\(A\) be a unital \Cstar\nb-algebra.  Any closed
  representation of~\(A\) on~\(\Hilm\) has domain \(\Hilms=\Hilm\)
  and is a unital \Star{}homomorphism to~\(\Bound(\Hilm)\).
\end{lemma}

\begin{proof}
  Let \(a\in A\).  There are a positive scalar \(C>0\) and \(b\in A\)
  with \(a^*a+b^*b = C\); say, take \(C=\norm{a}^2\) and
  \(b=\sqrt{C-a^*a}\).  Then
  \[
  \braket{\pi(a)\xi}{\pi(a)\xi}
  \le \braket{\pi(a)\xi}{\pi(a)\xi} + \braket{\pi(b)\xi}{\pi(b)\xi}
  = \braket{\xi}{\pi(a^*a+b^*b)\xi}
  = C\braket{\xi}{\xi}
  \]
  for all \(\xi\in\Hilms\).  Thus the graph topology on~\(\Hilms\)
  is equivalent to the norm topology on~\(\Hilm\).  Hence
  \(\Hilm=\Hilms\) for any closed representation.
\end{proof}

An \emph{isometry} \(I\colon \Hilm_1\injto\Hilm_2\)
between two Hilbert \(D\)\nb-modules \(\Hilm_1\)
and~\(\Hilm_2\)
is a right \(D\)\nb-module
map with \(\braket{I\xi_1}{I\xi_2} = \braket{\xi_1}{\xi_2}\)
for all \(\xi_1,\xi_2\in\Hilm_1\).

\begin{definition}
  \label{def:isometric_intertwiner}
  Let \((\Hilms_1,\pi_1)\)
  and \((\Hilms_2,\pi_2)\)
  be representations on Hilbert \(D\)\nb-modules
  \(\Hilm_1\)
  and~\(\Hilm_2\),
  respectively.  An \emph{isometric intertwiner} between them is an
  isometry \(I\colon \Hilm_1\injto\Hilm_2\)
  with \(I(\Hilms_1)\subseteq \Hilms_2\)
  and \(I\circ\pi_1(a) (\xi) = \pi_2(a) \circ I (\xi)\)
  for all \(a\in A\),
  \(\xi\in\Hilms_1\);
  equivalently, \(I\circ\pi_1(a)\subseteq \pi_2(a)\circ I\)
  for all \(a\in A\),
  that is, the graph of~\(\pi_2(a)\circ I\)
  contains the graph of~\(I\circ \pi_1(a)\).
  We neither ask~\(I\)
  to be adjointable nor \(I(\Hilms_1)=\Hilms_2\).
  Let \(\Rep(A,D)\)
  be the category with closed representations of~\(A\)
  on Hilbert \(D\)\nb-modules
  as objects, isometric intertwiners as arrows, and the usual
  composition.  The unit arrow on~\((\Hilms,\pi)\)
  is the identity operator on~\(\Hilm\).
\end{definition}

\begin{lemma}
  \label{lem:closure_functorial}
  Let \((\Hilms_1,\pi_1)\) and \((\Hilms_2,\pi_2)\) be
  representations on Hilbert \(D\)\nb-modules \(\Hilm_1\)
  and~\(\Hilm_2\), respectively, and let \(I\colon
  \Hilm_1\injto\Hilm_2\) be an isometric intertwiner.  Then~\(I\) is
  also an intertwiner between the closures of \((\Hilms_1,\pi_1)\)
  and \((\Hilms_2,\pi_2)\).
\end{lemma}

\begin{proof}
  Since~\(I\) intertwines the representations, it is continuous for
  the graph topologies on \(\Hilms_1\) and~\(\Hilms_2\).
  Hence~\(I\) maps the domain of the closure~\(\cl{\pi_1}\) into the
  domain of~\(\cl{\pi_2}\).  This extension is still an intertwiner
  because it is an intertwiner on a dense subspace.
\end{proof}

\begin{proposition}
  \label{pro:intertwiner_strong_generators}
  Let \((\Hilms_1,\pi_1)\) and \((\Hilms_2,\pi_2)\) be closed
  representations of~\(A\) on Hilbert \(D\)\nb-modules \(\Hilm_1\)
  and~\(\Hilm_2\), respectively.  Let \(S\subseteq A\) be a strong
  generating set.  An isometry \(I\colon \Hilm_1\injto \Hilm_2\) is
  an intertwiner from \((\Hilms_1,\pi_1)\) to \((\Hilms_2,\pi_2)\)
  if and only if \(I\circ \cl{\pi_1(a)} \subseteq \cl{\pi_2(a)}\circ
  I\) for all \(a\in S\).
\end{proposition}

\begin{proof}
  First let~\(I\) satisfy \(I\circ\cl{\pi_1(a)} \subseteq
  \cl{\pi_2(a)}\circ I\) for all \(a\in S\).  Then~\(I\) maps the
  domain of~\(\cl{\pi_1(a)}\) into the domain of~\(\cl{\pi_2(a)}\)
  for each \(a\in S\).  Now~\eqref{eq:domain_strong_generating_set}
  implies \(I(\Hilms_1)\subseteq \Hilms_2\).  Since \(\pi_i(a) =
  \cl{\pi_i(a)}|_{\Hilms_i}\), we get \(I(\pi_1(a)(\xi)) =
  \pi_2(a)(I(\xi))\) for all \(a\in S\), \(\xi\in\Hilms_1\).
  Since~\(S\) generates~\(A\) as an algebra and
  \(\pi_i(A)\Hilms_i\subseteq\Hilms_i\), this implies
  \(I\circ\pi_1(a) = \pi_2(a)\circ I\) for all \(a\in A\), that is,
  \(I\) is an intertwiner.

  Conversely, assume that~\(I\) is an intertwiner from
  \((\Hilms_1,\pi_1)\) to~\((\Hilms_2,\pi_2)\).  Equivalently,
  \(I\circ \pi_1(a) \subseteq \pi_2(a)\circ I\) for all \(a\in A\).
  We have \(I\circ \cl{\pi_1(a)} = \cl{I\circ \pi_1(a)}\) because~\(I\)
  is an isometry, and \(\cl{\pi_2(a)\circ I} \subseteq
  \cl{\pi_2(a)}\circ I\).  Thus \(I\circ \cl{\pi_1(a)} \subseteq
  \cl{\pi_2(a)}\circ I\) for all \(a\in A\).
\end{proof}

Now we relate the categories \(\Rep(A,D)\) for different
\Cstar\nb-algebras~\(D\).

\begin{definition}
  \label{def:Cstar-correspondence}
  Let \(D_1\) and~\(D_2\) be two \Cstar\nb-algebras.  A
  \emph{\Cstar\nb-correspondence} from~\(D_1\) to~\(D_2\) is a
  Hilbert \(D_2\)\nb-module with a representation of~\(D_1\) by
  adjointable operators (representations of \Cstar\nb-algebras are
  tacitly assumed nondegenerate).  An \emph{isometric intertwiner}
  between two correspondences from~\(D_1\) to~\(D_2\) is an
  isometric map on the underlying Hilbert \(D_2\)\nb-modules that
  intertwines the left \(D_1\)\nb-actions.  Let \(\Rep(D_1,D_2)\)
  denote the category of correspondences from~\(D_1\) to~\(D_2\)
  with isometric intertwiners as arrows and the usual composition.
\end{definition}

By Lemma~\ref{lem:Cstar-rep_bounded}, our two definitions of
\(\Rep(A,D)\) for unital \Star{}algebras and \Cstar\nb-algebras
coincide if~\(A\) is a unital \Cstar\nb-algebra.  So our notation is
not ambiguous.  There is no need to define representations of a
non-unital \Star{}algebra~\(A\) because we may adjoin a unit
formally.  A representation of~\(A\) extends uniquely to a
representation of the unitisation~\(\tilde{A}\).  Thus the
nondegenerate representations of~\(A\) are contained in
\(\Rep(\tilde{A})\).  To get rid of degenerate representations, we
may require nondegeneracy on~\(A\) when defining the integrable
representations of~\(\tilde{A}\), compare
Example~\ref{exa:nondegeneracy_condition}.

Let~\(\Hilm\) be a Hilbert \(D_1\)\nb-module and~\(\Hilm[F]\) a
correspondence from~\(D_1\) to~\(D_2\).  The interior tensor product
\(\Hilm\otimes_{D_1} \Hilm[F]\) is the (Hausdorff) completion of the
algebraic tensor product \(\Hilm\odot \Hilm[F]\) to a Hilbert
\(D_2\)\nb-module, using the inner product
\begin{equation}
  \label{eq:interior_tensor}
  \braket{\xi_1\otimes\eta_1}{\xi_2\otimes\eta_2}
  = \braket{\eta_1}{\braket{\xi_1}{\xi_2}_{D_1}\cdot\eta_2}_{D_2},
\end{equation}
see the discussion around \cite{Lance:Hilbert_modules}*{Proposition
  4.5} for more details.  We may use the balanced tensor product
\(\Hilm\odot_{D_1} \Hilm[F]\) instead of \(\Hilm\odot \Hilm[F]\)
because the inner product~\eqref{eq:interior_tensor} descends to this
quotient.  If we want to emphasise the left action \(\varphi\colon
D_1\to\Bound(\Hilm[F])\) in the
\Cstar\nb-correspondence~\(\Hilm[F]\), we write \(\Hilm\otimes_\varphi
\Hilm[F]\) for \(\Hilm\otimes_{D_1} \Hilm[F]\).

In addition, let~\((\Hilms,\pi)\) be a closed representation of~\(A\)
on~\(\Hilm\).  We are going to build a closed representation
\((\Hilms\otimes_{D_1}\Hilm[F], \pi\otimes_{D_1}1)\) of~\(A\) on
\(\Hilm\otimes_{D_1}\Hilm[F]\).  First let \(X\subseteq
\Hilm\otimes_{D_1}\Hilm[F]\) be the image of \(\Hilms\odot \Hilm[F]\)
or \(\Hilms\odot_{D_1} \Hilm[F]\) under the canonical map to
\(\Hilm\otimes_{D_1}\Hilm[F]\).

\begin{lemma}
  \label{lem:tensor_rep_with_corr}
  For \(a\in A\), there is a unique linear operator
  \(\pi(a)\otimes1\colon X\to X\) with \((\pi(a)\otimes1)
  (\xi\otimes\eta) = \pi(a)(\xi)\otimes \eta\) for all
  \(\xi\in\Hilms\), \(\eta\in\Hilm[F]\).  The map \(a\mapsto
  \pi(a)\otimes 1\) is a representation of~\(A\) with domain~\(X\).
\end{lemma}

\begin{proof}
  Write \(\omega,\zeta\in X\) as images of elements of \(\Hilms\odot
  \Hilm[F]\):
  \[
  \omega = \sum_{i=1}^n \xi_i\otimes\eta_i,\qquad
  \zeta = \sum_{j=1}^m \alpha_j\otimes\beta_j
  \]
  with \(\xi_i,\alpha_j\in\Hilms\), \(\eta_i,\beta_j\in\Hilm[F]\) for
  \(1\le i\le n\), \(1\le j\le m\).  Then
  \begin{equation}
    \label{eq:tensor_rep_well-def}
    \left< \zeta,
      \sum_{i=1}^n\pi(a)\xi_i\otimes\eta_i \right>
    = \left< \sum_{j=1}^m \pi(a^*)\alpha_j\otimes\beta_j,
      \omega \right>.
  \end{equation}
  An element \(\omega'\in\Hilm\otimes_{D_1}\Hilm[F]\) is determined
  uniquely by its inner products \(\braket{\zeta}{\omega'}=0\) for all
  \(\zeta\in X\) because~\(X\) is dense in
  \(\Hilm\otimes_{D_1}\Hilm[F]\).  The right hand side
  in~\eqref{eq:tensor_rep_well-def} does not depend on how we
  decomposed~\(\omega\).  Hence \((\pi(a)\otimes 1)\omega \defeq
  \sum_{i=1}^n\pi(a)\xi_i\otimes\eta_i\) well-defines an operator
  \(\pi(a)\otimes 1\colon X\to X\).  This is a right \(D_2\)\nb-module
  map, and \(a\mapsto \pi(a)\otimes 1\) is linear and multiplicative
  because~\(\pi\) is.  Equation~\eqref{eq:tensor_rep_well-def} says
  that \(\braket{\zeta}{(\pi(a)\otimes 1)\omega} =
  \braket{(\pi(a^*)\otimes 1)\zeta}{\omega}\) for all
  \(\omega,\zeta\in X\).  Thus \(\pi\otimes 1\) is a representation.
\end{proof}

\begin{definition}
  Let \((\Hilms\otimes_{D_1}\Hilm[F], \pi\otimes_{D_1}1)\) be the
  closure of the representation on~\(\Hilm\otimes_{D_1}\Hilm[F]\)
  defined in Lemma~\ref{lem:tensor_rep_with_corr}.
\end{definition}

\begin{lemma}
  \label{lem:rep_tensor_corr_functor}
  Let \(I\colon \Hilm_1\injto\Hilm_2\) be an isometric intertwiner
  between two representations \((\Hilms_1,\pi_1)\) and
  \((\Hilms_2,\pi_2)\), and let \(J\colon \Hilm[F]_1\injto \Hilm[F]_2\)
  be an isometric intertwiner of \Cstar\nb-correspondences.  Then
  \(I\otimes_{D_1} J\colon \Hilm_1\otimes_{D_1}\Hilm[F]_1 \injto
  \Hilm_2\otimes_{D_1}\Hilm[F]_2\) is an isometric intertwiner between
  \((\Hilms_1\otimes_{D_1}\Hilm[F]_1, \pi_1\otimes1)\) and
  \((\Hilms_2\otimes_{D_1}\Hilm[F]_2, \pi_2\otimes1)\).
\end{lemma}

\begin{proof}
  The isometry \(I\otimes_{D_1} J\) maps the image~\(X_1\) of
  \(\Hilms_1\odot\Hilm[F]_1\) to the image~\(X_2\) of
  \(\Hilms_2\odot\Hilm[F]_2\) and intertwines the operators
  \(\pi_1(a)\otimes 1\) on~\(X_1\) and \(\pi_2(a)\otimes 1\)
  on~\(X_2\) for all \(a\in A\).  That is, it intertwines the
  representations defined in Lemma~\ref{lem:tensor_rep_with_corr}.
  It also intertwines their closures by
  Lemma~\ref{lem:closure_functorial}.
\end{proof}

The lemma gives a bifunctor
\begin{equation}
  \label{eq:interior_tensor_bifunctor}
  \otimes_{D_1}\colon \Rep(A,D_1)\times \Rep(D_1,D_2) \to
  \Rep(A,D_2).
\end{equation}
The corresponding bifunctor
\[
\otimes_{D_1}\colon \Rep(B,D_1)\times \Rep(D_1,D_2) \to \Rep(B,D_2)
\]
for a \Cstar\nb-algebra~\(B\) is the usual composition of
\Cstar\nb-correspondences.  This composition is associative up to
canonical unitaries
\begin{equation}
  \label{eq:tensor_associative}
  \Hilm \otimes_{D_1} (\Hilm[F] \otimes_{D_2} \Hilm[G]) \congto
  (\Hilm \otimes_{D_1} \Hilm[F]) \otimes_{D_2} \Hilm[G],\qquad
  \xi \otimes (\eta\otimes \zeta) \mapsto
  (\xi \otimes \eta)\otimes \zeta,
\end{equation}
for all triples of composable \Cstar\nb-correspondences.

\begin{lemma}
  \label{lem:tensor_associative}
  If~\(\Hilm\) carries a representation~\((\Hilms,\pi)\) of a
  \Star{}algebra~\(A\), then the unitary
  in~\eqref{eq:tensor_associative} is an intertwiner \((\Hilms,\pi)
  \otimes_{D_1} (\Hilm[F]\otimes_{D_2} \Hilm[G]) \congto
  \bigl((\Hilms,\pi) \otimes_{D_1} \Hilm[F]\bigr) \otimes_{D_2}
  \Hilm[G]\).
\end{lemma}

\begin{proof}
  The bilinear map from~\(\Hilms \times \Hilm[F]\) to~\(\Hilms
  \otimes_{D_1} \Hilm[F]\) is separately continuous with respect to
  the graph topologies on \(\Hilms\) and \(\Hilms \otimes_{D_1}
  \Hilm[F]\) and the norm topology on~\(\Hilm[F]\).  Since the image
  of \(\Hilm[F]\odot \Hilm[G]\) in the Hilbert module
  \(\Hilm[F]\otimes_{D_2} \Hilm[G]\) is dense in the norm topology,
  the image of \(\Hilms\odot \Hilm[F]\odot \Hilm[G]\) in
  \(\Hilm \otimes_{D_1} (\Hilm[F] \otimes_{D_2} \Hilm[G])\) is a
  core for the representation \((\Hilms,\pi) \otimes_{D_1}
  (\Hilm[F]\otimes_{D_2} \Hilm[G])\).  Since the image of
  \(\Hilms\odot \Hilm[F]\) in \(\Hilms \otimes_{D_1}
  \Hilm[F]\) is dense in the graph topology, the image of
  \(\Hilms\odot \Hilm[F]\odot \Hilm[G]\) in \((\Hilm
  \otimes_{D_1} \Hilm[F]) \otimes_{D_2} \Hilm[G]\) is a core for the
  representation \(\bigl((\Hilms,\pi) \otimes_{D_1} \Hilm[F]\bigr)
  \otimes_{D_2} \Hilm[G]\).  The unitary
  in~\eqref{eq:tensor_associative} intertwines between these
  cores.  Hence it also intertwines between the resulting closed
  representations by Lemma~\ref{lem:closure_functorial}.
\end{proof}

\begin{definition}
  \label{def:star-intertwiner}
  Let \((\Hilms_1,\pi_1)\) and~\((\Hilms_2,\pi_2)\) be two
  representations of~\(A\) on Hilbert \(D\)\nb-modules \(\Hilm_1\)
  and~\(\Hilm_2\).  An adjointable operator \(x\colon
  \Hilm_1\to\Hilm_2\) is an \emph{intertwiner} if
  \(x(\Hilms_1)\subseteq \Hilms_2\) and \(x\pi_1(a)\xi =
  \pi_2(a)x\xi\) for all \(a\in A\), \(\xi\in\Hilms_1\).  It is
  a \emph{\Star{}intertwiner} if both \(x\) and~\(x^*\) are
  intertwiners.
  \label{def:adjointable_intertwiner}
\end{definition}

Any adjointable intertwiner between two representations of a
\Cstar\nb-algebra~\(B\) is a \Star{}intertwiner.  In contrast, for
a general \Star{}algebra, even the adjoint of a unitary
intertwiner~\(u\) fails to be an intertwiner if
\(u(\Hilms_1)\subsetneq \Hilms_2\).

\begin{example}
  \label{exa:Friedrichs_extension}
  Let~\(t\) be a positive symmetric operator on a Hilbert
  space~\(\Hilm[H]\).  Assume that \(\bigcap_{n\in\N} \dom t^n\) is
  dense in~\(\Hilm[H]\), so that~\(t\) generates a
  representation~\(\pi\) of the polynomial algebra~\(\C[x]\)
  on~\(\Hilm[H]\).  The Friedrichs extension of~\(t\) is a positive
  self-adjoint operator~\(t'\) on~\(\Hilm[H]\).  It generates
  another representation~\(\pi'\) of~\(\C[x]\) on~\(\Hilm[H]\).  The
  identity map on~\(\Hilm[H]\) is a unitary intertwiner
  \(\pi\injto\pi'\).  It is not a \Star{}intertwiner unless
  \(t=t'\).
\end{example}

The following proposition characterises when an adjointable isometry
\(I\colon \Hilm_1\injto \Hilm\) between two representations
on Hilbert \(D\)\nb-modules is a \Star{}intertwiner.  Since
\(\Hilm\cong \Hilm_1 \oplus \Hilm_1^\bot\) if~\(I\) is adjointable, we
may as well assume that~\(I\) is the inclusion of a direct summand.

\begin{proposition}
  \label{pro:isometry_star-intertwiner}
  Let \(\Hilm_1\) and~\(\Hilm_2\) be Hilbert modules over a
  \Cstar\nb-algebra~\(D\) and let \((\Hilms_1,\pi_1)\) and
  \((\Hilms,\pi)\) be representations of~\(A\) on \(\Hilm_1\)
  and~\(\Hilm_1\oplus\Hilm_2\), respectively.  The following are
  equivalent:
  \begin{enumerate}
  \item \label{enum:isometry_star-intertwiner1} the canonical
    inclusion \(I\colon \Hilm_1\injto\Hilm_1\oplus\Hilm_2\) is a
    \Star{}intertwiner from~\(\pi_1\) to~\(\pi\);
  \item \label{enum:isometry_star-intertwiner2} the canonical
    inclusion \(I\colon \Hilm_1\injto\Hilm_1\oplus\Hilm_2\) is an
    intertwiner from~\(\pi_1\) to~\(\pi\) and \(\Hilms =
    \Hilms_1 + (\Hilms\cap \Hilm_2)\);
  \item \label{enum:isometry_star-intertwiner3} there is a
    representation \((\Hilms_2,\pi_2)\) on~\(\Hilm_2\) such that \(\pi
    = \pi_1 \oplus \pi_2\).
  \end{enumerate}
\end{proposition}

\begin{proof}
  We view \(\Hilm_1\) and~\(\Hilm_2\) as subspaces of
  \(\Hilm_1\oplus\Hilm_2\), so we may drop the isometry~\(I\) from our
  notation.  The implication
  \ref{enum:isometry_star-intertwiner3}\(\Rightarrow\)\ref{enum:isometry_star-intertwiner1}
  is trivial.  We are going to prove
  \ref{enum:isometry_star-intertwiner1}\(\Rightarrow\)\ref{enum:isometry_star-intertwiner2}\(\Rightarrow\)\ref{enum:isometry_star-intertwiner3}.
  First assume that~\(I\) is a \Star{}intertwiner.  Then~\(I\) is an
  intertwiner.  In particular, \(\Hilms_1\subseteq \Hilms\).  Write
  \(\xi\in\Hilms\) as \(\xi=\xi_1+\xi_2\) with \(\xi_1\in\Hilm_1\),
  \(\xi_2\in\Hilm_2\).  Since~\(I^*\) is an intertwiner, \(\xi_1 =
  I^*(\xi)\in\Hilms_1\).  Hence \(\xi_2 = \xi-\xi_1 \in \Hilms\cap
  \Hilm_2\).  Thus~\ref{enum:isometry_star-intertwiner1}
  implies~\ref{enum:isometry_star-intertwiner2}.

  If~\ref{enum:isometry_star-intertwiner2} holds, then
  \(\Hilms_1\subseteq \Hilms\) is \(\pi\)\nb-invariant and
  \(\pi|_{\Hilms_1}=\pi_1\) because~\(I\) is an intertwiner.  We claim
  that \(\Hilms_2\defeq \Hilms\cap \Hilm_2\) is \(\pi\)\nb-invariant
  as well.  Let \(\xi\in\Hilms_2\) and \(\eta\in\Hilms_1\).  Then
  \(\braket{\eta}{\pi(a)\xi} = \braket{\pi(a^*)\eta}{\xi} =
  \braket{\pi_1(a^*)\eta}{\xi} = 0\) because \(\pi_1(a^*)\eta\in
  \Hilm_1\) is orthogonal to~\(\Hilm_2\).  Since~\(\Hilms_1\) is dense
  in~\(\Hilm_1\), this implies \(\pi(a)\xi \in \Hilm_1^\bot =
  \Hilm_2\), and this
  implies our claim.

  The condition~\ref{enum:isometry_star-intertwiner2} implies
  \(\Hilms = \Hilms_1 \oplus \Hilms_2\) as a vector space with
  \(\Hilms_2 = \Hilm_2 \cap \Hilms\) because \(\Hilm_1\cap \Hilm_2 =
  \{0\}\).  Then~\(\Hilms_2\) is dense in~\(\Hilm_2\)
  because~\(\Hilms\) is dense in~\(\Hilm_1\oplus\Hilm_2\).  Thus
  \((\Hilms_2,\pi|_{\Hilms_2})\) is a representation of~\(A\)
  on~\(\Hilm_2\).  And \((\Hilms,\pi)\) is the direct sum of
  \((\Hilms_1,\pi_1)\) and \((\Hilms_2,\pi|_{\Hilms_2})\) because
  \(\Hilms = \Hilms_1 \oplus\Hilms_2\) and \(\pi_1=\pi|_{\Hilms_1}\).
  Thus~\ref{enum:isometry_star-intertwiner2}
  implies~\ref{enum:isometry_star-intertwiner3}.
\end{proof}

\section{Integrable representations and \Cstar-hulls}
\label{sec:Cst_generated}

From now on, \emph{we tacitly assume representations to be closed}.
Proposition~\ref{pro:closure_rep} shows that this is no serious loss of
generality.

Let~\(A\) be a unital \Star{}algebra.  We assume that a class of
``integrable'' (closed) representations of~\(A\) on Hilbert modules
is chosen.  Let \(\Repi(A,D)\subseteq \Rep(A,D)\) be the full
subcategory with integrable representations on Hilbert
\(D\)\nb-modules as objects.  Being full means that the set of
arrows between two integrable representations of~\(A\) is still the
set of all isometric intertwiners.  We sometimes write \(\Repi(A)\)
and \(\Rep(A)\) for the collection of all the categories
\(\Repi(A,D)\) and \(\Rep(A,D)\) for all \Cstar\nb-algebras~\(D\).
A \Cstar\nb-hull is a \Cstar\nb-algebra~\(B\) with natural
isomorphisms \(\Rep(B,D) \cong \Repi(A,D)\) for all
\Cstar\nb-algebras~\(D\).  More precisely:

\begin{definition}
  \label{def:Cstar-hull}
  A \emph{\Cstar\nb-hull} for the integrable representations
  of~\(A\) is a \Cstar\nb-algebra~\(B\) with a family of bijections
  \(\Phi=\Phi^{\Hilm}\) from the set of representations of~\(B\)
  on~\(\Hilm\) to the set of integrable representations of~\(A\)
  on~\(\Hilm\) for all Hilbert modules~\(\Hilm\) over all
  \Cstar\nb-algebras~\(D\) with the following properties:
  \begin{itemize}
  \item \emph{compatibility with isometric intertwiners}: an isometry
    \(\Hilm_1\injto\Hilm_2\) (not necessarily adjointable) is an
    intertwiner between two representations \(\varrho_1\)
    and~\(\varrho_2\) of~\(B\) \emph{if and only if} it is an
    intertwiner between \(\Phi(\varrho_1)\) and~\(\Phi(\varrho_2)\);

  \item \emph{compatibility with interior tensor products}: if
    \(\Hilm[F]\) is a correspondence from~\(D_1\) to~\(D_2\),
    \(\Hilm\) is a Hilbert \(D_1\)\nb-module, and~\(\varrho\) is a
    representation of~\(B\) on~\(\Hilm\), then
    \(\Phi(\varrho\otimes_{D_1} 1_{\Hilm[F]}) =
    \Phi(\varrho)\otimes_{D_1} 1_{\Hilm[F]}\) as representations
    of~\(A\) on \(\Hilm\otimes_{D_1} \Hilm[F]\).
  \end{itemize}
\end{definition}

The compatibility with isometric intertwiners means that the
bijections~\(\Phi\) for all~\(\Hilm\) with fixed~\(D\) form an
isomorphism of categories \(\Rep(B,D)\cong \Repi(A,D)\) which, in
addition, does not change the underlying Hilbert \(D\)\nb-modules.  The
compatibility with interior tensor products expresses that these
isomorphisms of categories for different~\(D\) are natural with
respect to \Cstar\nb-correspondences.

\begin{definition}
  \label{def:weak_Cstar-hull}
  A \emph{weak \Cstar\nb-hull} for the integrable representations
  of~\(A\) is a \Cstar\nb-algebra~\(B\) with a family of
  bijections~\(\Phi\) between representations of~\(B\) and integrable
  representations of~\(A\) on Hilbert modules that is compatible with
  unitary \Star{}intertwiners and interior tensor products.
\end{definition}

Much of the general theory also works for weak \Cstar\nb-hulls.  But
the Induction Theorem~\ref{the:Fell_bundle_sections_hull} fails for
weak \Cstar\nb-hulls, as shown by a counterexample
in~§\ref{sec:counter_induction}.

\begin{proposition}
  \label{pro:adjointable_intertwiner}
  Let a class of integrable representations of~\(A\) have a weak
  \Cstar\nb-hull~\(B\).  Let \((\Hilms_1,\pi_1)\) and
  \((\Hilms_2,\pi_2)\) be integrable representations of~\(A\) on
  Hilbert \(D\)\nb-modules \(\Hilm_1\) and~\(\Hilm_2\), and
  let~\(\varrho_i\) be the corresponding representations of~\(B\)
  on~\(\Hilm_i\) for \(i=1,2\).  An adjointable operator \(x\colon
  \Hilm_1\to\Hilm_2\) is a \Star{}intertwiner from
  \((\Hilms_1,\pi_1)\) to~\((\Hilms_2,\pi_2)\) if and only if it is
  an intertwiner from \(\varrho_1\) to~\(\varrho_2\).
\end{proposition}

\begin{proof}
  Working with the direct sum representations on
  \(\Hilm_1\oplus\Hilm_2\) and the adjointable operator
  \(\bigl(\begin{smallmatrix} 0&x\\0&0 \end{smallmatrix}\bigr)\), we
  may assume without loss of generality
  that \(\Hilm_1=\Hilm_2=\Hilm\), \(\Hilms_1=\Hilms_2=\Hilms\),
  \(\pi_1=\pi_2=\pi\), and \(\varrho_1=\varrho_2=\varrho\).  The
  adjointable intertwiners for the representation~\(\varrho\)
  of~\(B\) form a \Cstar\nb-algebra~\(B'\): the commutant of~\(B\) in
  \(\Bound(\Hilm)\).  We claim that the \Star{}intertwiners for the
  representation~\(\pi\) of~\(A\) also form a
  \Cstar\nb-algebra~\(A'\).  Intertwiners and hence
  \Star{}intertwiners form an algebra.  Thus~\(A'\) is a
  \Star{}algebra.  We show that it is closed.

  Let \((x_i)_{i\in\N}\) be a sequence of adjointable intertwiners
  for~\((\Hilms,\pi)\) that converges in norm to
  \(x\in\Bound(\Hilm)\).  Let \(\xi\in\Hilms\).  Then
  \(x_i(\xi)\in\Hilms\) because each~\(x_i\) is an intertwiner.
  Since \(\pi(a)(x_i\xi) = x_i\pi(a)\xi\) is norm-convergent for
  each \(a\in A\), the sequence \(x_i(\xi)\) is a Cauchy sequence
  for the graph topology on~\(\Hilms\).  Since representations are
  tacitly assumed to be closed, this Cauchy sequence converges
  in~\(\Hilms\), so that \(x(\Hilms)\subseteq\Hilms\).  Moreover,
  \(x(\pi(a)\xi) = \pi(a) x(\xi)\) for all \(a\in A\),
  \(\xi\in\Hilms\), so~\(x\) is again an intertwiner.  Thus the
  algebra of intertwiners is norm-closed.  This implies that~\(A'\)
  is a \Cstar\nb-algebra.

  Since the family of bijections \(\Repi(A)\cong\Rep(B)\) is
  compatible with unitary \Star{}intertwiners, a unitary operator
  on~\(\Hilm\) is a \Star{}intertwiner for~\(A\) if and only if it
  is an intertwiner for~\(B\).  That is, the unital
  \Cstar\nb-subalgebras \(A',B'\subseteq \Bound(\Hilm)\) contain the
  same unitaries.  A unital \Cstar\nb-algebra is the linear span
  of its unitaries because any self-adjoint element~\(t\) of norm
  at most~\(1\) may be written as
  \[
  t = \nicefrac12\bigl(t+\ima \sqrt{1-t^2}\bigr)
  + \nicefrac12\bigl(t-\ima \sqrt{1-t^2}\bigr)
  \]
  and \(t\pm \ima \sqrt{1-t^2}\) are unitary.  Thus \(A'=B'\).  This
  is what we had to prove.
\end{proof}

\begin{corollary}
  \label{cor:integrable_sums}
  Let \(\Repi(A)\) have a weak \Cstar\nb-hull~\(B\).  Direct sums
  and summands of integrable representations remain integrable, and
  the family of bijections \(\Repi(A)\cong\Rep(B)\) preserves direct
  sums.
\end{corollary}

\begin{proof}
  Let \(\pi_1,\pi_2\) be representations of~\(A\) on Hilbert
  \(D\)\nb-modules \(\Hilm_1,\Hilm_2\).  Let \(S_i\colon \Hilm_i\injto
  \Hilm_1\oplus\Hilm_2\) for \(i=1,2\) be the inclusion maps.  First
  we assume that \(\pi_1,\pi_2\) are integrable.  Let~\(\varrho_i\) be
  the representation of~\(B\) on~\(\Hilm_i\) corresponding
  to~\(\pi_i\), and let \(\pi\) be the integrable representation
  of~\(A\) on \(\Hilm_1\oplus\Hilm_2\) corresponding to the
  representation \(\varrho_1\oplus\varrho_2\) of~\(B\).  The
  isometries~\(S_i\) are intertwiners from~\(\varrho_i\)
  to~\(\varrho_1\oplus\varrho_2\).  By
  Proposition~\ref{pro:adjointable_intertwiner}, they are
  \Star{}intertwiners from~\(\pi_i\) to~\(\pi\).  Hence \(\pi =
  \pi_1\oplus \pi_2\) by
  Proposition~\ref{pro:isometry_star-intertwiner}.  Thus \(\pi_1\oplus
  \pi_2\) is integrable and the family of bijections
  \(\Repi(A)\cong\Rep(B)\) preserves direct sums.  The same argument
  works for infinite direct sums.

  Now we assume instead that \(\pi_1\oplus\pi_2\) is integrable.
  Let~\(\varrho\) be the representation of~\(B\) corresponding to
  \(\pi_1\oplus\pi_2\).  The orthogonal projection onto~\(\Hilm_1\)
  is a \Star{}intertwiner on the representation \(\pi_1\oplus\pi_2\)
  by Proposition~\ref{pro:isometry_star-intertwiner}, and hence also
  on~\(\varrho\) by Proposition~\ref{pro:adjointable_intertwiner}.
  Thus \(\varrho = \varrho_1\oplus\varrho_2\) for some
  representations~\(\varrho_i\) of~\(B\) on~\(\Hilm_i\).
  Let~\(\pi_i'\) be the integrable representation of~\(A\)
  corresponding to~\(\varrho_i\).  The isometry~\(S_i\) is a
  \Star{}intertwiner from~\(\varrho_i\)
  to~\(\varrho_1\oplus\varrho_2\) and hence from~\(\pi_i'\)
  to~\(\pi_1\oplus\pi_2\) by
  Proposition~\ref{pro:adjointable_intertwiner}.  This implies
  \(\pi_i' = \pi_i\), so that~\(\pi_i\) is integrable for \(i=1,2\).
\end{proof}

\begin{definition}
  \label{def:universal_integrable_rep}
  Let~\(B\) be a weak \Cstar\nb-hull for~\(A\).  The
  \emph{universal} integrable representation of~\(A\) is the
  integrable representation~\((\Hilms[B],\mu)\) of~\(A\) on~\(B\)
  that corresponds to the identity representation of~\(B\) on
  itself.
\end{definition}

\begin{proposition}
  \label{pro:Phi_simplifies}
  Let~\(B\) with a family of bijections~\(\Phi\) between
  representations of~\(B\) and integrable representations of~\(A\) on
  Hilbert modules be a weak \Cstar\nb-hull for the integrable
  representations of~\(A\).  Let~\((\Hilms[B],\mu)\) be the universal
  integrable representation of~\(A\).  Then \(\Phi(\Hilm) \cong
  (\Hilms[B],\mu) \otimes_B \Hilm\) for any
  \Cstar\nb-correspondence~\(\Hilm\) from~\(B\) to~\(D\).
  \textup{(}The proof makes this isomorphism more precise.\textup{)}
\end{proposition}

\begin{proof}
  Let \(\varrho\colon B\to\Bound(\Hilm)\) be a representation
  of~\(B\) on a Hilbert module~\(\Hilm\).  Then \(u\colon
  B\otimes_\varrho \Hilm \congto \Hilm\), \(b\otimes\xi\mapsto
  \varrho(b) \xi\), is a unitary \Star{}intertwiner between the
  interior tensor product of the identity representation of~\(B\)
  with~\(\Hilm\) and the representation~\(\varrho\) on~\(\Hilm\).
  As~\(\Phi\) is compatible with interior tensor products and
  unitary \Star{}intertwiners, \(u\) is a unitary \Star{}intertwiner
  between \((\Hilms[B],\mu) \otimes_B \Hilm\) and~\(\Phi(\varrho)\).
  Therefore, the image \(u(\Hilms[B] \odot \Hilm) =
  \varrho(\Hilms[B])\Hilm\) is a core for~\(\Phi(\varrho)\), and
  \(a\in A\) acts on this core by \(a\mapsto u(\mu(a)\otimes 1)u^*\)
  or, explicitly,
  \(a\cdot (\varrho(b)\xi) = \varrho(\mu(a)b) \xi\) for all
  \(a\in A\), \(b\in\Hilms[B]\), \(\xi\in\Hilm\).
\end{proof}

Put in a nutshell, the whole isomorphism between integrable
representations of~\(A\) and representations of~\(B\) is encoded in
the single representation~\((\Hilms[B],\mu)\) of~\(A\) on~\(B\).
This is similar to Schmüdgen's approach
in~\cite{Schmuedgen:Well-behaved}.  In the following, we disregard
the canonical unitary~\(u\) in the proof of
Proposition~\ref{pro:Phi_simplifies} and write \(\Phi(\varrho) =
(\Hilms[B],\mu) \otimes_B \Hilm\).

A (weak) \Cstar\nb-hull~\(B\) does not solve the problem of describing
the integrable representations of~\(A\).  It only reduces it to the
study of the representations of the \Cstar\nb-algebra~\(B\).  This
reduction is useful because it gets rid of unbounded operators.
If~\(B\) is of type~I, then any Hilbert space representation of~\(B\)
is a direct integral of irreducible representations, and irreducible
representations may, in principle, be classified.  Thus integrable
Hilbert space representations of~\(A\) are direct integrals of
irreducible integrable representations, and the latter may, in
principle, be classified.  But if~\(B\) is not of type~I, then the
integrable Hilbert space representations of~\(A\) are exactly as
complicated as the Hilbert space representations of~\(B\), and giving
the \Cstar\nb-algebra~\(B\) may well be the best one can say about
them.

\begin{proposition}
  \label{pro:unique_hull}
  A class of integrable representations has at most one weak
  \Cstar\nb-hull.
\end{proposition}

\begin{proof}
  Let \(B_1\)
  and~\(B_2\)
  be weak \Cstar\nb-hulls for the same class of integrable
  representations of~\(A\).
  The identity map on~\(B_1\),
  viewed as a representation of~\(B_1\)
  on itself, corresponds first to an integrable representation
  of~\(A\)
  on~\(B_1\)
  and further to a representation of~\(B_2\)
  on~\(B_1\).
  This is a ``morphism'' from~\(B_2\)
  to~\(B_1\),
  that is, a nondegenerate \Star{}homomorphism \(B_2\to\Mult(B_1)\).
  Similarly, we get a morphism from~\(B_1\)
  to~\(B_2\).
  These morphisms \(B_1\leftrightarrow B_2\)
  are inverse to each other with respect to the composition of
  morphisms because the maps they induce on representations of \(B_1\)
  and~\(B_2\)
  on \(B_1\) and~\(B_2\) are inverse to each other.  An
  isomorphism in the category of morphisms is an isomorphism of
  \Cstar\nb-algebras in the usual sense by
  \cite{Buss-Meyer-Zhu:Higher_twisted}*{Proposition 2.10}.
\end{proof}

Now take any representation \((\Hilms[B],\mu)\) of~\(A\) on~\(B\).
When is this the universal integrable representation of a (weak)
\Cstar\nb-hull?  Let~\(D\) be a \Cstar\nb-algebra and~\(\Hilm\) a
Hilbert \(D\)\nb-module.  For a representation \(\varrho\colon
B\to\Bound(\Hilm)\), let
\(\Phi(\varrho)=(\Hilms[B],\mu)\otimes_\varrho \Hilm\) be the induced
representation of~\(A\) on~\(\Hilm\) as in the proof of
Proposition~\ref{pro:Phi_simplifies}.  A representation of~\(A\) is
called \emph{\(B\)\nb-integrable} if it is in the image of~\(\Phi\).

\begin{proposition}
  \label{pro:Cstar-generated_by_unbounded_multipliers}
  The \Cstar\nb-algebra~\(B\) is a weak \Cstar\nb-hull for the
  \(B\)\nb-integrable representations of~\(A\) if and only if
  \begin{enumerate}
  \item \label{enum:Cstar-gen1} if two representations
    \(\varrho_1,\varrho_2\colon B\rightrightarrows\Bound(\Hilm[H])\)
    on the same Hilbert space~\(\Hilm[H]\) satisfy \(\mu\otimes_B
    \varrho_1=\mu\otimes_B \varrho_2\) as closed representations
    of~\(A\), then \(\varrho_1=\varrho_2\).
  \end{enumerate}
  It is a \Cstar\nb-hull if and only if \ref{enum:Cstar-gen1} and
  the following equivalent conditions hold:
  \begin{enumerate}[resume]
  \item \label{enum:Cstar-gen2} Let \((\Hilms[H],\pi)\) be a
    representation of~\(A\) on a Hilbert space~\(\Hilm[H]\) and let
    \((\Hilms[H]_0,\pi_0)\) be a subrepresentation on a closed
    subspace \(\Hilm[H]_0\subseteq\Hilm[H]\); that is,
    \(\Hilms[H]_0\subseteq \Hilms[H]\) and \(\pi_0(a)=
    \pi(a)|_{\Hilms[H]_0}\) for all \(a\in A\).  If both \(\pi_0\)
    and~\(\pi\) are \(B\)\nb-integrable, then \(\Hilms[H] =
    \Hilms[H]_0 \oplus (\Hilms[H]\cap \Hilm[H]_0^\bot)\) as vector
    spaces.
  \item \label{enum:Cstar-gen2'} Isometric intertwiners
    between \(B\)\nb-integrable Hilbert space representations of~\(A\)
    are \Star{}intertwiners.
  \item \label{enum:Cstar-gen2''} \(B\)\nb-integrable
    subrepresentations of \(B\)\nb-integrable Hilbert space
    representations of~\(A\) are direct summands.
  \end{enumerate}
  The conditions \ref{enum:Cstar-gen1}--\ref{enum:Cstar-gen2''}
  together are equivalent to
  \begin{enumerate}[resume]
  \item \label{enum:Cstar-gen12} let \(\varrho\colon
    B\to\Bound(\Hilm[H])\) be a Hilbert space representation and
    let~\((\Hilms[H],\pi)\) be the associated representation of~\(A\)
    on~\(\Hilm[H]\).  If \((\Hilms[H]_0,\pi|_{\Hilms[H]_0})\) is a
    \(B\)\nb-integrable subrepresentation
    of~\((\Hilms[H],\pi)\) on a closed subspace
    \(\Hilm[H]_0\subseteq\Hilm[H]\), then the projection
    onto~\(\Hilm[H]_0\) commutes with~\(\varrho(B)\).
  \end{enumerate}
\end{proposition}

\begin{proof}
  The map~\(\Phi\) is compatible with interior tensor products by
  Lemma~\ref{lem:tensor_associative}.  The
  condition~\ref{enum:Cstar-gen1} says that~\(\Phi\) is injective on
  Hilbert space representations.  We claim that this implies
  injectivity also for representations on a Hilbert module~\(\Hilm\)
  over a \(\Cst\)\nb-algebra~\(D\).  Let \(\varrho_1,\varrho_2\) be
  representations of \(B\) on~\(\Hilm\) with
  \(\mu\otimes_B\varrho_1=\mu\otimes_B\varrho_2\).  Let
  \(D\to\Bound(\Hilm[H])\) be a faithful representation.  Then the
  representations \(\varrho_1\otimes_D 1\) and \(\varrho_2\otimes_D
  1\) on the Hilbert space \(\Hilm\otimes_D \Hilm[H]\) satisfy
  \(\mu\otimes_B\varrho_1\otimes_D 1= \mu\otimes_B
  \varrho_2\otimes_D 1\) by Lemma~\ref{lem:tensor_associative}.
  Then condition~\ref{enum:Cstar-gen1} implies \(\varrho_1\otimes_D
  1=\varrho_2\otimes_D 1\).  Since the representation
  \(\Bound(\Hilm) \to \Bound(\Hilm\otimes_D\Hilm[H])\) is faithful,
  this implies \(\varrho_1=\varrho_2\).  So~\(\Phi\) is injective
  also for representations on~\(\Hilm\).

  The image of~\(\Phi\) consists exactly of the \(B\)\nb-integrable
  representations of~\(A\) by definition.  A unitary operator~\(u\)
  \Star{}intertwines two representations \((\Hilms_1,\pi_1)\) and
  \((\Hilms_2,\pi_2)\) of~\(A\) if and only if \(\pi_2 = u\pi_1
  u^*\), where~\(u\pi_1 u^*\) denotes the representation with domain
  \(u(\Hilms_1)\) and \((u\pi_1 u^*)(a) = u\pi_1(a) u^*\).
  Similarly, \(u\) intertwines two representations \(\varrho_1\)
  and~\(\varrho_2\) of~\(B\) if and only if \(\varrho_2 = u\varrho_1
  u^*\).  Hence~\ref{enum:Cstar-gen1} implies that a unitary that
  \Star{}intertwines two \(B\)\nb-integrable representations
  of~\(A\) also intertwines the corresponding representations
  of~\(B\).  The converse is clear.  So~\(B\) is a weak
  \Cstar\nb-hull for the \(B\)\nb-integrable representations if and
  only if~\ref{enum:Cstar-gen1} holds.

  The equivalence between \ref{enum:Cstar-gen2},
  \ref{enum:Cstar-gen2'} and~\ref{enum:Cstar-gen2''} follows from
  Proposition~\ref{pro:isometry_star-intertwiner} by writing
  \(\Hilm[H] = \Hilm[H]_0 \oplus \Hilm[H]_0^\bot\).
  Assume that~\(B\) is a \Cstar\nb-hull.  An isometric
  intertwiner for~\(A\) is also one for~\(B\).  Then it is a
  \Star{}intertwiner for~\(A\) and its range projection is an
  intertwiner for~\(B\) by
  Proposition~\ref{pro:adjointable_intertwiner}.  Thus
  both~\ref{enum:Cstar-gen2'} and~\ref{enum:Cstar-gen12} follow
  if~\(B\) is a \Cstar\nb-hull.

  Conversely, assume \ref{enum:Cstar-gen1}
  and~\ref{enum:Cstar-gen2'}.  We are going to prove
  that~\(B\) is a \Cstar\nb-hull for the \(B\)\nb-integrable
  representations of~\(A\).  We have already seen that~\(B\) is a
  weak \Cstar\nb-hull.  We must check compatibility with isometric
  intertwiners.

  Let \(D\) be a \Cstar\nb-algebra and let \(\Hilm_1,\Hilm_2\) be
  Hilbert \(D\)\nb-modules with representations
  \(\varrho_1,\varrho_2\) of~\(B\).  The corresponding
  representations~\((\Hilms_i,\pi_i)\) of~\(A\) for \(i=1,2\) are the
  closures of the representations on \(\varrho_i(\Hilms[B])\Hilm_i\)
  given by \(\pi_i(a)(\varrho_i(b)\xi)
  \defeq \varrho_i(\mu(a) b)(\xi)\) for \(a\in A\), \(b\in\Hilms[B]\),
  \(\xi\in\Hilm_i\).  Hence an isometric intertwiner for~\(B\) is also
  one for~\(A\).  Conversely, let \(I\colon \Hilm_1\injto\Hilm_2\) be
  a Hilbert module isometry with \(I(\Hilms_1) \subseteq \Hilms_2\)
  and \(\pi_2(a)(I\xi) = I(\pi_1(a)\xi)\) for all \(a\in A\),
  \(\xi\in\Hilms_1\).  We must prove \(\varrho_2(b)I =
  I\varrho_1(b)\) for all \(b\in B\).

  Let \(\varphi\colon D\injto\Bound(\Hilm[K])\) be a faithful
  representation on a Hilbert space~\(\Hilm[K]\).  Equip
  \(\Hilm[H]_i\defeq \Hilm_i\otimes_\varphi \Hilm[K]\) with the
  induced representations~\(\tilde\varrho_i\) of~\(B\) and
  \(\tilde\pi_i\)~of~\(A\) for \(i=1,2\).  Since the family of
  bijections \(\Phi\colon \Rep(B) \congto \Repi(A)\) is
  compatible with interior tensor products, it maps
  \(\tilde\varrho_i\) to~\(\tilde\pi_i\).  The operator~\(I\)
  induces an isometric intertwiner~\(\tilde{I}\)
  from~\(\tilde\pi_1\) to~\(\tilde\pi_2\) by
  Lemma~\ref{lem:rep_tensor_corr_functor}.

  Since \(\tilde\pi_1\) and~\(\tilde\pi_2\) are \(B\)\nb-integrable,
  we are in the situation of~\ref{enum:Cstar-gen2'}.
  So~\(\tilde{I}\) is a \Star{}intertwiner from~\(\tilde\pi_1\)
  to~\(\tilde\pi_2\).  Thus~\(\tilde{I}\) is an intertwiner
  from~\(\tilde{\varrho}_1\) to~\(\tilde{\varrho}_2\) by
  Proposition~\ref{pro:adjointable_intertwiner}.  That is,
  \(\tilde{I}\tilde{\varrho}_1(b) = \tilde{\varrho}_2(b)\tilde{I}\)
  for all \(b\in B\).  Equivalently, \((I\varrho_1(b)\xi)\otimes\eta
  = (\varrho_2(b)I\xi)\otimes\eta\) in \(\Hilm_2\otimes_\varphi
  \Hilm[H]\) for all \(b\in B\), \(\xi\in\Hilm\),
  \(\eta\in\Hilm[H]\).  Since the representation~\(\varphi\) is
  faithful, this implies \(I\varrho_1(b)\xi = \varrho_2(b)I\xi\) for
  all \(b,\xi\), so that \(I\varrho_1(b) = \varrho_2(b)I\) for
  all~\(b\), that is, \(I\) intertwines \(\varrho_1\)
  and~\(\varrho_2\).  Thus~\(\Phi\) is compatible with isometric
  intertwiners.

  Since~\ref{enum:Cstar-gen12} holds for \Cstar\nb-hulls, we have
  proved along the way that \ref{enum:Cstar-gen1}
  and~\ref{enum:Cstar-gen2'} imply~\ref{enum:Cstar-gen12}.  It remains
  to show, conversely, that~\ref{enum:Cstar-gen12} implies
  \ref{enum:Cstar-gen2'} and~\ref{enum:Cstar-gen1}.  In the situation
  of~\ref{enum:Cstar-gen2'}, the projection~\(P\) onto~\(\Hilm[H]_0\)
  commutes with~\(B\) by~\ref{enum:Cstar-gen12}.  Thus the
  representation of~\(B\) on~\(\Hilm[H]\) is a direct sum of
  representations on \(\Hilm[H]_0\) and~\(\Hilm[H]_0^\bot\).  This is
  inherited by the induced representation of~\(A\) and its domain.
  So~\ref{enum:Cstar-gen12} implies~\ref{enum:Cstar-gen2'}.

  In the situation of~\ref{enum:Cstar-gen1}, form the direct sum
  representation \(\varrho_1\oplus\varrho_2\) on
  \(\Hilm[H]\oplus\Hilm[H]\) and let \(\Hilm[H]_0 \defeq
  \{(\xi,\xi)\mid \xi\in\Hilm[H]\}\).  The representation of~\(A\)
  corresponding to \(\varrho_1\oplus\varrho_2\) is
  \(\mu\otimes_{\varrho_1} \Hilm[H] \oplus \mu\otimes_{\varrho_2}
  \Hilm[H]\).  Since \(\mu\otimes_{\varrho_1} \Hilm[H] =
  \mu\otimes_{\varrho_2} \Hilm[H]\) by assumption, the domain of
  \(\mu\otimes_{\varrho_1} \Hilm[H] \oplus \mu\otimes_{\varrho_2}
  \Hilm[H]\) is \(\Hilms[H]\oplus\Hilms[H]\) for some dense subspace
  \(\Hilms[H]\subseteq \Hilm[H]\), and \(\Hilms[H]_0 \defeq
  \{(\xi,\xi)\mid \xi\in\Hilms[H]\}\) is a dense subspace
  in~\(\Hilm[H]_0\) that is invariant for the representation
  \(\mu\otimes_{\varrho_1} \Hilm[H] \oplus \mu\otimes_{\varrho_2}
  \Hilm[H]\).  The restricted representation on this subspace is
  \(B\)\nb-integrable because it is unitarily equivalent to
  \(\mu\otimes_{\varrho_1} \Hilm[H] = \mu\otimes_{\varrho_2}
  \Hilm[H]\).  Therefore, the projection onto~\(\Hilm[H]_0\)
  commutes with the representation of~\(B\)
  by~\ref{enum:Cstar-gen12}.  Thus \(\varrho_1=\varrho_2\).
  So~\ref{enum:Cstar-gen12} implies~\ref{enum:Cstar-gen1}.
\end{proof}

The equivalent conditions
\ref{enum:Cstar-gen2}--\ref{enum:Cstar-gen2''} may be easier to
check than~\ref{enum:Cstar-gen12} because they do not involve the
\Cstar\nb-hull.

\begin{corollary}
  \label{pro:envelope_unique}
  Let~\(A\) be a \Star{}algebra and let~\(B_i\) be \Cstar\nb-algebras
  with representations \((\Hilms[B]_i,\mu_i)\) of~\(A\) for \(i=1,2\).
  Assume that for each Hilbert space~\(\Hilm[H]\), the maps
  \(\Phi_i\colon \Rep(B_i,\Hilm[H]) \to \Rep(A,\Hilm[H])\),
  \(\varrho_i\mapsto (\Hilms[B]_i,\mu_i) \otimes_{\varrho_i}
  \Hilm[H]\), are injective and have the same image.  Then there is a
  unique isomorphism \(B_1\cong B_2\) intertwining the representations
  \((\Hilms[B]_i,\mu_i)\) of~\(A\) for \(i=1,2\).
\end{corollary}

Hence a \Cstar\nb-envelope as defined
in~\cite{Dowerk-Savchuk:Induced} is unique if it exists.

\begin{proof}
  Both \(B_1\) and~\(B_2\) are weak \Cstar\nb-hulls for the same
  class of representations of~\(A\) by
  Proposition~\ref{pro:Cstar-generated_by_unbounded_multipliers}.
  Proposition~\ref{pro:unique_hull} gives the isomorphism \(B_1\cong
  B_2\).
\end{proof}

\begin{remark}
  \label{rem:envelope_unique}
  The Hilbert space representations of a \Cstar\nb-algebra only
  determine its bidual \(\textup{W}^*\)\nb-algebra, not the
  \Cstar\nb-algebra itself.  Hence it is remarkable that the
  conditions in
  Proposition~\ref{pro:Cstar-generated_by_unbounded_multipliers} and
  Corollary~\ref{pro:envelope_unique} only need Hilbert space
  representations.  For Corollary~\ref{pro:envelope_unique}, this
  works because the bijection between the representations is of a
  particular form, induced by representations of~\(A\).
\end{remark}

The condition~\ref{enum:Cstar-gen1} in
Proposition~\ref{pro:Cstar-generated_by_unbounded_multipliers} is
required in several other theories
that associate a \Cstar\nb-algebra to a \Star{}algebra, such as the
host algebras of Grundling~\cites{Grundling:Group_algebra,
  Grundling-Neeb:Full_regularity}, the \Cstar\nb-envelopes of Dowerk
and Savchuk~\cite{Dowerk-Savchuk:Induced}, or the notion of a
\Cstar\nb-algebra generated by affiliated multipliers by
Woronowicz~\cite{Woronowicz:Cstar_generated}, see
\cite{Woronowicz:Cstar_generated}*{Theorem 3.3} or the proof of
Theorem~\ref{the:Woronowicz_local-global} below.

\begin{definition}
  \label{def:integrable_admissible}
  Let~\(A\) be a \Star{}algebra.  A class of ``integrable''
  representations of~\(A\) on Hilbert modules over
  \Cstar\nb-algebras is \emph{admissible} if it satisfies the
  conditions \ref{enum:admissible0}--\ref{enum:admissible3} below,
  and \emph{weakly admissible} if it satisfies
  \ref{enum:admissible0}--\ref{enum:admissible2}.
  \begin{enumerate}
  \item \label{enum:admissible0} If there is a unitary
    \Star{}intertwiner from an integrable representation to another
    representation, then the latter is integrable.
  \item \label{enum:admissible1} If \(D\) and~\(D'\) are
    \Cstar\nb-algebras, \(\Hilm[F]\) is a correspondence from~\(D\)
    to~\(D'\), and \((\Hilms,\pi)\) is an integrable representation
    of~\(A\) on a Hilbert \(D\)\nb-module~\(\Hilm\), then the
    representation \((\Hilms,\pi)\otimes_D \Hilm[F]\) on
    \(\Hilm\otimes_D \Hilm[F]\) is integrable.
  \item \label{enum:admissible2} Direct sums and summands of
    integrable representations are integrable.
  \item \label{enum:admissible3} Any integrable subrepresentation of
    an integrable representation of~\(A\) on a Hilbert space is a
    direct summand.
  \end{enumerate}
\end{definition}

\begin{lemma}
  \label{lem:integrable_admissible}
  Any class of integrable representations with a
  \textup{(}weak\textup{)} \Cstar\nb-hull is
  \textup{(}weakly\textup{)} admissible.
\end{lemma}

\begin{proof}
  If there is a \Cstar\nb-hull,
  Proposition~\ref{pro:Cstar-generated_by_unbounded_multipliers}
  implies~\ref{enum:admissible3} in
  Definition~\ref{def:integrable_admissible}.  If there is a weak
  \Cstar\nb-hull, then \ref{enum:admissible0}
  and~\ref{enum:admissible1} in
  Definition~\ref{def:integrable_admissible} follow from the
  compatibility with unitary \Star{}intertwiners and interior tensor
  products in the definition of a \Cstar\nb-hull,
  and~\ref{enum:admissible2} follows from
  Corollary~\ref{cor:integrable_sums}.
\end{proof}

\begin{proposition}
  \label{pro:representation_Hilm_Comp}
  Let~\(A\) be a unital \Star{}algebra and let~\(\Hilm\) be a
  Hilbert module over a \Cstar\nb-algebra~\(D\).  There is a natural
  bijection between the sets of representations of~\(A\) on
  \(\Hilm\) and~\(\Comp(\Hilm)\).  It preserves integrability if the
  class of integrable representations of~\(A\) is weakly admissible
  or, in particular, if it has a weak \Cstar\nb-hull.
\end{proposition}

\begin{proof}
  We may view~\(\Hilm\) as an imprimitivity bimodule between
  \(\Comp(\Hilm)\) and the ideal~\(I\) in~\(D\) that is spanned by
  the inner products \(\braket{\xi}{\eta}_D\) for \(\xi,\eta\in
  \Hilm\).  Let~\(\Hilm^*\) be the inverse imprimitivity bimodule,
  which is a Hilbert module over~\(\Comp(\Hilm)\) with
  \(\Comp(\Hilm^*)\cong I\).  Then \(\Comp(\Hilm) \cong \Hilm
  \otimes_D \Hilm^*\) and \(\Hilm^*\otimes_{\Comp(\Hilm)} \Hilm =
  I\).

  If \((\pi,\Hilms)\) is a representation of~\(A\) on~\(\Hilm\),
  then \((\pi,\Hilms) \otimes_D \Hilm^*\) is a representation
  of~\(A\) on \(\Hilm\otimes_D \Hilm^*= \Comp(\Hilm)\).  This
  maps~\(\Rep(A,\Hilm)\) to~\(\Rep(A,\Comp(\Hilm))\).  If
  \((\varrho,\Hilms[K])\) is a representation of~\(A\)
  on~\(\Comp(\Hilm)\), then
  \((\varrho,\Hilms[K])\otimes_{\Comp(\Hilm)} \Hilm\) is a
  representation of~\(A\) on \(\Comp(\Hilm)\otimes_{\Comp(\Hilm)}
  \Hilm \cong \Hilm\).  This maps \(\Rep(A,\Comp(\Hilm))\)
  to~\(\Rep(A,\Hilm)\).  We claim that these two maps are inverse to
  each other.  Both preserve integrability by~\ref{enum:admissible1} in
  Definition~\ref{def:integrable_admissible}.

  The map \(\Rep(A,\Hilm)\to\Rep(A,\Comp(\Hilm)) \to \Rep(A,\Hilm)\)
  sends a representation~\((\pi,\Hilms)\) of~\(A\) on~\(\Hilm\) to
  the representation \((\pi,\Hilms) \otimes_D (\Hilm^*
  \otimes_{\Comp(\Hilm)} \Hilm) = (\pi,\Hilms) \otimes_D I\)
  of~\(A\) on \(\Hilm\cong \Hilm\otimes_D I\) by
  Lemma~\ref{lem:tensor_associative}.  This is the restriction
  of~\(\pi\) to \(\Hilms\cdot I\subseteq \Hilms\).  Since~\(\Hilm\)
  is also a Hilbert module over~\(I\), it is nondegenerate as a
  right \(I\)\nb-module.  Therefore, if \((u_i)\) is an approximate
  unit in~\(I\), then \(\lim \xi u_i= \xi\) for all \(\xi\in\Hilm\).
  Then also \(\lim \pi(a)\xi u_i= \pi(a)\xi\) for all
  \(\xi\in\Hilms\), \(a\in A\), so \(\lim \xi u_i= \xi\) in the
  graph topology for all \(\xi\in\Hilms\).  Thus \(\Hilms \cdot
  I=\Hilms\), and we get the identity map on~\(\Rep(A,\Hilm)\).  A
  similar, easier argument shows that we also get the identity map
  on~\(\Rep(A,\Comp(\Hilm))\).
\end{proof}

\section{Polynomials in one variable I}
\label{sec:polynomials1}

Let \(A=\C[x]\) with \(x=x^*\).  A (not necessarily closed)
representation of~\(A\) on a Hilbert \(D\)\nb-module~\(\Hilm\) is
determined by a dense \(D\)\nb-submodule \(\Hilms\subseteq \Hilm\)
and a single symmetric operator \(\pi(x)\colon \Hilms\to\Hilms\),
that is, \(\pi(x)\subseteq\pi(x)^*\).  Then \(\pi(x^n)=\pi(x)^n\).

\begin{lemma}
  \label{lem:graph_topology_Cx}
  The graph topology on~\(\Hilms\) is generated by the increasing
  sequence of norms \(\norm{\xi}_n\defeq
  \norm{\braket{\xi}{(1+\pi(x^{2n}))\xi}}\) for \(n\in\N\).
\end{lemma}

\begin{proof}
  We must show that for any \(a\in \C[x]\) there are \(C>0\) and
  \(n\in\N\) with \(\norm{\xi}_a \le C\norm{\xi}_n\).  We choose~\(n\)
  so that~\(a\) has degree at most~\(n\).  Then there is \(C>0\) so
  that \(C(1+t^{2n}) > 1+\abs{a(t)}^2\) for all \(t\in\R\).  Thus the
  polynomial \(b \defeq C(1+x^{2n}) - (1+ a^*a)\) is positive
  on~\(\R\).  So the zeros of~\(b\) are complex and come in pairs
  \(\lambda_j\pm \ima\mu_j\) for \(j=1,\dotsc,n\) with
  \(\lambda_j,\mu_j\in\R\) by the Fundamental Theorem of Algebra.
  Then \(b = \prod_{j=1}^n \bigl((x-\lambda_j)^2 +\mu_j^2\bigr) =
  \sum_{k=1}^{2^n} b_k^2\), where each~\(b_k\) is a product of
  either \(x-\lambda_j\) or~\(\mu_j\) for \(j=1,\dotsc,k\), so \(b_k
  = b_k^*\).  Thus \(\norm{\xi}_a \le C\norm{\xi}_n\).
\end{proof}

Thus the monomials \(\{x^n\mid n\in\N\}\) form a strong generating
set for~\(\C[x]\).  A representation of~\(\C[x]\) is determined by
the closed operators~\(\cl{\pi(x^n)}\) for \(n\in\N\) by
Proposition~\ref{pro:equality_if_closure_equal}.  In contrast, it is
not yet determined by the single closed operator~\(\cl{\pi(x)}\)
because~\(\{x\}\) is not a \emph{strong} generating set:

\begin{example}
  \label{exa:x_not_strong_generator}
  We construct a closed representation of~\(\C[x]\)
  on a Hilbert space with
  \(\cl{\pi(x^2)} \subsetneq \bigl(\cl{\pi(x)}\bigr){}^2\).
  Let \(\Hilm[H]\defeq L^2(\T)\),
  viewed as the space of \(\Z\)\nb-periodic
  functions on~\(\R\).
  Let \(\Hilms[H]_0\defeq \Cont^\infty(\T)\)
  and let \(\pi_0\colon \C[x]\to \Endo(\Hilms[H]_0)\)
  be the polynomial functional calculus for the operator
  \(\ima \frac{\dd}{\dd t}\).
  The graph topology generated by this representation of~\(\C[x]\)
  is the usual Fr\'echet topology on~\(\Cont^\infty(\T)\).
  So the representation of~\(\C[x]\)
  on~\(\Cont^\infty(\T)\)
  is closed.  Now for some \(\lambda\in\T\), let
  \[
  \Hilms[H] \defeq \{f\in\Cont^\infty(\T)\mid
  f^{(n)}(\lambda)=0 \text{ for all }n\ge1\}.
  \]
  This is a closed, \(\C[x]\)\nb-invariant
  subspace in~\(\Hilms[H]_0\).
  Let~\(\pi\)
  be the restriction of~\(\pi_0\)
  to~\(\Hilms[H]_0\).
  This is also a closed representation of~\(\C[x]\).
  Its domain~\(\Hilms[H]\)
  is dense in~\(\Hilms[H]_0\)
  in the graph norm of~\(x\),
  but not in the graph norm of~\(x^2\).
  So \(\cl{\pi(x)} = \cl{\pi_0(x)}\)
  and
  \(\cl{\pi(x^2)} \subsetneq \cl{\pi_0(x^2)} =
  \bigl(\cl{\pi(x)}\bigr)^2\).
\end{example}

All notions of integrability for representations of~\(\C[x]\) that
we shall consider imply \(\cl{\pi(x^n)} = \cl{\pi(x)}{}^n\).  Under
this assumption, an \emph{integrable} representation of~\(\C[x]\) is
determined by the single closed operator~\(\cl{\pi(x)}\).

Let \(B\defeq \Cont_0(\R)\).  Let~\(X\) be the identity function
on~\(\R\), viewed as an unbounded multiplier of~\(B\).  We define a
closed representation~\((\Hilms[B],\mu)\) of~\(A\)
on~\(\Cont_0(\R)\) by
\begin{equation}
  \label{eq:domain_Cx_on_Cont0R}
  \Hilms[B]\defeq \{f\in B\mid \forall n\colon X^n\cdot f\in B\}
  \quad\text{and}\quad
  \mu(x^n)f \defeq X^n\cdot f\quad\text{for }f\in\Hilms[B],\ n\in\N.
\end{equation}

\begin{theorem}
  \label{the:regular_self-adjoint_Cstar-hull}
  Let~\((\Hilms,\pi)\) be a representation of \(A=\C[x]\) on a
  Hilbert module~\(\Hilm\) over a \Cstar\nb-algebra~\(D\).  The
  following are equivalent:
  \begin{enumerate}
  \item \label{enum:regular_self-adjoint_Cstar-hull1}
    \(\pi=\mu\otimes_\varrho 1_{\Hilm}\) for a representation
    \(\varrho\colon B\to\Bound(\Hilm)\);
  \item \label{enum:regular_self-adjoint_Cstar-hull2}
    \(\cl{\pi(a)}\) is regular and self-adjoint for each \(a\in A_h
    \defeq \{a\in A\mid a=a^*\}\);
  \item \label{enum:regular_self-adjoint_Cstar-hull3}
    \(\cl{\pi(x^n)}\) is regular and self-adjoint for each
    \(n\in\N\);
  \item \label{enum:regular_self-adjoint_Cstar-hull4}
    \(\cl{\pi(x)}\) is regular and self-adjoint and \(\cl{\pi(x^n)}
    = \cl{\pi(x)}{}^n\) for all \(n\in\N\);
  \item \label{enum:regular_self-adjoint_Cstar-hull5}
    \(\cl{\pi(x)}\) is regular and self-adjoint and \(\Hilms =
    \bigcap_{n=1}^\infty \dom \cl{\pi(x)}{}^n\).
  \end{enumerate}
  Call representations with these equivalent properties integrable.
  The \Cstar\nb-algebra \(\Cont_0(\R)\) is a \Cstar\nb-hull for the
  integrable representations of~\(A\) with~\((\Hilms[B],\mu)\) as
  the universal integrable representation.
\end{theorem}

\begin{proof}
  If \(a\in A_h\), then~\(\cl{\mu(a)}\) is a self-adjoint,
  affiliated multiplier of~\(B\).  Hence \(\mu(a)\otimes_D1\) is a
  regular, self-adjoint operator on \(B\otimes_\varrho \Hilm \cong
  \Hilm\) for any representation~\(\varrho\) of~\(B\) on~\(\Hilm\)
  by \cite{Lance:Hilbert_modules}*{Proposition 9.10}.
  Thus~\ref{enum:regular_self-adjoint_Cstar-hull1}
  implies~\ref{enum:regular_self-adjoint_Cstar-hull2}.  The
  implication
  \ref{enum:regular_self-adjoint_Cstar-hull2}\(\Rightarrow\)\ref{enum:regular_self-adjoint_Cstar-hull3}
  is trivial.  The operator~\(\cl{\pi(x^n)}\) is always contained in
  the \(n\)-fold power~\(\cl{\pi(x)}{}^n\).  The latter is
  symmetric, and a proper suboperator of a symmetric operator cannot
  be self-adjoint.  Thus~\ref{enum:regular_self-adjoint_Cstar-hull3}
  implies~\ref{enum:regular_self-adjoint_Cstar-hull4}.
  The set~\(\{x^n\mid n\in\N\}\) is a strong generating set
  for~\(\C[x]\) by Lemma~\ref{lem:graph_topology_Cx}.
  Equation~\eqref{eq:domain_strong_generating_set} gives \(\Hilms =
  \bigcap_{n=1}^\infty \dom \cl{\pi(x^n)}\) for any (closed)
  representation.  Thus~\ref{enum:regular_self-adjoint_Cstar-hull4}
  implies~\ref{enum:regular_self-adjoint_Cstar-hull5}.

  Assume~\ref{enum:regular_self-adjoint_Cstar-hull5} and abbreviate
  \(t=\cl{\pi(x)}\).
  The functional calculus for~\(t\)
  is a nondegenerate \Star{}homomorphism
  \(\varrho\colon \Cont_0(\R) \to \Bound(\Hilm)\)
  (see \cite{Lance:Hilbert_modules}*{Theorem 10.9}).  Let~\(\pi'\)
  be the representation \(\mu\otimes_\varrho 1\)
  of~\(A\)
  on~\(\Hilm\)
  associated to~\(\varrho\).
  We claim that \(\pi=\pi'\).
  The functional calculus extends to affiliated multipliers and maps
  the identity function on~\(\R\)
  to the regular, self-adjoint operator~\(t\).
  This means that \(\cl{\pi'(x)} = t\).
  Then \(\cl{\pi'(x^n)} \subseteq t^n\).
  This implies \(\cl{\pi'(x^n)} = t^n\)
  because \(\cl{\pi'(x^n)}\)
  is self-adjoint and~\(t^n\)
  is symmetric.  Since the set \(\{x^n\mid n\in\N\}\)
  is a strong generating set for~\(\C[x]\),
  the domain of~\(\pi'\)
  is \(\bigcap \dom \cl{\pi'(x^n)} = \Hilms\)
  by condition~\ref{enum:regular_self-adjoint_Cstar-hull5}
  and~\eqref{eq:domain_strong_generating_set}.  On this domain,
  \(\pi(x)\)
  and~\(\pi'(x)\)
  act by the same operator because they have the same closure.  Thus
  \(\pi=\pi'\)
  and~\ref{enum:regular_self-adjoint_Cstar-hull5}
  implies~\ref{enum:regular_self-adjoint_Cstar-hull1}.  So all five
  conditions in the theorem are equivalent.

  To show that~\(B\) is a \Cstar\nb-hull for the class of
  representations described
  in~\ref{enum:regular_self-adjoint_Cstar-hull1}, we
  check~\ref{enum:Cstar-gen12} in
  Proposition~\ref{pro:Cstar-generated_by_unbounded_multipliers}.
  An integrable representation of~\(A\) on a Hilbert
  space~\(\Hilm[H]\) corresponds to a self-adjoint operator~\(t\)
  on~\(\Hilm[H]\) by~\ref{enum:regular_self-adjoint_Cstar-hull5}.
  An integrable subrepresentation is a closed
  subspace~\(\Hilm[H]_0\) of~\(\Hilm[H]\) with a self-adjoint
  operator~\(t_0\) on~\(\Hilm[H]_0\) whose graph is contained in
  that of~\(t\).  Since~\(t_0\) is self-adjoint, the subspaces
  \((t_0\pm\ima)(\dom(t_0)) = (t\pm\ima)(\dom(t_0))\) are equal
  to~\(\Hilm[H]_0\).  The Cayley transform~\(u\) of~\(t\) maps
  \((t+\ima)(\dom(t_0))\) onto \((t-\ima)(\dom(t_0))\).  Thus it
  maps~\(\Hilm[H]_0\) onto itself.  Since \(u-1\) generates the
  image of \(B=\Cont_0(\R)\) under the functional calculus, the
  projection onto~\(\Hilm[H]_0\) is \(B\)\nb-invariant.
\end{proof}

\begin{example}
  \label{exa:irregular_selfadjoint}
  Regularity and self-adjointness are independent properties of a
  symmetric operator.  Examples of regular symmetric operators that
  are not self-adjoint are easy to find,
  see~§\ref{sec:polynomials2}.  We are going to construct a
  representation~\(\pi\) of~\(\C[x]\) on a Hilbert module for
  which~\(\cl{\pi(a)}\) is self-adjoint for each \(a\in \C[x]\) with
  \(a=a^*\), but~\(\cl{\pi(x)}\) is not regular.  We follow the
  example after Théorème~1.3 in~\cite{Pierrot:Reguliers}, which
  Pierrot attributes to Hilsum.

  Let~\(\Hilm[H]\) be the Hilbert space~\(L^2([0,1])\) and let
  \(T_1\) and~\(T_2\) be the operators \(\ima \frac{\dd}{\dd x}\)
  on~\(\Hilm[H]\) with the following domains.  For~\(T_1\), we take
  \(1\)\nb-periodic smooth functions; for~\(T_2\), we take the
  restrictions to~\([0,1]\) of smooth functions on~\(\R\) satisfying
  \(f(x+1) = - f(x)\).  Both \(T_1\) and~\(T_2\) are essentially
  self-adjoint.  Let \(D\defeq \Cont([-1,1])\) and \(\Hilm\defeq
  \Cont([-1,1],\Hilm[H])\).  Let \(\Hilms\subseteq \Hilm\) be the
  dense subspace of all functions \(f\colon [-1,1]\times
  [0,1]\to\C\) such that \(\frac{\partial^n}{\partial^n x} f(t,x)\)
  is continuous for each \(n\in\N\),
  \begin{equation}
    \label{eq:irregular_selfadjoint_1}
    \frac{\partial^n}{\partial^n x} f(t,1)
    = \sign(t)\cdot \frac{\partial^n}{\partial^n x} f(t,0)
  \end{equation}
  for all \(t\in[-1,1]\), \(x\in\R\), \(t\neq0\), and
  \begin{equation}
    \label{eq:irregular_selfadjoint_0}
    \frac{\partial^n}{\partial^n x} f(0,0)
    = \frac{\partial^n}{\partial^n x} f(0,1) = 0.
  \end{equation}
  Equivalently, \(f(t,\blank)\) belongs to the domain
  of \(\cl{T_1}{}^n = \cl{T_1^n}\) for all \(n\in\N\), \(t\le 0\) and to the
  domain of \(\cl{T_2}{}^n = \cl{T_2^n}\) for all \(n\in\N\), \(t\ge 0\); indeed,
  this forces \(\frac{\partial^n}{\partial^n x} f\) to be continuous
  on \([-1,1]\times [0,1]\) and to satisfy the boundary
  conditions~\eqref{eq:irregular_selfadjoint_1}.  These
  imply~\eqref{eq:irregular_selfadjoint_0} by continuity.  Let
  \(x^n\in \C[x]\) act on~\(\Hilms\) by \(\biggl(\ima \frac{\dd}{\dd x}\biggr)^n\).
  This defines a closed \Star{}representation of~\(\C[x]\)
  on~\(\Hilm\) with \(\Hilms = \bigcap_{n\in\N} \dom
  \cl{\pi(x)}{}^n\).

  The closure~\(\cl{\pi(x)}\) is the irregular self-adjoint operator
  described in~\cite{Pierrot:Reguliers}.  Let \(a\in \C[x]\) with
  \(a=a^*\).  Then~\(\cl{\varrho(a)}\) is (regular and) self-adjoint
  for any \emph{integrable} representation~\(\varrho\) of~\(\C[x]\)
  by Theorem~\ref{the:regular_self-adjoint_Cstar-hull}.  Therefore,
  the restriction of~\(\cl{\pi(a)}\) to a single fibre of~\(\Hilm\)
  at some \(t\in [-1,1]\setminus\{0\}\) is a self-adjoint operator
  on~\(L^2([0,1])\) because \(\cl{T_1}\) and~\(\cl{T_2}\) are self-adjoint and
  \(\Hilms = \bigcap_{n\in\N} \dom \cl{\pi(x)}{}^n\).  The
  restriction of~\(\cl{\pi(a)}{}^*\) at \(t=0\) is contained in the
  self-adjoint operators \(\cl{a(T_1)}\) and~\(\cl{a(T_2)}\) by
  continuity.  We claim that \(\cl{a(T_1)} \cap \cl{a(T_2)} =
  \cl{\pi(a)}|_{t=0}\).  This claim implies that~\(\cl{\pi(a)}{}^*\)
  is contained in~\(\cl{\pi(a)}\), that is, \(\cl{\pi(a)}\) is
  self-adjoint.

  Let \(a\in\C[x]\) have degree~\(n\).  Then the graph norms for
  \(a\) and~\(x^n\) are equivalent in any representation by the
  proof of Lemma~\ref{lem:graph_topology_Cx}.  Hence \(\cl{a(T_i)}\)
  and~\(\cl{T_i^n}\) have the same domain.  The domain
  of~\(\cl{T_i^n}\) consists of functions \([0,1]\to\C\) whose
  \(n\)th derivative lies in~\(L^2\) and whose derivatives of order
  strictly less than~\(n\) satisfy the boundary condition
  for~\(T_i\).  Hence the domain of \(\cl{T_1^n} \cap \cl{T_2^n}\)
  consists of those functions \([0,1]\to\C\) whose \(n\)th
  derivative lies in~\(L^2\) and whose derivatives of order strictly
  less than~\(n\) vanish at the boundary points \(0\) and~\(1\).
  This is exactly the domain of the closure of \((T_1\cap T_2)^n =
  \pi(x^n)|_{t=0}\).  On this domain the operators \(\cl{a(T_1)}
  \cap \cl{a(T_2)}\) and \(\cl{\pi(a)}|_{t=0}\) both act by the
  differential operator \(a(\ima \frac{\dd}{\dd x})\).
\end{example}

The algebra \(A=\C[x]\) has many Hilbert space representations
coming from closed symmetric operators that are not self-adjoint.
There is, however, no larger admissible class of integrable
representations:

\begin{proposition}
  \label{pro:symmetric_no_hull}
  Assume that an admissible class of integrable representations of
  \(A=\C[x]\) contains all representations coming from self-adjoint
  Hilbert space operators.  Then any integrable representation
  of~\(A\) on a Hilbert module comes from a regular, self-adjoint
  operator.
\end{proposition}

\begin{proof}
  We first prove that there can be no more integrable Hilbert space
  representations than those coming from self-adjoint operators.
  Let \((\Hilms[H],\pi)\) be an integrable representation on a
  Hilbert space~\(\Hilm[H]\).  We may extend the closed symmetric
  operator \(t \defeq \cl{\pi(x)}\) on~\(\Hilm[H]\) to a
  self-adjoint operator~\(t_2\) on a larger Hilbert
  space~\(\Hilm[H]_2\).  This gives a representation~\(\pi_2\)
  of~\(A\) on~\(\Hilm[H]_2\) as in
  Theorem~\ref{the:regular_self-adjoint_Cstar-hull}, which is
  integrable by assumption.  The inclusion map
  \(\Hilm[H]\injto\Hilm[H]_2\) is an isometric intertwiner
  from~\(\pi\) to~\(\pi_2\).  Hence~\(\pi\) is a direct summand
  of~\(\pi_2\) by~\ref{enum:admissible3} in
  Definition~\ref{def:integrable_admissible}.
  Thus~\(\cl{\pi(x^n)}\) is self-adjoint for each \(n\in\N\),
  and~\(\pi\) is the representation induced by~\(t\).

  Now let \((\Hilms,\pi)\) be an integrable representation of~\(A\) on
  a Hilbert \(D\)\nb-module~\(\Hilm\).  For any Hilbert space
  representation \(\varphi\colon D\to\Bound(\Hilm[H])\), the induced
  representation of~\(A\) on the Hilbert space
  \(\Hilm\otimes_\varphi\Hilm[H]\) is also integrable
  by~\ref{enum:admissible1} in
  Definition~\ref{def:integrable_admissible}.  Thus
  \(\cl{\pi(x^n)\otimes_\varphi 1_{\Hilm[H]}}\) is self-adjoint for any
  Hilbert space representation \(\varphi\colon D\to\Bound(\Hilm[H])\).

  A closed, densely defined, symmetric operator~\(T\) on a Hilbert
  \(D\)\nb-module~\(\Hilm\) is self-adjoint and regular if and only
  if, for any state~\(\omega\) on~\(D\), the closure of \(T\otimes_D
  1\) on the Hilbert spaces \(\Hilm\otimes_D \Hilm[H]_\omega\) is
  self-adjoint; here~\(\Hilm[H]_\omega\) means the
  GNS-representation for~\(\omega\).  This is called the
  Local--Global Principle by Kaad and Lesch
  (\cite{Kaad-Lesch:Local_global}*{Theorem 1.1}); the result was
  first proved by Pierrot (\cite{Pierrot:Reguliers}*{Théorème
    1.18}).  We will take up Local--Global Principles more
  systematically in~§\ref{sec:local_global}.  Thus~\(\cl{\pi(x^n)}\)
  is regular and self-adjoint for each \(n\in\N\).  So~\(\pi\) is
  obtained from the regular self-adjoint operator~\(\cl{\pi(x)}\) as
  in Theorem~\ref{the:regular_self-adjoint_Cstar-hull}.
\end{proof}

\begin{example}
  \label{exa:smaller_integrable_than_regular_self-adjoint}
  There are many admissible classes of representations of~\(\C[x]\)
  that are smaller than the class in
  Theorem~\ref{the:regular_self-adjoint_Cstar-hull}.  There are even
  many such classes that contain the same Hilbert space
  representations.  For instance, let \(B\defeq \Cont_0((-\infty,0))
  \oplus \Cont_0([0,\infty))\) with the representation of
  polynomials by pointwise multiplication.  This is a \Cstar\nb-hull
  for a class of representations of~\(\C[x]\) by
  Theorem~\ref{the:commutative_hull} below.  Since the standard
  topologies on~\(\R\) and \((-\infty,0)\sqcup [0,\infty)\) have the
  same Borel sets, both \Cstar\nb-hulls \(\Cont_0((-\infty,0))
  \oplus \Cont_0([0,\infty))\) and \(\Cont_0(\R)\) give the same
  integrable Hilbert space representations because of the Borel
  functional calculus.  But there are regular,
  self-adjoint operators on Hilbert modules that do not give a
  \(B\)\nb-integrable representation.  The obvious example is the
  multiplier~\(X\) of~\(\Cont_0(\R)\) that generates the universal
  integrable representation of~\(\C[x]\).
\end{example}

Can there be an \emph{admissible} class of representations
of~\(\C[x]\) that contains some representation on a Hilbert space
that does not come from a self-adjoint operator?  We cannot rule
this out completely.  But such a class would have to be rather
strange.  By Proposition~\ref{pro:symmetric_no_hull}, it cannot
contain all self-adjoint operators.  By
Example~\ref{exa:Friedrichs_extension}, it cannot contain all
representations coming from positive symmetric operators because
then there would be isometric intertwiners among integrable
representations that are not \Star{}intertwiners.  The following
example rules out symmetric operators with one deficiency
index~\(0\):

\begin{example}
  \label{exa:deficiency_index_0}
  Let~\(t\) be a closed symmetric operator on a Hilbert
  space~\(\Hilm[H]\) of deficiency indices~\((0,n)\) for some
  \(n\in[1,\infty]\).  Then \(\dom^\infty(t) \defeq
  \bigcap_{n=1}^\infty \dom (t^n)\) is a core for each power~\(t^k\)
  by \cite{Schmudgen:Unbounded_book}*{Proposition 1.6.1}.  Thus
  there is a closed representation~\(\pi\) of~\(\C[x]\) with domain
  \(\dom^\infty(t)\) and \(\cl{\pi(x^k)} = t^k\) for all \(k\in\N\).
  By assumption, the operator \(t+\ima\) is surjective, but
  \(t-\ima\) is not.  That is, the Cayley transform \(c\defeq
  (t-\ima)(t+\ima)^{-1}\) is a non-unitary isometry.  The
  operator~\(t\) may be reconstructed from~\(c\) as in
  \cite{Lance:Hilbert_modules}*{Equation (10.11)}.  Here~\(c^*\) is
  surjective, so this simplifies to \(\dom(t) = (1-c)c^*\Hilm[H] =
  (1-c)\Hilm[H]\), and \(t(1-c)\xi = \ima (1+c)\xi\) for all
  \(\xi\in\Hilm[H]\).  Thus \(c(\dom t)\subseteq \dom t\) and \(c t
  \subseteq t c\) because
  \[
  c t\bigl((1-c)\xi\bigr)
  = \ima c(1+c)\xi
  = \ima (1+c)(c\xi)
  = t(1-c)(c\xi)
  = (t c)\bigl((1-c)\xi\bigr).
  \]
  Then \(c t^n \subseteq t^n c\) for all \(n\in\N\).  Thus~\(c\) is an
  isometric intertwiner from~\(\pi\) to itself by
  Proposition~\ref{pro:intertwiner_strong_generators}.  If~\(c^*\)
  were an intertwiner as well, then \(c^*(\dom t) \subseteq \dom t\)
  and \(c^*(t\pm\ima)\xi= (t\pm\ima)c^*\xi\) for all
  \(\xi\in\dom(t)\).  So
  \[
  c^* c (t+\ima)\xi
  = c^* (t-\ima)\xi
  = (t-\ima)c^* \xi
  = c (t+\ima)c^* \xi
  = c c^* (t+\ima)\xi.
  \]
  This is impossible because \(c^* c
  \neq c c ^*\) and \(t+\ima\) is surjective.  So the isometry~\(c\)
  is an intertwiner, but not a \Star{}intertwiner.  This is
  forbidden for admissible classes of integrable representations.

  If~\(t\) has deficiency indices~\((n,0)\) instead, then~\(-t\) has
  deficiency indices~\((0,n)\) and its Cayley transform is an
  isometric intertwiner that is not a \Star{}intertwiner by the
  argument above.
\end{example}

\section{Local--Global principles}
\label{sec:local_global}

\begin{definition}
  \label{def:local_global}
  Let~\(A\) be a \Star{}algebra with a weakly admissible class of
  integrable representations
  (Definition~\ref{def:integrable_admissible}).

  The \emph{Local--Global Principle} says that a
  representation~\(\pi\) of~\(A\) on a Hilbert
  \(D\)\nb-module~\(\Hilm\) is integrable if (and only if) the
  representations \(\pi\otimes_\varrho 1\) are integrable for all
  Hilbert space representations \(\varrho\colon
  D\to\Bound(\Hilm[H])\).

  The \emph{Strong Local--Global Principle} says that a
  representation~\(\pi\) of~\(A\) on a Hilbert
  \(D\)\nb-module~\(\Hilm\) is integrable if (and only if) the
  representations \(\pi\otimes_\varrho 1\) are integrable for all
  \emph{irreducible} Hilbert space representations \(\varrho\colon
  D\to\Bound(\Hilm[H])\).
\end{definition}

Roughly speaking, the Local--Global Principle says that the class of
integrable representations on Hilbert modules is determined by the
class of integrable representations on Hilbert spaces.  Examples
where the Local--Global Principle fails are constructed in
§\ref{sec:polynomials2} and~§\ref{sec:commutative_hulls}.  We do not
know an example with the Local--Global Principle for which the
Strong Local--Global Principle fails.

An irreducible representation \(\varrho\colon D\to\Bound(\Hilm[H])\)
is unitarily equivalent to the GNS-\alb{}representation for a pure
state~\(\psi\) on~\(D\).  The tensor product \(\Hilm\otimes_\varrho
\Hilm[H]\) is canonically isomorphic to the completion~\(\Hilm_\psi\)
of~\(\Hilm\) to a Hilbert space for the scalar-valued inner product
\(\braket{x}{y}_\C \defeq \psi(\braket{x}{y}_D)\).  The induced
representation \(\pi\otimes_D 1\) of~\(A\) on~\(\Hilm[H]_\psi\) is the
closure of the representation~\(\pi\) with domain \(\Hilms\subseteq
\Hilm\subseteq \Hilm_\psi\).

Any representation \(\varrho\colon D\to\Bound(\Hilm[H])\) is a
direct sum of cyclic representations, and these are
GNS-\alb{}representations of states.  Since any weakly admissible
class of integrable representations is closed under direct sums, the
Local--Global Principle holds if and only if integrability of
\(\pi\otimes_\varrho 1\) for all GNS-representations~\(\varrho\) of
states on~\(D\) implies integrability of~\(\pi\).

\begin{example}
  \label{exa:local-global_Cx}
  Define integrable representations of the polynomial
  algebra~\(\C[x]\) as in
  Theorem~\ref{the:regular_self-adjoint_Cstar-hull}.  Thus they
  correspond to regular, self-adjoint operators on Hilbert modules.
  The main result in~\cite{Kaad-Lesch:Local_global} says that the
  integrable representations of~\(\C[x]\) satisfy the Local--Global
  Principle.  This is where our notation comes from.  We already
  used this to prove Proposition~\ref{pro:symmetric_no_hull}.  The
  Strong Local--Global Principle for integrable representations
  of~\(\C[x]\) is only conjectured
  in~\cite{Kaad-Lesch:Local_global}.  This conjecture had already
  been proved by Pierrot in \cite{Pierrot:Reguliers}*{Théorème 1.18}
  before~\cite{Kaad-Lesch:Local_global} was written.  It is based on the
  following Hahn--Banach type theorem for Hilbert submodules:
\end{example}

\begin{theorem}[\cite{Pierrot:Reguliers}*{Proposition 1.16}]
  \label{the:Hahn-Banach_submodules}
  Let~\(D\) be a \Cstar\nb-algebra and let~\(\Hilm\) be a Hilbert
  \(D\)\nb-module.  Let \(\Hilm[F]\subsetneq \Hilm\) be a proper,
  closed Hilbert submodule.  There is an irreducible Hilbert
  space representation \(\varrho\colon D\to\Bound(\Hilm[H])\) with
  \(\Hilm[F]\otimes_\varrho \Hilm[H] \subsetneq \Hilm\otimes_\varrho
  \Hilm[H]\).
\end{theorem}

\begin{corollary}[\cite{Pierrot:Reguliers}*{Corollaire 1.17}]
  \label{cor:Hahn-Banach_submodules}
  Let~\(\Hilm\) be a Hilbert module over a \Cstar\nb-algebra~\(D\).
  Let \(\Hilm[F]_1,\Hilm[F]_2\subsetneq \Hilm\) be two closed Hilbert
  submodules.  If \(\Hilm[F]_1\neq\Hilm[F]_2\), then there is an
  irreducible Hilbert space representation \(\varrho\colon
  D\to\Bound(\Hilm[H])\) with \(\Hilm[F]_1\otimes_\varrho \Hilm[H]
  \neq \Hilm[F]_2\otimes_\varrho \Hilm[H]\) as closed subspaces in
  \(\Hilm\otimes_\varrho \Hilm[H]\).
\end{corollary}

\begin{corollary}[\cite{Pierrot:Reguliers}*{Théorème 1.18}]
  \label{cor:local_global_self-adjoint}
  Let~\(T\) be a closed, semiregular operator on a Hilbert
  \(D\)\nb-module~\(\Hilm\).  The operator~\(T\) is regular if and
  only if, for each irreducible representation \(\varrho\colon
  D\to\Bound(\Hilm[H])\) on a Hilbert space~\(\Hilm[H]\), the closures
  of \(T\otimes_\varrho1\) and \(T^*\otimes_\varrho1\)
  on~\(\Hilm\otimes_\varrho\Hilm[H]\) are adjoints of each other.

  Hence~\(T\) is regular and self-adjoint if and only if
  \(\cl{T\otimes_\varrho1}\) is a self-adjoint operator
  on~\(\Hilm\otimes_\varrho\Hilm[H]\) for each irreducible Hilbert
  space representation \(\varrho\colon D\to\Bound(\Hilm[H])\).
\end{corollary}

We now apply the above results of Pierrot.  First we deduce a
criterion for representations to be equal.  Then we prove that
certain definitions of integrability automatically satisfy the Strong
Local--Global Principle.

\begin{theorem}
  \label{the:identify_representations}
  Let~\(A\) be a \Star{}algebra and let~\(\pi_i\) for \(i=1,2\) be
  \textup{(}closed\textup{)} representations of~\(A\) on a Hilbert
  module~\(\Hilm\) over a \Cstar\nb-algebra~\(D\).  The following
  are equivalent:
  \begin{enumerate}
  \item \(\pi_1 = \pi_2\);
  \item \(\pi_1\otimes_\varrho \Hilm[H] = \pi_2\otimes_\varrho
    \Hilm[H]\) for each irreducible Hilbert space
    representation~\(\varrho\) of~\(D\);
  \item \(\cl{\pi_1(a)} = \cl{\pi_2(a)}\) for each \(a\in A\).
  \end{enumerate}
\end{theorem}

\begin{proof}
  The equivalence (3)\(\iff\)(1) is
  Proposition~\ref{pro:equality_if_closure_equal}, and~(1) clearly
  implies (2).  Thus we only have to
  prove that not~(3) implies not~(2).
  Assume that there is \(a\in A\) with \(\cl{\pi_1(a)} \neq
  \cl{\pi_2(a)}\).  The graphs \(\Gamma_1\) and~\(\Gamma_2\) of
  \(\cl{\pi_1(a)}\) and \(\cl{\pi_2(a)}\) are different Hilbert
  submodules of \(\Hilm\oplus\Hilm\).
  Corollary~\ref{cor:Hahn-Banach_submodules} gives an irreducible
  representation~\(\varrho\) of~\(D\) with \(\Gamma_1\otimes_\varrho
  \Hilm[H] \neq \Gamma_2\otimes_\varrho \Hilm[H]\).  This says that
  \(\cl{\pi_1(a)\otimes_\varrho 1_{\Hilm[H]}} \neq
  \cl{\pi_2(a)\otimes_\varrho 1_{\Hilm[H]}}\) because
  \(\Gamma_i\otimes_\varrho \Hilm[H]\) is the graph of
  \(\cl{\pi_i(a)\otimes_\varrho 1_{\Hilm[H]}}\).
\end{proof}

How do we specify which representations~\(\pi\) of a
\Star{}algebra~\(A\) are integrable?  There are two basically
different ways.  The ``universal way'' specifies the universal
integrable representation.  That is, it starts with a
representation~\((\Hilms[B],\mu)\) on a \Cstar\nb-algebra~\(B\) that
satisfies~\ref{enum:Cstar-gen1} in
Proposition~\ref{pro:Cstar-generated_by_unbounded_multipliers} and
takes the class of \(B\)\nb-integrable representations.  The
``operator way'' imposes conditions on the operators~\(\pi(a)\),
such as regularity and self-adjointness of~\(\cl{\pi(a)}\) or strong
commutation relations.

In good cases, the same class of integrable representations may be
specified in both ways.  For instance,
Theorem~\ref{the:regular_self-adjoint_Cstar-hull} shows that several
classes of representations of~\(\C[x]\) are equal.  The first is
defined by the universal representation on~\(\Cont_0(\R)\).  The
second asks~\(\cl{\pi(a)}\)
to be regular and self-adjoint for all \(a\in A_h\).

We are going to make the ``operator way'' more precise so that
\emph{all} classes of representations defined in this way satisfy
the Strong Local--Global Principle.  This is a powerful method
to prove Local--Global Principles.

\begin{definition}
  \label{def:restriction_representation}
  Let~\(A\) be a \Star{}algebra and~\(\Rep'(A)\) some weakly
  admissible class of representations of~\(A\) on Hilbert modules
  over \Cstar\nb-algebras.  A \emph{natural construction of Hilbert
    submodules} (of rank \(n\in\N_{\ge1}\)) associates to each
  representation~\(\pi\) on a Hilbert module~\(\Hilm\) that belongs
  to~\(\Rep'(A)\) a Hilbert submodule \(\Hilm[F](\pi) \subseteq
  \Hilm^n\), such that
  \begin{enumerate}
  \item \label{enum:restriction0} if \(u\colon
    \Hilm_1\congto\Hilm_2\) is a unitary \Star{}intertwiner between
    two representations \(\pi_1\) and~\(\pi_2\) in~\(\Rep'(A)\),
    then \(u^{\oplus n}\colon \Hilm_1^n \to \Hilm_2^n\) maps
    \(\Hilm[F](\pi_1)\) onto \(\Hilm[F](\pi_2)\);
  \item \label{enum:restriction1} let \(D_1\) and~\(D_2\) be
    \Cstar\nb-algebras and let~\(\Hilm[G]\) be a
    \(D_1,D_2\)-correspondence; let~\(\pi\) be a
    representation in~\(\Rep'(A)\) on a Hilbert
    \(D_1\)\nb-module~\(\Hilm\); then the canonical isomorphism
    \(\Hilm^n \otimes_{D_1} \Hilm[G] \congto (\Hilm \otimes_{D_1}
    \Hilm[G])^n\) maps \(\Hilm[F](\pi) \otimes_{D_1}
    \Hilm[G]\) onto \(\Hilm[F](\pi \otimes_{D_1}
    \Hilm[G])\);
  \item \label{enum:restriction2} if \(\pi_i\) for~\(i\) in a
    set~\(I\) are representations in~\(\Rep'(A)\) on Hilbert
    \(D\)\nb-modules~\(\Hilm_i\) over the same
    \Cstar\nb-algebra~\(D\), then the canonical isomorphism
    \(\bigl(\bigoplus \Hilm_i\bigr)^n \congto \bigoplus \Hilm_i^n\)
    maps \(\Hilm[F]\bigl(\bigoplus \pi_i\bigr)\) onto \(\bigoplus
    \Hilm[F](\pi_i)\).
  \end{enumerate}
  In brief, \(\Hilm[F](\pi) \subseteq \Hilm^n\) is compatible with
  unitary \Star{}interwiners, interior tensor products, and direct
  sums.

  A smaller class of representations \(\Rep''(A)\subseteq \Rep'(A)\)
  is \emph{defined by a submodule condition} if there are two
  natural constructions of Hilbert submodules \(\Hilm[F]_i(\pi)\),
  \(i=1,2\), of the same rank~\(n\), such that a
  representation~\(\pi\) in~\(\Rep'(A)\) belongs to~\(\Rep''(A)\) if
  and only if \(\Hilm[F]_1(\pi) = \Hilm[F]_2(\pi)\).

  A class of representations \(\Repi(A)\subseteq \Rep(A)\) is
  \emph{defined by submodule conditions} if it is defined by
  transfinite recursion by repeating the step in the previous
  paragraph.  More precisely, there are a well-ordered set~\(I\) with
  a greatest element~\(M\) and least element~\(0\) and subclasses
  \(\Rep^i(A)\subseteq \Rep(A)\) for \(i \in I\) such that
  \begin{enumerate}
  \item \(\Rep^0(A)=\Rep(A)\) and \(\Rep^M(A) = \Repi(A)\);
  \item \(\Rep^{i+1}(A)\subseteq \Rep^i(A)\) is defined by
    a submodule condition for each \(i\in I\);
  \item \(\Rep^i(A) = \bigcap_{i'<i}
    \Rep^{i'}(A)\) if \(i\neq 0\) and \(i\neq i'+1\)
    for all \(i'\in I\).
  \end{enumerate}
\end{definition}

The following lemma makes this definition meaningful, the following
theorem makes it interesting.

\begin{lemma}
  \label{lem:operator_conditions}
  If \(\Rep'(A)\subseteq \Rep(A)\) is weakly admissible and
  \(\Rep''(A)\subseteq \Rep'(A)\) is defined by a submodule
  condition, then \(\Rep''(A)\) is also weakly admissible.  If
  \((\Rep^i(A))_{i\in I}\) is a set of weakly admissible subclasses,
  then \(\bigcap_{i\in I} \Rep^i(A)\) is weakly admissible.  Any
  class of representations defined by submodule conditions is weakly
  admissible.
\end{lemma}

\begin{theorem}
  \label{the:operator_conditions}
  If \(\Repi(A)\subseteq \Rep(A)\) is defined by submodule conditions,
  then it satisfies the Strong Local--Global Principle.
\end{theorem}

Before we prove these two results, we give examples of classes of
representations defined by one or more submodule conditions, and a
few counterexamples.  These show that a class of integrable
representations defined in the operator way is often but not always
defined by submodule conditions.

\begin{example}
  \label{exa:regularity_condition}
  The \emph{regularity condition} for \(a\in A_h\)
  requires~\(\cl{\pi(a)}\) to be regular and self-adjoint.
  Equivalently, the closures of \((\pi(a)\pm \ima)(\Hilms)\) for
  both signs are dense in~\(\Hilm\); this is equivalent
  to~\(\cl{\pi(a)}\) having a unitary Cayley transform.
  Sending~\(\pi\) to the image of \(\cl{\pi(a)+\ima}\) or
  \(\cl{\pi(a)-\ima}\) is a natural construction of a Hilbert submodule.
  Hence the condition that~\(\cl{\pi(a)}\) is regular and
  self-adjoint is equivalent to the combination of two submodule
  conditions of rank~\(1\).

  Alternatively, we may proceed as in the definition of regularity
  for non-self-adjoint operators.
  Let~\(\Gamma(T)\) denote the closure of the graph of an
  operator~\(T\).  A closed operator~\(T\) is regular if and only if
  the direct sum of \(\Gamma(T)\) and~\(U_0(\Gamma(T^*))\) is
  \(\Hilm\oplus\Hilm\), where \(U_0(\xi_1,\xi_2) \defeq
  (\xi_2,-\xi_1)\).  If \(a\in A_h\), then regularity
  and self-adjointness of~\(\cl{\pi(a)}\) together are equivalent to
  the equality of
  \[
  \Hilm[F]_1(\pi) \defeq
  \Gamma(\pi(a)) \oplus U_0(\Gamma(\pi(a^*)))
  \quad\text{and}\quad
  \Hilm[F]_2(\pi)\defeq \Hilm\oplus\Hilm.
  \]
  We claim that \(\Hilm[F]_1\) and~\(\Hilm[F]_2\) are natural
  constructions of Hilbert submodules of rank~\(2\).  This is
  trivial for~\(\Hilm[F]_2\).  That~\(\Hilm[F]_1\) is compatible
  with unitary intertwiners and direct sums is an easy exercise.
  The construction~\(\Hilm[F]_1\) is compatible with interior tensor
  products because the graph of \(\cl{(\pi\otimes_{D_1}
    1_{\Hilm[G]})(a)}\) is \(\Gamma(\pi(a))\otimes_{D_1} \Hilm[G]\).

  For instance, \ref{enum:regular_self-adjoint_Cstar-hull2} in
  Theorem~\ref{the:regular_self-adjoint_Cstar-hull} defines
  integrable representations of~\(\C[x]\) by regularity conditions.
  We generalise this in
  Theorem~\ref{the:regular_admissible_local-global} below.
\end{example}

\begin{example}
  \label{exa:alone_regularity}
  The class of representations where~\(\cl{\pi(a)}\) is regular for
  some \(a\in A\) is always weakly admissible by
  \cite{Lance:Hilbert_modules}*{Proposition 9.10}.  The first
  example in~§\ref{sec:polynomials2} shows a class of
  representations defined by such a condition that does not satisfy
  the Local--Global Principle, in contrast to
  Theorem~\ref{the:operator_conditions}.  Hence asking
  for~\(\cl{\pi(a)}\) to be regular for some \(a\in A\) cannot be a
  submodule condition.  The problem is that the inclusion
  \(\Gamma(\pi(a)^*)\otimes_D 1_{\Hilm[G]} \subseteq
  \Gamma\bigl((\pi(a)\otimes_D 1_{\Hilm[G]})^*\bigr)\) for a
  correspondence~\(\Hilm[G]\) may be strict.
\end{example}




\begin{example}
  \label{exa:strong_commutation}
  Let \(a_1,a_2\in A_h\) and suppose that \(t_1\defeq
  \cl{\pi(a_1)}\) and \(t_2\defeq \cl{\pi(a_2)}\) are self-adjoint,
  regular operators for all representations in~\(\Rep'(A)\); we may
  achieve this by submodule conditions as in
  Example~\ref{exa:regularity_condition} in previous steps of a
  recursive definition.  We say that \(t_1\) and~\(t_2\)
  \emph{strongly commute} if their Cayley transforms \(u_1\)
  and~\(u_2\) commute.  Equivalently, \(u_1\) commutes with~\(t_2\),
  that is, \(u_1 t_2 u_1^* = t_2\).  The graphs of \(t_2\) and \(u_1
  t_2 u_1^*\) are natural constructions of Hilbert submodules
  of rank~\(2\).  Therefore, strong commutation of
  \(\cl{\pi(a_1)}\) and~\(\cl{\pi(a_2)}\) is a submodule condition.
\end{example}

\begin{example}
  \label{exa:nondegeneracy_condition}
  Let \(I\idealin A\) be an ideal.  A \emph{nondegeneracy condition}
  for~\(I\) asks the closed linear span of~\(\pi(I)\Hilms\) to be
  all of~\(\Hilm\); here~\(\Hilms\) is the domain of~\(\pi\).  This
  means that \(\Hilm[F]_1(\pi)\), the closed linear span of
  \(\pi(a)\xi\) for \(a\in I\), \(\xi\in\Hilms\), is equal to
  \(\Hilm[F]_2(\pi) = \Hilm\).  These are natural
  constructions of Hilbert submodules of rank~\(1\).  So a
  nondegeneracy condition is a submodule condition.

  For instance, let~\(I\) be a non-unital \Star{}algebra and let
  \(A=\tilde{I}\) be its unitisation.  Any representation of~\(I\)
  extends uniquely to a unital representation of~\(A\).  The class
  of nondegenerate representations of~\(I\) inside the class of all
  representations of~\(A\) is defined by a submodule condition.

  More generally, let \(V_1,V_2\subseteq A\) be vector subspaces and
  ask the closed linear spans of \(\pi(a)\xi\) for \(a\in V_j\),
  \(\xi\in\Hilms\) to be equal for \(j=1,2\).  This is a submodule
  condition as well.  For instance, the condition
  \(\cl{\pi(a+\ima) \Hilms} = \Hilm\) for \(a\in A_h\) is of this
  form.  It holds if and only if the Cayley transform
  of~\(\cl{\pi(a)}\) is an isometry (possibly without adjoint).

  Often we need a mild generalisation of the above construction, see
  Example~\ref{exa:spectral_condition} below.  Suppose that we have
  constructed a representation \(\varphi(\pi)\) of a unital
  \Star{}algebra~\(A'\) on~\(\Hilm\) for any representation~\(\pi\)
  in~\(\Rep'(A)\), such that \(\pi\mapsto\varphi(\pi)\) is
  compatible with unitary \Star{}intertwiners, direct sums, and
  interior tensor products; the last property means that
  \(\varphi(\pi\otimes_{D_1} 1_{\Hilm[G]}) = \varphi(\pi)\otimes_{D_1}
  1_{\Hilm[G]}\) as representations on \(\Hilm\otimes_{D_1}
  \Hilm[G]\).  Then we may ask the nondegeneracy condition for an
  ideal in~\(A'\) instead.  In particular, \(A'\) may be a weak
  \Cstar\nb-hull for some class of representations
  containing~\(\Rep'(A)\).
\end{example}

\begin{example}
  \label{exa:spectral_condition}
  Let \(a_1,\dotsc,a_n\in A_h\) be commuting, symmetric elements and
  suppose that~\(\cl{\pi(a_j)}\) for \(j=1,\dotsc,n\) are strongly
  commuting, self-adjoint, regular operators for all representations
  in~\(\Rep'(A)\); we may achieve all this by previous submodule
  conditions as in Examples \ref{exa:regularity_condition}
  and~\ref{exa:strong_commutation}.  A \emph{closed spectral
    condition} asks the joint spectrum of
  \(\cl{\pi(a_1)},\dotsc,\cl{\pi(a_n)}\) to be contained in a closed
  subset \(X\subseteq\R^n\).

  We claim that this is a submodule condition.  Under our assumptions,
  the functional calculus \(\Phi\colon \Cont_0(\R^n)\to\Bound(\Hilm)\)
  exists.  Our spectral condition means that
  \(\Phi(\Cont_0(\R^n\setminus X))\Hilm=0\).
  The construction of~\(\Phi\)
  is clearly compatible with unitary \Star{}intertwiners and direct
  sums.  It is also compatible with interior tensor products, that is,
  the functional calculus for
  \(\cl{\pi\otimes 1(a_1)},\dotsc,\cl{\pi\otimes 1(a_n)}\)
  maps \(f\mapsto \Phi(f)\otimes 1\).
  Hence \(\Phi(\Cont_0(\R^n\setminus X))\Hilm\)
  is a naturally constructed Hilbert submodule of~\(\Hilm\).
  So our spectral condition for closed \(X\subseteq\R^n\)
  is a submodule condition.

  More generally, let \(X\subseteq \R^n\) be locally closed, that
  is, \(X\) is relatively open in its closure~\(\cl{X}\).  Suppose
  that the spectral condition for~\(\cl{X}\) holds for all
  representations in~\(\Rep'(A)\), say, by previous recursion steps.
  Then the functional calculus homomorphism for
  \(\cl{\pi(a_1)},\dotsc,\cl{\pi(a_n)}\) exists and descends to
  \(\Cont_0(\cl{X})\).  The \emph{spectral condition} for~\(X\) asks
  the restriction of this homomorphism to the ideal
  \(\Cont_0(X)\idealin \Cont_0(\cl{X})\) to be nondegenerate.  This
  is a submodule condition by
  Example~\ref{exa:nondegeneracy_condition}.
\end{example}

\begin{example}[see \cite{Woronowicz:Unbounded_affiliated}*{§3}]
  \label{exa:Weyl_regular}
  Let~\(A_\mu\) for some \(\mu\in\R\setminus\{0\}\) be the unital
  \Star{}algebra generated by two elements~\(v,n\) with the
  relations \(v^* v = v v^* = 1\), \(n^* n = n n^*\), \(v^* n v =
  \mu n\).  This is the algebra of polynomial functions on the
  quantum group~\(\textup{E}_\mu(2)\).  The relations allow to write
  any element as a linear combination of \(v^k\cdot g(n,n^*)\) for
  \(k\in\Z\) and a polynomial~\(g\).  It follows that the graph
  topology of a representation of~\(A_\mu\) is generated by the
  graph norms of~\((n^* n)^k\) for \(k\in\N\).  Thus a
  representation is closed if and only if its domain is
  \(\bigcap_{k=0}^\infty \cl{\pi((n^* n)^k)}\), compare the proof
  of~\eqref{eq:domain_strong_generating_set}.

  The \Cstar\nb-algebra of~\(\textup{E}_\mu(2)\) is a \Cstar\nb-hull
  for a certain class of integrable representations of~\(A_\mu\)
  that is defined by submodule conditions.  First, we
  require~\(\cl{\pi(n)}\) to be a regular, normal operator;
  equivalently, \(\cl{\pi(n+n^*)}\) and \(-\ima \cl{\pi(n-n^*)}\)
  are regular and self-adjoint, and they strongly commute; these are
  submodule conditions by Examples \ref{exa:regularity_condition}
  and~\ref{exa:strong_commutation}.  Secondly, we require the
  spectrum of~\(\cl{\pi(n)}\) (or the joint spectrum of its real and
  imaginary part) to be contained in \(X_\mu \defeq \{z\in\C\mid
  \abs{z} \in \mu^\Z\}\cup\{0\}\); this is a submodule condition by
  Example~\ref{exa:spectral_condition}.  Finally, we require
  \(\cl{\pi(n^* n)^k}\) to be regular and self-adjoint for all
  \(k\ge1\).  These are submodule conditions by
  Example~\ref{exa:regularity_condition}.

  We claim that an integrable representation on~\(\Hilm\) is
  equivalent to a pair~\((V,N)\) consisting of a unitary
  operator~\(V\) and a regular, normal operator~\(N\) on~\(\Hilm\)
  with spectrum contained in~\(X_\mu\), subject to the relation
  \(V^* N V = \mu N\).  First, any such pair~\((V,N)\) gives an
  integrable representation of~\(A_\mu\) with domain
  \(\bigcap_{k=0}^\infty \dom(N^k)\).  Conversely, if~\(\pi\) is an
  integrable representation, then let \(N\defeq \cl{\pi(n)}\), \(V\defeq
  \cl{\pi(v)}\).  These have the properties required above.  Since
  \(\cl{\pi((n^* n)^k)}\) is self-adjoint and contained in the
  symmetric operator~\((N^* N)^k\), we must have \(\cl{\pi((n^*
    n)^k)} = (N^* N)^k\).  So the domain of the representation
  of~\(A_\mu\) is \(\bigcap_{k=0}^\infty \dom(N^k)\).

  The regular, normal operator~\(N\)
  with spectrum in~\(X_\mu\)
  defines a functional calculus~\(\varrho\)
  on \(\Cont_0(X_\mu)\).
  The commutation relation \(v^* n v = \mu n\)
  is equivalent to \(V^* \varrho(f) V = \varrho(\alpha(f))\)
  for the automorphism \(\alpha(f)(x) \defeq f(\mu x)\)
  on \(\Cont_0(X_\mu)\).
  As a consequence, the crossed product \Cstar\nb-algebra
  \(\Cont_0(X_\mu)\rtimes_\alpha \Z\)
  is a \Cstar\nb-hull for our class of integrable representations.

  By the way, this also follows from our Induction Theorem.  For this,
  we give~\(A_\mu\)
  the unique \(\Z\)\nb-grading
  where~\(v\)
  has degree~\(1\)
  and~\(n\)
  has degree~\(0\).
  Then \((A_\mu)_0 = \C[n,n^*]\),
  and we call a representation of~\(\C[n,n^*]\)
  integrable if~\(n\)
  is regular and normal with spectrum contained in~\(X_\mu\).
  The \(\Cst\)\nb-hull
  for this class of integrable representations of~\(\C[n,n^*]\)
  is~\(\Cont_0(X_\mu)\).
  In this case, all representations of~\(\C[n,n^*]\)
  are inducible to~\(A_\mu\),
  and the induced \(\Cst\)\nb-hull
  for~\(A_\mu\) is \(\Cont_0(X_\mu)\rtimes_\alpha \Z\).
\end{example}

Interesting classes of representations defined by submodule
conditions occur in Theorems \ref{the:integrable_Ug}
and~\ref{the:commutative_regular_admissible}.  The examples in
\cites{Savchuk-Schmudgen:Unbounded_induced, Dowerk-Savchuk:Induced}
are also defined by submodule conditions, compare
Proposition~\ref{pro:integrable_induced_nice}.

\begin{example}
  \label{exa:continuity_condition}
  If the algebra~\(A\) carries a topology, then we may restrict
  attention to representations of~\(A\) that are continuous in some
  sense.  For instance, if~\(G\) is a topological group and \(A=\C[G]\)
  is the group ring of the underlying discrete group, then
  representations of~\(A\) are unitary representations of~\(G\),
  possibly discontinuous.  Among them, we may restrict to the continuous
  representations (compare the definition of a host algebra for~\(G\)
  in~\cite{Grundling-Neeb:Full_regularity}).  If~\(G\) is an
  infinite-dimensional Lie group, we may restrict further to
  representations of~\(\C[G]\) that are smooth in the sense that the
  smooth vectors are dense.  I do not expect continuity or smoothness to
  be a submodule condition, and I do not know when the classes of
  continuous or smooth representations satisfy the Local--Global
  Principle or its strong variant.

  Semiboundedness conditions ask for certain (regular) self-adjoint
  operators to be bounded above,
  see~\cite{Neeb:Semibounded_Hilbert_loop}.  If we specify the upper
  bound on the spectrum, this is a spectral condition as in
  Example~\ref{exa:spectral_condition}.  When we let the upper bound
  go to~\(\infty\), however, then direct sums no longer preserve
  semiboundedness.  Therefore, semiboundedness conditions seem close
  enough to submodule conditions to be tractable, but the details
  require further thought.
\end{example}

\begin{proof}[Proof of Lemma~\textup{\ref{lem:operator_conditions}}]
  Let \(\Hilm[F]_1\) and~\(\Hilm[F]_2\) be natural constructions of
  Hilbert submodules of rank~\(n\) that define \(\Rep''(A)\)
  inside~\(\Rep'(A)\), and let~\(\Rep'(A)\) be weakly admissible.
  Let~\(\pi_i\) for \(i=1,2\) be representations on Hilbert
  \(D\)\nb-modules~\(\Hilm_i\) for a \Cstar\nb-algebra~\(D\)
  that belong to~\(\Rep'(A)\).
  Let \(u\colon \Hilm_1\congto\Hilm_2\) be a unitary
  \Star{}intertwiner from~\(\pi_1\) to~\(\pi_2\).  If
  \(\Hilm[F]_1(\pi_1) = \Hilm[F]_2(\pi_1)\), then
  \(\Hilm[F]_1(\pi_2) = u^{\oplus n}\bigl( \Hilm[F]_1(\pi_1)\bigr) =
  u^{\oplus n}\bigl( \Hilm[F]_2(\pi_1)\bigr) = \Hilm[F]_2(\pi_2)\).
  Thus~\(\pi_2\) belongs to~\(\Rep''(A)\) if~\(\pi_1\) does.  This
  verifies~\ref{enum:admissible0} in
  Definition~\ref{def:integrable_admissible}
  using~\ref{enum:admissible0} in
  Definition~\ref{def:restriction_representation}.  Similarly,
  \ref{enum:restriction1} and~\ref{enum:restriction2} in
  Definition~\ref{def:restriction_representation} show that
  \(\Rep''(A)\) inherits \ref{enum:admissible1}
  and~\ref{enum:admissible2} in
  Definition~\ref{def:integrable_admissible} from~\(\Rep'(A)\).
  Thus \(\Rep''(A)\) is again weakly admissible.

  It is trivial that weak admissibility is hereditary for
  intersections.  By transfinite induction, it follows that any
  class of representations defined by submodule conditions is weakly
  admissible.
\end{proof}

\begin{proof}[Proof of Theorem~\textup{\ref{the:operator_conditions}}]
  Let \(\Hilm[F]_1\) and~\(\Hilm[F]_2\) be natural constructions of
  Hilbert submodules of rank~\(n\) that define \(\Rep''(A)\)
  inside~\(\Rep'(A)\), and assume that~\(\Rep'(A)\) satisfies the
  Strong Local--Global Principle.  Let~\(\pi\) be a representation
  on a Hilbert \(D\)\nb-module~\(\Hilm\) that does not belong
  to~\(\Rep''(A)\).  We must find an irreducible
  representation~\(\varrho\) of~\(D\) on a Hilbert
  space~\(\Hilm[H]\) such that \(\pi\otimes_\varrho \Hilm[H]\) does
  not belong to~\(\Rep''(A)\).  If the representation does not even
  belong to~\(\Rep'(A)\), this is possible because~\(\Rep'(A)\)
  satisfies the Strong Local--Global Principle by assumption.  So we
  may assume that~\(\pi\) belongs to~\(\Rep'(A)\) but not
  to~\(\Rep''(A)\).  Thus \(\Hilm[F]_1(\pi)\) and
  \(\Hilm[F]_2(\pi)\) are well defined and different Hilbert
  submodules of~\(\Hilm^n\).
  Corollary~\ref{cor:Hahn-Banach_submodules} gives an irreducible
  representation~\(\varrho\) of~\(D\) on a Hilbert
  space~\(\Hilm[H]\) such that
  \(\Hilm[F]_1(\pi)\otimes_\varrho \Hilm[H] \neq
  \Hilm[F]_2(\pi)\otimes_\varrho \Hilm[H]\) as closed
  subspaces of \(\Hilm^n\otimes_\varrho \Hilm[H]\).  Identify these
  with subspaces of \((\Hilm\otimes_\varrho \Hilm[H])^n\).
  The condition~\ref{enum:restriction1} in
  Definition~\ref{def:restriction_representation} gives
  \[
  \Hilm[F]_1(\pi\otimes_\varrho \Hilm[H]) =
  \Hilm[F]_1(\pi)\otimes_\varrho \Hilm[H] \neq
  \Hilm[F]_2(\pi)\otimes_\varrho \Hilm[H] =
  \Hilm[F]_2(\pi\otimes_\varrho \Hilm[H]).
  \]
  That is, \(\pi\otimes_\varrho \Hilm[H]\) does not belong
  to~\(\Rep''(A)\).  Thus~\(\Rep''(A)\) inherits the Strong
  Local--Global Principle from~\(\Rep'(A)\).

  The Strong Local--Global Principle is easily seen to be hereditary
  for intersections.  Hence any class of representations defined by
  submodule conditions satisfies the Strong Local--Global Principle
  by transfinite induction.
\end{proof}

\begin{theorem}
  \label{the:regular_admissible_local-global}
  Let~\(A\) be a \Star{}algebra and let \(S\subseteq A_h\).  Let
  \(\Rep^S(A)\) be the class of all representations where the elements
  of~\(S\) act by regular, self-adjoint operators.  This class is
  defined by submodule conditions and hence satisfies the Strong
  Local--Global Principle.  It is admissible if~\(S\) is a strong
  generating set for~\(A\).
\end{theorem}

\begin{proof}
  Asking~\(\cl{\pi(a)}\) to be regular and self-adjoint for a single
  \(a\in S\) is a submodule condition by
  Example~\ref{exa:regularity_condition}.  In order to ask this
  simultaneously for a set~\(S\), let~\(\prec\) be a
  well-ordering on~\(S\), and add an element~\(M\) with
  \(a\prec M\) for all \(a\in S\).  Let \(\Rep^a(A)\subseteq
  \Rep(A)\) for \(a\in S\cup\{M\}\) be the class of all
  representations~\(\pi\) where~\(\cl{\pi(b)}\) is regular and
  self-adjoint for all \(b\in S\) with \(b\prec a\).  These
  subclasses form a recursive definition of \(\Rep^S(A)\) by
  submodule conditions as in
  Definition~\ref{def:restriction_representation}.
  Thus~\(\Rep^S(A)\) is defined by submodule conditions.  Then it is
  weakly admissible and satisfies the Strong Local--Global Principle
  by Lemma~\ref{lem:operator_conditions} and
  Theorem~\ref{the:operator_conditions}.

  From now on, we assume that~\(S\) is a strong generating set.
  For~\(\Rep^S(A)\) to be admissible, we must prove that any
  isometric intertwiner \(I\colon (\Hilms_0,\pi_0)\injto
  (\Hilms,\pi)\) between two Hilbert space representations in
  \(\Rep^S(A)\) is a \Star{}intertwiner.

  If \(a\in S\), then \(\cl{\pi(a)}\) and \(\cl{\pi_0(a)}\) are
  regular, self-adjoint operators.  Hence they generate integrable
  representations of~\(\C[x]\) as in
  Theorem~\ref{the:regular_self-adjoint_Cstar-hull}.  The
  isometry~\(I\) intertwines these representations of~\(\C[x]\).
  Hence it is a \Star{}intertwiner by
  Theorem~\ref{the:regular_self-adjoint_Cstar-hull}.  In particular,
  \(I^*\) maps \(\dom \cl{\pi(a)}\) to \(\dom \cl{\pi_0(a)}\) for
  each \(a\in S\).  Since~\(S\) is a strong generating set,
  \eqref{eq:domain_strong_generating_set} gives \(\Hilms_0 =
  \bigcap_{a\in S} \dom \cl{\pi_0(a)}\) and similarly for~\(\pi\).
  So \(I^*(\Hilms) \subseteq \Hilms_0\).  Then \(\Hilms = \Hilms_0 +
  (\Hilms\cap \Hilms_0^\bot)\) and~\(I\) is a \Star{}intertwiner by
  Proposition~\ref{pro:isometry_star-intertwiner}.
\end{proof}

\begin{corollary}
  \label{cor:weak_Cstar-hull_full}
  Let \(S\subseteq A_h\) be a strong generating set for a
  \Star{}algebra~\(A\) and let~\(B\) with universal
  representation~\(\mu\) be a weak \Cstar\nb-hull.  If the closed
  multipliers~\(\cl{\mu(a)}\) for \(a\in S\) are self-adjoint and
  affiliated with~\(B\), then~\(B\) is a \Cstar\nb-hull.
\end{corollary}

\begin{proof}
  All \(B\)\nb-integrable representations belong to~\(\Rep^S(A)\)
  because the latter is weakly admissible and contains the universal
  \(B\)\nb-integrable representation.  Since \(\Rep^S(A)\) is
  admissible by Theorem~\ref{the:regular_admissible_local-global},
  any smaller class of integrable representations inherits the
  equivalent conditions
  \ref{enum:Cstar-gen2}--\ref{enum:Cstar-gen2''} in
  Proposition~\ref{pro:Cstar-generated_by_unbounded_multipliers},
  which characterise \Cstar\nb-hulls among weak \Cstar\nb-hulls.
\end{proof}

\begin{theorem}
  \label{the:Woronowicz_local-global}
  Let~\(A\) be a \Star{}algebra, \(B\) a \Cstar\nb-algebra,
  \((\Hilms[B],\mu)\) a representation of~\(A\) on~\(B\), and
  \(T_1,\dotsc,T_n\in A\).  Assume that
  \(\cl{\mu(T_1)},\dotsc,\cl{\mu(T_n)}\) are self-adjoint and
  affiliated with~\(B\) and generate~\(B\) in the sense of
  Woronowicz, see \cite{Woronowicz:Cstar_generated}*{Definition
    3.1}.  Then~\(B\) is a \Cstar\nb-hull for the
  \(B\)\nb-integrable representations of~\(A\) defined
  by~\((\Hilms[B],\mu)\), and these satisfy the Strong Local--Global
  Principle.
\end{theorem}

\begin{proof}
  To show that~\(B\) is a \Cstar\nb-hull, we check the
  condition~\ref{enum:Cstar-gen12} in
  Proposition~\ref{pro:Cstar-generated_by_unbounded_multipliers}.
  Let \(\varrho\colon B\to\Bound(\Hilm[H])\) be a representation
  of~\(B\) on a Hilbert space~\(\Hilm[H]\) and let \((\Hilms,\pi)\)
  be the corresponding \(B\)\nb-integrable representation of~\(A\).
  Let \((\Hilms_0,\pi|_{\Hilms_0})\) be a \(B\)\nb-integrable
  representation on a closed subspace \(\Hilm_0\subseteq \Hilm\) and
  let \(P\in\Bound(\Hilm)\) be the projection onto~\(\Hilm_0\).  We
  must show that~\(\varrho(B)\) is contained in the commutant
  of~\(P\).  Equivalently, \(\varrho\) is a morphism in the notation
  of~\cite{Woronowicz:Cstar_generated} to the algebra
  \(K = \Comp(\Hilm_0) \oplus \Comp(\Hilm_0^\bot)\) of all compact
  operators on~\(\Hilm\) that commute with~\(P\).

  Let \(1\le i \le n\).  Since \(T_i\) is self-adjoint and regular
  as an adjointable operator on the Hilbert \(B\)\nb-module~\(B\),
  it generates an integrable representation of the polynomial
  algebra~\(\C[x]\) on~\(B\) as in
  Theorem~\ref{the:regular_self-adjoint_Cstar-hull}.  These
  integrable representations form an admissible class.  Therefore, a
  \(B\)\nb-integrable representation of~\(A\) gives an integrable
  representation of~\(\C[x]\) when we compose with the canonical map
  \(j_i\colon \C[x]\to A\), \(x\mapsto T_i\), and take the closure.
  And since \(\pi\) and~\(\pi|_{\Hilms_0}\) are both
  \(B\)\nb-integrable, \(\cl{\pi|_{\Hilms_0}\circ j_i}\) is a direct
  summand in~\(\cl{\pi\circ j_i}\).  Equivalently, the unbounded
  operator~\(\cl{\pi(T_i)}\) is affiliated with~\(K\).

  The extension of~\(\varrho\) to affiliated multipliers
  maps~\(\cl{\mu(T_i)}\) to~\(\cl{\pi(T_i)}\), which is affiliated
  with~\(K\).  Hence~\(\varrho\) is a morphism to~\(K\) because
  these affiliated multipliers generate~\(B\).  Thus~\(B\) is a
  \Cstar\nb-hull for the \(B\)\nb-integrable representations by
  Proposition~\ref{pro:Cstar-generated_by_unbounded_multipliers}.

  Now we check the Strong Local--Global Principle.
  Let~\((\Hilms,\pi)\) be a representation of~\(A\) on a Hilbert
  \(D\)\nb-module~\(\Hilm\).  Assume that the representation
  \((\Hilms,\pi)\otimes_\omega \Hilm[H]_\omega\) is integrable for
  each irreducible representation~\(\omega\) of~\(D\) on a Hilbert
  space~\(\Hilm[H]_\omega\) in the sense that it comes from a
  representation of~\(B\).  We must show that the
  representation~\((\Hilms,\pi)\) is integrable.

  The condition that~\(\cl{\pi(T_i)}\) be self-adjoint and regular
  is a submodule condition by
  Example~\ref{exa:regularity_condition}.  Hence the class of
  representations with this property satisfies the Strong
  Local--Global Principle by Theorem~\ref{the:operator_conditions}.
  Therefore, \(\cl{\pi(T_i)}\) is a regular, self-adjoint operator
  on~\(\Hilm\) for \(i=1,\dotsc,n\).

  Let~\(\omega\) be the direct sum of all irreducible
  representations of~\(D\); this is a faithful representation
  of~\(D\) on some Hilbert space~\(\Hilm[H]\).  The induced
  representation~\(j\) of~\(\Comp(\Hilm)\) on \(\Hilm[K]\defeq
  \Hilm\otimes_D \Hilm[H]\) is faithful as well.  By assumption, the
  representation \(\pi\otimes_D1\) of~\(A\) on~\(\Hilm[K]\) is
  integrable, so it comes from a representation~\(\sigma\) of~\(B\).
  The extension of~\(\sigma\) to affiliated multipliers maps
  \(\cl{\mu(T_i)}\affil B\) to \(\cl{(\pi\otimes_D1)(T_i)}\).
  Since~\(\cl{\pi(T_i)}\) is a regular operator on~\(\Hilm\), it is
  an affiliated multiplier of~\(\Comp(\Hilm)\),
  see~\cite{Pal:Regular} or
  Proposition~\ref{pro:representation_Hilm_Comp}.  Thus
  \(\cl{(\pi\otimes_D1)(T_i)}\) is affiliated with the image
  of~\(\Comp(\Hilm)\) on~\(\Hilm[K]\) by
  \cite{Lance:Hilbert_modules}*{Proposition 9.10}.  Thus
  \(\sigma(\cl{\mu(T_i)}) \affil \Comp(\Hilm)\) for
  \(i=1,\dotsc,n\).  Since the affiliated
  multipliers~\(\cl{\mu(T_i)}\) generate~\(B\) in the sense of
  Woronowicz, \(\sigma\) factors through a morphism \(\tau\colon
  B\to\Comp(\Hilm)\).  This is the same as a representation of~\(B\)
  on~\(\Hilm\).  Let~\(\pi'\) be the representation of~\(A\)
  on~\(\Hilm\) associated to~\(\tau\).  If~\(\varrho\) is an
  irreducible Hilbert space representation of~\(D\), then
  \(\pi\otimes_\varrho \Hilm[H] = \pi'\otimes_\varrho \Hilm[H]\) by
  construction of~\(\tau\).  Hence
  Theorem~\ref{the:identify_representations} gives \(\pi=\pi'\).
  Since~\(\pi'\) is integrable by construction, so is~\(\pi\).
\end{proof}

The first counterexample in~§\ref{sec:polynomials2} exhibits a
\emph{symmetric} affiliated multiplier that generates a
\Cstar\nb-algebra, such that the Local--Global Principle fails
and~\(B\) is not a \Cstar\nb-hull.  Without self-adjointness, we
only get the following much weaker statement:

\begin{lemma}
  \label{lem:Woronowicz_weak_hull}
  Let~\(A\) be a \Star{}algebra, \(B\) a \Cstar\nb-algebra,
  \((\Hilms[B],\mu)\) a representation of~\(A\) on~\(B\), and
  \(T_1,\dotsc,T_n\in A\).  Assume that
  \(\cl{\mu(T_1)},\dotsc,\cl{\mu(T_n)}\) are affiliated with~\(B\)
  and generate~\(B\) in the sense of Woronowicz.  Then~\(B\) is a
  weak \Cstar\nb-hull for the \(B\)\nb-integrable representations
  of~\(A\).
\end{lemma}

\begin{proof}
  To show that~\(B\) is a weak \Cstar\nb-hull, we
  check~\ref{enum:Cstar-gen1} in
  Proposition~\ref{pro:Cstar-generated_by_unbounded_multipliers}.
  Let \(\varrho_1,\varrho_2\) be representations of~\(B\) on a
  Hilbert space~\(\Hilm[H]\) with \((\Hilms[B],\mu)
  \otimes_{\varrho_1} \Hilm[H] = (\Hilms[B],\mu) \otimes_{\varrho_2}
  \Hilm[H]\).  We claim that \(\varrho_1\oplus\varrho_2\colon
  B\to\Bound(\Hilm[H]^2) = \Mat_2(\Bound(\Hilm[H]))\) maps~\(B\)
  into the multiplier algebra of the diagonally embedded copy~\(K\)
  of~\(\Comp(\Hilm[H])\).  This is equivalent to
  \(\varrho_1=\varrho_2\).  Since \((\Hilms[B],\mu)
  \otimes_{\varrho_1} \Hilm[H] = (\Hilms[B],\mu) \otimes_{\varrho_2}
  \Hilm[H]\), the extension of \(\varrho_1\oplus\varrho_2\) to
  affiliated multipliers maps \(\cl{\mu(T_i)} \affil B\) to an
  operator of the form~\((X_i,X_i)\) for \(i=1,\dotsc,n\); these are
  affiliated with~\(K\).  Since these affiliated multipliers
  generate~\(B\), \(\varrho_1\oplus\varrho_2\) is a morphism
  from~\(B\) to~\(K\).  Thus~\(B\) is a weak \Cstar\nb-hull
  for~\(A\).
\end{proof}

\subsection{Universal enveloping algebras}
\label{sec:Ug}

We illustrate our theory by an example.  Let~\(\Lg\) be a
finite-dimensional Lie algebra over~\(\R\) and let \(A=U(\Lg)\) be
its universal enveloping algebra with the usual involution, where
elements of~\(\Lg\) are skew-symmetric.  A representation of~\(A\)
on~\(\Hilm\), possibly not closed, is equivalent to a dense
submodule \(\Hilms\subseteq\Hilm\) with a Lie algebra representation
\(\pi\colon \Lg\to\Endo_D(\Hilms)\) satisfying
\(\braket{\xi}{\pi(X)(\eta)} = -\braket{\pi(X)(\xi)}{\eta}\) for all
\(X\in\Lg\), \(\xi,\eta\in\Hilms\).

Let~\(G\) be a simply connected Lie group with Lie algebra~\(\Lg\)
and let \(B\defeq \Cst(G)\).  A representation of~\(\Cst(G)\) on a
Hilbert module~\(\Hilm\) is equivalent to a strongly continuous,
unitary representation of~\(G\) on~\(\Hilm\).  Given such a
representation, let~\(\Hilm^\infty\subseteq \Hilm\) be its subspace
of smooth vectors.  This is the domain of a closed representation
of~\(U(\Lg)\).  We call a representation of~\(U(\Lg)\) integrable if
it comes from a unitary representation of~\(G\) in this way.

In particular, \(G\) acts continuously on~\(\Cst(G)\) by left
multiplication with unitary multipliers.  Let
\(\Hilms[B]=\Cst(G)^\infty\) be the right ideal of smooth elements
for this \(G\)\nb-action, equipped with the canonical
\(U(\Lg)\)-module structure~\(\mu\).  By the universal property
of~\(\Cst(G)\), the pair \((\Hilms[B],\mu)\) is the universal
integrable representation.  That is, a representation of~\(U(\Lg)\)
is integrable if and only if it is of the form \((\Hilms[B],\mu)
\otimes_\varrho \Hilm\) for a representation~\(\varrho\)
of~\(\Cst(G)\).

Let \(X_1,\dotsc,X_d\) form a basis of~\(\Lg\).  The
\emph{Laplacian} is \(L\defeq -\sum_{i=1}^d X_i^2 \in U(\Lg)\).

\begin{theorem}[\cite{Pierrot:Reguliers}*{Théorème 2.12}]
  \label{the:integrable_Ug}
  A representation \((\pi,\Hilms)\) of~\(U(\Lg)\) is integrable if
  and only if~\(\cl{\pi(L^n)}\) is regular and self-adjoint for all
  \(n\in\N\).
\end{theorem}

\begin{proof}
  Since Pierrot does not require representations to be closed, his
  statement is slightly different from ours.  Pierrot shows that
  there is a continuous representation~\(\varrho\) of~\(G\) with
  differential \(X\mapsto \cl{\pi(X)}\) if and only if \(T\defeq
  \cl{\pi(L)}\) is self-adjoint and regular.  His proof shows that
  all elements of \(\bigcap_{n=1}^\infty \dom T^n\) are smooth
  vectors for~\(\varrho\).  Conversely, all smooth vectors must
  belong to this intersection.  A representation of~\(U(\Lg)\) is
  determined by its domain and the closed operators~\(\cl{\pi(X)}\)
  for \(X\in\Lg\).  So a closed representation~\((\Hilms,\pi)\) of~\(U(\Lg)\)
  is integrable if and only if~\(T\) is
  self-adjoint and regular and \(\Hilms=\bigcap_{n=1}^\infty
  \dom T^n\).  Moreover, the proof shows that the graph topology for
  a representation with regular self-adjoint~\(T\) is determined by
  the graph norms of~\(L^n\) for all \(n\in\N\).  If
  \(\cl{\pi(L^n)}\) is self-adjoint, then it must be equal
  to~\(T^n\) because \(\cl{\pi(L^n)}\subseteq T^n\) and~\(T^n\) is
  symmetric.  Therefore, if \(\cl{\pi(L)}\) is regular and
  self-adjoint, then the domain of~\(\pi\) is \(\bigcap_{n=1}^\infty
  \dom T^n\) if and only if \(\cl{\pi(L^n)}\) is regular and
  self-adjoint also for all \(n\ge2\).
\end{proof}

\begin{theorem}
  \label{the:Cstar_hull_Lie_algebra}
  The class of integrable representations of~\(U(\Lg)\)
  has~\(\Cst(G)\) as a \Cstar\nb-hull and is defined by submodule
  conditions.  So it satisfies the Strong Local--Global Principle.
\end{theorem}

\begin{proof}
  By Theorem~\ref{the:integrable_Ug}, a representation is integrable
  if and only if all elements of the set \(\{L^n\mid n\in\N\}\) act
  by a regular and self-adjoint operator.  Hence the assertion
  follows from Theorem~\ref{the:regular_admissible_local-global}.

  Alternatively, the closed multipliers of~\(\Cst(G)\) associated to
  \(\ima X_1,\dotsc, \ima X_d\) are regular and affiliated
  with~\(\Cst(G)\) and generate~\(\Cst(G)\) by
  \cite{Woronowicz:Cstar_generated}*{Example~3 in~§3}.  Hence
  \(\Cst(G)\) is a \Cstar\nb-hull and the Strong Local--Global
  Principle holds by Theorem~\ref{the:Woronowicz_local-global}.
\end{proof}

The results of Vassout~\cite{Vassout:Unbounded_groupoids} get close
to proving an analogue of Theorem~\ref{the:Cstar_hull_Lie_algebra}
for an \(s\)\nb-simply connected Lie groupoid~\(G\) with compact base.
This analogue would
replace~\(\Lg\) by the space of smooth sections of the Lie
algebroid~\(A(G)\), and \(U(\Lg)\) by the \Star{}algebra of
\(G\)\nb-equivariant differential operators on~\(G\), a subalgebra
of the \Star{}algebra of \(G\)\nb-pseudodifferential operators.  Any
symmetric, elliptic element of~\(U(\Lg)\) should be a possible
replacement for the Laplacian in
Theorem~\ref{the:Cstar_hull_Lie_algebra}.

\section{Polynomials in one variable II}
\label{sec:polynomials2}

We discuss two classes of ``integrable'' representations of the
\Star{}algebra~\(\C[x]\) with \(x=x^*\) which are weakly admissible,
but not admissible, and which violate the Local--Global Principle.  Both
examples have a weak \Cstar\nb-hull, on which all powers of the
generator~\(x\) act by an affiliated multiplier.  In the first
example, these affiliated multipliers generate the weak
\Cstar\nb-hull, but not in the second.
Neither Theorem~\ref{the:regular_admissible_local-global} nor
Theorem~\ref{the:Woronowicz_local-global} apply because the generating
affiliated multipliers are not self-adjoint.  The first example shows
that a \Cstar\nb-algebra generated by affiliated multipliers in the
sense of Woronowicz need not be a \Cstar\nb-hull, though it is always
a weak \Cstar\nb-hull by Lemma~\ref{lem:Woronowicz_weak_hull}.
The second example shows that a weak
\Cstar\nb-hull need not be generated by affiliated multipliers.

Let \(S\in\Bound(\ell^2\N)\) be the unilateral shift.  Let~\(Q\) be
the closed symmetric operator on~\(\ell^2\N\) with Cayley
transform~\(S\).  Thus~\(Q\) has deficiency index~\((0,1)\).  The
domain of~\(Q\) is \((1-S)\ell^2\N\), and \(Q(1-S)\xi \defeq \ima
(1+S)\xi\) for all \(\xi\in\ell^2\N\) (see also
Example~\ref{exa:deficiency_index_0}).  We may identify~\(\ell^2\N\)
with the Hardy space~\(H^2\).  Then~\(Q\) becomes the Toeplitz
operator with the unbounded symbol \(\ima(1+z)(1-z)^{-1}\).

Let~\(\Toep\) be the Toeplitz \Cstar\nb-algebra, that is, the
\Cstar\nb-subalgebra of~\(\Bound(\ell^2\N)\) generated by~\(S\).
Every element in~\(\Toep\) is of the form \(T_\varphi+K\),
where~\(T_\varphi\) is the Toeplitz operator with symbol
\(\varphi\in\Cont(S^1)\) and~\(K\) is a compact operator.  Let
\(\Toep_0 \idealin \Toep\) be the kernel of the unique
\Star{}homomorphism \(\Toep\to\C\) that maps~\(S\) to~\(1\).

\begin{proposition}
  \label{pro:Toeplitz_generated}
  There is a symmetric, affiliated multiplier~\(Q\) of~\(\Toep_0\)
  with domain \((1-S)\cdot\Toep_0\) and \(Q\cdot (1-S)\cdot t \defeq
  \ima (1+S) \cdot t\) for all \(t\in\Toep_0\).  It
  generates~\(\Toep_0\) in the sense of Woronowicz.
\end{proposition}

\begin{proof}
  We claim that the right ideal \((1-S)S^*\Toep_0\subseteq \Toep_0\)
  is dense.  This would fail for~\(\Toep\) because the continuous
  \Star{}homomorphism \(\Toep\to\C\), \(S\mapsto 1\), annihilates
  this right ideal.  First, \((1-S)S^*\Comp(\ell^2\N)\) is dense
  in~\(\Comp(\ell^2\N)\) because \((1-S)S^*\) has dense range
  on~\(\ell^2\N\).  So the closure of \((1-S)S^*\Toep_0\)
  contains~\(\Comp(\ell^2\N)\).  Secondly,
  \((1-S)S^*\Toep_0/\Comp(\ell^2\N)\) is dense in
  \(\Toep_0/\Comp(\ell^2\N) \cong \Cont_0(S^1\setminus\{1\})\)
  because the function \((1-z)\conj{z}\) on~\(S^1\) vanishes only
  at~\(1\).

  An affiliated multiplier of~\(\Toep_0\) is the same as a regular
  operator on~\(\Toep_0\), viewed as a Hilbert module over itself.
  Since~\((1-S)S^*\Toep_0\) is dense in~\(\Toep_0\), there is a
  regular, symmetric operator~\(Q'\) on~\(\Toep_0\) that has~\(S\)
  as its Cayley transform, see \cite{Lance:Hilbert_modules}*{Chapter
    10}.  The operator~\(Q'\) has the domain~\((1-S)S^*\Toep_0\) and
  acts by \(Q' \cdot (1-S)S^* t \defeq \ima (1+S)S^* t\).  Rewriting
  any \(t\in\Toep_0\) as \(t= S^* S t\), we may replace~\(S^* t\)
  by~\(t\) here.  Thus \(Q'=Q\).

  Since~\(Q+\ima\) maps~\((1-S)t\) to \(\ima (1+S)t + \ima (1-S)t =
  2\ima t\), it is surjective, and \((Q+\ima)^{-1} = \frac{1}{2\ima}
  (1-S)\) belongs to~\(\Toep_0\).  Hence \((Q+\ima)^* = Q^* - \ima\)
  is the inverse of \(\frac{1}{-2\ima} (1-S^*)\).  So~\(Q^*\) has
  domain \((1-S^*)\Toep_0\) and maps \((1-S^*)t\mapsto \ima (1-S^*)t
  -2\ima t = -\ima (1+S^*) t\).  As expected, \(Q^*\)
  contains~\(Q\): we may write \((1-S)t = S^* S t - S t = (1-S^*)(-S
  t)\), and~\(Q^*\) maps this to \(-\ima(1+S^*)(-S t) =
  \ima(S+1)t\).

  Next we show that \(Q^* Q+1\) is the inverse of \(\frac{1}{4}
  (1-S)(1-S^*) \in\Toep_0\).  We compute
  \begin{multline*}
    Q^* Q (1-S)(1-S^*) t
    = \ima Q^* (1+S)(1-S^*) t
    = \ima Q^* (1+S-S^* - S S^*) t
    \\= \ima Q^* (1-S^*) (2+S -S S^*) t
    = (1+S^*) (2+S -S S^*) t
    = (4- (1-S)(1-S^*)) t.
  \end{multline*}
  This implies \((Q^* Q +1)(1-S)(1-S^*) t = 4 t\).  Since this is
  already surjective and \(Q^*Q+1\) is injective, the domain of
  \(Q^* Q+1\) is exactly \((1-S)(1-S^*)\Toep_0\), and \(Q^* Q +1\)
  is the inverse of \(\frac{1}{4} (1-S)(1-S^*) \in\Toep_0\) as
  asserted.

  Let \(\varrho_1\) and~\(\varrho_2\) be two Hilbert space
  representations of~\(\Toep_0\) whose extension to affiliated
  multipliers maps~\(Q\) to the same unbounded operator.  Then they
  also map the Cayley transform~\(S\) of~\(Q\) to the same partial
  isometry.  So \(\varrho_1(S) = \varrho_2(S)\), which gives
  \(\varrho_1=\varrho_2\).  Thus~\(Q\) separates the representations
  of~\(\Toep_0\).  Since \((Q^* Q+1)^{-1}\in\Toep_0\) as well,
  \cite{Woronowicz:Cstar_generated}*{Theorem 3.3} shows that the
  affiliated multiplier~\(Q\) generates~\(\Toep_0\).
\end{proof}

The domain of~\(Q^n\) is the right ideal \((1-S)^n\cdot\Toep_0\),
which is dense in~\(\Toep_0\) for the same reason as
\((1-S)\cdot\Toep_0\).  Even more, the right ideal
\((1-S)^{n+1}\cdot\Toep_0\) is dense in \((1-S)^n\cdot\Toep_0\) in the
graph norm of~\(Q^n\).  Thus the intersection~\(\Hilms[T]\) of this
decreasing chain of dense right ideals \((1-S)^n\Toep_0\) is still
dense in~\(\Toep_0\) by \cite{Schmudgen:Unbounded_book}*{Lemma 1.1.2}.
This intersection is the domain of a closed representation~\(\mu\)
of~\(\C[x]\) on~\(\Toep_0\) with \(\cl{\mu(x^n)} = Q^n\).
We call a representation of~\(\C[x]\) on a Hilbert
module~\(\Hilm\) \emph{Toeplitz integrable} if it is of the form
\((\Hilms[T],\mu) \otimes_\varrho \Hilm\) for some representation
\(\varrho\colon \Toep_0\to\Bound(\Hilm)\).

\begin{proposition}
  \label{pro:Toeplitz_integrable}
  The class of Toeplitz integrable representations of~\(\C[x]\) is
  weakly admissible with the weak \Cstar\nb-hull~\(\Toep_0\).  It is
  not admissible, so~\(\Toep_0\) is not a \Cstar\nb-hull.  The
  Toeplitz integrable representations violate the Local--Global
  Principle.

  A representation~\((\Hilms,\pi)\) of~\(\C[x]\) on a Hilbert
  module~\(\Hilm\) over a \Cstar\nb-algebra~\(D\) is Toeplitz
  integrable if and only if it has the following properties:
  \begin{enumerate}
  \item \label{pro:Toeplitz_integrable_1}%
    \(\cl{\pi(x+\ima)^n\Hilms} = \Hilm\) for all \(n\in\N_{\ge1}\);
  \item \label{pro:Toeplitz_integrable_2}%
    \(\cl{\pi(x)}\) is regular.
  \end{enumerate}
  Toeplitz integrable representations on~\(\Hilm\) are in bijection
  with regular, symmetric operators~\(T\) on~\(\Hilm\) for which
  \(T+\ima\) is surjective.
\end{proposition}

\begin{proof}
  We checked condition~\ref{enum:Cstar-gen1} in
  Proposition~\ref{pro:Cstar-generated_by_unbounded_multipliers} in
  the proof of Proposition~\ref{pro:Toeplitz_generated}.
  Thus~\(\Toep_0\) is a weak \Cstar\nb-hull for the Toeplitz
  integrable representations, and this class is weakly admissible.
  Any self-adjoint operator on a Hilbert space generates a Toeplitz
  integrable representation of~\(\C[x]\) because
  \(\Toep_0/\Comp(\ell^2\N) \cong \Cont_0(\R)\); so does~\(Q\)
  itself.  Thus both Example~\ref{exa:deficiency_index_0} and
  Proposition~\ref{pro:symmetric_no_hull} show that the class of
  Toeplitz integrable representations is not admissible.
  So~\(\Toep_0\) is not a \Cstar\nb-hull.

  We claim that the representation~\((\Hilms[T],\mu)\) of~\(\C[x]\)
  on~\(\Toep_0\) has the properties \ref{pro:Toeplitz_integrable_1}
  and~\ref{pro:Toeplitz_integrable_2} in the proposition.  First,
  \((\mu(x)+\ima)^n\) acts by \((2\ima)^n (1-S)^{-n}\) on its dense
  domain \(\Hilms[T] \defeq \bigcap_{k=1}^\infty (1-S)^k\Toep_0\).
  Since \((1-S)^{k+1}\Toep_0\) is norm dense in \((1-S)^k\Toep_0\),
  the closure of \((\mu(x)+\ima)^n\) is equal to \((2\ima)^n
  (1-S)^{-n}\) with its natural domain \((1-S)^n\Toep_0\), and this
  operator is surjective.  Secondly, \(\mu(x)=Q\) is regular.

  The property~\ref{pro:Toeplitz_integrable_1} is a sequence of
  submodule conditions, see
  Example~\ref{exa:nondegeneracy_condition}.  Hence it is inherited
  by interior tensor products by
  Lemma~\ref{lem:operator_conditions}.  So is the
  property~\ref{pro:Toeplitz_integrable_2} by
  \cite{Lance:Hilbert_modules}*{Proposition 9.10}.  Hence both
  \ref{pro:Toeplitz_integrable_1}
  and~\ref{pro:Toeplitz_integrable_2} are necessary for a
  representation~\((\Hilms,\pi)\) to be Toeplitz integrable.

  Conversely, let~\((\Hilms,\pi)\) be a representation of~\(\C[x]\)
  on~\(\Hilm\) that satisfies \ref{pro:Toeplitz_integrable_1}
  and~\ref{pro:Toeplitz_integrable_2}.  Then the closed, symmetric
  operator \(T\defeq\cl{\pi(x)}\) on~\(\Hilm\) is regular
  by~\ref{pro:Toeplitz_integrable_2}.  So its Cayley transform~\(s\)
  is an adjointable partial isometry such that \((1-s)s^*\) has
  dense range (see \cite{Lance:Hilbert_modules}*{Chapter 10}).  Even
  more, \(s\)~is an isometry because \(\cl{(T+\ima)\Hilms} =
  \Hilm\).  Thus~\(s\) generates a unital representation~\(\varrho\)
  of~\(\Toep\).  The restriction of~\(\varrho\) to~\(\Toep_0\) is
  nondegenerate because~\((1-s)s^*\) has dense range.  Let
  \(\pi'\defeq \mu\otimes_\varrho 1\) be the representation
  of~\(\C[x]\) associated to~\(\varrho\).  Then
  \[
  \cl{\pi'((x+\ima)^n)}
  = (2\ima)^n (1-s)^{-n}
  \supseteq \cl{\pi((x+\ima)^n)}.
  \]
  Assumption~\ref{pro:Toeplitz_integrable_1} implies that~\(\Hilms\)
  is dense in the domain of \((2\ima)^n (1-s)^{-n}\) in the graph
  norm of~\((2\ima)^n (1-s)^{-n}\).  Hence even
  \(\cl{\pi'((x+\ima)^n)} = (2\ima)^n (1-s)^{-n} =
  \cl{\pi((x+\ima)^n)}\).  Since the domains of~\(\pi(x)^n\) form a
  decreasing sequence, induction on~\(n\) now shows that
  \(\cl{\pi'(x^n)} = \cl{\pi(x^n)}\).  The set~\(\{x^n\}\) is a
  strong generating set for~\(\C[x]\) by
  Lemma~\ref{lem:graph_topology_Cx}.  Thus \(\pi=\pi'\) by
  Proposition~\ref{pro:equality_if_closure_equal}.  This finishes
  the proof that Toeplitz integrable representations of~\(\C[x]\)
  are characterised by the properties
  \ref{pro:Toeplitz_integrable_1}
  and~\ref{pro:Toeplitz_integrable_2} and that they are in bijection
  with regular, symmetric operators~\(T\) for which~\(T+\ima\) is
  surjective.

  For a counterexample to the Local--Global Principle, let
  \(\bar{\N} = \N\cup\{\infty\}\) be the one-point compactification
  of~\(\N\) and \(D=\Cont(\bar{\N})\).  Let \(\Hilm\subseteq
  \Cont(\bar{\N},\ell^2\N)\) consist of all continuous functions
  \(f\colon \bar{\N}\to\ell^2\N\) with \(f(\infty)\bot \delta_0\).  The
  unilateral shift~\(S\) on \(\Cont(\bar{\N},\ell^2\N)\) restricts
  to a non-adjointable isometry~\(s\) on this subspace.  Let~\(T\)
  be the inverse Cayley transform of~\(s\).  This is a closed,
  symmetric operator on~\(\Hilm\) that is irregular because its
  Cayley transform is not adjointable.  If \(\varrho\colon
  D\to\Bound(\Hilm[H])\) is a Hilbert space representation, then the
  induced representation of~\(\C[x]\) is associated to the closed
  operator \(T\otimes_\varrho 1\).  The operator \((T\otimes_\varrho
  1)+\ima\) remains surjective, and \(T\otimes_\varrho 1\) is
  regular because it acts on a Hilbert space.  So \(T\otimes_\varrho
  1\) generates a Toeplitz integrable representation for all
  representations~\(\varrho\) of~\(D\).  Since~\(T\) itself does not
  generate a Toeplitz integrable representation, the Local--Global
  Principle is violated.
\end{proof}

Condition~\ref{pro:Toeplitz_integrable_1} in
Proposition~\ref{pro:Toeplitz_integrable} is a submodule condition.
If regularity without self-adjointness were a submodule condition as
well, then the Toeplitz integrable representations of~\(\C[x]\)
would be defined by submodule conditions; so the failure of the
Local--Global Principle for them would contradict
Theorem~\ref{the:Woronowicz_local-global}.

The identical inclusion \(\Toep_0\injto \Mult(\Comp(\ell^2\N))\) is
a representation of the weak \Cstar\nb-hull~\(\Toep_0\)
on~\(\Comp(\ell^2\N)\) and thus corresponds to a Toeplitz integrable
representation of~\(\C[x]\) on~\(\Comp(\ell^2\N)\).  This is simply
the restriction of~\((\Hilms[T],\mu)\) to the Hilbert
\(\Toep_0\)\nb-submodule \(\Comp(\ell^2\N) \subset\Toep_0\), with
domain \(\Hilms[T]\cap \Comp(\ell^2\N)\) and the same action~\(\mu\)
of~\(\C[x]\).  Call a representation \emph{purely Toeplitz
  integrable} if it is of the form
\((\Hilms[T]\cap\Comp(\ell^2\N),\mu) \otimes_\varrho \Hilm\) for
some representation \(\varrho\colon
\Comp(\ell^2\N)\to\Bound(\Hilm)\).

\begin{proposition}
  \label{pro:Comp_integrable_Cx}
  The purely Toeplitz integrable representations of~\(\C[x]\) form a
  weakly admissible class that is not admissible,
  and~\(\Comp(\ell^2\N)\) is a weak \Cstar\nb-hull for it, but not a
  \Cstar\nb-hull.  This class violates the Local--Global Principle.
  The closed multiplier \(Q = \cl{\mu(x)}\) of~\(\Toep_0\) is
  affiliated with~\(\Comp(\ell^2\N)\) but does not
  generate~\(\Comp(\ell^2\N)\).

  A representation~\((\Hilms,\pi)\) of~\(\C[x]\) on a Hilbert
  module~\(\Hilm\) over a \Cstar\nb-algebra~\(D\) is purely Toeplitz
  integrable if and only if it has the following property in
  addition to those in
  Proposition~\textup{\ref{pro:Toeplitz_integrable}}:
  \begin{enumerate}[start=3]
  \item the closure of \(\bigcup_{n=1}^\infty (\pi(x-\ima)^n
    \Hilms)^\bot\) is~\(\Hilm\).
  \end{enumerate}
\end{proposition}

\begin{proof}
  Since~\(\Comp(\ell^2\N)\) has fewer representations
  than~\(\Toep_0\), the condition~\ref{enum:Cstar-gen1} in
  Proposition~\ref{pro:Cstar-generated_by_unbounded_multipliers}
  for~\(\Comp(\ell^2\N)\) follows from the corresponding property
  for~\(\Toep_0\), which we have already checked in the proof of
  Proposition~\ref{pro:Toeplitz_generated}.
  Hence~\(\Comp(\ell^2\N)\) is a weak \Cstar\nb-hull for the purely
  Toeplitz representations of~\(\C[x]\).

  Since~\(Q\) gives a purely Toeplitz representation of~\(\C[x]\)
  on~\(\ell^2(\N)\), the class of purely Toeplitz integrable
  representations is not admissible by
  Example~\ref{exa:deficiency_index_0}.  Therefore, its weak
  \Cstar\nb-hull is not a \Cstar\nb-hull.  The same counterexample
  as in the proof of Proposition~\ref{pro:Toeplitz_integrable} shows
  that the Local--Global Principle fails for the purely Toeplitz
  representations.

  Any closed operator on~\(\ell^2\N\) is affiliated with~\(\Comp(\ell^2\N)\).
  In particular, so is~\(Q\).  In the identical representation
  of~\(\Comp(\ell^2\N)\) on the Hilbert space~\(\ell^2\N\), the image of~\(Q\)
  is affiliated with~\(\Toep_0\) by
  Proposition~\ref{pro:Toeplitz_generated}.  But the representation
  of~\(\Comp(\ell^2\N)\) is not by a morphism to~\(\Toep_0\) because the
  inclusion map \(\Comp(\ell^2\N)\injto\Toep_0\) is degenerate.
  Hence~\(Q\) does not generate~\(\Comp(\ell^2\N)\) in the sense of
  Woronowicz.

  The element \(P_n\defeq 1- S^n (S^*)^n\in\Comp(\ell^2\N)\subseteq
  \Toep_0\) is the orthogonal projection onto the span of
  \(\delta_0,\dotsc,\delta_{n-1}\).  A representation of~\(\Toep_0\)
  maps~\(P_n\) to an orthogonal projection whose image is the
  orthogonal complement of the image of~\(S^n\).  This is also the
  orthogonal complement of the image of~\(\pi(x-\ima)^n\).  These
  orthogonal complements form an increasing chain of complementable
  submodules, and~\(\pi\) is purely Toeplitz if and only if their
  union is all of~\(\Hilm\).  This proves our characterisation of
  purely Toeplitz representations.
\end{proof}

\section{Bounded and locally bounded representations}
\label{sec:locally_bounded}

Let~\(A\) be a \Star{}algebra.  A \emph{bounded representation}
of~\(A\) on a Hilbert module~\(\Hilm\) is a \Star{}homomorphism
\(\pi\colon A\to\Bound(\Hilm)\).  Corollary~\ref{cor:bounded_rep} says
that a closed representation is bounded once~\(\cl{\pi(a)}\) is
globally defined for~\(a\) in a strong
generating set of~\(A\).  Finite-dimensional representations are
always bounded.  In particular, characters are bounded.
Thus commutative \Star{}algebras have many bounded representations.
Many other \Star{}algebras, such as the Weyl algebra, have no
bounded representations.  In this section, we are going to study
\Cstar\nb-hulls related to bounded representations.  These are only
relevant if~\(A\) has many bounded representations.

Any bounded representation~\(\pi\) of~\(A\) is bounded in some
\Cstar\nb-seminorm~\(q\) on~\(A\), that is, \(\norm{\pi(a)}\le q(a)\)
for all \(a\in A\).  Then~\(\pi\) extends to the (Hausdorff)
completion~\(\mathcal{A}_q\) of~\(A\) in the seminorm~\(q\), which is
a unital \Cstar\nb-algebra.

If \(p,q\) are two \Cstar\nb-seminorms on~\(A\), then \(\max \{p,q\}\)
is a \Cstar\nb-seminorm as well.  Thus the set~\(\mathcal{N}(A)\) of
\Cstar\nb-seminorms on~\(A\) is directed.  For \(q,q'\in
\mathcal{N}(A)\) with \(q\le q'\), let \(\varphi_{q, q'}\colon
\mathcal{A}_{q'} \to \mathcal{A}_q\) be the \Star{}homomorphism
induced by the identity map on~\(A\).  The
\Cstar\nb-algebras~\(\mathcal{A}_q\) and the
\Star{}homomorphisms~\(\varphi_{q,q'}\) for \(q\le q'\)
in~\(\mathcal{N}(A)\) form a projective system of \Cstar\nb-algebras.
Each \Star{}homomorphism~\(\varphi_{q, q'}\) is unital and surjective
because its image contains~\(A\), which is unital and dense
in~\(\mathcal{A}_{q'}\).

The \Cstar\nb-seminorms in~\(\mathcal{N}(A)\) define a locally
convex topology on~\(A\), where a net converges if and only if it
converges in any \Cstar\nb-seminorm.  Let~\(\mathcal{A}\) with the
canonical map \(j\colon A\to \mathcal{A}\) be the completion
of~\(A\) in this topology.  This is a \Cstar\nb-algebra if and only
if there is a \emph{largest} \Cstar\nb-seminorm on~\(A\).  In
general, \(\mathcal{A}\) is the projective limit of the diagram of
unital \Cstar\nb-algebras \((\mathcal{A}_q,\varphi_{q, q'})\)
described above.  Thus~\(\mathcal{A}\) is a unital
pro-\Cstar\nb-algebra, see~\cite{Phillips:Inverse}.

As a concrete example, we describe~\(\mathcal{A}\) for a
commutative \Star{}algebra~\(A\).

\begin{definition}
  \label{def:characters}
  Let~\(\hat{A}\) be the set of \Star{}homomorphisms \(A\to\C\), which
  we briefly call \emph{characters}.  Each \(a\in A\) gives a function
  \(\hat{a}\colon \hat{A}\to\C\), \(\hat{a}(\chi) \defeq \chi(a)\).
  We equip~\(\hat{A}\) with the coarsest topology making these
  functions continuous.  That is, a net~\((\chi_i)_{i\in I}\)
  in~\(\hat{A}\) converges to \(\chi\in\hat{A}\) if and only if \(\lim
  \chi_i(a) = \chi(a)\) for all \(a\in A\).  Let~\(\tau_c\) be the
  \emph{compactly generated topology} associated to this topology,
  that is, a subset in~\(\hat{A}\) is closed in~\(\tau_c\) if and only
  if its intersection with any compact subset in~\(\hat{A}\) is closed.
\end{definition}

If \(a\in A\), then its \emph{Gelfand transform}~\(\hat{a}\) is a
continuous function on~\(\hat{A}\).  This
defines a \Star{}homomorphism \(A\to\Cont(\hat{A})\).  If the usual
topology on~\(\hat{A}\) is locally compact or metrisable, then it is
already compactly generated and hence equal to~\(\tau_c\).
The topology~\(\tau_c\) may have more closed subsets and hence
more continuous functions to~\(\C\).  So
\(\Cont(\hat{A})\subseteq \Cont(\hat{A},\tau_c)\).

\begin{proposition}
  \label{pro:commutative_pro-completion}
  Let~\(A\) be a commutative \Star{}algebra.  The directed
  set~\(\mathcal{N}(A)\) of \Cstar\nb-seminorms on~\(A\) is isomorphic
  to the directed set of compact subsets of~\(\hat{A}\), where
  \(K\subseteq \hat{A}\) corresponds to the \Cstar\nb-seminorm
  \[
  \norm{a}_K \defeq \sup  \{\abs{\hat{a}(\chi)} \mid \chi\in K\}.
  \]
  The \Cstar\nb-completion of~\(A\) in this \Cstar\nb-seminorm
  is~\(\Cont(K)\).  And \(\mathcal{A} \cong \Cont(\hat{A},\tau_c)\),
  where the inclusion map \(j\colon A\to\mathcal{A}\) is the Gelfand
  transform \(A\to\Cont(\hat{A},\tau_c)\), \(a\mapsto \hat{a}\).
\end{proposition}

\begin{proof}
  Let~\(q\) be a \Cstar\nb-seminorm on~\(A\).  Let
  \(\hat{A}_q\subseteq \hat{A}\) be the subspace of all
  \(q\)\nb-bounded characters, that is, \(\chi\in\hat{A}_q\) if and
  only if \(\abs{\chi(a)} \le q(a)\) for all \(a\in A\).  These are
  precisely the characters that extend to characters on the
  \Cstar\nb-completion~\(\mathcal{A}_q\).  Conversely, since~\(A\) is
  dense in~\(\mathcal{A}_q\), any character on~\(\mathcal{A}_q\) is
  the unique continuous extension of a \(q\)\nb-bounded character
  on~\(A\).  And the subspace topology on \(\hat{A}_q\subseteq
  \hat{A}\) is equal to the canonical topology on the spectrum
  of~\(\mathcal{A}_q\): a net of \(q\)\nb-bounded characters that
  converges uniformly on~\(A\) also converges uniformly
  on~\(\mathcal{A}_q\).  Thus
  \[
  \mathcal{A}_q \cong \Cont(\hat{A}_q)
  \]
  by the Gelfand--Naimark Theorem, and so \(\hat{A}_q\subseteq
  \hat{A}\) is compact for each \(q\in\mathcal{N}(A)\).

  If \(q\le q'\), then \(\hat{A}_q\subseteq \hat{A}_{q'}\) and
  \(\varphi_{q q'}\colon \mathcal{A}_{q'} \prto \mathcal{A}_q\) is the
  restriction map for the subspace \(\hat{A}_q\subseteq
  \hat{A}_{q'}\).  The pro-\Cstar\nb-algebra~\(\mathcal{A}\) is the
  limit of this diagram of commutative \Cstar\nb-algebras.  Since all
  the maps \(\hat{A}_q\subseteq \hat{A}_{q'}\) are injective,
  \(\mathcal{A}\)~is
  the algebra of continuous functions on
  \(\bigcup_{q\in\mathcal{N}(A)} \hat{A}_q\) with the
  inductive limit topology.  That is, a subset of
  \(\bigcup_{q\in\mathcal{N}(A)} \hat{A}_q\) is closed if and only if
  its intersection with each~\(\hat{A}_q\) is closed,
  where~\(\hat{A}_q\) carries the (compact) subspace topology
  from~\(\hat{A}\).

  Any character~\(\chi\) on~\(A\) is bounded with respect to some
  \Cstar\nb-seminorm; for instance, \(\norm{a}_\chi\defeq
  \abs{\chi(a)}\).  Thus \(\bigcup_{q\in\mathcal{N}(A)} \hat{A}_q =
  \hat{A}\) as a set.  If \(K\subseteq \hat{A}\) is compact, then
  \(\hat{a}\in\Cont(\hat{A})\) for \(a\in A\) must be uniformly
  bounded on~\(K\), so that
  \[
  \norm{a}_K \defeq \sup  \{\abs{\hat{a}(\chi)} \mid \chi\in K\}
  \]
  is a \Cstar\nb-seminorm on~\(A\).  Thus \(K\subseteq \hat{A}_q\)
  for some \(q\in\mathcal{N}(A)\).  Hence the inductive limit
  topology on \(\bigcup_{q\in\mathcal{N}(A)} \hat{A}_q\)
  is~\(\tau_c\).
\end{proof}

We return to the general noncommutative case.
The class of \(q\)\nb-bounded representations for a fixed
\(q\in\mathcal{N}(A)\) is easily seen to be weakly admissible.  The
class of bounded representations with variable~\(q\) is not weakly
admissible unless~\(A\) has a largest \Cstar\nb-seminorm because it is
not closed under direct sums.  We are going to define the larger class
of ``locally bounded'' representations to rectify this.  Roughly
speaking, a representation is locally bounded if and only if it comes
from a representation of the pro-\Cstar\nb-algebra~\(\mathcal{A}\).
Before we define locally bounded representations, we characterise
\(q\)\nb-bounded representations by some slightly weaker estimates.

\begin{proposition}
  \label{pro:bounded_element_fixed_seminorm}
  Let~\(A\) be a \Star{}algebra and let~\(q\) be a
  \Cstar\nb-seminorm on~\(A\).  Let~\((\Hilms,\pi)\) be a
  representation of~\(A\) on a Hilbert module~\(\Hilm\) over some
  \Cstar\nb-algebra~\(D\) and let \(\xi\in\Hilms\).  The following
  are equivalent:
  \begin{enumerate}
  \item there is \(C>0\) with \(\norm{\braket{\xi}{\pi(a)\xi}} \le C
    q(a)\) for all \(a\in A\);
  \item there is \(C>0\) with \(\norm{\pi(a)\xi} \le C q(a)\) for all
    \(a\in A\);
  \item \(\norm{\pi(a)\xi} \le \norm{\xi} q(a)\) for all \(a\in A\).
  \end{enumerate}
  The set of vectors~\(\xi\) with these equivalent properties is a
  norm-closed \(A,D\)-submodule of~\(\Hilm\).  The representation
  of~\(A\) on this submodule extends to the
  \Cstar\nb-completion~\(\mathcal{A}_q\).
\end{proposition}

\begin{proof}
  The implications (3)\(\Rightarrow\)(2)\(\Rightarrow\)(1) are
  trivial.  Conversely, assume~(1) and let \(a\in A\).  Let
  \((b_n)_{n\in\N}\) be a sequence in~\(A\) that converges
  in~\(\mathcal{A}_q\) towards the positive square-root of \(q(a)^2 -
  a^*a\).  Then the sequence \((a^* a + b_n^* b_n)\) in~\(A\)
  converges in the norm~\(q\) to \(q(a)^2\in A\).  If~(1) holds, then
  \[
  \lim_{n\to\infty} {}\braket{\xi}{\pi(a^* a+b_n^* b_n)\xi}
  = q(a)^2 \braket{\xi}{\xi}.
  \]
  Since \(0 \le \braket{\pi(a)\xi}{\pi(a)\xi} \le
  \braket{\pi(a)\xi}{\pi(a)\xi} + \braket{\pi(b_n)\xi}{\pi(b_n)\xi} =
  \braket{\xi}{\pi(a^* a + b_n^* b_n)\xi}\) for all~\(n\), this
  implies \(\norm{\pi(a)\xi} \le \lim \norm{\braket{\xi}{\pi(a^* a +
      b_n^* b_n)\xi}} = q(a)^2 \norm{\xi}^2\).  Thus~(1) implies~(3).

  The set~\(\Hilm_q\) of vectors \(\xi\in\Hilms\) satisfying~(2) is a
  vector subspace and closed under left multiplication by elements
  of~\(A\) and right multiplication by elements of~\(D\).  On this
  subspace, the graph and norm topologies coincide because of~(3).
  The subspace~\(\Hilm_q\) is closed in the norm topology by the
  Principle of Uniform Boundedness.  The \Star{}representation
  of~\(A\) on this submodule is globally defined and bounded by the
  \Cstar\nb-seminorm~\(q\).  Hence it extends to a representation
  of~\(\mathcal{A}_q\).
\end{proof}

\begin{definition}
  \label{def:locally_bounded_representation}
  Let~\((\Hilms,\pi)\) be a representation of~\(A\) on a Hilbert
  module~\(\Hilm\).  A vector \(\xi\in\Hilms\) is \emph{bounded} if it
  satisfies the equivalent conditions in
  Proposition~\ref{pro:bounded_element_fixed_seminorm} for some
  \(q\in\mathcal{N}(A)\).  The
  representation is \emph{locally bounded} if the bounded vectors are
  dense in~\(\Hilms\) in the graph topology.
\end{definition}

By Proposition~\ref{pro:bounded_element_fixed_seminorm}, the
\(q\)\nb-bounded vectors in~\(\Hilms\) for a fixed \(q\in
\mathcal{N}(A)\) form a
closed \(A,D\)\nb-submodule \(\Hilm_q\subseteq\Hilm\), on which the
representation of~\(A\) extends to the
\Cstar\nb-completion~\(\mathcal{A}_q\) and hence to a representation
of~\(\mathcal{A}\).  Since~\(\mathcal{N}(A)\) is directed and
\(\Hilm_q \subseteq \Hilm_{q'}\) if \(q\le q'\), the
family of sub-bimodules \(\Hilm_q\subseteq \Hilms\) is directed.
The set of bounded vectors is the increasing union
\[
\Hilms_\mathrm{b} \defeq \bigcup_{q\in\mathcal{N}(A)} \Hilm_q.
\]
Since \(\pi|_{\Hilm_q}\) extends to~\(\mathcal{A}\) for each~\(q\),
there is a representation~\(\bar\pi\) of the
pro-\Cstar\nb-algebra~\(\mathcal{A}\) on
\(\Hilms_\mathrm{b}\subseteq \Hilm\).  The
representation~\((\Hilms,\pi)\) is locally bounded if and only if
\((\Hilms_\mathrm{b},\bar\pi\circ j)\) is a core for it.
Thus~\((\Hilms,\pi)\) is the closure of the ``restriction''
\(\bar\pi\circ j\) of~\(\bar\pi\) to~\(A\).

We do not claim that~\(\bar\pi\) is closed, and neither do we claim
that \(\bar\pi\circ j\) is locally bounded for any representation
of~\(\mathcal{A}\): we need the representation of~\(\mathcal{A}\) to
be locally bounded as well:

\begin{definition}
  \label{def:pro_locally_bounded}
  A representation \((\pi,\Hilms)\) of a
  pro-\Cstar\nb-algebra~\(\mathcal{A}\) is \emph{locally bounded} if
  the vectors \(\xi\in\Hilms\) for which \(\mathcal{A}\to\Hilm\),
  \(a\mapsto \pi(a)\xi\), is continuous form a core.
\end{definition}

\begin{proposition}
  \label{pro:locally_bounded_pro-Cstar}
  Composition with \(j\colon A\to \mathcal{A}\) induces an
  equivalence between the categories of locally bounded
  representations of \(A\) and~\(\mathcal{A}\) which is compatible
  with isometric intertwiners and interior tensor products.
\end{proposition}

\begin{proof}
  The \Star{}homomorphism~\(j\)
  induces an isomorphism between the directed sets of
  \Cstar\nb-seminorms on \(A\)
  and~\(\mathcal{A}\).
  Therefore, a representation~\(\bar\pi\)
  of~\(\mathcal{A}\)
  is locally bounded if and only if the vectors~\(\xi\)
  with \(\norm{\bar\pi(a)\xi}\le q(a)\norm{\xi}\)
  for all \(a\in \mathcal{A}\),
  for some \(q\in\mathcal{N}(A)\),
  form a core.  Since \(j(A)\)
  is dense in~\(\mathcal{A}\),
  this is equivalent to \(\norm{\pi(a)\xi}\le q(a)\norm{\xi}\)
  for all \(a\in A\).
  Thus the closure of \(\bar\pi\circ j\)
  is locally bounded if and only if~\(\bar\pi\) is.

  An isometric intertwiner \(\bar\pi_1 \injto \bar\pi_2\) also
  intertwines the closures of \(\bar\pi_1\circ j\) and
  \(\bar\pi_2\circ j\) by Lemma~\ref{lem:closure_functorial}.
  Conversely, an isometric intertwiner between two locally bounded
  representations of~\(A\) must map \(q\)\nb-bounded vectors to
  \(q\)\nb-bounded vectors for any \(q\in\mathcal{N}(A)\).  Thus it
  remains an isometric intertwiner between the canonical extensions
  of the representations to~\(\mathcal{A}\).  So the equivalence
  between the locally bounded representations of \(A\)
  and~\(\mathcal{A}\) is compatible with isometric intertwiners.  It
  is also compatible with interior tensor products, that is, the
  closure of \((\bar\pi\otimes_D 1_{\Hilm[G]})\circ j\) is
  \(\cl{\bar\pi\circ j}\otimes_D 1_{\Hilm[G]}\).
\end{proof}

\begin{proposition}
  \label{pro:locally_bounded_irreducible}
  All irreducible, locally bounded Hilbert space representations are
  bounded.
\end{proposition}

\begin{proof}
  If~\(\pi\) is irreducible, then the closed
  \(A\)\nb-submodule~\(\Hilm_q\) for a \Cstar\nb-seminorm~\(q\) is
  either \(\{0\}\) or~\(\Hilm\).  The latter must happen for
  some~\(q\) if~\(\pi\) is locally bounded.
\end{proof}

Thus local boundedness is not an interesting notion for irreducible
representations.

If~\(A\) has no \Cstar\nb-seminorms, then \(\mathcal{A}=\{0\}\)
and~\(A\) has no locally bounded representations, so that the
following discussion will be empty.  Even if the map \(j\colon
A\to\mathcal{A}\) is not injective, there are examples where all
``integrable'' representations of~\(A\) come from~\(\mathcal{A}\).  An
important case is the unit fibre for the canonical \(\Z\)\nb-grading
on the Virasoro algebra studied in
\cite{Savchuk-Schmudgen:Unbounded_induced}*{§9.3}.  In this case,
\(A\) is not commutative, but all irreducible, integrable
representations are characters and hence locally bounded.

\begin{proposition}
  \label{pro:locally_bounded_regular}
  If~\(\pi\) is a locally bounded representation,
  then~\(\cl{\pi(a)}\) is regular and self-adjoint for each \(a\in
  A_h\).
\end{proposition}

\begin{proof}
  The map of left multiplication by~\(j(a)\pm\ima\) on~\(\mathcal{A}\)
  is invertible because \(j(a)\in \mathcal{A}\) is symmetric
  and~\(\mathcal{A}\) is a pro-\Cstar\nb-algebra.  Therefore,
  \(\bar\pi(j(a)) \pm\ima \subseteq \pi(a) \pm\ima\) has dense range
  on~\(\Hilm\).  Thus~\(\cl{\pi(a)}\) is regular and self-adjoint, see
  \cite{Lance:Hilbert_modules}*{Chapter 10}.
\end{proof}

\begin{corollary}
  \label{cor:locally_bounded_admissible}
  Let~\(A\) be a \Star{}algebra.  The class \(\Rep_\mathrm{b}(A)\)
  of locally bounded representations of~\(A\) is admissible.
\end{corollary}

\begin{proof}
  Being locally bounded is clearly invariant under unitary
  \Star{}intertwiners and direct sums.  It is also invariant under
  direct summands because a \Star{}intertwiner maps bounded vectors to
  bounded vectors.  If \(\xi\in\Hilm\)
  is bounded, then \(\xi\otimes\eta \in \Hilm\otimes_D \Hilm[F]\)
  is bounded for any \Cstar\nb-correspondence~\(\Hilm[F]\).
  Thus a locally bounded representation on~\(\Hilm\)
  induces one on~\(\Hilm\otimes_D \Hilm[F]\).

  Since~\(A_h\) is a strong generating set for~\(A\) by
  Example~\ref{exa:strong_generating}, the class of representations
  for which all \(a\in A_h\) act by a regular and self-adjoint
  operator is admissible by
  Theorem~\ref{the:regular_admissible_local-global}.  This class
  contains the locally bounded representations by
  Proposition~\ref{pro:locally_bounded_regular}.  Hence this subclass
  is also admissible.
\end{proof}

Any pro-\Cstar\nb-algebra~\(\mathcal{A}\) contains a dense unital
\Cstar\nb-subalgebra~\(\mathcal{A}_\mathrm{b}\) of bounded elements,
see \cite{Phillips:Inverse}*{Proposition 1.11}.  For instance,
if~\(A\) is commutative, so that \(\mathcal{A} \cong
\Cont(\hat{A},\tau_c)\) by
Proposition~\ref{pro:commutative_pro-completion}, then
\(\mathcal{A}_\mathrm{b} = \Contb(\hat{A},\tau_c)\) consists of the
\emph{bounded} continuous functions.

Let~\((\Hilms,\pi)\) be a locally bounded representation of~\(A\).
This comes from a locally bounded representation
\((\Hilms_\mathrm{b},\bar\pi)\) of~\(\mathcal{A}\) by
Proposition~\ref{pro:locally_bounded_pro-Cstar}.  The closure of the
restriction of~\(\bar\pi\) to~\(\mathcal{A}_\mathrm{b}\) is a
representation of a unital \Cstar\nb-algebra.  Hence it is a unital
\Star{}homomorphism \(\varrho\colon \mathcal{A}_\mathrm{b}\to
\Bound(\Hilm)\) by Lemma~\ref{lem:Cstar-rep_bounded}.

\begin{proposition}
  \label{pro:equality_locally_bounded}
  Two locally bounded representations \(\pi_1\) and~\(\pi_2\) of~\(A\)
  on a Hilbert module~\(\Hilm\) are equal if and only if they induce
  the same representation of~\(\mathcal{A}_\mathrm{b}\).
\end{proposition}

\begin{proof}
  Of course, \(\pi_1\) and~\(\pi_2\) induce the same representation
  of~\(\mathcal{A}_\mathrm{b}\) if \(\pi_1=\pi_2\).  Conversely,
  assume that \(\pi_1\) and~\(\pi_2\) induce the same
  representation~\(\varrho\) of~\(\mathcal{A}_\mathrm{b}\).  If \(a\in
  A_h\), then the Cayley transform~\(c_a\) of \(j(a)\in\mathcal{A}\)
  is a unitary element of~\(\mathcal{A}_\mathrm{b}\).  The Cayley
  transforms of \(\cl{\pi_1(a)}\) and \(\cl{\pi_2(a)}\) are both equal
  to~\(\varrho(c_a)\).  Hence \(\cl{\pi_1(a)}=\cl{\pi_2(a)}\).  Since
  this holds for all \(a\in A_h\),
  Proposition~\ref{pro:equality_if_closure_equal} gives \(\pi_1=\pi_2\).
\end{proof}

The \Cstar\nb-algebra~\(\mathcal{A}_\mathrm{b}\) usually has many
representations that do not come from locally bounded representations
of~\(A\).  Hence it is not a \Cstar\nb-hull.  It is, however, a useful
tool to decide when a representation~\(\mu\) of~\(A\) on a
\Cstar\nb-algebra~\(B\) is a weak \Cstar\nb-hull, that is, when~\(A\)
separates the Hilbert space representations of~\(B\):

\begin{proposition}
  \label{pro:criterion_Cstar-hull_locally_bounded}
  Let~\(\mu\) be a locally bounded representation of~\(A\) on a
  \Cstar\nb-algebra~\(B\) and let \(\varrho\colon
  \mathcal{A}_\mathrm{b} \to \Mult(B) = \Bound(B)\) be the associated
  representation of~\(\mathcal{A}_\mathrm{b}\).  The image
  of~\(\varrho\) is dense in~\(\Mult(B)\) in the strict topology if
  and only if~\(B\) is a weak \Cstar\nb-hull for the class of
  \(B\)\nb-integrable representations of~\(A\) defined by~\(\mu\).
\end{proposition}

\begin{proof}
  Combine Proposition~\ref{pro:equality_locally_bounded} and the
  following proposition for \(D=\mathcal{A}_\mathrm{b}\).
\end{proof}

\begin{proposition}
  \label{pro:Cstar-hull_bounded_criterion}
  Let~\(\mu\)
  be a representation of~\(A\)
  on a \Cstar\nb-algebra~\(B\).
  Let~\(D\)
  be a \Cstar\nb-algebra and \(\varphi\colon D\to\Mult(B)\)
  a \Star{}homomorphism.  Assume that two representations
  \(\varrho_1,\varrho_2\)
  of~\(B\)
  on a Hilbert space~\(\Hilm[H]\)
  satisfy
  \(\mu\otimes_{\varrho_1} 1_{\Hilm[H]} = \mu\otimes_{\varrho_2}
  1_{\Hilm[H]}\)
  if and only if
  \(\bar\varrho_1\circ \varphi = \bar\varrho_2 \circ\varphi\),
  where \(\bar\varrho_1\)
  and~\(\bar\varrho_2\)
  denote the unique strictly continuous extensions of
  \(\varrho_1,\varrho_2\)
  to \(\Mult(B)\).
  Then~\(B\)
  is a weak \Cstar\nb-hull for a class of integrable representations
  of~\(A\)
  if and only if \(\varphi(D)\)
  is dense in~\(\Mult(B)\) in the strict topology.
\end{proposition}

\begin{proof}
  We use the criterion for weak \Cstar\nb-hulls
  in~\ref{enum:Cstar-gen1} in
  Proposition~\ref{pro:Cstar-generated_by_unbounded_multipliers}.
  Assume first that \(\varphi(D)\)
  is strictly dense in~\(\Mult(B)\).
  Let \(\varrho_1,\varrho_2\)
  be two Hilbert space representations of~\(B\)
  that satisfy
  \(\mu\otimes_{\varrho_1} 1_{\Hilm[H]} = \mu\otimes_{\varrho_2}
  1_{\Hilm[H]}\).
  Extend \(\varrho_1,\varrho_2\)
  to strictly continuous representations
  \(\bar\varrho_1,\bar\varrho_2\)
  of~\(\Mult(B)\).
  By assumption,
  \(\bar\varrho_1\circ \varphi = \bar\varrho_2 \circ\varphi\),
  that is, \(\bar\varrho_1\)
  and~\(\bar\varrho_2\)
  are equal on~\(\varphi(D)\).
  Since they are strictly continuous and~\(\varphi(D)\)
  is strictly dense, we get \(\bar\varrho_1 = \bar\varrho_2\)
  and hence \(\varrho_1 = \varrho_2\).
  Thus the condition~\ref{enum:Cstar-gen1} in
  Proposition~\ref{pro:Cstar-generated_by_unbounded_multipliers} is
  satisfied, making~\(B\) a weak \(\Cst\)\nb-hull of~\(A\).

  Conversely, assume that~\(\varphi(D)\)
  is not strictly dense in~\(\Mult(B)\).
  We claim that the image of~\(D\)
  is not weakly dense in the bidual
  \(\mathrm{W}^*\)\nb-algebra~\(B^{**}\).
  Any positive linear functional on~\(B\)
  extends to a strictly continuous, positive linear functional
  on~\(\Mult(B)\)
  by extending its GNS-representation to a strictly continuous
  representation of~\(\Mult(B)\).
  By the Jordan decomposition, the same remains true for self-adjoint
  linear functionals and hence for all bounded linear functionals
  on~\(B\).
  Furthermore, such extensions are unique because~\(B\)
  is strictly dense in~\(\Mult(B)\).
  Hence restriction to~\(B\)
  maps the space of strictly continuous linear functionals
  on~\(\Mult(B)\)
  isomorphically onto the dual space~\(B^*\)
  of~\(B\),
  which is also the space of weakly continuous linear functionals
  on~\(B^{**}\).
  If \(\varphi(D)\)
  is not strictly dense in~\(\Mult(B)\),
  then the Hahn--Banach Theorem gives a non-zero functional in~\(B^*\)
  that vanishes on the image of~\(D\).
  When viewed as a weakly continuous functional on~\(B^{**}\),
  it witnesses that~\(\varphi(D)\) is not weakly dense in~\(B^{**}\).

  Let \(\varrho\colon B\to \Bound(\Hilm[H])\)
  be the direct sum of all cyclic representations of~\(B\).
  Then~\(\varrho\)
  extends to an isomorphism of \(\mathrm{W}^*\)\nb-algebras
  from~\(B^{**}\)
  onto the double commutant~\(\varrho(B)''\)
  of~\(B\)
  in~\(\Bound(\Hilm[H])\).
  The extension of~\(\varrho\)
  to~\(\Mult(B)\)
  restricts to a representation
  \(\bar\varrho\circ\varphi\colon D \to \Bound(\Hilm[H])\).
  Since we assume that the image of~\(D\)
  is not strictly dense in~\(\Mult(B)\),
  our claim shows that \(\bar\varrho\circ\varphi(D)\)
  is not weakly dense in~\(\varrho(B)''\).
  By the bicommutant theorem, this is equivalent to
  \(\bar\varrho\circ\varphi(D)' \neq \varrho(B)'\).

  Since these commutants are \(\Cst\)\nb-algebras,
  they are the linear spans of the unitaries that they contain.  So
  there is a unitary operator~\(U\)
  in \(\bar\varrho\circ\varphi(D)'\)
  that is not contained in~\(\varrho(B)'\).
  So \(\varrho_2\defeq U\varrho U^*\neq \varrho\),
  but \(\bar\varrho_2\circ\varphi = \bar\varrho\circ\varphi\).
  By assumption, the latter implies
  \(\mu \otimes_\varrho 1_{\Hilm[H]} = \mu \otimes_{\varrho_2}
  1_{\Hilm[H]}\).
  So~\(A\)
  fails to separate the representations \(\varrho,\varrho_2\)
  of~\(B\)
  although they are not equal.  Hence~\(B\)
  is not a weak \(\Cst\)\nb-hull of~\(A\).
\end{proof}

\begin{remark}
  \label{rem:criterion_density_weak-Cstar-hull}
  Proposition~\ref{pro:Cstar-hull_bounded_criterion} applies whenever
  we can somehow produce enough bounded operators from a
  representation of~\(A\) so that these bounded operators and the
  original representation have the same unitary \Star{}intertwiners.
  For instance, it applies if the elements of a strong generating
  set for~\(A\) act by regular operators, so that we may take their
  bounded transforms.
\end{remark}

The quotient map \(\mathcal{A}_q \prto \mathcal{A}_{q'}\) for \(q
\ge q'\) in~\(\mathcal{N}(A)\) identifies the primitive ideal space
\(\Prim(\mathcal{A}_{q'})\) with a closed subspace of
\(\Prim(\mathcal{A}_q)\).  Let \(\Prim \mathcal{A} \defeq
\bigcup_{q\in\mathcal{N}(A)} \Prim(\mathcal{A}_q)\).  Let
\(a\in\mathcal{A}\) and
\(\mathfrak{p}\in \Prim(\mathcal{A})\).  Then the
norm~\(\norm{a}_\mathfrak{p}\) of the image of~\(a\) in
\(\mathcal{A}_q/\mathfrak{p}\) is the same for all
\(q\in\mathcal{N}(A)\) with \(\mathfrak{p}\in
\Prim(\mathcal{A}_q)\).  Hence the function \(\mathfrak{p}\mapsto
\norm{a}_\mathfrak{p}\) on~\(\Prim(\mathcal{A})\) is well defined.

\begin{definition}
  \label{def:compact_in_pro-Cstar}
  An element \(a\in\mathcal{A}\) \emph{vanishes at~\(\infty\)} if
  for every \(\varepsilon>0\) there is \(q\in\mathcal{N}(A)\) such
  that \(\norm{a}_\mathfrak{p}<\varepsilon\) for \(\mathfrak{p} \in
  \Prim(\mathcal{A}) \setminus \Prim(\mathcal{A}_q)\).  An element
  \(a\in\mathcal{A}\) is \emph{compactly supported} if there is
  \(q\in\mathcal{N}(A)\) with \(a\in \mathfrak{p}\) for all
  \(\mathfrak{p} \in \Prim(\mathcal{A}) \setminus
  \Prim(\mathcal{A}_q)\).  Let \(\Cont_0(\mathcal{A})\) and
  \(\Contc(\mathcal{A})\) be the subsets of elements that vanish
  at~\(\infty\) and have compact support, respectively.
\end{definition}

It may happen that \(\Cont_0(\mathcal{A})=\{0\}\).  In the
following, we are interested in the case where
\(\Cont_0(\mathcal{A})\) is dense in~\(\mathcal{A}\).  For instance,
\(\Cont_0(\R)\) is dense in~\(\Cont(\R)\).

\begin{lemma}
  \label{lem:compactly_supported_ideal}
  The subset \(\Cont_0(\mathcal{A})\) is a closed ideal
  in~\(\mathcal{A}_\mathrm{b}\).  The subspace
  \(\Contc(\mathcal{A})\) is a two-sided \Star{}ideal
  in~\(\mathcal{A}\).  It is norm-dense in \(\Cont_0(\mathcal{A})\).
  More generally, if~\(D\) is a \Cstar\nb-algebra and
  \(\varphi\colon D\to\mathcal{A}\) is a \Star{}homomorphism, then
  \(\varphi^{-1}(\Contc(\mathcal{A}))\) is dense in
  \(\varphi^{-1}(\Cont_0(\mathcal{A}))\).
\end{lemma}

\begin{proof}
  The quotient maps \(\mathcal{A} \prto \mathcal{A}_q \prto
  \mathcal{A}_q/\mathfrak{p}\) for
  \(\mathfrak{p}\in\Prim(\mathcal{A}_q)\) are \Star{}homomorphisms.
  Thus \(\Cont_0(\mathcal{A})\) is a \Star{}subalgebra
  of~\(\mathcal{A}\).  An element \(a\in \mathcal{A}\) is bounded if
  and only if the norms of its images in~\(\mathcal{A}_q\) for
  \(q\in\mathcal{N}(A)\) are uniformly bounded.  The norm of~\(a\)
  in~\(\mathcal{A}_q\) is the maximum of~\(\norm{a}_\mathfrak{p}\) for
  \(\mathfrak{p}\in\Prim(\mathcal{A}_q)\).  Hence~\(a\) is bounded if
  and only if the function \(\norm{a}_\mathfrak{p}\)
  on~\(\Prim(\mathcal{A})\) is bounded.  Thus \(\Cont_0(\mathcal{A})\)
  consists of bounded elements, and it is an ideal
  in~\(\mathcal{A}_\mathrm{b}\).  We claim that the limit~\(a\) of a
  norm-convergent sequence \((a_n)_{n\in\N}\)
  in~\(\Cont_0(\mathcal{A})\) again vanishes at~\(\infty\).  Given
  \(\varepsilon>0\), there is \(n_0\in\N\) so that
  \(\norm{a-a_n}_\mathfrak{p} \le \norm{a-a_n}<\varepsilon/2\) for all
  \(n\ge n_0\) and all \(\mathfrak{p} \in \Prim(\mathcal{A})\).
  Since~\(a_n\) vanishes at~\(\infty\), there is
  \(q\in\mathcal{N}(A)\) with
  \(\norm{a_n}_\mathfrak{p}<\varepsilon/2\) for \(\mathfrak{p}\notin
  \Prim(\mathcal{A}_q)\).  Thus \(\norm{a}_\mathfrak{p}<\varepsilon\)
  for \(\mathfrak{p}\notin \Prim(\mathcal{A}_q)\).  Thus
  \(\Cont_0(\mathcal{A})\) is a closed ideal
  in~\(\mathcal{A}_\mathrm{b}\).

  The condition \(a\in \mathfrak{p}\) for fixed \(\mathfrak{p}\in
  \Prim(\mathcal{A})\) defines a closed two-sided \Star{}ideal
  in~\(\mathcal{A}\).  Hence~\(\Contc(\mathcal{A})\) is a two-sided
  \Star{}ideal in~\(\mathcal{A}\).  Let \(a\in
  \Cont_0(\mathcal{A})\) and \(\varepsilon>0\).  Let
  \(f_\varepsilon\in\Contb([0,\infty))\) be increasing and satisfy
  \(f_\varepsilon(t)=0\) for \(0\le t<\varepsilon\) and
  \(f_\varepsilon(t)=1\) for \(2\varepsilon\le t\).  Then \(\norm{a
    - a f_\varepsilon(a^* a)} \le 2\varepsilon\), and
  \(f_\varepsilon(a^* a)\in \mathfrak{p}\) if \(\norm{a^*
    a}_{\mathfrak{p}}\le \varepsilon\).  Hence \(a f_\varepsilon(a^*
  a) \in \Contc(\mathcal{A})\) for all \(\varepsilon>0\).  Thus
  \(\Contc(\mathcal{A})\) is dense in~\(\Cont_0(\mathcal{A})\).
  Similarly, if \(\varphi\colon D\to\mathcal{A}\) is a
  \Star{}homomorphism, \(x\in D\), and \(\varphi(x)\in
  \Cont_0(\mathcal{A})\), then \(\varphi(x f_\varepsilon(x^* x)) \in
  \Contc(\mathcal{A})\) and \(\lim_{\varepsilon\to0} x
  f_\varepsilon(x^* x) = x\) in the norm topology on~\(D\).
\end{proof}

\begin{theorem}
  \label{the:Prim_A_locally_compact_gives_hull}
  Let~\(A\) be a \Star{}algebra and let~\(\mathcal{A}\) be its
  pro-\Cstar\nb-algebra completion.  If
  \(\Cont_0(\mathcal{A})\) is dense in~\(\mathcal{A}\), then
  \(\Cont_0(\mathcal{A})\) is a \Cstar\nb-hull for the class of
  locally bounded representations of~\(A\).
\end{theorem}

We shall prove a more general theorem that still applies if
\(\Cont_0(\mathcal{A})\) is not dense in~\(\mathcal{A}\).  Then
probably there is no \Cstar\nb-hull for the class of all locally
bounded representations.  We may, however, find \Cstar\nb-hulls for
smaller classes of representations.  We describe such classes of
representations by a generalisation of the spectral conditions in
Example~\ref{exa:spectral_condition}.  The spectral condition for a
locally closed subset in~\(\R^n\) implicitly uses a subquotient
of~\(\Cont_0(\R^n)\).  We are going to describe subquotients of the
pro-\Cstar-algebra~\(\mathcal{A}\).  We then associate a class
\(\Rep_\mathrm{b}(A,\mathcal{K})\) of locally bounded
representations of~\(A\) to a
subquotient~\(\mathcal{K}\).  If \(\Cont_0(\mathcal{K})\) is dense
in~\(\mathcal{K}\), then \(\Cont_0(\mathcal{K})\) is a
\Cstar\nb-hull for \(\Rep_\mathrm{b}(A,\mathcal{K})\).
Theorem~\ref{the:Prim_A_locally_compact_gives_hull} is the special
case \(\mathcal{K} = \mathcal{A}\).

Let \(\mathcal{J}\idealin \mathcal{A}\) be a closed, two-sided
\Star{}ideal in the pro-\Cstar\nb-algebra~\(\mathcal{A}\).  Being
closed, the ideal~\(\mathcal{J}\) is complete in the subspace
topology, so it is also a pro-\Cstar\nb-algebra.  Thus \(\mathcal{J}
= \varprojlim \mathcal{J}_q\), where \(\mathcal{J}_q\idealin
\mathcal{A}_q\) is the image of~\(\mathcal{J}\) in the
quotient~\(\mathcal{A}_q\).  The
quotient~\(\mathcal{A}/\mathcal{J}\) is complete
if~\(\mathcal{A}\) is metrisable, that is, its topology is
defined by a sequence of \Cstar\nb-seminorms.  It need not be
complete in general, however.  Therefore, we replace the
quotient~\(\mathcal{A}/\mathcal{J}\) by its
completion~\(\mathcal{B}\), which is a pro-\Cstar-algebra as well.
It is the projective limit of the
quotients~\(\mathcal{A}_q/\mathcal{J}_q\) for all \(q\in
\mathcal{N}(A)\).  A \emph{subquotient} of~\(\mathcal{A}\) is a
closed, two-sided \Star{}ideal \(\mathcal{K}\idealin \mathcal{B}\)
with~\(\mathcal{B}\) as above.

Let \(\Rep_\mathrm{b}(A,\mathcal{K})\) consist of all
representations~\(\pi\) of~\(A\) with the following properties:
\begin{enumerate}
\item \(\pi\) is locally bounded, so it comes from a locally bounded
  representation~\(\pi'\) of~\(\mathcal{A}\);
\item the representation~\(\pi'\) annihilates~\(\mathcal{J}\);
\item the representation~\(\bar\pi\) of~\(\mathcal{B}\) induced
  by~\(\pi'\)
  is nondegenerate on~\(\mathcal{K}\), that is,
  \(\bar\pi(\mathcal{K})(\Hilms)\) is a core for~\(\bar\pi\).
\end{enumerate}
Define the \Cstar\nb-algebra \(\Cont_0(\mathcal{K})\) and its dense
ideal \(\Contc(\mathcal{K})\) by replacing \(\mathcal{A}\)
by~\(\mathcal{K}\) in Definition~\ref{def:compact_in_pro-Cstar}.
Equivalently,
\(\Cont_0(\mathcal{K}) = \Cont_0(\mathcal{B}) \cap \mathcal{K}\).

We may choose \(\mathcal{J}=0\) and \(\mathcal{K} = \mathcal{A}\).
Then \(\Rep_\mathrm{b}(A,\mathcal{A}) = \Rep_\mathrm{b}(A)\) simply
consists of all locally bounded representations of~\(A\).  Hence
Theorem~\ref{the:Prim_A_locally_compact_gives_hull} is the special
case \(\mathcal{K} = \mathcal{A}\) of the following theorem:

\begin{theorem}
  \label{the:locally_bounded_relative_hull}
  If \(\Cont_0(\mathcal{K})\) is dense in~\(\mathcal{K}\), then
  \(\Cont_0(\mathcal{K})\) is a \Cstar\nb-hull for
  \(\Rep_\mathrm{b}(A,\mathcal{K})\).
\end{theorem}

\begin{proof}
  First we claim that \(\Rep_\mathrm{b}(A,\mathcal{K})\) is
  equivalent to the class of nondegenerate, locally bounded
  representations of the pro-\Cstar-algebra~\(\mathcal{K}\) as in
  Definition~\ref{def:pro_locally_bounded}.  If \(\mathcal{K}=
  \mathcal{A}\), this is
  Proposition~\ref{pro:locally_bounded_pro-Cstar}.
  A locally bounded representation~\(\pi'\) of~\(\mathcal{A}\)
  descends to a representation~\(\pi''\) of the
  quotient~\(\mathcal{A}/\mathcal{J}\) if and only if it
  annihilates~\(\mathcal{J}\); the induced representation
  of~\(\mathcal{A}/\mathcal{J}\) remains locally bounded with
  respect to the family of \Cstar\nb-seminorms from the quotient
  mappings \(\mathcal{A}/\mathcal{J} \prto
  \mathcal{A}_q/\mathcal{J}_q\).  Hence it extends uniquely to a
  locally bounded representation~\(\bar\pi''\) of the
  completion~\(\mathcal{B}\).  Thus locally bounded representations
  of~\(A\) for which the corresponding representation
  of~\(\mathcal{A}\) annihilates~\(\mathcal{J}\) are equivalent to
  locally bounded representations of~\(\mathcal{B}\).

  We claim that a nondegenerate, locally bounded
  representation~\(\varrho\) of~\(\mathcal{K}\) extends uniquely
  to~\(\mathcal{B}\).  Let~\(q\) be a continuous seminorm
  on~\(\mathcal{B}\), also write~\(q\) for its restriction
  to~\(\mathcal{K}\).  The \(q\)\nb-bounded vectors for~\(\varrho\)
  form a nondegenerate \(\mathcal{K}_q\)\nb-module.  The module
  structure extends uniquely to the multiplier algebra
  of~\(\mathcal{K}_q\), and~\(\mathcal{B}_q\) maps to this
  multiplier algebra because \(\mathcal{K}_q \idealin
  \mathcal{B}_q\).  Letting~\(q\) vary gives a locally bounded
  representation of~\(\mathcal{B}\) that remains nondegenerate
  on~\(\mathcal{K}\).  Conversely, any such
  representation of~\(\mathcal{B}\) is obtained in this way from its
  restriction to~\(\mathcal{K}\).  Thus representations of~\(A\)
  that belong to \(\Rep_\mathrm{b}(A,\mathcal{K})\) are equivalent
  to nondegenerate, locally bounded representations of the
  pro-\Cstar\nb-algebra~\(\mathcal{K}\).  The equivalence above is
  compatible with isometric intertwiners, direct sums and interior
  tensor products, compare
  Proposition~\ref{pro:locally_bounded_pro-Cstar}.

  Lemma~\ref{lem:compactly_supported_ideal} shows that
  \(\Contc(\mathcal{K}) \defeq \Contc(\mathcal{B})\cap \mathcal{K}\)
  is dense in~\(\mathcal{K}\).  This is an ideal in~\(\mathcal{B}\)
  as an intersection of two ideals.  Hence left multiplication
  defines a representation of~\(\mathcal{B}\) on~\(\mathcal{K}\)
  with core~\(\Contc(\mathcal{K})\), which is locally bounded by
  construction.  Through the canonical homomorphisms \(A\to
  \mathcal{A}\to \mathcal{B}\), this becomes a representation
  of~\(A\).  This representation clearly belongs to
  \(\Rep_\mathrm{b}(A,\mathcal{K})\).  We claim that it is the
  universal representation for the class
  \(\Rep_\mathrm{b}(A,\mathcal{K})\).
  So let~\(\pi\)
  be any representation in \(\Rep_\mathrm{b}(A,\mathcal{K})\).
  Then~\(\pi\)
  comes from a unique nondegenerate, locally bounded
  representation~\(\bar\pi\)
  of~\(\mathcal{K}\).
  We must show that it comes from a unique nondegenerate
  representation~\(\varrho\)
  of~\(\Cont_0(\mathcal{K})\).

  Let~\(\varrho\)
  be the restriction of~\(\bar\pi\)
  to~\(\Cont_0(\mathcal{K})\).
  Then
  \(\varrho(\Contc(\mathcal{K}))\Hilm \subseteq \Hilm_\mathrm{b}
  \subseteq \Hilms\).  We are going to prove that this is a core.
  The bilinear map \(\mathcal{K} \odot \Hilm_\mathrm{b} \to \Hilms\)
  is separately continuous with respect to the pro-\Cstar-algebra
  topology on~\(\mathcal{K}\)
  and the inductive limit topology on
  \(\Hilm_\mathrm{b} = \bigcup_{q\in \mathcal{N}(A)} \Hilm_q\).
  We have assumed that it has dense range.  Since
  \(\Contc(\mathcal{K})\)
  is dense in~\(\mathcal{K}\),
  the image of \(\Contc(\mathcal{K}) \odot \Hilm_\mathrm{b}\)
  is a core.  Thus \(\varrho(\Contc(\mathcal{K}))\Hilm\)
  is dense in~\(\Hilms\)
  in the graph topology.  The representation~\(\varrho\)
  is nondegenerate, and the associated representation of~\(A\)
  is~\(\pi\).
  So~\(\pi\) comes from a representation of~\(\Cont_0(\mathcal{K})\).

  The uniqueness of~\(\varrho\)
  means that~\(\Cont_0(\mathcal{K})\)
  is a weak \(\Cst\)\nb-hull
  for some class of integrable representations of~\(A\).
  We check this using
  Proposition~\ref{pro:criterion_Cstar-hull_locally_bounded}.  For
  \(q\in \mathcal{N}(A)\),
  the image of~\(\mathcal{A}_\mathrm{b}\)
  in~\(\Mult(\mathcal{K}_q)\)
  contains~\(\mathcal{K}_q\)
  and hence is strictly dense.  This implies that the image
  of~\(\mathcal{A}_\mathrm{b}\)
  in~\(\Mult(\Cont_0(\mathcal{K}))\)
  is strictly dense.  So~\(\Cont_0(\mathcal{K})\)
  is a weak \(\Cst\)\nb-hull
  for a class of representations of~\(A\)
  by Proposition~\ref{pro:criterion_Cstar-hull_locally_bounded} It is
  even a \Cstar\nb-hull because the class of locally bounded
  representations is admissible by
  Corollary~\ref{cor:locally_bounded_admissible}.
\end{proof}

\section{Commutative \Cstar-hulls}
\label{sec:commutative_hulls}

Let~\(A\) be a commutative \Star{}algebra.  We are going to describe
all \emph{commutative} weak \Cstar\nb-hulls for~\(A\).  Actually, we
describe all locally bounded weak \Cstar\nb-hulls, and these turn out
to be the same as the commutative ones.  We
study when a \Cstar\nb-hull satisfies the (Strong) Local--Global
Principle and when the class of
\emph{all} locally bounded representations has a \Cstar\nb-hull.  We compare
the class of locally bounded representations with the class of
representations defined by requiring all \(a\in A_h\) to act by a
regular, self-adjoint operator.

\begin{proposition}
  \label{pro:commutative_representation}
  Let~\(A\) be a \Star{}algebra and let \(B=\Cont_0(X)\)
  be a commutative \Cstar\nb-algebra.  Any representation of~\(A\)
  on~\(B\) has \(\Contc(X)\) as a core and is locally bounded.
  There is a natural bijection between representations of~\(A\)
  on~\(B\), \Star{}homomorphisms \(A\to\Cont(X)\), and continuous
  maps \(\hat{A}\to X\).
\end{proposition}

\begin{proof}
  Let~\((\Hilms[B],\mu)\) be a representation.  Since~\(\Hilms[B]\)
  is dense in~\(B\), for any \(x\in X\) there is \(f\in \Hilms[B]\)
  with \(f(x)\neq0\).  Then there is an open neighbourhood of~\(x\)
  on which~\(f\) is non-zero.  A compact subset~\(K\) of~\(X\) may
  be covered by finitely many such open neighbourhoods.  This gives
  finitely many functions \(f_1,\dotsc,f_n\in\Hilms[B]\) so that
  \(\sum f_i\cdot \conj{f_i}(x)>0\) for all \(x\in K\).  This sum
  again belongs to the right ideal~\(\Hilms[B]\), and
  hence~\(\Hilms[B]\) contains~\(\Contc(X)\).  There is an
  approximate unit~\((u_i)_{i\in I}\) for~\(\Cont_0(X)\) that
  belongs to~\(\Contc(X)\).  If \(b\in\Hilms[B]\), then \(\lim
  \mu(a)b u_i = \mu(a) b\) for all \(a\in A\).  That is, \(\lim b
  u_i = b\) in the graph topology.  Since \(b u_i\in \Contc(X)\),
  \(\Contc(X)\) is a core for~\((\Hilms[B],\mu)\).

  Given \(a\in A\), we define a function \(f_a\colon X\to\C\) by
  \(f_a(x) \defeq (\mu(a)b)(x)\cdot b(x)^{-1}\) for any
  \(b\in\Contc(X)\) with \(b(x)\neq0\).  This does not depend on the
  choice of~\(b\), and~\(f_a\) is continuous in the open subset
  where \(b\neq0\).  Thus \(f_a\in\Cont(X)\).  The map
  \(A\to\Cont(X)\), \(a\mapsto f_a\), is a \Star{}homomorphism.
  Conversely, any \Star{}homomorphism \(A\to\Cont(X)\) gives a
  representation of~\(A\) on~\(\Cont_0(X)\) with core~\(\Contc(X)\)
  by \(\mu(a)b = f_a\cdot b\) for all \(a\in A\), \(b\in\Contc(X)\).
  The maps that go back and forth between representations
  on~\(\Cont_0(X)\) and \Star{}homomorphisms \(A\to\Cont(X)\) are
  inverse to each other.

  A \Star{}homomorphism \(f\colon A\to\Cont(X)\) gives a continuous
  map \(X\to\hat{A}\) by mapping \(x\in X\) to the character
  \(a\mapsto f(a)(x)\).  Conversely, a continuous map \(g\colon
  X\to\hat{A}\) induces a \Star{}homomorphism \(g^*\colon
  A\to\Cont(X)\), \(g^*(a)(x) \defeq g(x)(a)\), and these two
  constructions are inverse to each other.

  Let \(f\colon X\to\hat{A}\) be a continuous map.  Then~\(f\) maps
  compact subsets in~\(X\) to compact subsets of~\(\hat{A}\).  If
  \(K\subseteq X\) is compact, then
  any element in \(\Cont_0(K\setminus\partial K)\subseteq \Cont_0(X)\)
  is \(\norm{\blank}_{f(K)}\)-bounded for the \Cstar\nb-seminorm
  on~\(A\) associated to the compact subset \(f(K)\subseteq
  \hat{A}\).  Thus all elements in~\(\Contc(X)\) are bounded.
  Since~\(\Contc(X)\) is a core for the representation of~\(A\)
  associated to~\(f\), this representation is locally bounded.
\end{proof}

\begin{theorem}
  \label{the:commutative_hull}
  Let~\(A\) be a commutative \Star{}algebra, let \(B=\Cont_0(X)\) be
  a commutative \Cstar\nb-algebra, let \(f\colon X\to \hat{A}\) be a
  continuous map, and let~\((\Hilms[B],\mu)\) be the corresponding
  representation of~\(A\) on~\(B\).  Call a representation of~\(A\)
  on a Hilbert module~\(\Hilm\) \emph{\(X\)\nb-integrable} if it is
  isomorphic to \((\Hilms[B],\mu)\otimes_\varrho\Hilm\) for a
  representation~\(\varrho\) of~\(B\) on~\(\Hilm\).  The following
  are equivalent:
  \begin{enumerate}
  \item \(f\colon X\to \hat{A}\) is injective;
  \item \(B\) is a weak \Cstar\nb-hull for the \(X\)\nb-integrable
    representations;
  \item \(B\) is a \Cstar\nb-hull for the \(X\)\nb-integrable
    representations.
  \end{enumerate}
  Furthermore, any locally bounded weak \Cstar\nb-hull of~\(A\) is
  commutative.
\end{theorem}

\begin{proof}
  If~\(f\) is not injective, then there are \(x\neq y\) in~\(X\)
  with \(f(x) = f(y)\).  The evaluation maps at \(x\) and~\(y\) are
  different \(1\)\nb-dimensional representations of~\(B\) that
  induce the same representation of~\(A\).
  Hence the condition~\ref{enum:Cstar-gen1} in
  Proposition~\ref{pro:Cstar-generated_by_unbounded_multipliers} is
  violated and so~\(B\) is not a weak \Cstar\nb-hull.  Conversely,
  assume that~\(f\) is injective.

  The representation of~\(A\) on~\(B\) associated to~\(f\) is locally
  bounded by Proposition~\ref{pro:commutative_representation} and
  hence induces a representation of the unital \Cstar{}algebra
  \(\Contb(\hat{A},\tau_c)\) of bounded elements in \(\mathcal{A}
  \cong \Cont(\hat{A},\tau_c)\), see
  Proposition~\ref{pro:commutative_pro-completion}.  Explicitly, this
  representation composes functions with~\(f\).  Since~\(f\) is
  injective, \(D\defeq f^*(\Contb(\hat{A},\tau_c)) \subseteq
  \Contb(X)\) separates the points of~\(X\).  We show that~\(D\) is
  strictly dense in \(\Contb(X) \cong \Mult(B)\).

  If \(K\subseteq X\) is compact, then the image of
  \(f^*(\Contb(\hat{A},\tau_c))|_K\) in the quotient~\(\Cont(K)\)
  of~\(\Cont_0(X)\) separates the points of~\(K\).  Since this image
  is again a \Cstar\nb-algebra, it is equal to~\(\Cont(K)\) by the
  Stone--Weierstraß Theorem.  Let \(f\in \Contb(X)\).  For any compact
  subset \(K\subseteq X\), there is \(d_K\in D\) with \(d_K|_K = f\).
  By functional calculus, we may arrange that \(\norm{d_K}_\infty \le
  \norm{f}\).  The net~\((d_K)\) indexed by the directed set of
  compact subsets \(K\subseteq X\) is uniformly bounded and converges
  towards~\(f\) in the topology of uniform convergence on compact
  subsets.  Hence it converges towards~\(f\) in the strict topology
  (compare \cite{Grundling-Neeb:Full_regularity}*{Lemma A.1}).  This
  finishes the proof that \(f^*(\mathcal{A}_\mathrm{b})\) is strictly
  dense in~\(\Mult(\Cont_0(X))\).
  Proposition~\ref{pro:criterion_Cstar-hull_locally_bounded} shows
  that~\(B\) is a weak \Cstar\nb-hull for the \(B\)\nb-integrable
  representations of~\(A\).

  Any \(X\)\nb-integrable representation of~\(A\) is locally
  bounded.  The class \(\Rep_\mathrm{b}(A)\) of locally bounded
  representations of~\(A\) is admissible by
  Corollary~\ref{cor:locally_bounded_admissible}.  Hence the smaller
  class of \(X\)\nb-integrable representations inherits the
  equivalent conditions
  \ref{enum:Cstar-gen2}--\ref{enum:Cstar-gen2''} in
  Proposition~\ref{pro:Cstar-generated_by_unbounded_multipliers}.
  Thus~\(\Cont_0(X)\) is even a \Cstar\nb-hull.

  Let~\(B\) with the universal representation \((\Hilms[B],\mu)\) be
  a locally bounded weak \Cstar\nb-hull.  Then the image of
  \(\Contb(\hat{A},\tau_0)\) in the multiplier algebra of~\(B\) is
  strictly dense by
  Proposition~\ref{pro:criterion_Cstar-hull_locally_bounded}.
  Thus~\(\Mult(B)\) is commutative, and then so is~\(B\).  Thus a
  locally bounded weak \Cstar\nb-hull is commutative.
\end{proof}

\begin{theorem}
  \label{the:commutative_local-global}
  Let~\(A\) be a commutative \Star{}algebra, let \(B=\Cont_0(X)\) be
  a commutative \Cstar\nb-algebra, and let \(f\colon X\to \hat{A}\)
  be an injective continuous map.  Let \(\Repi(A,X)\) be the class
  of \(X\)\nb-integrable representations.  The following statements
  are equivalent if~\(\hat{A}\) is metrisable:
  \begin{enumerate}
  \item \label{enum:commutative_local-global1} \(f\colon X\to
    \hat{A}\) is a homeomorphism onto its image;
  \item \label{enum:commutative_local-global2} \(\Repi(A,X)\) is
    defined by submodule conditions;
  \item \label{enum:commutative_local-global3} \(\Repi(A,X)\)
    satisfies the Strong Local--Global Principle;
  \item \label{enum:commutative_local-global4} \(\Repi(A,X)\)
    satisfies the Local--Global Principle;
  \item \label{enum:commutative_local-global5} if \(\lim f(x_n) =
    f(x)\) for a sequence~\((x_n)_{n\in\N}\) in~\(X\) and \(x\in
    X\), then already \(\lim x_n = x\).
  \end{enumerate}
  The implications
  \ref{enum:commutative_local-global1}\(\Rightarrow\)\ref{enum:commutative_local-global2}\(\Rightarrow\)\ref{enum:commutative_local-global3}\(\Rightarrow\)\ref{enum:commutative_local-global4}\(\Rightarrow\)\ref{enum:commutative_local-global5}
  hold without assumptions on~\(\hat{A}\).
\end{theorem}

I do not know whether
\ref{enum:commutative_local-global1}--\ref{enum:commutative_local-global4}
are equivalent in general.  The
condition~\ref{enum:commutative_local-global5} is there to allow to
go back from~\ref{enum:commutative_local-global4}
to~\ref{enum:commutative_local-global1} at least for
metrisable~\(\hat{A}\).

\begin{proof}
  First we check
  \ref{enum:commutative_local-global5}\(\Rightarrow\)\ref{enum:commutative_local-global1}
  if~\(\hat{A}\) is metrisable.  If~\(f\) is not a homeomorphism
  onto its image, then there is a subset \(U\subseteq X\) that is
  open, such that \(f(U)\) is not open in the subspace topology on
  \(f(X)\subseteq\hat{A}\).  Since~\(\hat{A}\) is metrisable, there
  is \(x\in U\) and a sequence in \(f(X)\setminus f(U)\) that
  converges towards~\(f(x)\).  This lifts to a sequence
  \((x_n)_{n\in\N}\) in \(X\setminus U\) such that \(\lim f(x_n) =
  f(x)\).  We cannot have \(\lim x_n = x\) because~\(x_n\) never
  enters the open neighbourhood~\(U\) of~\(x\).

  The implication
  \ref{enum:commutative_local-global2}\(\Rightarrow\)\ref{enum:commutative_local-global3}
  is Theorem~\ref{the:operator_conditions}, and
  \ref{enum:commutative_local-global3}\(\Rightarrow\)\ref{enum:commutative_local-global4}
  is trivial.  We are going to verify
  \ref{enum:commutative_local-global1}\(\Rightarrow\)\ref{enum:commutative_local-global2}
  and
  \ref{enum:commutative_local-global4}\(\Rightarrow\)\ref{enum:commutative_local-global5}.
  This will finish the proof of the theorem.

  Assume~\ref{enum:commutative_local-global1}.  Let~\(\pi\) be a
  representation in~\(\Repi(A,X)\).  Then~\(\pi\) is locally
  bounded, and the
  operators~\(\cl{\pi(a)}\) for \(a\in A_h\) are regular and
  self-adjoint by Proposition~\ref{pro:locally_bounded_regular}.
  Furthermore, their Cayley transforms belong to the image of
  \(\mathcal{A}_\mathrm{b} \cong \Contb(\hat{A},\tau_c)\), which is commutative.
  Hence the operators~\(\cl{\pi(a)}\) for \(a\in A_h\) strongly
  commute with each other.  The class \(\Repi(A)\) of representations
  of~\(A\) with the property that all~\(\cl{\pi(a)}\), \(a\in A_h\), are regular and
  self-adjoint and strongly commute with each other is defined by
  submodule conditions by Examples \ref{exa:regularity_condition} and
  Example~\ref{exa:strong_commutation}.

  Let \(Y \defeq \prod_{a\in A_h} S^1\).  Given a representation
  in~\(\Repi(A)\), there is a unique representation \(\varrho\colon
  \Cont(Y)\to \Bound(\Hilm)\) that maps the \(a\)th coordinate
  projection to the Cayley transform of~\(\cl{\pi(a)}\).  We map
  \(\hat{A}\) to~\(Y\) by sending \(\chi\in\hat{A}\) to the point
  \((c_{\chi(a)})_{a\in A_h} \in Y\).  Here~\(c_{\chi(a)}\) is the
  Cayley transform of the number \(\chi(a)\in \R\) or, equivalently,
  the value of the Cayley transform of the unbounded function
  \(\hat{a}\in\Cont(\hat{A})\) at~\(\chi\).  This is a homeomorphism
  onto its image because for a net of characters~\((\chi_i)\) and a
  character~\(\chi\) on~\(A\), we have \(\lim \chi_i(a) = \chi(a)\) if
  and only if \(\lim c_{\chi_i(a)} = c_{\chi(a)}\).  Thus the
  composite map \(X\to \hat{A}\to Y\) is a homeomorphism onto its
  image as well.  This forces the image to be locally closed
  because~\(Y\) is compact and~\(X\) locally compact, and a subspace
  of a locally compact space is locally compact if and only if its
  underlying subset is locally closed (see
  \cite{Bourbaki:Topologie_generale}*{I.9.7, Propositions 12 and 13}).

  Let \(\cl{X}\subseteq Y\) be the closure of the image of~\(X\)
  in~\(Y\).  Then~\(X\) is open in~\(\cl{X}\).  All representations
  in~\(\Repi(A)\) carry a unital \Star{}homomorphism
  \(\Cont(Y)\to\Bound(\Hilm)\).  Asking for this to factor through the
  quotient~\(\Cont(\cl{X})\) of~\(\Cont(Y)\) is a submodule condition
  as in Example~\ref{exa:spectral_condition}.  Asking for the induced
  \Star{}homomorphism \(\Cont(\cl{X})\to\Bound(\Hilm)\) to remain
  nondegenerate on~\(\Cont_0(X)\) is another submodule condition as in
  Example~\ref{exa:spectral_condition}.

  The class \(\Rep'(A)\) defined by these two more submodule
  conditions is weakly admissible by
  Lemma~\ref{lem:operator_conditions}.  The universal
  \(X\)\nb-integrable representation belongs to \(\Rep'(A)\); by
  weak admissibility, this is inherited by all \(X\)\nb-integrable
  representations.  Conversely, we claim that any representation in
  \(\Rep'(A)\) is \(X\)\nb-integrable.

  If \(\pi\in\Rep'(A)\), then the unital
  \Star{}homomorphism \(\Cont(Y)\to\Bound(\Hilm)\) descends to a
  nondegenerate \Star{}homomorphism \(\varrho\colon
  \Cont_0(X)\to\Bound(\Hilm)\).  By construction, the extension
  of~\(\varrho\) to multipliers maps the Cayley transform of
  \(f^*(a)\in\Cont(X)\) for \(a\in A_h\) to the Cayley transform
  of~\(\cl{\pi(a)}\).  Let~\(\pi'\) be the \(X\)\nb-integrable
  representation of~\(A\) associated to~\(\varrho\).  The
  regular, self-adjoint operators \(\cl{\pi'(a)}\)
  and~\(\cl{\pi(a)}\) have the same Cayley transform for all \(a\in
  A_h\).  Hence \(\cl{\pi'(a)} = \cl{\pi(a)}\) for all \(a\in A_h\).
  The subset~\(A_h\) is a strong generating set for~\(A\)
  by Example~\ref{exa:strong_generating}.  Hence
  Proposition~\ref{pro:equality_if_closure_equal} gives \(\pi'=\pi\).
  Thus \(\Repi(A,X)\) is the class of representations defined by the
  submodule conditions above.  This finishes the proof that
  \ref{enum:commutative_local-global1}\(\Rightarrow\)\ref{enum:commutative_local-global2}.

  Now we prove
  \ref{enum:commutative_local-global4}\(\Rightarrow\)\ref{enum:commutative_local-global5}
  by contradiction.  Let \((x_n)_{n\in\N}\) and~\(x\) be as
  in~\ref{enum:commutative_local-global5}.  Let \(\bar{\N} =
  \N\cup\{\infty\}\) be the one-point compactification of~\(\N\) and
  view the sequence~\((x_n)\) and~\(x\) as a map \(\xi\colon
  \bar{\N}\to X\).  This map is not continuous, but composition
  with~\(f\) gives a continuous map \(\bar{\N}\to\hat{A}\).  Hence
  Proposition~\ref{pro:commutative_representation} gives a
  representation~\((\Hilms[D],\mu)\) of~\(A\)
  on~\(\Cont(\bar{\N})\).  This is not \(X\)\nb-integrable because
  the map \(\bar{\N}\to X\) is not continuous.  We claim, however,
  that the representation \((\Hilms[D],\mu) \otimes_\varrho
  \Hilm[H]\) is \(X\)\nb-integrable for any
  GNS-representation~\(\varrho\) on a Hilbert space~\(\Hilm[H]\).
  A state on~\(\Cont(\bar{\N})\) is the same as a Radon measure
  on~\(\bar{\N}\).  Since~\(\bar{\N}\) is countable, any Radon
  measure is atomic.  Thus the resulting GNS-representation is a
  direct sum of irreducible representations associated to
  characters.  Each character on~\(\Cont(\bar{\N})\) gives an
  \(X\)\nb-integrable representation because \(\xi(\bar{\N})
  \subseteq f(X)\).
  Hence~\((\Hilms[D],\mu)\) is a counterexample to the
  Local--Global Principle.  So~\ref{enum:commutative_local-global4}
  cannot hold if~\ref{enum:commutative_local-global5} fails.
\end{proof}

\begin{example}
  \label{exa:pure-point}
  Let \(A=\C[x]\), so that \(\hat{A}=\R\).  Let~\(X\) be~\(\R\) with
  the discrete topology, and let \(f\colon X\to\R\) be the identity
  map.  This is a continuous bijection, but not open.  Hence the
  class of \(X\)\nb-integrable representations violates the
  Local--Global Principle by
  Theorem~\ref{the:commutative_local-global}.  Nevertheless,
  \(\Cont_0(X)\) is a \Cstar\nb-hull for the class of
  \(X\)\nb-integrable representations of~\(A\) by
  Theorem~\ref{the:commutative_hull}.  An \(X\)\nb-integrable
  representation of~\(A\) is integrable as in
  Theorem~\ref{the:regular_self-adjoint_Cstar-hull}, and so it comes
  from a single regular, self-adjoint operator \(T\defeq
  \cl{\pi(x)}\).  The representation of~\(\C[x]\) associated
  to~\(T\) is \(X\)\nb-integrable if and only if \(\Hilm =
  \bigoplus_{\lambda\in\R} \Hilm_\lambda\), where \(\Hilm_\lambda
  \defeq \{\xi\in\Hilm \mid T\xi= \lambda\xi\}\) for \(\lambda\in\R\)
  is the \(\lambda\)\nb-eigenspace of~\(T\).
\end{example}

Another example of a \Cstar\nb-hull for~\(\C[x]\) where \(X\to\R\)
is bijective but not a homeomorphism onto its image is discussed in
Example~\ref{exa:smaller_integrable_than_regular_self-adjoint}.

\begin{theorem}
  \label{the:no_Cstar-hull_commutative_not_locally_compact}
  There is a \Cstar\nb-hull for
  \(\Rep_\mathrm{b}(A)\) if and only if the compactly generated
  topology~\(\tau_c\) on~\(\hat{A}\) is locally compact, and then
  the \Cstar\nb-hull is \(\Cont_0(\hat{A},\tau_c)\).
\end{theorem}

\begin{proof}
  Assume first that~\((\hat{A},\tau_c)\) is locally compact.  The
  pro-\Cstar\nb-algebra completion~\(\mathcal{A}\) that acts on
  locally bounded representations of~\(A\) is
  \(\Cont(\hat{A},\tau_c)\) by
  Proposition~\ref{pro:commutative_pro-completion}.  The primitive
  ideal space of \(\Cont(K)\) for a compact subspace \(K\subseteq
  \hat{A}\) is simply~\(K\), and \(\norm{a}_\mathfrak{p} =
  \abs{a(\mathfrak{p})}\) for \(a\in\Cont(\hat{A},\tau_c)\) and
  \(\mathfrak{p}\in \Prim \Cont(K) \cong K\).  Therefore, a function
  \(f\in \Cont(\hat{A},\tau_c)\) vanishes at~\(\infty\) in the sense
  of Definition~\ref{def:compact_in_pro-Cstar} if and only if it
  vanishes at~\(\infty\) in the usual sense.  The subalgebra
  \(\Cont_0(\mathcal{A}) = \Cont_0(\hat{A},\tau_c)\) is dense
  in~\(\mathcal{A}\)
  because~\(\tau_c\) is locally compact.  Now
  Theorem~\ref{the:Prim_A_locally_compact_gives_hull} shows that
  \(\Cont_0(\mathcal{A}) = \Cont_0(\hat{A},\tau_c)\) is a
  \Cstar\nb-hull for the class of locally bounded representations
  of~\(A\).

  Conversely, let~\(B\) be a (weak) \Cstar\nb-hull for the locally
  bounded representations of~\(A\).  Then~\(B\) is commutative by
  Theorem~\ref{the:commutative_hull}.  Let~\(Y\) be the spectrum
  of~\(B\).  The representation of~\(A\) on \(B\cong \Cont_0(Y)\)
  corresponds to a continuous map \(f\colon Y\to\hat{A}\) by
  Proposition~\ref{pro:commutative_representation}.  Let
  \(D=\Cont_0(X)\) be a commutative \Cstar\nb-algebra.  Any
  representation of~\(A\) on~\(D\) is locally bounded.  So the
  bijection \(\Rep_\mathrm{b}(A,D)\cong \Rep(B,D)\) is a bijection
  between the spaces of continuous maps \(X\to\hat{A}\) and \(X\to
  Y\).  More precisely, this bijection is composition with~\(f\).

  For the one-point space~\(X\), this bijection says that \(f\colon
  Y\to\hat{A}\) is bijective.  The bijection for all compact~\(X\)
  means that~\(f\) becomes a homeomorphism if we replace the
  topologies on \(Y\) and~\(\hat{A}\) by the associated compactly
  generated ones.  The topology on~\(Y\) is already compactly
  generated because~\(Y\) is locally compact.  Hence~\(f\) is a
  homeomorphism from~\(Y\) to \((\hat{A},\tau_c)\).  So~\(\tau_c\)
  is locally compact.
\end{proof}

Let \(\Repi(A)\) be the class of all representations with the
property that~\(\cl{\pi(a)}\) is regular and self-adjoint for all
\(a\in A_h\).  We are going to compare \(\Repi(A)\) and
\(\Rep_\mathrm{b}(A)\).
Proposition~\ref{pro:locally_bounded_regular} gives
\(\Rep_\mathrm{b}(A)\subseteq\Repi(A)\).

\begin{theorem}
  \label{the:commutative_regular_admissible}
  The class \(\Repi(A)\) is admissible and defined by submodule
  conditions.  Hence it satisfies the Strong Local--Global Principle.
  The operators~\(\cl{\pi(a)}\) for \(a\in A_h\) strongly commute for
  all \(\pi\in\Repi(A)\).

  Let \(S\subseteq A_h\) be a strong generating set for~\(A\).
  If\/~\(\cl{\pi(a)}\) is regular and self-adjoint for all \(a\in S\),
  then already \(\pi\in\Repi(A)\).
\end{theorem}

\begin{proof}
  The class \(\Rep^S(A)\) of representations defined by
  requiring~\(\cl{\pi(a)}\) to be regular and self-adjoint for all
  \(a\in S\) for a strong generating set~\(S\) is admissible and
  defined by submodule conditions by
  Theorem~\ref{the:regular_admissible_local-global}.  The class
  \(\Repi(A)\) is defined by submodule conditions as well by
  Example~\ref{exa:regularity_condition}.  So is the subclass
  \(\Rep'(A)\) of all representations in~\(\Repi(A)\) for which the
  operators~\(\cl{\pi(a)}\) for all \(a\in A_h\) strongly commute
  (Example~\ref{exa:strong_commutation}).  Hence our three classes of
  representations satisfy the Strong Local--Global Principle by
  Theorem~\ref{the:operator_conditions}.

  The classes \(\Repi(A)\) and~\(\Rep'(A)\) have the same Hilbert
  space representations by \cite{Schmudgen:Unbounded_book}*{Theorem
    9.1.2}.  Since~\(S\) is a strong generating set, the domain of any
  representation~\(\pi\) in \(\Rep^S(A)\) is \(\bigcap_{a\in S}
  \dom(\cl{\pi(a)})\) by~\eqref{eq:domain_strong_generating_set}.
  This contains \(\bigcap_{a\in S, n\in\N} \dom(\cl{\pi(a)}{}^n)\).
  Now \cite{Schmudgen:Unbounded_book}*{Theorem 9.1.3} shows that
  \(\Rep^S(A)\) and \(\Repi(A)\) contain the same Hilbert space
  representations.  Since our three classes of representations satisfy
  the (Strong) Local--Global Principle and have the same Hilbert space
  representations, they are equal.
\end{proof}

\begin{theorem}
  \label{the:locally_bounded_versus_regular}
  If~\(A\) is commutative and countably generated, then
  \[
  \Repi(A)=\Rep_\mathrm{b}(A).
  \]
\end{theorem}

\begin{proof}
  Proposition~\ref{pro:locally_bounded_regular} gives
  \(\Rep_\mathrm{b}(A)\subseteq \Repi(A)\).  Conversely,
  let~\((\Hilms,\pi)\) be a representation on a Hilbert
  module~\(\Hilm\) in~\(\Repi(A)\); that is, \(\cl{\pi(a)}\) is
  regular and self-adjoint for each \(a\in A\).  Let
  \((a_i)_{i\in\N}\) be a countable generating set for~\(A\).  We
  may assume without loss of generality that \(a_i = a_i^*\) for all
  \(i\in\N\) and that~\((a_i)\) is a basis for~\(A\) and hence a
  strong generating set.  Let \(\xi\in\Hilms\).  We are going to
  approximate~\(\xi\) by bounded vectors for~\(\pi\).  This will
  show that~\(\pi\) is locally bounded.

  For each \(i\in\N\), there is a canonical homomorphism
  \(\alpha_i\colon \C[x]\to A\) mapping \(x\mapsto a_i\).  The
  closure of \(\pi\circ\alpha_i\) is an integrable representation
  of~\(\C[x]\) as in
  condition~\ref{enum:regular_self-adjoint_Cstar-hull2} in
  Theorem~\ref{the:regular_self-adjoint_Cstar-hull}.  Hence it
  corresponds to a representation \(\varrho_i\colon
  \Cont_0(\R)\to\Bound(\Hilm)\), the functional calculus
  of~\(\cl{\pi(a_i)}\).  The operators~\(\cl{\pi(a)}\) for \(a\in
  A_h\) strongly commute by
  Theorem~\ref{the:commutative_regular_admissible}.  Thus the Cayley
  transform of~\(a_i\) commutes with~\(\cl{\pi(a)}\) and, in
  particular, maps the domain of~\(\cl{\pi(a)}\) to itself.  The
  same remains true for \(\varrho_i(f)\) for all \(f\in\Cont_0(\R)\)
  because we get them by the (bounded) functional calculus for the
  Cayley transform of~\(\cl{\pi(a_i)}\).  So
  \(\varrho_i(f)(\Hilms)\subseteq \Hilms\)
  by~\eqref{eq:domain_closure} and \(\varrho_i(f) \pi(a) = \pi(a)
  \varrho_i(f)\) for all \(f\in\Cont_0(\R)\), \(a\in A\) as
  operators on~\(\Hilms\).  Now we show that~\(\pi\circ\alpha_i\) is
  locally bounded.  If \(f\in\Contc(\R)\) is supported in a compact
  subset \(K\subseteq \R\), then
  \[
  \norm{\pi(h(a_i))\varrho_i(f)\xi}
  = \norm{\varrho_i(h\cdot f)\xi}
  \le C \sup \{\abs{h(x)}\mid x\in K\}
  \]
  for all \(h\in\C[x]\); thus~\(\varrho_i(f)\xi\) is bounded for the
  representation~\(\pi\circ\alpha_i\).  There is an approximate
  unit~\((f_n)\) for~\(\Cont_0(\R)\) that lies in~\(\Contc(\R)\).
  Then \(\lim \varrho_i(f_n)\xi=\xi\) for all \(\xi\in\Hilms\), even
  in the graph topology for~\(\pi\) because \(\pi(a)\varrho_i(f_n)\xi =
  \varrho_i(f_n)\pi(a)\xi\) for all \(a\in A\),
  \(f_n\in\Cont_0(\R)\), \(\xi\in\Hilms\).  Therefore, the bounded
  vectors of the form~\(\varrho(f)\xi\) with \(f\in\Contc(\R)\),
  \(\xi\in\Hilms\) form a core for \(\pi\circ\alpha_i\).  So
  \(\pi\circ\alpha_i\) is locally bounded.

  We now refine this construction to approximate~\(\xi\) by bounded
  vectors for the whole representation~\(\pi\).  We
  construct~\(\varrho_i\) as above.  Fix \(i,k\in\N\) and let \(\xi'
  \defeq \bigl(1+\pi(a_0^2) + \dotsb + \pi(a_k^2)\bigr)\xi
  \in\Hilms\).  The argument above gives \(f_{i,k}\in\Contc(\R)\)
  with \(0\le f_{i,k}\le 1\) and \(\norm{\varrho_i(f_{i,
      k})\xi'-\xi'} < 2^{-k}\).  Thus \(\norm{\varrho_i(f_{i,k})\xi
    - \xi}_{a_j}<2^{-k}\) in the graph norm for~\(a_j\) for \(0\le
  j\le k\).  For \(k,l\in\N\), let
  \[
  \xi_{k,l} \defeq \varrho_0(f_{0,k}) \varrho_1(f_{1,k+1}) \dotsm
  \varrho_l(f_{l,k+l})\xi.
  \]
  The operators~\(\varrho_i(f_{i,j})\) are norm-contracting,
  map~\(\Hilms\) into itself, and commute with each other and with
  the unbounded operators~\(\pi(a)\) for all \(a\in A\).  Hence
  \begin{multline*}
    \norm{\xi_{k,l} - \xi_{k,l+d}}_{a_j}
    \le \sum_{i=1}^d {}\norm{\xi_{k,l+i-1} - \xi_{k,l+i}}_{a_j}
    \\\le \sum_{i=1}^d {}\norm{\varrho_{l+i}(f_{l+i,k+l+i})\xi - \xi}_{a_j}
    \le \sum_{i=1}^d 2^{-k-l-i} = 2^{-k-l}
  \end{multline*}
  for all \(k,l,d\in\N\), \(0\le j\le k+l+1\).  Since we
  assumed~\((a_j)\) to be a strong generating set, the graph norms
  for~\(a_j\) generate the graph topology.  So the estimate above
  shows that~\((\xi_{k,l})_{l\in\N}\) with fixed~\(k\) is a Cauchy
  sequence in~\(\Hilms\) in the graph topology.  Thus it converges
  to some \(\xi_k\in\Hilms\).  Letting \(\xi_{k,-1}\defeq\xi\), the
  above estimate remains true for \(l=-1\) and gives
  \(\norm{\xi_{k,l} - \xi}_{a_j} \le 2^{-k+1}\) for all \(j\le k\),
  uniformly in~\(l\in\N\).  This implies \(\norm{\xi_k-\xi}_{a_j}\le
  2^{-k+1}\) for \(j\le k\), so that \(\lim \xi_k=\xi\) in the graph
  topology.  It
  remains to show that each~\(\xi_k\) is a bounded vector.

  Fix \(k,i\in\N\) and let \(b\in A\).  Choose \(R_i>0\) so
  that~\(f_{i,k+i}\) is supported in \([-R_i,R_i]\).  If \(l\ge i\),
  then \(\pi(b)\xi_{k,l}\in \varrho_i(\Cont_0(-R_i,R_i))\Hilm\)
  because \(\varrho_i(f_{i,k+i})\) occurs in the definition
  of~\(\xi_{k,l}\).  As above, this implies
  \(\norm{\pi(a_i)\pi(b)\xi_k} \le R_i \norm{\pi(b)\xi_k}\) for all
  \(b\in A\).  Thus
  \[
  q(a) \defeq \sup_{b\in A}
  \frac{\norm{\pi(a)\pi(b)\xi_k}}{\norm{\pi(b)\xi_k}}
  \]
  is finite for \(a=a_i\).  Since~\(a_i\) is a basis for~\(A\)
  and~\(q\) is subadditive, we get \(q(a)<\infty\) for all \(a\in
  A\).  Since \(q(a)\) is the operator norm of
  \(\pi(a)|_{\pi(A)\xi_k}\), it is a \Cstar\nb-seminorm on~\(A\).
  By construction, \(\norm{\pi(a)\xi_k} \le q(a)\) for all \(a\in
  A\), that is, \(\xi_k\) is bounded.
\end{proof}

\begin{proposition}
  \label{pro:commutative_regular_versus_locally_bounded_necessary}
  If\/~\(\Repi(A)\) has a weak \Cstar\nb-hull, then \(\Repi(A) =
  \Rep_\mathrm{b}(A)\).
\end{proposition}

\begin{proof}
  Let~\(B\) with the universal representation~\((\Hilms[B],\mu)\) be a
  weak \Cstar\nb-hull for \(\Repi(A)\).  First we claim that~\(B\) is
  commutative.  Let \(\omega\colon B\injto\Bound(\Hilm[H])\) be a
  faithful representation.  This corresponds to an integrable
  representation~\(\pi\) of~\(A\).  Since the equivalence
  \(\Repi(A,\Hilm[H])\cong\Rep(B,\Hilm[H])\) is compatible with
  unitary \Star{}intertwiners, the commutant of~\(\omega(B)\) is the
  \Cstar\nb-algebra of \Star{}intertwiners of~\(\pi\) by
  Proposition~\ref{pro:adjointable_intertwiner}.  The commutant of
  this is a commutative von Neumann algebra by
  \cite{Schmudgen:Unbounded_book}*{Theorem 9.1.7}.  So the bicommutant
  of~\(\omega(B)\) is commutative.  This forces~\(B\) to be
  commutative.

  Any representation of~\(A\) on a commutative \Cstar\nb-algebra is
  locally bounded by Theorem~\ref{the:commutative_hull}.  If the
  universal representation for~\(\Repi(A)\) is locally bounded, then
  all representations in~\(\Repi(A)\) are locally bounded, so that
  \(\Repi(A) = \Rep_\mathrm{b}(A)\).  Thus \(\Repi(A)\) only has a
  weak \Cstar\nb-hull if \(\Repi(A) = \Rep_\mathrm{b}(A)\).
\end{proof}

\begin{example}
  \label{exa:Repi_infinite_generators}
  Let~\(A\) be the \Star{}algebra \(\C[(x_i)_{i\in\N}]\) of
  polynomials in countably many symmetric generators.  Then
  \(\hat{A} \cong \prod_\N \R\) with the product topology.  This is
  metrisable.  So~\(\tau_c\) is the usual product topology.  Since
  this is not locally compact, \(\Rep_\mathrm{b}(A)\) has no
  \Cstar\nb-hull, not even a weak one
  (Theorem~\ref{the:no_Cstar-hull_commutative_not_locally_compact}).
  Since~\(A\) is countably generated, \(\Repi(A) =
  \Rep_\mathrm{b}(A)\) by
  Theorem~\ref{the:locally_bounded_versus_regular}.  A commutative
  (weak) \Cstar\nb-hull for some class of representations of~\(A\)
  is equivalent to an injective, continuous map \(X\to\hat{A}\) for
  a locally compact space~\(X\) by
  Theorem~\ref{the:commutative_hull}.
\end{example}

Let~\(G\) be a topological group.  A \emph{host algebra} for a~\(G\)
is defined in~\cite{Grundling-Neeb:Infinite_tensor} as a
\Cstar\nb-algebra~\(B\) with a continuous representation~\(\lambda\)
of~\(G\) by unitary multipliers, such that for each Hilbert
space~\(\Hilm[H]\), the map that sends a representation
\(\varrho\colon B\to \Bound(\Hilm[H])\) to a unitary representation
\(\varrho\circ\lambda\) of~\(G\) is injective.  We claim that
commutative \Cstar\nb-hulls for the polynomial algebra
\(\C[(x_i)_{i\in\N}]\) are equivalent to host algebras of the
topological group \(\R^{(\N)} \defeq \bigoplus_\N \R\).

Let~\(\Cst(G_d)\) be the \Cstar\nb-algebra of~\(G\) viewed as a
discrete group.  Representations of~\(\Cst(G_d)\) are equivalent to
representations of the discrete group underlying~\(G\) by unitary
multipliers.  Since any representation of~\(\Cst(G_d)\) is bounded,
any weakly admissible class of representations of~\(\Cst(G_d)\) is
admissible by Corollary~\ref{cor:locally_bounded_admissible}.  Call
a representation of~\(\Cst(G_d)\) continuous if the corresponding
representation of~\(G\) is continuous.  This class is easily seen to
be weakly admissible, hence admissible.  The unital
\Star{}homomorphism \(\Cst(G_d)\to \Mult(B)\) associated to the
unitary representation~\(\lambda\) for a host algebra~\(B\) is
continuous by assumption.  Thus \(B\)\nb-integrable representations
of~\(\Cst(G_d)\) are continuous.  The injectivity requirement in the
definition of a host algebra is exactly the
condition~\ref{enum:Cstar-gen1} in
Proposition~\ref{pro:Cstar-generated_by_unbounded_multipliers}, and
this is equivalent to~\(B\) being a \Cstar\nb-hull.  Thus a host
algebra for~\(G\) is the same as a \Cstar\nb-hull or weak
\Cstar\nb-hull for a class of continuous representations
of~\(\Cst(G_d)\).

In applications, we would rather study
continuous representations of~\(G\) through the Lie algebra of~\(G\)
instead of through the inseparable \Cstar\nb-algebra~\(\Cst(G_d)\).
The Lie algebra of \(G=\R^{(\N)}\) is the Abelian Lie
algebra~\(\R^{(\N)}\), and its universal enveloping algebra is the
polynomial algebra \(A=\C[(x_i)_{i\in\N}]\).  Call a representation
of~\(A\) integrable if it belongs to \(\Repi(A)= \Rep_\mathrm{b}(A)\).

Let~\(\Hilm\) be a Hilbert module.  We claim that an integrable
representation of~\(A\) on~\(\Hilm\) is equivalent to a strictly
continuous, unitary representation of the group~\(\R^{(\N)}\)
on~\(\Hilm\).  Indeed, a unitary representation of~\(\R\) is
equivalent to a representation of \(\Cst(\R)\cong \Cont_0(\R)\), and
these are equivalent to integrable representations of~\(\C[x]\) as
in Theorem~\ref{the:regular_self-adjoint_Cstar-hull}.  In an
integrable representation of~\(\C[(x_i)_{i\in\N}]\), the operators
\(\cl{\pi(x_i)}\) for \(i\in\N\) strongly commute
by Theorem~\ref{the:commutative_regular_admissible}.  Hence the
resulting representations of \(\Cont_0(\R)\) commute.  Equivalently,
the resulting continuous representations of~\(\R\) commute, so that
we may combine them to a representation of the Abelian
group~\(\R^{(\N)}\).  Conversely, a continuous unitary
representation of~\(\R^{(\N)}\) provides nondegenerate
representations of~\(\Cont_0(\R^m)\) for all \(m\in\N\) by
restricting the representation to \(\R^m\subseteq \R^{(\N)}\).
These correspond to a compatible family of representations of the
polynomial algebras \(\C[x_1,\dotsc,x_m]\) for \(m\in\N\).  The
intersection of their domains is dense by
\cite{Schmudgen:Unbounded_book}*{Lemma 1.1.2}.  So these
representations combine to a representation
of \(A=\C[(x_i)_{i\in\N}]\).  Hence an integrable representation
of~\(A\) on a Hilbert module as in
Theorem~\ref{the:commutative_regular_admissible} is equivalent to a
continuous representation of~\(\R^{(\N)}\).

\section{From graded \Star{}algebras to Fell bundles}
\label{sec:graded_to_Fell}

Let~\(G\) be a discrete group with unit element~\(e\).

\begin{definition}
  \label{def:grading}
  A \emph{\(G\)\nb-graded \Star{}algebra} is a unital algebra~\(A\)
  with a linear direct sum decomposition \(A=\bigoplus_{g\in G} A_g\)
  with \(A_g\cdot A_h\subseteq A_{gh}\), \(A_g^*=A_{g^{-1}}\), and
  \(1\in A_e\) for all \(g,h\in G\).  Thus \(A_e\subseteq A\) is a
  unital \Star{}subalgebra.
\end{definition}

The articles \cites{Savchuk-Schmudgen:Unbounded_induced,
  Dowerk-Savchuk:Induced} study many examples of \(G\)\nb-graded
\Star{}algebras.

We fix some notation used throughout this section.  Let~\(\Hilm\) be a
Hilbert module over a \Cstar\nb-algebra~\(D\).  Let \((\Hilms,\pi)\)
be a representation of~\(A\) on~\(\Hilm\).  Let \(\pi_g\colon
A_g\to\Endo_D(\Hilms)\) for \(g\in G\) be the restrictions of~\(\pi\),
so \(\pi=\bigoplus_{g\in G} \pi_g\).  Since~\(\pi\) is a
\Star{}homomorphism,
\[
\pi_g(a_g)\pi_h(a_h)=\pi_{gh}(a_g\cdot a_h),
\qquad
\pi_{g^{-1}}(a_g^*) \subseteq \pi_g(a_g)^*
\]
for all \(a_g\in A_g\), \(a_h\in A_h\).  The last condition means that
\(\braket{\xi}{\pi_g(a_g)\eta} = \braket{\pi_{g^{-1}}(a_g^*)\xi}
{\eta}\) for all \(\xi,\eta\in\Hilms\).  In particular, \(\pi_e\colon
A_e\to\Endo(\Hilms)\) is a representation of~\(A_e\).

\begin{lemma}[compare \cite{Savchuk-Schmudgen:Unbounded_induced}*{Lemma 12}]
  \label{lem:restriction_closed}
  The families of norms \(\norm{\xi}_a\defeq \norm{\pi(a)\xi}\) for
  \(a\in A\) and for \(a\in A_e\) generate equivalent topologies
  on~\(\Hilms\).  Hence the representation \(\pi_e \colon
  A_e\to\Endo_D(\Hilms)\) is closed if and only if~\(\pi\) is
  closed.
\end{lemma}

\begin{proof}
  Any element of~\(A\) is a sum \(a=\sum_{g\in G} a_g\) with
  \(a_g\in A_g\) and only finitely many non-zero terms.  We estimate
  \(\norm{\xi}_a \le \sum_{g\in G} \norm{\xi}_{a_g}\), and
  \(\norm{\xi}_{a_g} \le \frac{5}{4}\norm{\xi}_{a_g^* a_g}\) by the
  proof of Lemma~\ref{lem:graph_norms_directed}.  Since \(a_g^*
  a_g\in A_e\), the graph topologies for \(\pi_e\) and~\(\pi\) are
  equivalent.
\end{proof}

\subsection{Integrability by restriction}
\label{sec:integrable_by_restrict}

\begin{definition}
  \label{def:induced_integrable}
  Let a weakly admissible class of integrable representations
  of~\(A_e\) on Hilbert modules be given.  We call a representation
  of~\(A\) on a Hilbert module \emph{integrable} if its restriction
  to~\(A_e\) is integrable.
\end{definition}

Here ``restriction of~\(\pi\)'' means the representation~\(\pi_e\)
with the same domain~\(\Hilms\) as~\(\pi\).  This is closed by
Lemma~\ref{lem:restriction_closed}.

\begin{proposition}
  \label{pro:integrable_induced_nice}
  If integrability for representations of~\(A_e\) is defined by
  submodule conditions, then the same holds for~\(A\).  If the
  Local--Global Principle holds for the integrable representations
  of~\(A_e\), it also holds for the integrable representations
  of~\(A\).  If the class of integrable representations of~\(A_e\)
  is admissible or weakly admissible, the same holds for~\(A\).
\end{proposition}

\begin{proof}
  The first two statements and the claim about weak admissibility
  are trivial because integrability for a representation of~\(A\)
  only involves its restriction to~\(A_e\).
  Lemma~\ref{lem:restriction_closed} shows that restriction
  from~\(A\) to~\(A_e\) does not change the domain.
  Hence~\ref{enum:Cstar-gen2} in
  Proposition~\ref{pro:Cstar-generated_by_unbounded_multipliers} is
  inherited by~\(A\) if it holds for~\(A_e\).  That is,
  admissibility of the integrable representations passes
  from~\(A_e\) to~\(A\).
\end{proof}

It is unclear whether~\(A\) also inherits the \emph{Strong}
Local--Global Principle from~\(A_e\).  This may often be bypassed
using Theorem~\ref{the:operator_conditions}.

\subsection{Inducible representations and induction}
\label{sec:induction}

Let~\(\Hilm[F]\) be
a Hilbert \(D\)\nb-module and let
\(\Hilms[F]\subseteq \Hilm[F]\) and \(\varphi_e\colon A_e\to
\Endo_D(\Hilms[F])\) be
a representation of~\(A_e\) on~\(\Hilm[F]\).  We try to
induce~\(\varphi_e\) to a representation of~\(A\) as
in~\cite{Savchuk-Schmudgen:Unbounded_induced}.  Thus we consider the
algebraic tensor product \(A\odot \Hilms[F]\) and equip it with the
obvious right \(D\)\nb-module structure and the unique sesquilinear
map that satisfies
\[
\braket{a_1\otimes \xi_1}{a_2\otimes \xi_2}
= \delta_{g,h}\braket{\xi_1}{\varphi_e(a_1^*a_2) \xi_2}
\]
for all \(g,h\in G\), \(a_1\in A_g\), \(a_2\in A_h\),
\(\xi_1,\xi_2\in\Hilms[F]\).  This map is sesquilinear and descends to
the quotient space \(A\odot_{A_e} \Hilms[F]\).  It is symmetric and
\(D\)\nb-linear in the sense that \(\braket{x}{y}=\braket{y}{x}^*\)
and \(\braket{x}{y d} = \braket{x}{y}d\).  Let~\(\pi\) be the action
of~\(A\) on \(A\odot_{A_e} \Hilms[F]\) by left multiplication.  This
representation is formally a \Star{}homomorphism in the sense that
\(\braket{x}{\pi(a)y} = \braket{\pi(a^*)x}{y}\) for all \(a\in A\),
\(x,y\in A\odot_{A_e} \Hilms[F]\).  The only thing that is missing to
get a representation of~\(A\) on a Hilbert \(D\)\nb-module is
positivity of the inner product.  This requires a subtle extra
condition.

\begin{proposition}
  \label{pro:inducible_criteria}
  The following are equivalent:
  \begin{enumerate}
  \item \label{en:inducible_1} the sesquilinear map on \(A\odot_{A_e}
    \Hilms[F]\) defined above is positive semidefinite;
  \item \label{en:inducible_2} for all \(g\in G\), \(n\in\N\) and all
    \(a_1,\dotsc,a_n\in
    A_g\), \(\xi_1,\dotsc,\xi_n\in \Hilms[F]\), the element
    \(\sum_{k,l=1}^n \braket{\xi_k}{\varphi_e(a_k^* a_l)\xi_l} \in D\)
    is positive;
  \item \label{en:inducible_3} for all \(g\in G\), \(n\in\N\) and all
    \(a_1,\dotsc,a_n\in
    A_g\), \(\xi_1,\dotsc,\xi_n\in \Hilms[F]\), the matrix
    \(\bigl(\braket{\xi_k}{\varphi_e(a_k^* a_l)\xi_l}\bigr)_{k,l} \in
    \Mat_n(D)\) is positive.
  \end{enumerate}
\end{proposition}

\begin{proof}
  The condition~\ref{en:inducible_2} for fixed \(g\in G\) says that the
  sesquilinear map on \(A_g\odot_{A_e} \Hilms[F]\) is positive
  semidefinite.  Since the subspaces \(A_g\odot_{A_e} \Hilms[F]\) for
  different~\(g\) are orthogonal, this is equivalent to positive
  semidefiniteness on \(A\odot_{A_e} \Hilms[F]\).  Thus
  \ref{en:inducible_1}\(\iff\)\ref{en:inducible_2}.

  We prove \ref{en:inducible_2}\(\iff\)\ref{en:inducible_3}.  Fix
  \(g\in G\), \(n\in\N\), \(a_1,\dotsc,a_n\in A_g\) and
  \(\xi_1,\dotsc,\xi_n\in \Hilms[F]\).  Let \(y = (y_{kl}) \in
  \Mat_n(D)\) be the matrix in~\ref{en:inducible_3}.  By
  \cite{Lance:Hilbert_modules}*{Lemma 4.1}, \(y\ge0\) in \(\Mat_n(D)
  \subseteq \Bound(D^n)\) if and only if \(\braket{d}{y d}\ge0 \) for
  all \(d=(d_1,\dotsc,d_n)\in D^n\).  That is, \(\sum_{k,l=1}^n d_k^*
  y_{kl} d_l \ge0\) for all \(d_1,\dotsc,d_n\in D\).
  Since~\(\Hilms[F]\) is a right \(D\)\nb-module, this condition for
  all \(\xi_i\in\Hilms[F]\), \(d_i\in D\) is equivalent
  to~\ref{en:inducible_2}.
\end{proof}

\begin{definition}
  \label{def:inducible_representation}
  A representation~\(\varphi_e\) of~\(A_e\) is \emph{inducible}
  (to~\(A\)) if it satisfies the equivalent conditions in
  Proposition~\ref{pro:inducible_criteria}.
\end{definition}

If~\(A_e\) were a \Cstar\nb-algebra, it would be enough to assume
\(\braket{\xi}{\varphi_e(a^*a) \xi} \ge0\) for all \(g\in G\), \(a\in
A_g\), \(\xi\in \Hilms[F]\), which amounts to the condition \(a^* a\ge
0\) in~\(A_e\) for all \(g\in G\), \(a\in A_g\).  This is part of the
definition of a Fell bundle over a group.  For more general
\Star{}algebras, the positivity conditions for different \(n\in\N\) in
Proposition~\ref{pro:inducible_criteria} may differ,
compare~\cite{Friedrich-Schmuedgen:n-positivity}.

Let \(A\otimes_{A_e} \Hilm[F]\) be the Hilbert module completion of
\(A\odot_{A_e} \Hilms[F]\) for the inner product above.  The
\Star{}algebra~\(A\) acts on \(A\odot_{A_e} \Hilms[F]\) by left
multiplication, \(a_1\cdot (a_2\otimes \xi) \defeq (a_1 a_2)\otimes
\xi\) for \(a_1,a_2\in A\), \(\xi\in\Hilms[F]\).  As in the proof of
Lemma~\ref{lem:tensor_rep_with_corr}, this module structure descends
to the image of \(A\odot_{A_e} \Hilms[F]\) in \(A\otimes_{A_e}
\Hilm[F]\) and gives a well defined representation~\(\pi\) of~\(A\)
on~\(A\otimes_{A_e} \Hilm[F]\).  Its closure is called the
\emph{induced representation} from~\(\varphi_e\), and its domain is
denoted by \(A\otimes_{A_e} \Hilms[F]\).

The decomposition \(A\odot_{A_e}\Hilms[F] = \bigoplus_{g\in G}
A_g\odot_{A_e} \Hilms[F]\) is \(A_e\)\nb-invariant and orthogonal
for the above inner product.  Hence
\begin{equation}
  \label{eq:decompose_induced}
  A\otimes_{A_e} \Hilm[F] \cong
  \bigoplus_{g\in G} A_g\otimes_{A_e} \Hilm[F],
\end{equation}
where \(A_g\otimes_{A_e} \Hilm[F]\) is the closure of the image of
\(A_g\odot_{A_e} \Hilms[F]\) or, equivalently, the Hilbert
\(D\)\nb-module completion of \(A_g\odot_{A_e} \Hilms[F]\) with
respect to the restriction of the inner product.  Each summand
\(A_g\otimes_{A_e} \Hilm[F]\) carries a closed representation
of~\(A_e\) with domain \(A_g\otimes_{A_e} \Hilms[F]\),
and~\(\pi|_{A_e}\) is the direct sum of these representations.

\begin{lemma}
  \label{lem:induced_repr_inducible}
  Let~\(\pi\) be any representation of~\(A\).  Then~\(\pi|_{A_e}\) is
  inducible.
\end{lemma}

\begin{proof}
  For \(g\in G\), \(a_1,\dotsc,a_n\in A_g\), \(\xi_1,\dotsc,\xi_n\in
  \Hilms\), let \(y\defeq \sum_{k=1}^n \pi(a_k)\xi_k\).  Then
  \[
  \sum_{k,l=1}^n \braket{\xi_k}{\pi|_{A_e}(a_k^* a_l)\xi_l}
  = \sum_{k,l=1}^n \braket{\pi(a_k)\xi_k}{\pi(a_l)\xi_l}
  = \braket{y}{y} \ge 0.\qedhere
  \]
\end{proof}

Lemma~\ref{lem:tensor_associative} about the associativity
of~\(\otimes\) has a variant for induction:

\begin{lemma}
  \label{lem:induction_associative}
  Let \(D_1,D_2\) be \(\Cst\)\nb-algebras, let \(\Hilm\) be a
  Hilbert \(D_1\)\nb-module and let \(\Hilm[F]\) be a
  \(\Cst\)\nb-correspondence between \(D_1,D_2\).  Let
  \((\varphi_e,\Hilms)\) be an inducible representation of~\(A\)
  on~\(\Hilm\).  Then the representation \(\varphi_e \otimes_{D_1}
  \Hilm[F]\) on \(\Hilm \otimes_{D_1} \Hilm[F]\) is inducible and
  there is a canonical unitary \Star{}intertwiner of representations of~\(A\),
  \[
  (A \otimes_{A_e} \Hilm) \otimes_{D_1} \Hilm[F]
  \cong A \otimes_{A_e} (\Hilm \otimes_{D_1} \Hilm[F]).
  \]
\end{lemma}

\begin{proof}
  Let \(\Hilms\otimes_{D_1} \Hilm[F] \subseteq \Hilm\otimes_{D_1}
  \Hilm[F]\) be the domain of \(\varphi_e \otimes_{D_1} \Hilm[F]\).
  Let \(g_1,\dotsc,g_n\in G\), \(a_1,\dotsc,a_n\in A_{g_i}\), and
  \(\omega_1,\dotsc,\omega_n \in \Hilms\otimes_{D_1} \Hilm[F]\).
  Let \(\zeta \defeq \sum_{k=1}^n a_k \otimes \omega_k \in A\odot
  (\Hilms\otimes_{D_1} \Hilm[F])\).  To show that \(\varphi_e
  \otimes_{D_1} \Hilm[F]\) is inducible, we must prove that
  \(\braket{\zeta}{\zeta}\in D_2\) is positive.  Vectors in
  \(\Hilms\odot \Hilm[F]\) form a core for \(\varphi_e \otimes_{D_1}
  \Hilm[F]\) by construction.  Hence there is a sequence of vectors
  of the form
  \[
  \omega_{j,\tau} \defeq \sum_{i=1}^{\ell_j}
  \xi_{\tau,j,i}\otimes\eta_{\tau,j,i},\qquad
  \xi_{\tau,j,i}\in \Hilms, \eta_{\tau,j,i}\in\Hilm[F],
  \]
  which, for \(\tau\to\infty\), converges to~\(\omega_j\) in the
  graph norms of the elements \(\delta_{g_m,g_k} a_m^* a_k \in A_e\)
  for all \(m,k=1,\dotsc,n\).  Let \(\zeta_\tau \defeq \sum_{j=1}^n
  a_j \otimes \omega_{j,\tau}\).  Then
  \[
  \lim_{\tau\to\infty} \braket{\zeta_\tau}{\zeta_\tau} =
  \lim_{\tau\to\infty} \braket{\zeta_\tau}{\zeta} =
  \lim_{\tau\to\infty} \braket{\zeta}{\zeta_\tau} =
  \braket{\zeta}{\zeta}
  \]
  in norm and
  \begin{multline*}
    \braket{\zeta_\tau}{\zeta_\tau} =
    \left< \sum_{i,j} a_j \otimes \xi_{\tau,j,i} \otimes \eta_{\tau,j,i},
      \sum_{m,k} a_k \otimes \xi_{\tau,k,m} \otimes \eta_{\tau,k,m} \right>
    \\= \sum_{i,j,k,m} \delta_{g_j,g_k} \braket{\eta_{\tau,j,i}}
    {\braket{\xi_{\tau,j,i}} {\varphi_e(a_j^* a_k)
        \xi_{\tau,k,m}}_{D_1} \cdot \eta_{\tau,k,m}}_{D_2}.
  \end{multline*}
  This is also the inner product of \(\zeta_\tau\) with itself in
  the tensor product \((A\otimes_{A_e} \Hilm) \otimes \Hilm[F]\).
  This is positive because~\(\varphi_e\) is inducible and the usual
  tensor product of the Hilbert \(D_1\)\nb-module \(A\otimes_{A_e}
  \Hilm\) with the \(D_1,D_2\)\nb-correspondence~\(\Hilm[F]\) is a
  Hilbert \(D_2\)\nb-module.  Hence
  \(\braket{\zeta_\tau}{\zeta_\tau} \ge0\) for all~\(\tau\).  Since
  the positive elements in~\(D_2\) form a closed subset, this
  implies \(\braket{\zeta}{\zeta}\ge0\).  Thus \(\varphi_e
  \otimes_{D_1} \Hilm[F]\) is inducible.  The argument above also
  shows that the linear span of vectors of the form \(a\otimes
  \xi\otimes \eta\) with \(a\in A\), \(\xi\in\Hilms\),
  \(\eta\in\Hilm[F]\) is a core for the representation of~\(A_e\) on
  \(A \otimes_{A_e} (\Hilm \otimes_{D_1} \Hilm[F])\).   Such vectors
  also form a core for the
  representation of~\(A\) on \((A \otimes_{A_e} \Hilm) \otimes_{D_1}
  \Hilm[F]\).  The left actions of~\(A\) and the \(D_2\)\nb-valued
  inner products coincide on such vectors.  Hence there is a unique
  unitary \Star{}intertwiner that maps
  the image of \(a\otimes \xi\otimes \eta\) in \((A \otimes_{A_e}
  \Hilm) \otimes_{D_1} \Hilm[F]\) to its image in \(A \otimes_{A_e}
  (\Hilm \otimes_{D_1} \Hilm[F])\).
\end{proof}

\subsection{\Cstar-Hulls for the unit fibre}
\label{sec:hull_Ae}

We assume that the chosen class of integrable representations
of~\(A_e\) has a (weak) \Cstar\nb-hull~\(B_e\).  We want to
construct a Fell
bundle whose section \Cstar\nb-algebra is a (weak) \Cstar\nb-hull
for the integrable representations of~\(A\).  At some point, we
need~\(B_e\) to be a full \Cstar\nb-hull (compatible with
isometric intertwiners) and one more extra condition.  But we may
begin the construction without these assumptions.  First we build
the unit fibre~\(B_e^+\) of the Fell bundle.  It is a (weak)
\Cstar\nb-hull for the inducible, integrable representations
of~\(A_e\).

Let~\((\Hilms[B]_e,\mu_e)\) be the universal integrable
representation of~\(A_e\) on~\(B_e\).  Let~\(x_-\) for a
self-adjoint element~\(x\) in a \Cstar\nb-algebra denote its
negative part.

\begin{definition}
  \label{def:restrict_to_inducible}
  Let~\(B_e^+\) be the quotient \Cstar\nb-algebra of~\(B_e\) by the
  closed two-sided ideal generated by elements of the form
  \begin{equation}
    \label{eq:inducible_kill}
    \biggl( \sum_{k,l=1}^n b_k^*\cdot
    \mu_e(a_k^* a_l)\cdot b_l\biggr)_-\qquad
    \text{for }g\in G,\ a_1,\dotsc,a_n\in A_g,\
    b_1,\dotsc,b_n\in\Hilms[B]_e.
  \end{equation}
  Let~\(\Hilms[B]_e^+\)
  be the image of~\(\Hilms[B]_e\)
  in~\(B_e^+\)
  and let \(\mu_e^+\colon A_e\to\Endo_{B_e^+}(\Hilms[B]_e^+)\)
  be the induced representation of~\(A_e\) on this quotient.
\end{definition}

The following proposition shows that the representation
\((\Hilms[B]_e^+,\mu_e^+)\)
of~\(A_e\)
on~\(B_e^+\)
is the universal inducible, integrable representation of~\(A_e\).

\begin{proposition}
  \label{pro:inducible_A_B}
  Let \((\Hilms[F],\varphi_e)\) be an integrable representation
  of~\(A_e\) on a Hilbert module~\(\Hilm[F]\).  Let
  \(\bar\varphi_e\colon B_e\to \Bound(\Hilm[F])\) be the
  corresponding representation of~\(B_e\).  Then~\(\varphi_e\) is
  inducible if and only if~\(\bar\varphi_e\) factors through the
  quotient map \(B_e\prto B_e^+\).  Thus~\(B_e^+\) is a
  \Cstar\nb-hull for the inducible, integrable representations
  of~\(A_e\).
\end{proposition}

\begin{proof}
Assume first that~\(\varphi_e\)
  is inducible.  Let \(\xi\in\Hilm[F]\)
  and let \(g\in G\),
  \(a_1,\dotsc,a_n\in A_g\)
  and \(b_1,\dotsc,b_n\in\Hilms[B]_e\)
  be as in~\eqref{eq:inducible_kill}.  Let
  \(\xi_k \defeq \varphi_e(b_k)\xi\).
  Since~\(\varphi_e\)
  is inducible, Proposition~\ref{pro:inducible_criteria} implies
  \begin{multline*}
    0 \le \sum_{k,l=1}^n \braket{\xi_k}{\varphi_e(a_k^* a_l)\xi_l}
    = \sum_{k,l=1}^n \braket{\xi}
    {\bar\varphi_e(b_k)^*\varphi_e(a_k^* a_l)  \bar\varphi_e(b_l)\xi}
    \\= \left< \xi, 
      \bar\varphi_e\left(\sum_{k,l=1}^n b_k^*\mu_e(a_k^* a_l) b_l\right)
      \xi \right>.
  \end{multline*}
  Since \(\xi\in\Hilm[F]\)
  is arbitrary, this means that
  \(\bar\varphi_e\left(\sum_{k,l=1}^n b_k\mu_e(a_k^* a_l) b_l\right)
  \ge0\)
  in \(\Bound(\Hilm[F])\).
  Equivalently, \(\bar\varphi_e\)
  annihilates the negative part of
  \(\sum_{k,l=1}^n b_k\mu_e(a_k^* a_l) b_l\).
  So~\(\bar\varphi_e\)
  descends to a homomorphism on the quotient~\(B_e^+\).
  Conversely, the representation \((\Hilms[B]_e^+,\mu_e^+)\)
  is inducible by Proposition~\ref{pro:inducible_criteria}.  If
  \(\bar\varphi_e^+\colon B_e^+\to\Bound(\Hilm[F])\)
  is a representation, then the representation
  \(\mu_e^+\otimes_{\bar\varphi_e^+} 1_{\Hilm[F]} \cong \varphi_e\)
  on \(B_e^+ \otimes_{B_e^+} \Hilm[F] \cong \Hilm[F]\)
  is inducible by Lemma~\ref{lem:induction_associative}.  That is,
  \(\varphi_e\)
  is inducible if~\(\bar\varphi_e\)
  factors through the quotient map \(B_e\prto B_e^+\).

  Summing up, the representation~\(\bar\varphi_e\) associated to an
  integrable representation~\(\varphi_e\) of~\(A_e\) descends
  to~\(B_e^+\) if and only if~\(\varphi_e\) is inducible.  The
  quotient map induces a fully faithful embedding \(\Rep(B_e^+,D)
  \injto \Rep(B_e,D)\).  The argument above shows that its image
  consists of those representations of~\(B_e\) that correspond to
  inducible, integrable representations of~\(A_e\) under the
  correspondence \(\Rep(B_e,D) \cong \Repi(A_e,D)\).
  Hence~\(B_e^+\) is a (weak) \Cstar\nb-hull for the class of
  inducible, integrable representations of~\(A_e\).
\end{proof}

\begin{definition}
  \label{def:fibre_g}
  Let \(B_g^+ \defeq A_g\otimes_{A_e} B_e^+\).  This is a well
  defined Hilbert \(B_e^+\)\nb-module because the representation
  \((\Hilms[B]_e^+,\mu_e^+)\) of~\(A_e\) on~\(B_e^+\) is inducible.
  Let \((\Hilms[B]_g^+,\mu_{e,g}^+)\) be the induced representation
  of~\(A_e\) on~\(B_g^+\).  It has the image of \(A_g\odot_{A_e}
  \Hilms[B]_e^+\) as a core, with the representation
  \(\mu_{e,g}^+(a_e)(a_g\otimes b) \defeq (a_ea_g)\otimes b\) for
  \(a_e\in A_e\), \(a_g\in A_g\), \(b\in \Hilms[B]_e^+\).
\end{definition}

By definition, the right \(B_e^+\)\nb-module structure and the inner
product on~\(B_g^+\) are the unique extensions of the following
pre-Hilbert module structure on \(A_g\odot_{A_e} \Hilms[B]_e^+\):
\((a_g\otimes b_1)\cdot b_2 \defeq a_g \otimes (b_1\cdot b_2)\) for
all \(a_g\in A_g\), \(b_1\in\Hilms[B]_e^+\), \(b_2\in B_e^+\), and
\begin{equation}
  \label{eq:innprod_Bgplus}
  \braket{a_1 \otimes b_1}{a_2\otimes b_2}
  \defeq b_1^* \mu_e^+(a_1^* a_2) b_2
\end{equation}
for \(a_1,a_2\in A_g\), \(b_1,b_2\in\Hilms[B]_e^+\).  This is positive
definite by Proposition~\ref{pro:inducible_A_B}.  By definition,
\(B_g^+\) is the norm completion of this pre-Hilbert
\(B_e^+\)\nb-module, and~\(\Hilms[B]_g^+\) is the completion of
\(A_g\odot_{A_e} \Hilms[B]_e^+\) in the graph topology for the
representation~\(\mu_{e,g}^+\) of~\(A_e\).

The Hilbert \(B_e^+\)\nb-modules~\(B_g^+\) are the fibres of our Fell
bundle.

The Fell bundle structure on~\((B^+_g)_{g\in G}\) only exists under
extra assumptions.  Before we turn to these, we construct
representations of the Hilbert \(B^+_e\)\nb-modules~\(B^+_g\) from
an integrable representation~\(\pi\) of~\(A\) on~\(\Hilm\).  Let
\(\pi_g\defeq \pi|_{A_g}\) and let \(\bar\pi_e\colon
B_e\to\Bound(\Hilm)\) be the representation of the \(\Cst\)\nb-hull
corresponding to~\(\pi_e\).  Since~\(\pi_e\) is inducible by
Lemma~\ref{lem:induced_repr_inducible}, \(\bar\pi_e\) descends to a
representation \(\bar\pi_e^+\colon B_e^+\to\Bound(\Hilm)\) by
Proposition~\ref{pro:inducible_A_B}.

Let \(a\in A_g\)
and \(b\in\Hilms[B]^+_e\).
The operator~\(\pi_g(a)\bar\pi^+_e(b)\)
is defined on all of~\(\Hilm\)
because~\(\bar\pi^+_e(b)\)
maps~\(\Hilm\)
into the domain~\(\Hilms\)
of~\(\pi_e\),
which is also the domain of~\(\pi_g(a)\)
by Lemma~\ref{lem:restriction_closed}.  Its adjoint contains the
densely defined operator~\(\bar\pi^+_e(b^*)\pi_{g^{-1}}(a^*)\),
and the operator
\[
\bar\pi^+_e(b^*)\pi_{g^{-1}}(a^*) \pi_g(a)\bar\pi^+_e(b)
= \bar\pi^+_e(b^*) \pi_e(a^* a) \bar\pi^+_e(b)
= \bar\pi^+_e(b^*\cdot \mu_e^+(a^* a) \cdot b)
\]
is bounded.  Hence~\(\bar\pi^+_e(b^*)\pi_{g^{-1}}(a^*)\) extends to a
bounded operator on~\(\Hilm\), which is adjoint to
\(\pi_g(a)\bar\pi^+_e(b)\).  Thus
\(\pi_g(a)\bar\pi^+_e(b)\in\Bound(\Hilm)\).  Define
\[
\bar\pi^+_g\colon A_g\odot \Hilms[B]^+_e\to\Bound(\Hilm),
\qquad
a\otimes b\mapsto \pi_g(a)\bar\pi^+_e(b).
\]
As above, we check that
\begin{equation}
  \label{eq:barpi_g_representation}
  \bar\pi^+_g(x_1)^*\bar\pi^+_g(x_2) = \bar\pi^+_e(\braket{x_1}{x_2}),\qquad
  \bar\pi^+_g(x\cdot b) = \bar\pi^+_g(x)\bar\pi^+_e(b)
\end{equation}
for all \(x_1,x_2,x\in A_g\odot \Hilms[B]^+_e\), \(b\in B^+_e\), where the
inner product is the one that defines~\(B^+_g\).  Thus~\(\bar\pi^+_g\)
extends uniquely to a bounded linear map
\[
\bar\pi^+_g\colon B^+_g\to\Bound(\Hilm),
\]
which still satisfies~\eqref{eq:barpi_g_representation}.  That is,
it is a representation of the Hilbert module~\(B^+_g\) with respect
to~\(\bar\pi^+_e\).

\begin{lemma}
  \label{lem:barpi_isometric}
  If \(\bar\pi^+_e\colon B^+_e\injto\Bound(\Hilm)\) is faithful
  \textup{(}hence isometric\textup{)}, then so is~\(\bar\pi^+_g\).
\end{lemma}

\begin{proof}
  Let \(\xi\in B^+_g\).  Then
  \[
  \norm{\xi} = \norm{\braket{\xi}{\xi}_{B^+_e}}^{\nicefrac12} =
  \norm{\bar\pi^+_e(\braket{\xi}{\xi}_{B^+_e})}^{\nicefrac12} =
  \norm{\bar\pi^+_g(\xi)^*\bar\pi^+_g(\xi)}^{\nicefrac12} =
  \norm{\bar\pi^+_g(\xi)}.\qedhere
  \]
\end{proof}

Next we want to prove that
\begin{equation}
  \label{eq:Fell_bundle_barpig}
  \bar\pi^+_g(B^+_g) \cdot \bar\pi^+_h(B^+_h)
  \subseteq \bar\pi^+_{g h}(B^+_{g h})
  \quad\text{and}\quad
  \bar\pi^+_g(B^+_g)^* = \bar\pi^+_{g^{-1}}(B^+_{g^{-1}})
\end{equation}
for all \(g,h\in G\) and for all integrable representations~\(\pi\)
of~\(A\).  This would give \((\bar\pi^+_g(B^+_g))_{g\in G}\) a Fell
bundle structure, which would lift to~\((B^+_g)_{g\in G}\) itself
if~\(\bar\pi^+_e\) is faithful.  Lemma~\ref{lem:Be_Bg_multiplicative}
below gives~\eqref{eq:Fell_bundle_barpig} provided the closed linear span
of \(\bar\pi^+_e(B^+_e)\cdot \bar\pi^+_g(B^+_g)\) is~\(\bar\pi^+_g(B^+_g)\) for
all \(g\in G\).  But this only holds if we impose two extra
assumptions.  First, compatibility of integrability and induction
gives~\(B^+_g\) a canonical left \(B^+_e\)\nb-module structure.
Secondly, compatibility of the weak \Cstar\nb-hull~\(B^+_e\) with
isometric intertwiners ensures that the representation~\(\bar\pi^+_g\)
is compatible with this left \(B^+_e\)\nb-module structure.

\subsection{Integrability and induction}
\label{sec:int_ind}

\begin{definition}
  \label{def:induced_actions_integrable}
  We say that \emph{integrability is compatible with induction} if
  induction of inducible representations preserves integrability; that
  is, if~\(\varphi_e\) is an inducible, integrable representation
  of~\(A_e\) on~\(\Hilm\) and~\(\pi\) is the representation of~\(A\)
  on~\(A\otimes_{A_e} \Hilm\) induced by~\(\varphi_e\), then the
  representation \(\pi_e \defeq \pi|_{A_e}\) of~\(A_e\) is again
  integrable.
\end{definition}

We shall use this assumption in Section~\ref{sec:Fell_structure} to
prove~\eqref{eq:Fell_bundle_barpig}.  But first, we study some
sufficient conditions for integrability to be compatible with
induction.


A direct sum of representations is integrable if and only if each
summand is integrable by Corollary~\ref{cor:integrable_sums}.  Hence
integrability is compatible with induction if and only if an
inducible, integrable representation~\(\varphi_e\) on~\(\Hilm[F]\)
induces integrable representations of~\(A_e\) on \(A_g\otimes_{A_e}
\Hilm[F]\) for all \(g\in G\).

\begin{proposition}
  \label{pro:inducible_universal}
  Integrability is compatible with induction if and only if the
  representations \((\Hilms[B]^+_g,\mu^+_{e,g})\) of~\(A_e\) on~\(B^+_g\)
  are integrable for all \(g\in G\).
\end{proposition}

\begin{proof}
  The representations \((\Hilms[B]^+_g,\mu^+_{e,g})\) of~\(A_e\)
  on~\(B^+_g\) are integrable for all \(g\in G\) if and only if
  their direct sum is integrable.  Denote this by \((A\otimes_{A_e}
  \Hilms[B]^+_e,\mu^+)\).  If integrability is compatible with
  induction, then \((A\otimes_{A_e} \Hilms[B]^+_e,\mu^+)\) must be
  integrable because it is the induced representation of the
  universal integrable (inducible)
  representation~\((\Hilms[B]^+_e,\mu^+_e)\) of~\(A_e\)
  on~\(B^+_e\).  Conversely, by
  Lemma~\ref{lem:induction_associative}, induction maps the
  representation \((\Hilms[B]^+_e,\mu^+_e)\otimes_\varrho \Hilm[F]\)
  of~\(A_e\) associated to a representation \(\varrho\colon
  B^+_e\to\Bound(\Hilm[F])\) to the representation \((A
  \otimes_{A_e} \Hilms[B]^+_e,\mu^+) \otimes_\varrho \Hilm[F]\),
  which is integrable if \((A\otimes_{A_e} \Hilms[B]^+_e,\mu^+)\)
  is, see Definition
  \ref{def:integrable_admissible}.\ref{enum:admissible1}.
\end{proof}

The (Strong) Local--Global Principle is useful to check that
integrability is compatible with induction:

\begin{proposition}
  \label{pro:induced_actions_integrable_Hils}
  Assume that the integrable representations of~\(A_e\) satisfy the
  Strong Local--Global Principle and that induction maps irreducible,
  inducible, integrable Hilbert space representations of~\(A_e\) to
  integrable Hilbert space representations of~\(A\).  Then
  integrability is compatible with induction.

  The same conclusion holds if the integrable representations
  of~\(A_e\) satisfy the Local--Global Principle and induction maps
  all inducible, integrable Hilbert space representations of~\(A_e\)
  to integrable Hilbert space representations of~\(A\).
\end{proposition}

\begin{proof}
  Let~\(B_e^+\) with the representation \((\Hilms[B]_e^+,\mu_e^+)\)
  of~\(A_e\) be the \Cstar\nb-hull for the inducible, integrable
  representations of~\(A_e\).  By
  Proposition~\ref{pro:inducible_universal}, it suffices to prove
  that the canonical representation of~\(A_e\) on \(A\otimes_{A_e}
  \Hilms[B]_e^+\) is integrable.

  By the Strong Local--Global Principle, this follows if the induced
  representation~\(\tilde\pi\) of~\(A_e\) on \((A\otimes_{A_e}
  \Hilms[B]_e^+)\otimes_\varrho \Hilm[H]\) is integrable for each
  irreducible representation~\(\varrho\) of~\(B_e^+\) on a Hilbert
  space~\(\Hilm[H]\).  The representation~\(\varrho\) is equivalent
  to an irreducible, inducible, integrable representation~\(\pi\)
  of~\(A_e\) on~\(\Hilm[H]\), and~\(\tilde\pi\) is the representation
  induced by~\(\pi\).  By assumption,
  \(\tilde\pi\) is integrable.  This finishes the proof in the case
  of the Strong Local--Global Principle.  The argument in the other
  case is the same without the word ``irreducible.''
\end{proof}

\begin{proposition}
  \label{pro:finit_dim_integrability_compatible_induction}
  Assume the following.  First, the integrable representations
  of~\(A_e\) satisfy the Strong Local--Global Principle.  Secondly,
  all \emph{irreducible}, integrable Hilbert space representations
  of~\(A_e\) are finite-dimensional.  Third, all finite-dimensional
  inducible representations of~\(A_e\) are integrable.  And fourth,
  each~\(A_g\) is finitely generated as a right \(A_e\)\nb-module.
  Then integrability is compatible with induction.
\end{proposition}

\begin{proof}
  First, since~\(B_e^+\) is a quotient of~\(B_e\), its irreducible
  representations form a subset of the irreducible representations
  of~\(B_e\).  Thus the irreducible, inducible, integrable Hilbert
  space representations of~\(A_e\) are finite-dimensional as well.
  By Proposition~\ref{pro:induced_actions_integrable_Hils}, it
  suffices to check that the induced representation of~\(A_e\) on
  \(A_g\otimes_{A_e} \Hilm[H]\) is integrable when~\(\Hilm[H]\) is a
  Hilbert space with an irreducible, inducible, integrable
  representation.  By our assumptions, \(\Hilm[H]\) is
  finite-dimensional and~\(A_g\) is finitely generated as an
  \(A_e\)\nb-module.  Hence \(A_g\otimes_{A_e} \Hilm[H]\) is
  finite-dimensional.  This representation is a direct summand in a
  representation of~\(A\) on \(A\otimes_{A_e} \Hilm[H]\) and hence
  inducible by Lemma~\ref{lem:induced_repr_inducible}.  By
  assumption, the induced representation of~\(A_e\) on
  \(A_g\otimes_{A_e} \Hilm[H]\) is integrable.
\end{proof}

\subsection{The Fell bundle structure}
\label{sec:Fell_structure}

If integrability is compatible with induction, the
representation~\(\mu^+_{e,g}\) of~\(A_e\) on~\(B^+_g\) is integrable.
It is inducible as well by Lemma~\ref{lem:induced_repr_inducible}
because it is a direct summand in a representation of~\(A\).  Hence
there is a unique (nondegenerate) representation~\(\bar\mu^+_{e,g}\)
of~\(B^+_e\) on~\(B^+_g\) such that \(\bar\mu^+_{e,g}(\Hilms[B]^+_e) B^+_g\)
is a core for~\(\mu^+_{e,g}\), and \(\mu^+_{e,g}(a_e)
(\bar\mu^+_{e,g}(b)x) = \bar\mu^+_{e,g}(\mu^+_e(a_e)b)x\) for all \(a\in
A_e\), \(b\in\Hilms[B]^+_e\), \(x\in B^+_g\).  Our next goal is to show
that the representations \(\bar\pi^+_e\colon B^+_e \to \Bound(\Hilm)\)
and \(\bar\pi^+_g\colon B^+_g \to \Bound(\Hilm)\) constructed
using~\eqref{eq:barpi_g_representation} are compatible in the sense
that
\begin{equation}
  \label{eq:pie_pig_compatible}
  \bar\pi^+_e(b_e)\cdot \bar\pi^+_g(b_g)
  = \bar\pi^+_g(\bar\mu^+_{e,g}(b_e)b_g)
  \qquad\text{for all }b_e\in B^+_e,\ b_g\in B^+_g.
\end{equation}
This is not automatic.  The following lemma is the most
subtle point in the proof of the Induction Theorem.

\begin{lemma}
  \label{lem:Be_Bg_multiplicative}
  Equation~\eqref{eq:pie_pig_compatible} holds if~\(B^+_e\) is a
  \Cstar\nb-hull, not just a weak \Cstar\nb-hull.  Then also
  \(\bar\pi^+_e(B^+_e)\cdot\bar\pi^+_g(B^+_g) = \bar\pi^+_g(B^+_g)\) for all
  \(g\in G\).
\end{lemma}

\begin{proof}
  Let \(\Hilm[F]\defeq B^+_g\otimes_{B^+_e} \Hilm\).  The linear map
  \(B^+_g\odot \Hilm\to\Hilm\), \(b\otimes\xi\mapsto
  \bar\pi^+_g(b)\xi\), for \(b\in B^+_g\), \(\xi\in\Hilm\),
  preserves the inner products by~\eqref{eq:barpi_g_representation}.
  Hence it extends to a well defined isometry \(I\colon \Hilm[F]
  \injto \Hilm\).  The representation~\(\bar\mu^+_{e,g}\)
  of~\(B^+_e\) on~\(B^+_g\) induces a
  representation~\(\bar\mu^+_{e,g}\otimes 1_{\Hilm}\) of~\(B^+_e\)
  on~\(\Hilm[F]\).  The meaning of~\eqref{eq:pie_pig_compatible} is
  that~\(I\) intertwines the representations
  \(\bar\mu^+_{e,g}\otimes1\) and~\(\bar\pi_e^+\) of~\(B^+_e\) on
  \(\Hilm[F]\) and~\(\Hilm\).  These representations correspond to
  the integrable representations \(\mu^+_{e,g}\otimes 1\)
  and~\(\pi_e\) of~\(A_e\) on \(\Hilm[F]\) and~\(\Hilm\),
  respectively.  Since~\(B^+_e\) is a \Cstar\nb-hull, it suffices to
  prove that~\(I\) intertwines these representations of~\(A_e\).

  We identify \(\Hilm \cong B_e^+ \otimes_{\bar\pi_e^+} \Hilm\) and
  describe~\(\pi_e\) as \(\mu_e^+ \otimes_{\bar\pi_e^+} 1_{\Hilm}\)
  as in Proposition~\ref{pro:Phi_simplifies}.  Then
  Lemma~\ref{lem:induction_associative} gives a canonical unitary
  \Star{}intertwiner
  \[
  \Hilm[F]
  \defeq (A_g \otimes_{A_e} B_e^+)\otimes_{B^+_e} \Hilm
  \cong A_g \otimes_{A_e} (B_e^+\otimes_{B^+_e} \Hilm)
  \cong A_g \otimes_{A_e} \Hilm
  \]
  of representations of~\(A_e\).  An inspection of the proof shows
  that~\(I\) corresponds to the isometry \(I'\colon A_g\otimes_{A_e}
  \Hilm \injto \Hilm\) defined by \(I'(a\otimes\xi) \defeq
  \pi_g(a)\xi\) for all \(a\in A_g\), \(\xi\in\Hilms\).
  Since~\(I'\) is an \(A_e\)\nb-intertwiner, so is~\(I\).  This
  finishes the proof of~\eqref{eq:pie_pig_compatible}.  Then
  \(\bar\pi^+_e(B^+_e)\cdot\bar\pi^+_g(B^+_g) = \bar\pi^+_g(B^+_g)\)
  follows because~\(\bar\mu^+_{e,g}\) is nondegenerate.
\end{proof}

\begin{lemma}
  \label{lem:Fell_bundle_rep}
  Assume \(\bar\pi^+_e(B^+_e)\cdot\bar\pi^+_g(B^+_g) = \bar\pi^+_g(B^+_g)\) for
  all \(g\in G\).  Then~\eqref{eq:Fell_bundle_barpig} holds.
\end{lemma}

\begin{proof}
  We write~\(\doteq\) to denote that two sets of operators have the
  same closed linear span.  By definition, \(\bar\pi^+_g(B^+_g) \doteq
  \pi_g(A_g) \bar\pi^+_e(\Hilms[B]^+_e)\), and \(\bar\pi^+_e(\Hilms[B]^+_e)^*
  \doteq \bar\pi^+_e(B^+_e)\) because~\(\Hilms[B]^+_e\) is dense in~\(B^+_e\).
  Our assumption \(\bar\pi^+_g(B^+_g) \doteq \bar\pi^+_e(B^+_e)\cdot
  \bar\pi^+_g(B^+_g)\) implies \(\bar\pi^+_g(B^+_g) \doteq
  \bar\pi^+_e(\Hilms[B]^+_e)^* \pi_g(A_g) \bar\pi^+_e(\Hilms[B]^+_e)\).  We
  have seen above~\eqref{eq:barpi_g_representation} that
  \[
  \bar\pi^+_e(b^*)\pi_{g^{-1}}(a^*)
  = \bar\pi^+_e(b)^*\pi_{g^{-1}}(a^*)
  \]
  for \(b\in\Hilms[B]^+_e\), \(a\in A_g\) extends to a bounded
  operator on~\(\Hilm\) that is adjoint to the bounded
  operator~\(\pi_g(a)\bar\pi^+_e(b)\).  Therefore,
  \[
  \bigl(\bar\pi^+_e(b_1)^* \pi_g(a) \bar\pi^+_e(b_2)\bigr)^*
  = \bar\pi^+_e(b_2)^* \pi_{g^{-1}}(a^*) \bar\pi^+_e(b_1)
  \]
  for all \(b_1,b_2\in \Hilms[B]^+_e\), \(a\in A_g\); both sides are
  globally defined bounded operators because
  \(\bar\pi^+_e(\Hilms[B]^+_e)\) maps~\(\Hilm\) into~\(\Hilms\).  The
  closed linear spans on the two sides of this equality are
  \(\bar\pi^+_g(B^+_g)^*\) and~\(\bar\pi^+_{g^{-1}}(B^+_{g^{-1}})\),
  respectively.  Thus \(\bar\pi^+_g(B^+_g)^* =
  \bar\pi^+_{g^{-1}}(B^+_{g^{-1}})\).  As above, the operators
  \(\cl{\bar\pi^+_e(b) \pi_g(a)}\) for \(b\in (\Hilms[B]^+_e)^*\), \(g\in
  G\), \(a\in A_g\) are bounded and generate \((B^+_{g^{-1}})^* = B^+_g\).
  Hence
  \begin{multline*}
    \bar\pi^+_g(B^+_g)\cdot\bar\pi^+_h(B^+_h)
    \doteq \bar\pi^+_e((\Hilms[B]^+_e)^*) \pi_g(A_g) \cdot
    \pi_h(A_h) \bar\pi^+_e(\Hilms[B]^+_e)
    \\\subseteq \bar\pi^+_e(\Hilms[B]^+_e)^* \cdot \pi_{gh}(A_{gh})
    \bar\pi^+_e(\Hilms[B]^+_e)
    \doteq \bar\pi^+_{gh}(B^+_{gh}).
  \end{multline*}
  We used here that~\(\pi\) is a homomorphism
  on~\(A\) and that \(A_g\cdot A_h\subseteq A_{gh}\).
\end{proof}

\begin{lemma}
  \label{lem:Fell_bundle_structure}
  Assume that~\(B_e\) is a \Cstar\nb-hull and that integrability is
  compatible with induction.  There is a unique Fell bundle
  structure on~\((B^+_g)_{g\in G}\) such that the maps
  \(\bar\pi^+_g\colon B^+_g\to \Bound(\Hilm)\) form a Fell bundle
  representation for any integrable representation~\(\pi\) of~\(A\)
  on a Hilbert module~\(\Hilm\).
\end{lemma}

\begin{proof}
  Lemmas \ref{lem:Be_Bg_multiplicative}
  and~\ref{lem:Fell_bundle_rep} show that
  \eqref{eq:Fell_bundle_barpig} holds under our assumptions.  Hence
  the multiplication and involution in~\(\Bound(\Hilm)\) restrict to
  a Fell bundle structure on the subspaces \(\bar\pi^+_g(B^+_g)\subseteq
  \Bound(\Hilm)\) for \(g\in G\), such that the inclusions
  \(\bar\pi^+_g(B^+_g)\injto \Bound(\Hilm)\) give a Fell bundle
  representation.

  The induced representation~\(\lambda\) of~\(A\) on the Hilbert
  \(B^+_e\)-module \(A\otimes_{A_e} B^+_e\) gives a faithful
  representation of~\(B^+_e\) because \(A\otimes_{A_e} B^+_e \supseteq
  A_e\otimes_{A_e} B^+_e = B^+_e\) contains the identity representation.
  Hence the resulting representations~\(\bar\lambda_g\) of~\(B^+_g\)
  are also faithful, even isometric, by
  Lemma~\ref{lem:barpi_isometric}.  So the Fell bundle structure
  on~\(\bar\lambda_g(B^+_g)\) lifts to~\(B^+_g\), so that the maps
  \(\bar\lambda_g\colon B^+_g\to \Bound(\Hilm)\) form a Fell bundle
  representation.

  Let~\(\pi\) be any integrable representation of~\(A\).  The
  exterior direct sum \(\pi\oplus\lambda\) on the Hilbert \(D\oplus
  B^+_e\)-module \(\Hilm'\defeq \Hilm\oplus (A\otimes_{A_e} B^+_e)\) is
  still integrable.  The resulting maps from~\(B^+_g\)
  to~\(\Bound(\Hilm')\) simply give block matrices
  \(\bar\pi^+_g(b)\oplus \bar\lambda_g(b)\) for \(b\in B^+_g\).  The
  compressions to the direct summands \(\Hilm\) and \(A\otimes_{A_e}
  B^+_e\) therefore restrict to Fell bundle representations with
  respect to the Fell bundle structure on \((\bar\pi^+_g\oplus
  \bar\lambda_g)(B^+_g)\) defined above.  Since~\(\lambda\) is
  faithful, the projection \((\bar\pi^+_g\oplus \bar\lambda_g)(B^+_g)
  \to \bar\lambda_g(B^+_g) \cong B^+_g\) is a Fell bundle isomorphism.
  Hence the map \(B^+_g \congto (\bar\pi^+_g\oplus \bar\lambda_g)(B^+_g)
  \to \bar\pi^+_g(B^+_g)\) is a Fell bundle representation.
\end{proof}

Let~\((\beta_g)_{g\in G}\) be a Fell bundle over a discrete
group~\(G\) (see \cite{Exel:Partial_dynamical}).  Then \(\beta\defeq
\bigoplus_{g\in G} \beta_g\) is a
\(G\)\nb-graded \Star{}algebra using the given multiplications and
involutions among the subspaces~\(\beta_g\).  The \emph{section
  \Cstar\nb-algebra} \(\Cst(\beta)\) of the Fell bundle is defined
as the completion of~\(\beta\) in the maximal \Cstar\nb-seminorm.
By construction, a representation of~\(\Cst(\beta)\) is equivalent
to a representation of the Fell bundle.  This holds also for
representations on Hilbert modules.

\begin{theorem}
  \label{the:Fell_bundle_sections_hull}
  Let~\(A\) be a graded \Star{}algebra for which~\(A_e\) has a
  \Cstar\nb-hull.  Assume that integrability is compatible with
  induction as in
  Definition~\textup{\ref{def:induced_actions_integrable}}.  The
  section \Cstar\nb-algebra~\(B\) of the Fell bundle~\((B^+_g)_{g\in
    G}\) constructed above is a \Cstar\nb-hull for the integrable
  representations of~\(A\).
\end{theorem}

\begin{proof}
  Representations of~\(B\) are in natural bijection with Fell bundle
  representations: restricting a representation of~\(B\) to the
  subspaces~\(B^+_g\) gives a Fell bundle representation, and
  conversely a Fell bundle representation gives a representation of
  the \Star{}algebra \(\bigoplus_{g\in G} B^+_g\), which extends
  uniquely to the \Cstar\nb-completion.
  Lemma~\ref{lem:Fell_bundle_structure} says that any integrable
  representation \(\pi = \bigoplus_{g\in G} \pi_g\) of~\(A\) induces
  a Fell bundle representation \((\bar\pi^+_g)_{g\in G}\)
  of~\((B^+_g)_{g\in G}\) and thus a representation of~\(B\).  By
  construction, this family of maps \(\Repi(A)\to\Rep(B)\) is
  compatible with interior tensor products and unitary
  \Star{}intertwiners.  We are going to show that this is a family
  of bijections.

  First we describe an integrable representation~\((\Hilms[B],\mu)\)
  of~\(A\) on~\(B\).  By construction, \(A\otimes_{A_e} \Hilms[B]^+_e
  = \bigoplus_{g\in G} \Hilms[B]^+_g\) is dense in~\(B\).  This
  subspace carries a representation of~\(A\) by left multiplication.
  We extend this to the right ideal in~\(B\) generated by
  \(A\otimes_{A_e} \Hilms[B]^+_e\) to get a representation of~\(A\)
  on~\(B\).  Let~\((\Hilms[B],\mu)\) be its closure.

  The representations~\(\bar\mu^+_{e,g}\) of~\(B^+_e\) on~\(B^+_g\) are
  defined so that \(\bar\mu^+_{e,g}(\Hilms[B]^+_e)B^+_g\) is another core
  for the representation~\((\Hilms[B]^+_g,\mu^+_{e,g})\) of~\(A_e\)
  on~\(B^+_g\).  Therefore, \(\Hilms[B]^+_e\cdot B\) is a core for the
  restriction of the representation~\((\Hilms[B],\mu)\) to~\(A_e\).
  This core shows that \((\Hilms[B],\mu|_{A_e}) =
  (\Hilms[B]^+_e,\mu^+_e) \otimes_{B_e} B\), where the interior tensor
  product is with respect to the canonical embedding \(B^+_e \injto
  B\).  Therefore, the restriction of~\((\Hilms[B],\mu)\) to~\(A_e\)
  is integrable and the corresponding representation~\(\bar\mu^+_e\)
  of~\(B^+_e\) is simply the inclusion map \(B^+_e\injto B\).  Thus the
  representation~\((\Hilms[B],\mu)\) of~\(A\) on~\(B\) is also
  integrable.

  The integrable representation~\((\Hilms[B],\mu)\) of~\(A\)
  on~\(B\) yields a representation~\(\bar\mu^+_g\) of the Fell bundle
  \((B^+_g)_{g\in G}\) in \(\Mult(B)=\Bound(B)\).  By construction,
  the image of \(a_g\otimes b\in A_g \odot_{A_e} \Hilms[B]^+_e\)
  in~\(B^+_g\) acts by \(\mu(a_g)\bar\mu^+_e(b) = \mu(a_g)\cdot b\).
  That is, \(B^+_g\) is represented by the canonical inclusion map
  \(B^+_g\injto B\).  The representation of~\(B\) associated to this
  Fell bundle representation is the identity map on~\(B\).

  Interior tensor product with~\((\Hilms[B],\mu)\) gives a family of
  maps \(\Rep(B)\to\Repi(A)\) that is compatible with unitary
  \Star{}intertwiners and interior tensor products.  Since the
  composite family of maps \(\Rep(B)\to\Repi(A)\to \Rep(B)\) is
  compatible with interior tensor products and maps the identity
  representation of~\(B\) to itself, the composite map
  on~\(\Rep(B)\) is the identity.

  Let~\((\Hilms,\pi)\) be an integrable representation of~\(A\)
  on a Hilbert \(D\)\nb-module~\(\Hilm\) for some
  \Cstar\nb-algebra~\(D\).  This yields a
  representation~\((\bar\pi^+_g)_{g\in G}\) of the Fell
  bundle~\((B^+_g)_{g\in G}\) and an associated
  representation~\(\bar\pi\) of~\(B\).  We claim that the integrable
  representation \((\Hilms',\pi') \defeq
  (\Hilms[B],\mu)\otimes_{\bar\pi} \Hilm\) is equal
  to~\((\Hilms,\pi)\).  Both representations have the same restriction
  to~\(A_e\) because
  \[
  (\Hilms[B],\mu|_{A_e}) \otimes_{\bar\pi} \Hilm
  \cong (\Hilms[B]^+_e,\mu^+_e) \otimes_{B_e} B \otimes_{\bar\pi} \Hilm
  \cong (\Hilms[B]^+_e,\mu^+_e) \otimes_{\bar\pi|_{B_e}} \Hilm
  \cong (\Hilms,\pi).
  \]
  Hence both representations have the same domain by
  Lemma~\ref{lem:restriction_closed}.
  And~\(\bar\pi^+_e(\Hilms[B]^+_e)\Hilm\) is a core for both.  On
  \(\bar\pi^+_e(\Hilms[B]^+_e) \Hilm\), \(a_g\in A_g\) acts by
  mapping \(\bar\pi^+_e(b_e)\xi\) to \(\pi_g(a_g)\bar\pi^+_e(b_e)\xi
  = \bar\pi(a_g\otimes b_e)\xi\) in both representations, where we
  view \(a_g\otimes b_e\in B^+_g\subseteq B\).  Since
  \((\Hilms,\pi)\) and~\((\Hilms',\pi')\) have a common core, they
  are equal.

  This finishes the proof that our two families of maps
  \(\Repi(A)\leftrightarrow \Rep(B)\) are inverse to each other.
  Thus~\(B\) is a weak \Cstar\nb-hull for the integrable
  representations of~\(A\).  Since~\(A_e\) is a \Cstar\nb-hull, the
  integrable representations of~\(A_e\) are admissible.  So are the
  integrable representations of~\(A\) by
  Proposition~\ref{pro:integrable_induced_nice}.  Thus~\(B\) is a
  \Cstar\nb-hull.
\end{proof}

\begin{remark}
  \label{rem:Fell_explicit}
  The fibres~\(B^+_e\) of the Fell bundle in
  Theorem~\ref{the:Fell_bundle_sections_hull} are described in
  Definitions \ref{def:restrict_to_inducible} and~\ref{def:fibre_g},
  including the right Hilbert \(B_e^+\)\nb-module structure
  on~\(B_g^+\).  The rest of the Fell bundle structure needs
  technical extra assumptions.  The simplest way to get it is by
  inducing the universal inducible, integrable representation
  of~\(A\) on~\(B_e^+\) to an integrable representation of~\(A\) on
  the Hilbert \(B_e^+\)\nb-module~\(A\otimes_{A_e} B_e^+\).  The
  Fell bundle \((B_g^+)_{g\in G}\) is represented faithfully
  in~\(\Bound(A\otimes_{A_e} B_e^+)\) by
  Lemma~\ref{lem:barpi_isometric}.  The multiplication, involution,
  and norm in our Fell bundle are simply the multiplication,
  involution and norm in the \Cstar\nb-algebra
  \(\Bound(A\otimes_{A_e} B_e^+)\).  The dense image of
  \(A_g\odot_{A_e} \Hilms[B]_e^+\) in~\(B_g^+\) acts on
  \(A\otimes_{A_e} B_e^+\) by \(a_g\otimes b\mapsto \pi_g(a_g)\cdot
  \bar\pi^+_e(b)\), where~\(\bar\pi^+_e(b)\) is the representation
  of the \Cstar\nb-hull~\(B_e^+\) associated to the induced
  representation of~\(A_e\) on~\(A\otimes_{A_e} B_e^+\), which is
  integrable by assumption.
\end{remark}

\subsection{Two counterexamples}
\label{sec:counter_induction}

Two assumptions limit the generality of the Induction
Theorem~\ref{the:Fell_bundle_sections_hull}.  First, integrability
must be compatible with induction.  Secondly, \(B_e\) should be a
\Cstar\nb-hull and not a weak \Cstar\nb-hull.  Equivalently, all
isometric intertwiners between integrable Hilbert space
representations of~\(A_e\) are \Star{}intertwiners.  We show by two
simple counterexamples that both assumptions are needed.  In
particular, there is no version of the Induction Theorem for weak
\Cstar\nb-hulls.

Both counterexamples involve the group \(G=\Z/2 = \{0,1\}\).  A
\(G\)\nb-graded \Star{}algebra is a \emph{\Star{}superalgebra}, that
is, a \Star{}algebra with a decomposition \(A= A_0 \oplus A_1\) such
that
\[
A_0\cdot A_0 + A_1\cdot A_1 \subseteq A_0,\quad
A_0\cdot A_1 + A_1\cdot A_0 \subseteq A_1,\quad
A_0^* = A_0,\quad
A_1^* = A_1,\quad
1\in A_0.
\]
In both examples, \(A_0=\C[x]\) with \(x=x^*\).

In the first example, \(A\) is the crossed product for the action
of~\(\Z/2\) on \(A_0=\C[x]\) through the involution \(x\mapsto -x\).
That is,
\[
A = \C\langle x,\varepsilon\mid \varepsilon^2=1,\ x\varepsilon =
-\varepsilon x,\ x=x^*,\ \varepsilon = \varepsilon^* \rangle,
\qquad
x\in A_0,\ \varepsilon\in A_1.
\]
Since \(A_1 = \varepsilon A_0 \cong A_0\) as a right
\(A_0\)\nb-module, any representation of~\(A_0\) is inducible.

Let \(B_0 = \Cont_0((0,\infty))\) with the representation of~\(A_0\)
from the inclusion map \((0,\infty)\injto \R = \widehat{A_0}\) (see
Proposition~\ref{pro:commutative_representation}).  This gives a
\Cstar\nb-hull for a class of representations of~\(A_0\) that is
defined by submodule conditions and satisfies the Strong
Local--Global Principle by Theorems \ref{the:commutative_hull}
and~\ref{the:commutative_local-global}.  The class of
\((0,\infty)\)-integrable representations consists of those
representations of~\(\C[x]\) that are generated by a regular,
self-adjoint, \emph{strictly positive} operator.

In a representation of~\(A\), the element \(\varepsilon\in A\) acts
by a unitary involution that conjugates~\(\cl{\pi(x)}\)
to~\(-\cl{\pi(x)}\).  Hence~\(\cl{\pi(x)}\) cannot be strictly
positive.  Thus the zero-dimensional representation is the only
representation of~\(A\) whose restriction to~\(A_0\) is
\(\Cont_0((0,\infty))\)-integrable.  The \Cstar\nb-hull for this
class is~\(\{0\}\).  Theorem~\ref{the:Fell_bundle_sections_hull}
does not apply here because induced representations of inducible,
integrable representations of~\(A_0\) are \emph{never} integrable
when they are non-zero.

The second example is the commutative \Star{}superalgebra
\[
A = \C\langle x,\varepsilon\mid \varepsilon^2=1+x^2,\ x\varepsilon =
\varepsilon x,\ x=x^*,\ \varepsilon = \varepsilon^*\rangle,
\qquad
x\in A_0,\ \varepsilon\in A_1.
\]
Thus \(A_1 = \varepsilon\C[x] \cong A_0\)
with the usual \(A_0\)\nb-bimodule structure and the inner product
\(\braket{\varepsilon a_1}{\varepsilon a_2} = (1+x^2)\cdot
\conj{a_1}\cdot a_2\).  Since \((1+x^2)\abs{a}^2\) is positive
in~\(\C[x]\) for any
\(a\in\C[x]\), any representation of~\(A_0\) is inducible.

Let~\((\Hilms,\pi)\) be a representation of~\(A_0\) on a Hilbert
module~\(\Hilm\) over a \Cstar\nb-algebra~\(D\).  The induced
representation \(A_1 \otimes_{A_0} (\Hilms,\pi)\) lives on the
Hilbert \(D\)\nb-module completion~\(\Hilm_1\) of~\(\Hilms\) for the
inner product \(\braket{\xi_1}{\xi_2}_1 \defeq
\braket{\xi_1}{\pi(1+x^2)\xi_2}\).  Its domain is~\(\Hilms\),
viewed as a dense \(D\)\nb-submodule in~\(\Hilm_1\), and the
representation of~\(A_0\) is~\(\pi\) again.  The
operator~\(\pi(x+\ima)\) on~\(\Hilms\) extends to an isometry
\(I\colon \Hilm_1 \injto \Hilm\) because
\[
\braket{\pi(x+\ima)\xi_1}{\pi(x+\ima)\xi_2}
= \braket{\xi_1}{\pi(x-\ima)\pi(x+\ima)\xi_2}
= \braket{\xi_1}{\pi(1+x^2)\xi_2}
= \braket{\xi_1}{\xi_2}_1
\]
for all \(\xi_1,\xi_2\in\Hilms\).  This isometry commutes
with~\(\pi(a)\) for all \(a\in A\), so it is an isometric
intertwiner \(A_1 \otimes_{A_0} (\Hilms,\pi) \injto (\Hilms,\pi)\).

Now let~\(B_0\) with the universal
representation~\((\Hilms[B]_0,\mu_0)\) be one of the two
noncommutative weak \Cstar\nb-hulls \(\Toep_0\)
or~\(\Comp(\ell^2\N)\) of~\(\C[x]\) described
in~§\ref{sec:polynomials2}.  In a Toeplitz integrable
representation, \(\pi(x+\ima)\) has dense range.  Even more,
\(\pi(x+\ima)\Hilms\) is dense in~\(\Hilms\) in the graph topology.
Thus~\(I\) is a unitary \Star{}intertwiner \(A_1 \otimes_{A_0}
(\Hilms,\pi) \congto (\Hilms,\pi)\) for any integrable
representation~\((\Hilms,\pi)\) of~\(A_0\).

Since all representations of~\(A_0\) are inducible, the unit fibre of
the Fell bundle should be~\(B_0\).  The other fibre~\(B_1\) is
\(A_1\otimes_{A_0} B_0\), which we have identified with~\(B_0\).  The
unitary \(A_1\otimes_{A_0} B_0\cong B_0\) is a \Star{}intertwiner
between the representations of~\(A_0\) as well.  Therefore,
integrability is compatible with induction.  And the left
\(B_0\)\nb-module structure~\(\bar\mu_{0,1}\) on~\(B_1\)
in~\eqref{eq:pie_pig_compatible} is simply left multiplication.

Next we describe the induced representation of~\(A\) on the Hilbert
\(B_0\)\nb-module
\[
A\otimes_{A_0} B_0
= A_0\otimes_{A_0} B_0 \oplus A_1\otimes_{A_0} B_0
\cong B_0 \oplus B_0.
\]
The representations of \(A\) and~\(A_0\) on \(A\otimes_{A_0} B_0\)
have the same domain by Lemma~\ref{lem:restriction_closed}, and
for~\(A_0\) the domain is \(\Hilms[B]_0\oplus \Hilms[B]_0\).  We
claim that~\(A\) acts on this domain by
\[
x\mapsto
\begin{pmatrix}
  \mu_0(x)&0\\0&\mu_0(x)
\end{pmatrix},\qquad
\varepsilon\mapsto
\begin{pmatrix}
  0&\mu_0(x-\ima)\\
  \mu_0(x+\ima)&0
\end{pmatrix}.
\]
We have already seen this for \(x\in A_0\).  Left multiplication
by~\(\varepsilon\) maps \(b\in\Hilms[B]_0\subseteq B_0\) first to
\(\varepsilon\otimes b\in A_1\otimes_{A_0} B_0\), which is mapped by
the isometry~\(I\) to \(\mu_0(x+\ima)b \in \Hilms[B]_0\subseteq B_0\).
And it maps the element \(\mu_0(x+\ima)b \in B_0\) for
\(b\in\Hilms[B]_0\), which corresponds to \(\varepsilon \otimes b\) in
the odd fibre, to \(\varepsilon^2\otimes b = \mu_0(x^2+1)b =
\mu_0(x-\ima)\mu_0(x+\ima)b \in B_0\).  This proves the formula for
the action of~\(\varepsilon\).

The representation~\(\bar\mu_0\) of~\(B_0\) on \(A\otimes_{A_0} B_0\)
is the representation of the weak \Cstar\nb-hull that corresponds to
the representation of \(A_0\subseteq A\) described above.  This is
\[
\bar\mu_0\colon B_0 \to \Mat_2(B_0),\qquad
b\mapsto
\begin{pmatrix}
  b&0\\0&b
\end{pmatrix}.
\]
Hence \(\varepsilon\otimes b\in A_1 \otimes_{A_0} B_0\) for
\(b\in\Hilms[B]_0\) acts by the matrix
\[
\begin{pmatrix}
  0&\mu_0(x-\ima)\\
  \mu_0(x+\ima)&0
\end{pmatrix}
\begin{pmatrix}
  b&0\\0&b
\end{pmatrix} =
\begin{pmatrix}
  0&\mu_0(x-\ima) b\\
  \mu_0(x+\ima)b&0
\end{pmatrix}.
\]
The map \(\mu_0(x+\ima) b\mapsto \mu_0(x-\ima)b\) is the Cayley
transform of~\(\mu_0(x)\).  For our two weak \Cstar\nb-hulls, this
is the unilateral shift \(S\in\Mult(B_0)\) by construction.  Thus
the odd fibre \(B_1\cong B_0\) of our Fell bundle should act by
\[
\bar\mu_1\colon B_0\to \Mat_2(B_0),\qquad
b\mapsto \begin{pmatrix}
  0&S b\\
  b&0
\end{pmatrix}.
\]
The map~\(\bar\mu_0\) is a \Star{}representation,
and~\eqref{eq:barpi_g_representation} gives
\[
\bar\mu_1(b_1)^* \bar\mu_1(b_2) = \bar\mu_0(b_1^* b_2),\qquad
\bar\mu_1(b_1) \bar\mu_0(b_2) = \bar\mu_1(b_1 b_2)
\]
for all \(b_1,b_2\in B_0\).  This is also obvious from our explicit
formulas.  But
\[
\bar\mu_0(b_1)\bar\mu_1(b_2) =
\begin{pmatrix}
  0& b_1 S b_2\\
  b_1 b_2&0
\end{pmatrix}
\quad\text{and}\quad
\bar\mu_0(b_1 b_2) =
\begin{pmatrix}
  0& S b_1 b_2\\
  b_1 b_2&0
\end{pmatrix}
\]
differ if, say \(b_1=S^*\), \(b_2=1\).  In fact,
\(\bar\mu_0(B_0)\cdot \bar\mu_1(B_0)\) is not contained
in~\(\bar\mu_1(B_0)\).  Hence there is no Fell bundle structure
on~\((B_g)_{g\in\Z/2}\) for which~\((\bar\mu_g)_{g\in\Z/2}\) would
be a Fell bundle representation.

\section{Locally bounded unit fibre representations}
\label{sec:locally_bounded_unit_fibre}

We now specialise the Induction
Theorem~\ref{the:Fell_bundle_sections_hull} to the case where the
universal integrable representation of the unit fibre~\(A_e\) is
locally bounded.  In this case, we may first construct a
pro-\Cstar-algebraic Fell bundle whose unit fibre is the
pro-\Cstar\nb-algebra completion of~\(A_e\).  This is relevant
because pro-\Cstar\nb-algebras are much closer to ordinary
\Cstar\nb-algebras than general \Star{}algebras.  We will see the
importance of this in the commutative
case, where the pro-\Cstar-algebraic Fell bundle gives us a twisted
partial group action on the space~\(\hat{A}_e^+\) of positive
characters.

As before, let~\(G\) be a group and let \(A= \bigoplus_{g\in G}
A_g\) be a \(G\)\nb-graded \Star{}algebra.  We are interested in the
locally bounded representations of~\(A_e\), and representations
of~\(A\) that restrict to locally bounded representations
on~\(A_e\).  The class~\(\Rep_\mathrm{b}(A_e)\) of locally bounded
representations of~\(A_e\) is admissible by
Corollary~\ref{cor:locally_bounded_admissible}.  So any weak
\Cstar\nb-hull for some smaller class of representations will be an
ordinary \Cstar\nb-hull.

Let~\(\mathcal{A}_e\) be the pro-\Cstar\nb-algebra completion of the
unit fibre~\(A_e\), that is, the completion of~\(A_e\) in the
topology defined by the directed set~\(\mathcal{N}(A_e)\) of all
\Cstar\nb-seminorms on~\(A_e\).  Locally bounded representations
of~\(A_e\) are equivalent to locally bounded representations
of~\(\mathcal{A}_e\) by
Proposition~\ref{pro:locally_bounded_pro-Cstar}.

When is a locally bounded representation inducible?

\begin{proposition}
  \label{pro:locally_bounded_inducible}
  A locally bounded representation~\((\Hilms,\varphi)\) of~\(A_e\)
  on a Hilbert module~\(\Hilm\) is inducible if and only if
  \(\varphi(a^* a)\ge0\) for all \(a\in A_g\), \(g\in G\).
\end{proposition}

The difference to the general criterion for inducibility in
Proposition~\ref{pro:inducible_criteria} is that we do not consider
matrices.

\begin{proof}
  The subspace \(\Hilm_\mathrm{b}\subseteq \Hilms\) of bounded
  vectors is a core for~\(\varphi\).  As in the proof of
  Proposition~\ref{pro:inducible_A_B}, it suffices to prove the
  positivity of the inner product for a finite linear combination
  \(\sum_{k=1}^n a_k \otimes \xi_k\) with \(a_k\in A_g\),
  \(\xi_k\in\Hilm_\mathrm{b}\) for a fixed \(g\in G\).  Since there
  are only finitely many~\(\xi_k\), there is a
  \Cstar\nb-seminorm~\(q\) on~\(A_e\) so that all~\(\xi_k\) are
  \(q\)\nb-bounded.  Thus we may replace~\(\Hilm\) by the Hilbert
  submodule \(\Hilm_q\) of \(q\)\nb-bounded vectors, where the
  representation of~\(A_e\) extends to the \Cstar\nb-completion
  \(D\defeq (\mathcal{A}_e)_q\) for~\(q\).  Since we assume
  \(\varphi(a^* a)\ge0\) for all \(a\in A_g\), this representation
  factors through the quotient of~\(D\) by the closed ideal~\(I\)
  generated by the negative parts~\((a^* a)_-\) for all \(a\in
  A_g\), \(g\in G\).
  The \(D/I\)\nb-valued inner product \(\braket{a_1}{a_2} \defeq
  a_1^* a_2 \bmod I\) on~\(A_g\) is positive definite by construction;
  since~\(D/I\) is a \Cstar\nb-algebra, we may use the usual notion
  of positivity here, which does not involve matrices.  Then the
  inner product on the tensor product \(A_g \otimes_{D/I} \Hilm_q\)
  is also positive definite.  This is what we had to prove.
\end{proof}

A pro-\Cstar\nb-algebra has a functional calculus for self-adjoint
elements.  Hence we may construct the negative parts \((a^* a)_- \in
\mathcal{A}_e\) for \(a\in A_g\), \(g\in G\).  We
let~\(\mathcal{A}_e^+\) be the completed quotient of~\(\mathcal{A}_e\) by the
closed two-sided ideal generated by these elements.  This is another
pro-\Cstar\nb-algebra, and it is the largest quotient in which \(a^*
a\ge0\) for all \(a\in A_g\), \(g\in G\).  By
Proposition~\ref{pro:locally_bounded_inducible}, a locally bounded
representation of~\(A_e\) is inducible if and only if the
corresponding locally bounded representation of~\(\mathcal{A}_e\)
factors through~\(\mathcal{A}_e^+\).

\begin{corollary}
  \label{cor:locally_bounded_inducible_pro-Cstar}
  There is an equivalence between the inducible, locally
  bounded representations of~\(A_e\) and the locally bounded
  representations of the pro-\Cstar\nb-algebra~\(\mathcal{A}_e^+\),
  which is compatible with isometric intertwiners and interior
  tensor products.
\end{corollary}

\begin{proof}
  Proposition~\ref{pro:locally_bounded_inducible} says that the
  equivalence in Proposition~\ref{pro:locally_bounded_pro-Cstar}
  maps the subclass in~\(\Rep_\mathrm{b}(A_e)\) of inducible,
  locally bounded representations of~\(A_e\) onto the subclass
  \(\Rep_\mathrm{b}(\mathcal{A}_e^+)\) in
  \(\Rep_\mathrm{b}(\mathcal{A}_e)\).
\end{proof}

Let~\(\mathcal{N}(A_e)^+\) be the directed set of
\Cstar\nb-seminorms on~\(\mathcal{A}_e^+\).  This is isomorphic to
the subset of~\(\mathcal{N}(A_e)\) consisting of all
\Cstar\nb-seminorms~\(q\) on~\(A_e\) for which \(a^* a\ge 0\) holds in the
\Cstar\nb-completion~\((\mathcal{A}_e)_q\) for all \(a\in A_g\),
\(g\in G\).  We would like to
complete~\(A\) to a \Star{}algebra \(\bigoplus_{g\in G}
\mathcal{A}_g^+\) with unit fibre~\(\mathcal{A}_e^+\), where
each~\(\mathcal{A}_g^+\) is a Hilbert bimodule
over~\(\mathcal{A}_e^+\).  But such a construction does not work in
the following example.

\begin{example}
  \label{exa:locally_bounded_induction_non-compatible}
  It can happen that the class of locally bounded representations
  of~\(A_e\) is not compatible with induction.  Let
  \(\Endo^*(\C[\N])\) be the \Star{}algebra of all
  \(\infty\times\infty\)-matrix with only finitely many entries in
  each row and each column, with the usual matrix multiplication and
  involution.  Let~\(A\) be the \(\Z/2\)-graded \Star{}algebra of
  block \(2\times 2\)-matrices
  \[
  \begin{pmatrix}
    a&b\\c&d
  \end{pmatrix},\qquad
  a\in \C,\ b\in \C[\N],\ c\in \C[\N],\ d\in \Endo^*(\C[\N]),
  \]
  with the grading where \(a,d\) are even and \(b,c\) are odd.  Here
  \(b\) and~\(c\) are infinite column and row vectors with only
  finitely many non-zero entries, respectively.  Thus \(A \cong
  \Endo^*(\C[\N])\) with the grading induced by the grading
  on~\(\C[\N]\) where \(\C\cdot\delta_0\) is the even part and the
  span of \(\delta_i\) for \(i>0\) is the odd part.

  The character \((a,d)\mapsto a\) is a bounded representation of
  the unit fibre~\(A_0\).  Induction gives the standard representation
  of~\(A\) on the Hilbert space \(\C\oplus\ell^2(\N) \cong
  \ell^2(\N)\) by matrix-vector multiplication.  This representation
  is irreducible because already the ideal of finite
  matrices~\(\Mat_\infty(\C)\) in~\(A\) acts irreducibly.  It is not
  bounded, that is, some elements in~\(\Endo^*(\C[\N])\) act by
  unbounded operators on~\(\ell^2(\N)\).  Hence it is not locally
  bounded by Proposition~\ref{pro:locally_bounded_irreducible}.
\end{example}

To rule out this problem, we now \emph{assume that induction
  from~\(A_e\) to~\(A\) and restriction back to~\(A_e\) maps bounded
  representations of~\(A_e\) again to bounded representations
  of~\(A_e\)}, briefly, that boundedness is compatible with
induction.  This implies that local boundedness is compatible with
induction because a locally bounded representation contains bounded
subrepresentations whose union is a core for it.  Our assumption is
equivalent to the boundedness of the induced representations
of~\(A_e\) on the Hilbert \((A_e^+)_q\)\nb-modules
\(A_g\otimes_{A_e} (A_e^+)_q\) for all \(g\in G\) and
\(q\in\mathcal{N}(A_e)^+\).  That is, there is another norm
\(q'\in\mathcal{N}(A_e)^+\) such that
\[
q(a^* b^* b a)
= \norm{b a}_q^2
\le \norm{b}_{q'}^2 \norm{a}_q^2
= q'(b)^2 \cdot q(a^* a)
\]
for all \(a\in A_g\), \(b\in A_e\).  Let~\(\mathcal{A}_g^+\) be the
completion of~\(A_g\) in the topology generated by the family of
norms~\(q(a^* a)\) for \(q\in \mathcal{N}(A_e)^+\).

\begin{lemma}
  \label{lem:mult_lb_continuous}
  The multiplication maps and the involutions in~\((A_g)_{g\in G}\)
  extend to continuous maps \(\mathcal{A}_g^+ \times \mathcal{A}_h^+
  \to \mathcal{A}_{g h}^+\) and \(\mathcal{A}_g^+ \to
  \mathcal{A}_{g^{-1}}^+\) for \(g,h\in G\).
\end{lemma}

\begin{proof}
  Given \(q\in \mathcal{N}(A_e)^+\), let \(q'\in
  \mathcal{N}(A_e)^+\) be such that \(q(a^* b^* b a)\le q'(b)^2\cdot
  q(a^* a)\) for all \(a\in A_h\), \(b\in A_e\).  If \(b\in A_g\),
  \(a\in A_h\), then
  \[
  \norm{b a}_q^2
  \defeq q(a^* b^* b a)
  = q(a^* (b^* b)^{\nicefrac12} (b^* b)^{\nicefrac12} a)
  \le q'((b^* b)^{\nicefrac12})^2\cdot q(a^* a)
  = \norm{b}_{q'}^2 \norm{a}_q^2
  \]
  That is, the multiplication is jointly continuous with respect to
  the topology defining~\((\mathcal{A}_g^+)_{g\in G}\) and hence
  extends to a jointly continuous map \(\mathcal{A}_g^+ \times
  \mathcal{A}_h^+ \to \mathcal{A}_{g h}^+\).

  Furthermore, \(q(a a^*)^2 = q(a a^* a a^*) \le q'(a^* a) \cdot q(a
  a^*)\) and hence \(q(a a^*) \le q'(a^* a)\) for all \(a\in A_h\).
  That is, \(\norm{a^*}_q^2 \le \norm{a}_{q'}^2\) for all \(a\in
  A_h\).  Thus the involution is continuous as well.
\end{proof}

The completion \(\mathcal{A}^+ \defeq \bigoplus_{g\in G}
\mathcal{A}_g^+\) of~\(A\) is again a \Star{}algebra by
Lemma~\ref{lem:mult_lb_continuous}.  By construction
of~\(\mathcal{A}_e^+\), the inner products \(a^* a\in
\mathcal{A}_e^+\) are positive for \(a\in A_g\), \(g\in G\); this
remains so for \(a\in \mathcal{A}_g^+\) because the subset of
positive elements in~\(\mathcal{A}_e^+\) is closed.  Thus
\((\mathcal{A}_g^+)_{g\in G}\) has the usual properties of a Fell
bundle over~\(G\), except that the fibres are only Hilbert bimodules
over a pro-\Cstar-algebra.  We
interpret~\((\mathcal{A}_g^+)_{g\in G}\) as a partial action
of~\(G\) on~\(\mathcal{A}_e^+\) by Hilbert bimodules as
in~\cite{Buss-Meyer:Actions_groupoids}.

Usually, the norms \(q(a^* a)\) and \(q(a a^*)\) on~\(A_g\) are not
equivalent for a fixed \(q\in\mathcal{N}(A)^+\).  This prevents us
from completing~\(\mathcal{A}^+\) to a pro-\Cstar\nb-algebra.  It
also means that the integrable representations of~\(A\) are not
locally bounded on~\(A\), but only on~\(A_e\).  This happens in
interesting examples such as the Weyl algebra discussed in
§\ref{sec:Weyl_twisted}.  This phenomenon for Fell bundles is
related to the known problem that crossed products for group actions
on pro-\Cstar\nb-algebras only work well if the action is strongly
bounded, that is, the invariant continuous \Cstar\nb-seminorms are
cofinal in the set of all continuous \Cstar\nb-seminorms,
see~\cite{Joita:New_crossed_products}.

\begin{proposition}
  \label{pro:locally_bounded_Fell_representations}
  Suppose that boundedness for representations of~\(A_e\) is
  compatible with induction to~\(A\).  Representations of~\(A\) that
  restrict to locally bounded representations on~\(A_e\) are
  equivalent to representations of the \Star{}algebra
  \(\mathcal{A}^+ = \bigoplus_{g\in G} \mathcal{A}_g^+\) that
  restrict to locally bounded representations
  on~\(\mathcal{A}_e^+\); this equivalence is compatible with
  isometric intertwiners and interior tensor products.
\end{proposition}

\begin{proof}
  Let~\(\pi\) be a representation of~\(A\) for which~\(\pi_e\) is a
  locally bounded representation of~\(A_e\).  The
  representation~\(\pi_e\) is inducible by
  Lemma~\ref{lem:induced_repr_inducible}.  Hence~\(\pi_e\) is the
  closure of the restriction of a locally bounded
  representation~\(\bar\pi^+_e\) of~\(\mathcal{A}_e^+\) by
  Corollary~\ref{cor:locally_bounded_inducible_pro-Cstar}.  The
  representation~\(\pi_g\) of~\(A_g\) for \(g\in G\) is continuous
  with respect to the topology defining~\(\mathcal{A}_g^+\) and the
  graph topology on the domain of~\(\pi_g\) because \(\pi_g(a)^*
  \pi_g(a) = \pi_e(a^* a)\).  Hence it extends uniquely
  to~\(\mathcal{A}_g^+\), and this gives a
  representation~\(\bar\pi^+\) of~\(\bigoplus \mathcal{A}_g^+\) such
  that~\(\pi\) is the closure of~\(\bar\pi^+\circ j\).  It is easy
  to see that this equivalence between the locally bounded
  representations of~\(A\) and the representations of
  \(\bigoplus_{g\in G} \mathcal{A}_g^+\) that are locally bounded
  on~\(\mathcal{A}_e^+\) is compatible with isometric intertwiners
  and interior tensor products.
\end{proof}

We will explore the consequences of this in the case of
commutative~\(A_e\) in~§\ref{sec:commutative_Ae}.  In that case,
boundedness is automatically compatible with induction, and the
pro-\Cstar-algebraic Fell bundle~\(\mathcal{A}_e^+\) gives rise to a
twisted groupoid with object space~\(\hat{A}_e^+\).  Thus the
\Cstar\nb-hull produced by the Induction
Theorem~\ref{the:Fell_bundle_sections_hull} is a twisted groupoid
\Cstar\nb-algebra when~\(A_e\) is commutative and the integrable
representations of~\(A_e\) are locally bounded.

Here we briefly consider the situation of
Theorem~\ref{the:Prim_A_locally_compact_gives_hull} where
\(\Cont_0(\mathcal{A}_e^+)\) is dense in~\(\mathcal{A}_e^+\) and
provides a \Cstar\nb-hull for the class of locally bounded
representations.  Then we define
\[
\Cont_0(\mathcal{A}_g^+) \defeq
\{a\in A_g \mid a^* a \in \Cont_0(\mathcal{A}_e^+) \}.
\]
That is, \(a\in \Cont_0(\mathcal{A}_g^+)\) if and only if for all
\(\varepsilon>0\) there is \(q\in \mathcal{N}(A_e)^+\) such that
\(\norm{a^* a}_\mathfrak{p} <\varepsilon\) for all \(\mathfrak{p} \in
\Prim(\mathcal{A}_e^+) \setminus\Prim(\mathcal{A}_e^+)_q\).  Since the
involutions \(A_g \to A_{g^{-1}}\) and \(A_{g^{-1}} \to A_g\) are both
continuous, they are homeomorphisms.  Thus \(a\in
\Cont_0(\mathcal{A}_g^+)\) if and only if \(a a^* \in
\Cont_0(\mathcal{A}_e^+)\).  The proof of
Lemma~\ref{lem:compactly_supported_ideal} shows that \(\mathcal{A}^+_g
\cdot \Contc(\mathcal{A}_e^+)\) and
\(\Contc(\mathcal{A}_e^+)\cdot\mathcal{A}^+_g\) are dense
in~\(\Cont_0(\mathcal{A}_g^+)\).

\begin{theorem}
  \label{the:Fell_lb}
  Assume that boundedness is compatible with induction from~\(A_e\)
  to~\(A\) and that \(\Cont_0(\mathcal{A}_e^+)\) is dense
  in~\(\mathcal{A}_e^+\).  Then \(\Cont_0(\mathcal{A}_g^+)_{g\in
    G}\) is a Fell bundle over~\(G\) whose section \Cstar\nb-algebra
  is a \Cstar\nb-hull for the class of all representations of~\(A\)
  that restrict to a locally bounded representation of~\(A_e\).
\end{theorem}

\begin{proof}
  The assumption that boundedness is compatible with induction allows
  us to build the pro-\Cstar-algebraic Fell
  bundle~\((\mathcal{A}_g^+)_{g\in G}\).  Call a representation of \(A
  = \bigoplus_{g\in G} A_g\) or \(\mathcal{A}^+ \defeq \bigoplus_{g\in
    G} \mathcal{A}_g^+\) integrable if the restriction to the unit
  fibre \(A_e\) or~\(\mathcal{A}_e^+\) is locally bounded,
  respectively.  These classes of integrable representations are
  equivalent by
  Proposition~\ref{pro:locally_bounded_Fell_representations}.

  Since \(\Cont_0(\mathcal{A}_e^+)\) is dense in~\(\mathcal{A}_e^+\),
  it is a \Cstar\nb-hull for the locally bounded representations
  of~\(\mathcal{A}_e^+\) by
  Theorem~\ref{the:locally_bounded_relative_hull}.  Equivalently, it
  is a \Cstar\nb-hull for the inducible, locally bounded
  representations of~\(A_e\).
  Let \(\Cont_0(\mathcal{A}^+) \defeq \bigoplus_{g\in G}
  \Cont_0(\mathcal{A}_g^+)\).  Representations of
  \(\Cont_0(\mathcal{A}^+)\) are equivalent to representations of the
  Fell bundle \(\Cont_0(\mathcal{A}^+_g)\).  Thus we must prove
  that the class of all representations of \(\Cont_0(\mathcal{A}^+)\)
  is equivalent to the class of integrable representations
  of~\(\mathcal{A}^+\).  More precisely, the equivalence maps a
  representation~\(\varrho\) of~\(\Cont_0(\mathcal{A}^+)\) on a
  Hilbert module~\(\Hilm\) to the representation~\(\pi\)
  of~\(\mathcal{A}^+\) with the core
  \(\varrho(\Contc(\mathcal{A}^+_e))\Hilm\) and \(\pi(a)\varrho(b)\xi
  \defeq \varrho(a\cdot b)\xi\) for all \(a\in\mathcal{A}^+\),
  \(b\in\Contc(\mathcal{A}^+_e)\), \(\xi\in\Hilm\); here \(a\cdot b\)
  is the product in~\(\mathcal{A}^+\), which belongs to
  \(\Cont_0(\mathcal{A}^+)\) if \(b\in\Contc(\mathcal{A}^+_e)\).

  In the converse direction, we may simply restrict a locally bounded
  representation of~\(\mathcal{A}^+\) to the \Star{}subalgebra
  \(\Cont_0(\mathcal{A}^+)\).  This restriction is nondegenerate
  because \(\Cont_0(\mathcal{A}_e^+)\subseteq\Cont_0(\mathcal{A}^+)\)
  acts nondegenerately in any integrable representation
  of~\(\mathcal{A}^+\): this is part of the equivalence between
  representations of \(\Cont_0(\mathcal{A}_e^+)\) and locally bounded
  representations of~\(\mathcal{A}_e^+\) in
  Theorem~\ref{the:locally_bounded_relative_hull}.  We claim that the
  maps from representations of \(\Cont_0(\mathcal{A}^+)\) to
  integrable representations of~\(\mathcal{A}^+\) and back are inverse
  to each other.

  Let~\(\pi\) be an integrable representation of~\(\mathcal{A}^+\) on
  a Hilbert module~\(\Hilm\).  The representations \(\pi\)
  and~\(\pi|_{\mathcal{A}_e^+}\) have the same domain by
  Lemma~\ref{lem:restriction_closed}.
  Since~\(\pi|_{\mathcal{A}_e^+}\) is locally bounded,
  \(\pi(\Contc(\mathcal{A}_e^+))\Hilms\) is a core
  for~\(\pi|_{\mathcal{A}_e^+}\).  Since
  \(\Contc(\mathcal{A}_e^+)\cdot \mathcal{A}^+_g =
  \mathcal{A}^+_g\cdot \Contc(\mathcal{A}_e^+)\) for all \(g\in G\),
  this subspace is \(\pi(\mathcal{A}^+)\)-invariant and thus a core
  for~\(\pi\).  The representation~\(\varrho\)
  of~\(\Cont_0(\mathcal{A}^+)\) is the closure of the restriction
  of~\(\pi\) to \(\Cont_0(\mathcal{A}^+) \subseteq \mathcal{A}^+\).
  By definition, the representation of~\(\mathcal{A}^+\) has the core
  \(\varrho(\Contc(\mathcal{A}_e^+))\Hilm\) and acts there by
  \(\pi'(a)\varrho(b)\xi = \varrho(a\cdot b)\xi\).  The subspace
  \(\varrho(\Contc(\mathcal{A}_e^+))\Hilms\) is a core for this
  representation because the map \(\xi\mapsto \pi'(a)\varrho(b)\xi\)
  is continuous in the norm topology on~\(\Hilm\) and~\(\Hilms\) is
  dense in~\(\Hilm\).  If \(\xi\in\Hilms\), then
  \(\varrho(b)\xi=\pi(b)\xi\) and hence \(\pi'(a)\pi(b)\xi =
  \pi(a)\pi(b)\xi\) for all \(a\in\mathcal{A}^+\), \(b\in
  \Contc(\mathcal{A}_e^+)\), \(\xi\in\Hilms\).  This implies
  \(\pi=\pi'\), as desired.

  Now start with a representation~\(\varrho\)
  of~\(\Cont_0(\mathcal{A}^+)\).  Let~\(\pi\) be the associated
  integrable representation of~\(\mathcal{A}^+\).  It has the core
  \(\varrho(\Contc(\mathcal{A}^+_e))\Hilm\) and acts there by
  \(\pi(a)\varrho(b)\xi = \varrho(a\cdot b)\xi\) for all
  \(a\in\mathcal{A}^+\), \(b\in \Contc(\mathcal{A}^+_e)\),
  \(\xi\in\Hilm\).  In particular, if \(a\in\Cont_0(\mathcal{A}^+)\),
  then \(\pi(a)\varrho(b)\xi = \varrho(a\cdot b)\xi =
  \varrho(a)\varrho(b)\xi\).  Since \(\Cont_0(\mathcal{A}^+_e) \cdot
  \Cont_0(\mathcal{A}^+_g)\) is dense in~\(\Cont_0(\mathcal{A}_g^+)\)
  for all \(g\in G\), the restriction of~\(\varrho\) to
  \(\Cont_0(\mathcal{A}^+_e)\) remains nondegenerate.  Therefore, the
  set of \(\varrho(b)\xi\) for \(b\in \Contc(\mathcal{A}^+_e)\),
  \(\xi\in\Hilm\) is dense in~\(\Hilm\).  Hence~\(\varrho\) is the
  restriction of~\(\pi\) to \(\Cont_0(\mathcal{A}^+)\subseteq
  \mathcal{A}^+\), as desired.
\end{proof}

The proof of Theorem~\ref{the:Fell_lb} does not use the
constructions in Section~\ref{sec:graded_to_Fell} and so provides an
alternative proof of the Induction Theorem in case the chosen class
of integrable representations of~\(A_e\) is the class of all locally
bounded representations.

\section{Fell bundles with commutative unit fibre}
\label{sec:commutative_Ae}

In this section, we apply the Induction Theorem in the case
where~\(A_e\) and the chosen \Cstar\nb-hull~\(B_e\) are commutative.
This is the only case considered
in~\cite{Savchuk-Schmudgen:Unbounded_induced}.  Extra assumptions
in~\cite{Savchuk-Schmudgen:Unbounded_induced} ensure that the
\Cstar\nb-hull for the integrable representations of~\(A\) is the
crossed product for a partial action of~\(G\) on the space
\(\hat{A}^+_e\subseteq \hat{A}_e\) of positive characters.  Without
these assumptions, we shall get a ``twisted'' crossed product for a
partial action.

So let~\(G\) be a discrete group and \(A = \bigoplus_{g\in G} A_g\) a
\(G\)\nb-graded \Star{}algebra such that~\(A_e\) is commutative.  We
have already classified the possible commutative \Cstar\nb-hulls
for~\(A_e\) in~§\ref{sec:commutative_hulls}.  In particular, all
commutative weak \Cstar\nb-hulls are already \Cstar\nb-hulls by
Theorem~\ref{the:commutative_hull}, and they correspond to injective,
continuous maps from locally compact spaces to the
spectrum~\(\hat{A}_e\) of~\(A_e\).

Explicitly, let~\(X\) be a locally compact space and let \(j\colon
X\to \hat{A}_e\) be an injective, continuous map.  Let
\(B_e=\Cont_0(X)\) and define a representation of~\(A_e\) on~\(B_e\)
with domain~\(\Contc(X)\) by \((a\cdot f)(x) = \hat{a}(j(x))\cdot
f(x)\) for all \(a\in A_e\), \(f\in\Contc(X)\), \(x\in X\), where
\(\hat{a}(\chi) = \chi(a)\) for \(\chi\in\hat{A}_e\).  Let~\(\mu_e\)
be the closure of this representation of~\(A_e\) on~\(B_e\).  The
\Cstar\nb-algebra~\(B_e\) with the universal
representation~\(\mu_e\) is a \Cstar\nb-hull for a class
\(\Repi(A_e,X)\) of representations of~\(A_e\) by
Theorem~\ref{the:commutative_hull}, and any commutative
\Cstar\nb-hull is of this form.

Let \(\Repi(A,X)\) be the class of representations of~\(A\) that
restrict to a representation in \(\Repi(A_e,X)\) on~\(A_e\), as in
Definition~\ref{def:induced_integrable}.  If \(\Repi(A_e,X)\) is
compatible with induction to~\(A\) as in
Definition~\ref{def:induced_actions_integrable}, then
Theorem~\ref{the:Fell_bundle_sections_hull} gives a Fell bundle
whose section \Cstar\nb-algebra is a \Cstar\nb-hull
for~\(\Repi(A,X)\).  We are going to characterise exactly when this
happens and describe the \Cstar\nb-hull for~\(\Repi(A,X)\) as a
twisted groupoid \Cstar\nb-algebra.

Any representation of~\(A_e\) on a commutative \Cstar\nb-algebra is
locally bounded by Proposition~\ref{pro:commutative_representation}.  Hence
the constructions in~§\ref{sec:locally_bounded_unit_fibre} specialise
to our commutative case.  Actually, we shall make these results more
explicit through independent proofs.  First we describe the
\Cstar\nb-hull~\(B_e^+\) for the inducible representations
in~\(\Repi(A_e,X)\) as in
Proposition~\ref{pro:locally_bounded_inducible}:

\begin{lemma}
  \label{lem:inducible_Ae_commutative}
  Call a character \(\chi\in\hat{A}_e\) \emph{positive} if \(\chi(a^*
  a)\ge0\) for all \(a\in A_g\) and all \(g\in G\).  These form a
  closed subset~\(\hat{A}_e^+\) of~\(\hat{A}_e\), and \(B_e^+ =
  \Cont_0\bigl(j^{-1}(\hat{A}_e^+)\bigr)\).
\end{lemma}

\begin{proof}
  The positive characters form a closed subset in~\(\hat{A}_e\) by
  definition of the topology on~\(\hat{A}_e\).  We have
  constructed~\(B_e^+\) in Proposition~\ref{pro:inducible_A_B} as a
  quotient of~\(B_e\), such that a representation is inducible if and
  only if it factors through~\(B_e^+\).  Thus~\(B_e^+\) corresponds to a
  certain closed subset of~\(\hat{A}_e\).  Its points are the
  inducible characters of~\(A_e\).  Let~\(\chi\) be a character.  Any
  vector in \(A_g\otimes_{A_e,\chi} \C\) is of the form \(a\otimes 1\)
  for some \(a\in A_g\), that is, there is no need to take linear
  combinations.  Hence the sesquilinear form on
  \(A_g\otimes_{A_e,\chi} \C\) for all \(g\in G\) is positive
  semidefinite if and only if \(\chi(a^* a)\ge0\) for all \(a\in A_g\)
  and all \(g\in G\), that is, \(\chi\) is positive.  Thus~\(B_e^+\)
  is the quotient corresponding to those \(x\in \hat{A}_e\) for which
  \(j(x)\in\hat{A}_e\) is positive.
\end{proof}

\begin{theorem}
  \label{the:induce_character}
  Let \(g\in G\) and \(\chi\in\hat{A}_e^+\).  Then \(\dim
  A_g\otimes_{A_e,\chi} \C\le1\).  The set
  \[
  \mathcal{D}_{g^{-1}}\defeq
  \{\chi\in \hat{A}_e^+ \mid \dim A_g\otimes_{A_e,\chi} \C =1\}
  \]
  is relatively open in~\(\hat{A}_e^+\).  The left \(A_e\)\nb-module
  structure on~\(A_g\otimes_{A_e,\chi} \C \cong \C\) for
  \(\chi\in\mathcal{D}_{g^{-1}}\) is by a character~\(\vartheta_g(\chi)\)
  that belongs to~\(\mathcal{D}_g\).  The map~\(\vartheta_g\) is a
  homeomorphism from~\(\mathcal{D}_{g^{-1}}\) onto~\(\mathcal{D}_g\),
  and these maps form a \emph{partial action} of~\(G\)
  on~\(\hat{A}_e^+\), that is, \(\vartheta_e=\id_{\hat{A}_e^+}\) and
  \(\vartheta_g \circ \vartheta_h \subseteq \vartheta_{g h}\) for all \(g,h\in
  G\).
\end{theorem}

\begin{proof}
  As in the proof of Lemma~\ref{lem:inducible_Ae_commutative},
  \(A_g\otimes_{A_e,\chi} \C\) is the Hausdorff completion
  of~\(A_g\) in the norm coming from the inner product
  \(\braket{a_1}{a_2} \defeq \chi(a_1^* a_2)\).  We write
  \(\lambda\cdot a\) for \(a\otimes \lambda\) for \(a\in A_g\),
  \(\lambda\in\C\) throughout this proof, and we write \(a \equiv
  b\) if \(a, b\in A\) have the same image in
  \(A_g\otimes_{A_e,\chi} \C\).  Let \(a,b\in A_g\) satisfy
  \(\chi(a^* a)\neq0\) and \(\chi(b^* b)\neq0\).  We must show that
  \(a\) and~\(b\) are parallel in \(A_g\otimes_{A_e,\chi} \C\).

  The following computation makes
  \cite{Dowerk-Savchuk:Induced}*{Footnote~3} explicit:
  \[
  (a^* a b^* b)^2
  = a^* a b^* (b a^*) (a b^*) b
  = a^* a (b^* a) (b^* b a^* b)
  = a^* a b^* b a^* b b^* a
  \]
  because~\(A_e\) is commutative and the terms in parentheses belong
  to~\(A_e\).  Hence
  \[
  \chi(a^* a)^2\chi(b^* b)^2 = \chi(a^* a) \chi(b^* b) \chi(a^*
  b)\chi(b^* a).
  \]
  Since \(\chi(a^* a)\neq0\) and \(\chi(b^* b)\neq0\), this implies
  \begin{equation}
    \label{eq:chi_aabb}
    \chi(a^* a) \chi(b^* b) = \chi(a^* b)\chi(b^* a)
    = \abs{\chi(a^* b)}^2 \neq 0.
  \end{equation}
  The inner product on \(A_g\otimes_{A_e,\chi} \C\) annihilates
  \(a\cdot c\otimes 1- a\otimes \chi(c)\), which we write as
  \(a\cdot c- \chi(c)a\), for all \(a\in A_g\), \(c\in A_e\).  Hence
  \begin{multline}
    \label{eq:compare_Ag_commutative}
    a = \frac{\chi(a^* b)\chi(b^* a)}{\chi(a^* a) \chi(b^* b)} a
    \equiv \frac{a a^* b b^* a}{\chi(a^* a) \chi(b^* b)}
    = \frac{b b^* a a^* a}{\chi(a^* a) \chi(b^* b)}
    \\= \frac{\chi(b^* a) \chi(a^* a)}{\chi(a^* a) \chi(b^* b)} b
    = \frac{\chi(b^* a)}{\chi(b^* b)} b.
  \end{multline}
  Thus all non-zero \(a,b\in A_g\otimes_{A_e,\chi}\C\) are parallel,
  that is, \(\dim A_g\otimes_{A_e,\chi}\C\le 1\).  The space
  \(A_g\otimes_{A_e,\chi}\C\) is non-zero if and only if there is
  \(a\in A_g\) with \(\chi(a^* a)\neq0\).  Thus
  \begin{equation}
    \label{eq:partial_act_Aeplus_dom}
    \mathcal{D}_{g^{-1}} =
    \{\chi\in\hat{A}_e^+\mid \chi(a^* a)\neq0\text{ for some }a\in A_g\}.
  \end{equation}
  The latter set is relatively open in~\(\hat{A}_e^+\).

  Let \(\chi\in\mathcal{D}_{g^{-1}}\).  Then 
  \(\dim A_g\otimes_{A_e,\chi}\C=1\).  Hence the representation
  of~\(A_e\) on it is by a character, which we denote
  by~\(\vartheta_g(\chi)\).  This character is an inducible
  representation by Lemma~\ref{lem:induced_repr_inducible}, and hence
  positive by Lemma~\ref{lem:inducible_Ae_commutative}.  There is
  \(b\in A_g\) with \(\chi(b^* b)>0\).  If \(a\in A_e\),
  then~\eqref{eq:compare_Ag_commutative} implies \(a b \equiv
  \frac{\chi(b^* a b)}{\chi(b^* b)} b\).  Thus
  \begin{equation}
    \label{eq:partial_act_Aeplus}
    \vartheta_g(\chi)(a) = \frac{\chi(b^* a b)}{\chi(b^* b)}
  \end{equation}
  for all \(a\in A_e\).  Hence \(\vartheta_g(\chi)\bigl((b^*)^*
  b^*\bigr)\neq0\), so that \(\vartheta_g(\chi) \in \mathcal{D}_g\)
  by~\eqref{eq:partial_act_Aeplus_dom}.  Thus~\(\vartheta_g\)
  maps~\(\mathcal{D}_{g^{-1}}\) to~\(\mathcal{D}_g\).
  Equation~\eqref{eq:partial_act_Aeplus} also implies that the
  map~\(\vartheta_g\) is continuous on the open set of characters
  in~\(A_e^+\) with \(\chi(b^* b)>0\).  Since these open sets for
  different \(b\in A_g\) cover~\(\mathcal{D}_{g^{-1}}\), the
  map~\(\vartheta_g\) is continuous on all of~\(\mathcal{D}_{g^{-1}}\).

  Let \(g,h \in G\) and let \(\chi\in\mathcal{D}_{h^{-1}}\) and
  \(\vartheta_h(\chi)\in \mathcal{D}_{g^{-1}}\).  Then there is \(b_h\in
  A_h\) with \(\chi(b_h^* b_h)>0\), and \(b_g\in A_g\) with
  \(\vartheta_h(\chi)(b_g^* b_g)>0\).  Thus \(\chi(b_h^* b_g^* b_g b_h) =
  \chi(b_h^* b_h)\cdot \vartheta_h(\chi)(b_g^* b_g)>0\), and
  so~\eqref{eq:partial_act_Aeplus} for \(b=b_g b_h \in A_{g h}\)
  describes~\(\vartheta_{g h}\).  Hence
  \[
  \vartheta_{g h}(\chi)(a)
  = \frac{\chi(b_h^* b_g^* a b_g b_h)}{\chi(b_h^* b_g^* b_g b_h)}
  = \frac{\vartheta_h(\chi)(b_g^* a b_g)}{\vartheta_h(\chi)(b_g^* b_g)}
  = \vartheta_g \bigl(\vartheta_h(\chi)\bigr)(a).
  \]
  Thus \(\vartheta_{g h}\subseteq \vartheta_g \vartheta_h\) for all
  \(g,h\in G\).  In addition, \(\vartheta_e = \id_{\hat{A}_e^+}\).
  So the maps~\(\vartheta_g\) form a partial action of~\(G\)
  on~\(\hat{A}_e^+\), see~\cite{Exel:Partial_dynamical}.  In
  particular, \(\vartheta_g\) is a homeomorphism
  from~\(\mathcal{D}_{g^{-1}}\) onto~\(\mathcal{D}_g\) with
  inverse~\(\vartheta_{g^{-1}}\).
\end{proof}

In the examples considered in
\cites{Savchuk-Schmudgen:Unbounded_induced, Dowerk-Savchuk:Induced},
the space~\(\hat{A}_e^+\) is locally compact and the \Cstar\nb-hull
for the integrable representations of~\(A\) is the crossed product for
the partial action of~\(G\) on~\(\hat{A}_e^+\) described above.  In
general, however, certain twists are possible.  The partial action
of~\(G\) on~\(\hat{A}_e^+\) may be encoded in a transformation
groupoid~\(G\ltimes \hat{A}_e^+\), which has object
space~\(\hat{A}_e^+\), arrow space \(\bigsqcup_{g\in G}
\mathcal{D}_{g^{-1}}\) with the disjoint union topology, range and
source maps \(s(g,\chi) \defeq \chi\), \(r(g,\chi) \defeq
\vartheta_g(\chi)\) for \(g\in G\), \(\chi\in \mathcal{D}_{g^{-1}}\), and
multiplication \((g,\vartheta_h(\chi))\cdot (h,\chi) \defeq (g\cdot
h,\chi)\) for all \(g,h\in G\), \(\chi\in \mathcal{D}_{h^{-1}} \cap
\vartheta_h^{-1}(\mathcal{D}_{g^{-1}})\).  The unit arrow on~\(\chi\)
is~\((1,\chi)\), and the inverse of~\((g,\chi)\) is
\((g^{-1},\vartheta_g(\chi))\).  This is an étale topological groupoid
because \(r\) and~\(s\) restrict to homeomorphisms on the open
subsets~\(\mathcal{D}_{g^{-1}}\) of the arrow space.  The object
space~\(\hat{A}_e^+\) need not be locally compact.

We are going to construct another topological groupoid~\(\Sigma\)
that is a central extension of \(G\ltimes \hat{A}_e^+\) by the
circle group~\(\T\).  That is, \(\Sigma\) comes with a canonical
functor to \(G\ltimes \hat{A}_e^+\) whose kernel is the group bundle
\(\hat{A}_e^+ \times \T\).  Such an extension is also called a
\emph{twisted groupoid} in
\cite{Renault:Cartan.Subalgebras}*{Section 4}, following a
definition by Kumjian~\cite{Kumjian:Diagonals}.  A twisted groupoid
with locally compact object space has a twisted groupoid
\Cstar\nb-algebra.  For a suitable injective continuous map \(X\to
\hat{A}_e^+\), we are going to identify the \Cstar\nb-hull of the
\(X\)\nb-integrable representations of~\(A\) with the twisted
groupoid \Cstar\nb-algebra of the restriction of~\(\Sigma\) to
\(j(X^+)\subseteq \hat{A}_e^+\).

A point in~\(\Sigma\) is a triple~\((g,\chi,[a])\), where \(g\in
G\), \(\chi\in \mathcal{D}_{g^{-1}}\), and~\([a]\) is a unit vector
in the \(1\)\nb-dimensional Hilbert space~\(A_g \otimes_{A_e,\chi}
\C\).  We represent unit vectors in \(A_g \otimes_{A_e,\chi} \C\) by
elements \(a\in A_g\) with \(\chi(a^* a)=1\); two elements
\(a,b\in A_g\) with \(\chi(a^* a) = \chi(b^* b)=1\) represent the
same unit vector \([a]=[b]\) if and only if \(\chi(a^* b)=1\).  We
get the same set of equivalence classes if we allow \(a\in A\) with
\(\chi(a^* a)>0\) and set \([a]=[b]\) if \(\chi(a^* b)>0\): then
\(a_1\defeq \chi(a^* a)^{-\nicefrac12} a\) and \(b_1\defeq \chi(b^*
b)^{-\nicefrac12} b\) satisfy \([a]=[a_1]\), \([b]=[b_1]\), and
\([a]=[b]\) if and only if \(\chi(a_1^* b_1)=1\)
by~\eqref{eq:chi_aabb}.  The circle group~\(\T\) acts on~\(\Sigma\)
by multiplication: \(\lambda\cdot (g,\chi,[a]) \defeq
(g,\chi,[\lambda a])\).  The orbit space projection for this circle
action is the coordinate projection \(F\colon \Sigma \prto G\ltimes
\hat{A}_e^+\), \((g,\chi,[a])\mapsto (g,\chi)\).  Next we
equip~\(\Sigma\) with a topology so that this coordinate projection
is a locally trivial principal \(\T\)\nb-bundle.

For \(a\in A\), let \(U_a \defeq \{\chi\in \hat{A}_e^+\mid \chi(a^*
a)\neq0\}\).  This is an open subset in~\(\hat{A}_e^+\), and
\(\chi(a^* a)>0\) if \(\chi\in U_a\) because~\(\chi\) is positive.
The map \(\sigma_a\colon \{g\}\times U_a \to \Sigma\),
\((g,\chi)\mapsto (g,\chi,[a])\), for \(a\in A_g\) is a local
section for the coordinate projection~\(F\).  If \(a,b\in A_g\), and
\(\chi\in U_a\cap U_b\), then
\[
[a] = \left[\frac{\chi(b^* a)} {\chi(a^* a)^{\nicefrac12} \chi(b^*
    b)^{\nicefrac12}} b\right],
\]
by~\eqref{eq:compare_Ag_commutative}.
Since the functions sending~\(\chi\) to \(\chi(b^* a)\), \(\chi(a^*
a)\) and \(\chi(b^* b)\) are continuous on~\(A_e\), the two
trivialisations induce the same topology on the restriction
of~\(\Sigma\) to \(\{g\}\times (U_a\cap U_b)\).  For any \(\chi\in
\mathcal{D}_{g^{-1}}\), there is \(a\in A_g\) with \(\chi(a^* a)>0\).
Thus the open subsets~\(U_a\) cover~\(\mathcal{D}_{g^{-1}}\).
Consequently, there is a unique topology on~\(\Sigma\) that makes the
local sections~\(\sigma_a\) for all \(a\in A_g\) continuous, and this
topology turns~\(\Sigma\) into a locally trivial \(\T\)\nb-bundle
over~\(G\ltimes \hat{A}_e^+\).

We define a groupoid with object space~\(A_e^+\), arrow
space~\(\Sigma\), and
\[
r(g,\chi,[a])\defeq \vartheta_g(\chi),\quad
s(g,\chi,[a])\defeq \chi,\quad
(g,[\vartheta_h(\chi)],[a])\cdot (h,\chi,[b])
\defeq (g\cdot h,\chi, [a\cdot b]);
\]
we must show that this multiplication is well defined.  We have \(a b
\in A_{g h}\) and
\[
\chi(b^* a^* a b) = \vartheta_h(\chi)(a^* a)\cdot \chi(b^* b) \neq0
\]
by~\eqref{eq:partial_act_Aeplus}, so \((g\cdot h,\chi, [a\cdot b]) \in
\Sigma\).  If \(\chi(b^* b_1)>0\) and \(\vartheta_h(\chi)(a^* a_1)>0\),
then \(\chi(b^* a^* a_1 b_1)>0\) by computations as in the proof of
Theorem~\ref{the:induce_character}.  Hence the multiplication is well
defined.  It is clearly associative.  The unit arrow on~\(\chi\) is
\(1_\chi \defeq (1,\chi,[1])\), and \((g,\chi,[a])^{-1} =
(g^{-1},\vartheta_g(\chi),[a^*])\).  The multiplication, unit map and
inversion are continuous and the range and source maps are open
surjections (even locally trivial).  So~\(\Sigma\) is a topological
groupoid.

The identity map on objects and the coordinate projection \(F\colon
\Sigma \to G\ltimes \hat{A}_e^+\) on arrows form a functor, which is
a locally trivial, open surjection on arrows.  The kernel of~\(F\)
consists of those \((g,\chi,[a])\in\Sigma\) for which
\(F(g,\chi,[a])\) is a unit arrow in~\(G\ltimes \hat{A}_e^+\).  Then
\(g=1\), and \(a\in A_e\) is equivalent to \([a] = [\chi(a)\cdot 1]\)
because \(\chi(a^* \chi(a) 1)>0\).  The map \((g,\chi,[a])\mapsto
(\chi,\chi(a))\) is an isomorphism of topological groupoids from the
kernel of~\(F\) onto the trivial group bundle
\(\hat{A}_e^+\times\T\).  Thus we have an extension of
topological groupoids
\[
\hat{A}_e^+\times\T \into \Sigma \prto G\ltimes \hat{A}_e^+.
\]
The three groupoids above are clearly Hausdorff.

To construct \Cstar\nb-algebras, we need groupoids with a locally
compact object space.  Therefore, we replace~\(\hat{A}_e\) by a
locally compact space~\(X\) with an injective, continuous map
\(j\colon X\to\hat{A}_e\).  Then \(B_e = \Cont_0(X)\) is a
\Cstar\nb-hull for a class \(\Repi(A_e,X)\) of representations
of~\(A_e\).  By Lemma~\ref{lem:inducible_Ae_commutative}, the
\Cstar\nb-hull for the class of \(X\)\nb-integrable, inducible
representations of~\(A_e\) is \(B_e^+ = \Cont(X^+)\) with \(X^+ \defeq
j^{-1}(\hat{A}_e^+)\subseteq X\).

\begin{proposition}
  \label{pro:commutative_integrability_induction}
  Let \(j\colon X\to\hat{A}_e\) be an injective, continuous map.
  The class \(\Repi(A,X)\) is compatible with induction if and only
  if \(j(X^+)\subseteq \hat{A}_e^+\) is invariant under the partial
  maps~\(\vartheta_g\) in Theorem~\textup{\ref{the:induce_character}}
  and the resulting partial maps on~\(X^+\) are continuous in the
  topology of~\(X^+\).  We briefly say that the partial action
  of~\(G\) on~\(\hat{A}_e^+\) \emph{restricts} to~\(X^+\).
\end{proposition}

\begin{proof}
  By Proposition~\ref{pro:inducible_universal}, it suffices to check
  that the induced representation of~\(A_e\) on~\(A_g\otimes_{A_e}
  B_e^+\) is \(X\)\nb-integrable for \(g\in G\) if and only if the
  partial map \(\vartheta_g\circ j\) on~\(X\) factors through~\(j\) and
  the resulting partial map \(j^{-1}\circ \vartheta_g \circ j\) on~\(X\)
  is again continuous.
  View the Hilbert module \(A_g\otimes_{A_e} B_e^+\) as a continuous
  field of Hilbert spaces over~\(X^+\).  The fibres of this field have
  dimension at most~\(1\) by Theorem~\ref{the:induce_character},
  and the set where the fibre is non-zero is the open subset
  \(j^{-1}(\mathcal{D}_{g^{-1}})\).  Hence \(\Comp(A_g\otimes_{A_e}
  B_e^+) \cong \Cont_0(j^{-1}(\mathcal{D}_{g^{-1}}))\).  The
  representation of~\(A_e\) on \(A_g\otimes_{A_e} B_e^+\) is
  equivalent to a representation on \(\Comp(A_g\otimes_{A_e} B_e^+)\)
  by Proposition~\ref{pro:representation_Hilm_Comp}.  This is
  equivalent to a continuous map \(j^{-1}(\mathcal{D}_{g^{-1}}) \to
  \hat{A}_e\) by Proposition~\ref{pro:commutative_representation}.
  This map is~\(\vartheta_g \circ j\) by a fibrewise computation.  Hence
  the induced representation of~\(A_e\) on~\(A_g \otimes_{A_e} B_e^+\)
  is \(X\)\nb-integrable if and only if \(\vartheta_g \circ j\) has
  values in \(j(X)\) and the partial maps \(j^{-1}\circ \vartheta_g\circ
  j\) on~\(X\) are continuous.
\end{proof}

From now on, we assume that the partial action of~\(G\)
on~\(\hat{A}_e^+\) restricts to~\(X^+\).  By
Proposition~\ref{pro:commutative_integrability_induction}, this
assumption is necessary and sufficient for \(X\)\nb-integrability to
be compatible with induction.  The ``restriction'' of the partial
action on~\(\hat{A}_e^+\) to~\(X^+\) is a partial action of~\(G\)
on~\(X^+\) by partial homeomorphisms.  Its transformation groupoid
\(G\ltimes X^+\) is constructed like \(G\ltimes \hat{A}_e^+\).  Its
set of arrows is the subset of \(G\ltimes \hat{A}_e^+\) of arrows with
range and/or source in~\(j(X^+)\), and the topology on the arrow space
is the unique one that makes the inclusion \(G\ltimes X^+ \to G\ltimes
\hat{A}_e^+\) and the range and source maps \(G\ltimes X^+ \to X^+\)
continuous.  There is also a unique topology on the
restriction~\(\Sigma_X\) of~\(\Sigma\) to~\(j(X^+)\) so that there is
an extension of topological groupoids
\[
X^+\times\T \into \Sigma_X \prto G\ltimes X^+.
\]
Since~\(X^+\) is locally compact, the groupoids in this extension
are locally compact, Hausdorff groupoids.  Since \(G\ltimes X^+\) is
étale, it carries a canonical Haar system, namely, the family of
counting measures.  There is also a unique normalised Haar system
on~\(X^+\times\T\).  These produce a unique Haar system
on~\(\Sigma_X\) by \cite{Buss-Meyer:Groupoid_fibrations}*{Theorem 5.1}, so
that the groupoid \Cstar\nb-algebra \(\Cst(\Sigma_X)\) is
defined.  The \emph{twisted groupoid \Cstar\nb-algebra}
\(\Cst(G\ltimes X^+,\Sigma_X)\) of~\(G\ltimes X^+\) with respect to
the \emph{twist}~\(\Sigma_X\) is defined
in~\cite{Renault:Representations}.  It is related to the groupoid
\Cstar\nb-algebra of~\(\Sigma_X\) in
\cite{Buss-Meyer:Groupoid_fibrations}*{Corollary 7.2}.

\begin{theorem}
  \label{the:twisted_groupoid}
  Let~\(G\) be a discrete group and let~\(A\) be a \(G\)\nb-graded
  \Star{}algebra with commutative~\(A_e\).  Let \(j\colon X\to
  \hat{A}_e\) be an injective, continuous map, such that the partial
  action of~\(G\) on~\(\hat{A}_e^+\) in
  Theorem~\textup{\ref{the:induce_character}} restricts to~\(X^+\)
  as in
  Proposition~\textup{\ref{pro:commutative_integrability_induction}}.
  Then \(\Cst(G\ltimes X^+,\Sigma_X)\) is a \Cstar\nb-hull for
  \(\Repi(A,X)\).
\end{theorem}

\begin{proof}
  The \Cstar\nb-algebra \(\Cst(G\ltimes X^+,\Sigma_X)\) may be defined
  as the full section \Cstar\nb-algebra of a certain Fell line bundle
  over the étale, locally compact groupoid~\(G\ltimes X^+\).  The Fell
  line bundle involves the space of sections of the Hermitian complex
  line bundle \(\mathcal{L} \defeq \Sigma_X \times_{\T} \C\)
  associated to the principal \(\T\)\nb-bundle \(\Sigma_X\prto
  G\ltimes X^+\) and the multiplication maps \(\mathcal{L}_g \times
  \mathcal{L}_h \to \mathcal{L}_{g h}\) induced by the multiplication
  of~\(\Sigma_X\) (see~\cite{Buss-Meyer:Groupoid_fibrations}).  By
  construction, the Hilbert \(B_e^+\)-module \(B_g^+ = A_g
  \otimes_{A_e} B_e^+\) is isomorphic to the continuous sections of
  this line bundle~\(\mathcal{L}\) over the subset \(\{g\}\times
  \mathcal{D}_{g^{-1}}^X\) of \(\{g\}\times X^+\): an element
  \(a\otimes b\) is mapped to the continuous section that
  sends~\((g,x)\) for \(x\in X\) with \(j(x)\in \mathcal{D}_{g^{-1}}\)
  to \(b(x)\cdot \chi(a^* a)^{\nicefrac12} [a]\).  The multiplication
  in~\(\Sigma_X\) is defined so that the multiplication maps \(B_g^+
  \otimes_{B_e^+} B_h^+ \to B_{g h}^+\) are exactly the
  multiplication maps in the Fell line bundle associated
  to~\(\Sigma_X\).

  Thus the Fell bundle~\((B_g^+)_{g\in G}\) constructed in
  Theorem~\ref{the:Fell_bundle_sections_hull} is isomorphic to the
  Fell bundle~\((\beta_g)_{g\in G}\), where~\(\beta_g\) is the space
  of \(\Cont_0\)\nb-sections of~\(\mathcal{L}\) over \(\{g\}\times
  \mathcal{D}_{g^{-1}}\) and the multiplication and involution come
  from the Fell line bundle structure on~\(\mathcal{L}\) over the
  groupoid~\(G\ltimes X^+\).  The full section \Cstar\nb-algebra of
  this Fell bundle is canonically isomorphic to the section
  \Cstar\nb-algebra of the corresponding Fell bundle over the
  groupoid \(G\ltimes X^+\) by results
  of~\cite{BussExel:Fell.Bundle.and.Twisted.Groupoids}.  The small
  issue to check here is that it makes no difference whether we use
  \(\Cont_0\)\nb-sections or compactly supported continuous sections
  of~\(\mathcal{L}\) over \(\{g\}\times \mathcal{D}_{g^{-1}}\).
  Both have the same \Cstar\nb-completion.  This is a special case
  of general results about Fell bundles over étale locally compact
  groupoids.
\end{proof}

If~\((B_g)_{g\in G}\) is any Fell bundle over~\(G\), then
\(\bigoplus_{g\in G} B_g\) is a \Star{}algebra, to which we may apply
our machinery although all its representations are bounded.  Thus any
Fell bundle over~\(G\) may come up for some choice of the
\(G\)\nb-graded \Star{}algebra~\(A\).  Thus the section
\Cstar\nb-algebra of a Fell bundle~\((B_g)_{g\in G}\)
with commutative unit fibre is always a twisted groupoid
\Cstar\nb-algebras of a twist of an étale groupoid, namely, the
transformation groupoid of a certain partial action on the spectrum of
the unit fibre associated to the Fell bundle.  This result is already
known, even for Fell bundles over inverse semigroups with commutative
unit fibre, see~\cite{BussExel:Fell.Bundle.and.Twisted.Groupoids}.

If \(\Sigma_X \cong (G\ltimes X^+)\times \T\) as a groupoid,
then \(\Cst(G\ltimes X^+,\Sigma_X) \cong \Cst(G\ltimes X^+)\).  This
is the same as the crossed product for the partial action of~\(G\)
on~\(X^+\).  This happens in all the examples in
\cites{Savchuk-Schmudgen:Unbounded_induced, Dowerk-Savchuk:Induced}.
The possible twists have two levels.  First, \(\Sigma_X\) may be
non-trivial as a principal circle bundle over~\(G\ltimes X^+\).
Secondly, if it is trivial as a principal circle
bundle, the multiplication may create a non-trivial twist.

The circle bundle \(\Sigma_X \prto G\ltimes X^+\) is trivial if and
only if its restriction to \(\{g\}\times \mathcal{D}_{g^{-1}}\) is
trivial for each \(g\in G\).  For a circle bundle, this means that
there is a nowhere vanishing section.  For instance, if there is
\(a\in A_g\) that generates~\(A_g\) as a right \(A_e\)\nb-module,
then \(U_a = \mathcal{D}_{g^{-1}}\) and~\(\sigma_a\) is a global
trivialisation of~\(\Sigma_X|_{\{g\}\times\mathcal{D}_{g^{-1}}}\).

The complex line bundles over a space~\(X\) are classified by the
second cohomology group \(H^2(X,\Z)\).  If~\(\mathcal{L}\) is a line
bundle, then the spaces of \(\Cont_0\)\nb-sections of
\(\mathcal{L}^{\otimes n}\) for \(n\in\Z\) form a Fell bundle
over~\(\Z\), and the direct sum of these spaces of sections is a
\(\Z\)\nb-graded \Star{}algebra such that the given line
bundle~\(\mathcal{L}\) appears in the resulting twisted groupoid.
If \(H^2(X,\Z)\neq0\), the space~\(X\) is at least
\(2\)\nb-dimensional.  There are indeed non-trivial complex line
bundles over all \emph{compact} oriented \(2\)\nb-dimensional
manifolds.  The resulting \Star{}algebra, however, has only
\Star{}representations by bounded operators if~\(X\) is compact.
Examples where unbounded operators appear must involve a non-trivial
line bundle over a noncompact space.  These first appear in
dimension~\(3\).  It is easy to write down a \(\Z\)\nb-graded
\Star{}algebra~\(A\) where~\(B_e^+\) is, say, \(S^2\times \R\)
and~\(B_g^+\) involves the Bott line bundle over~\(S^2\).  These
examples seem artificial, however.

Now assume that~\(\Sigma_X\) is trivial as a principal circle bundle
over~\((G\ltimes X^+)^1\), that is, \(\Sigma_X \cong (G\ltimes
X^+)^1\times\T\) as a \(\T\)\nb-space.  We may choose this
homeomorphism to be the obvious one on the open subset \((1\ltimes
X^+)\times \T\) corresponding to \(1\in G\).  The multiplication
must be of the form
\[
(g_1,\vartheta_{g_2}(x),\lambda_1)\cdot (g_2,x,\lambda_2)
= (g_1\cdot g_2, x,\varphi(g_1,g_2,x)\cdot
\lambda_1\cdot\lambda_2)
\]
for some continuous \(\T\)\nb-valued function~\(\varphi\) with
\(\varphi(1,g,x) = 1 =\varphi(g,1,x)\) for all \(g,x\);
here~\(\varphi\) is defined on the space of all triples
\((g_1,g_2,x_2)\in G\times G\times X^+\) with
\(x_2\in\mathcal{D}_{g_2^{-1}}\) and
\(\vartheta_{g_2}(x_2)\in\mathcal{D}_{g_1^{-1}}\); this space is
homeomorphic to the space~\((G\ltimes X^+)^2\) of pairs of
composable arrows in~\(G\ltimes X^+\).  The associativity of the
multiplication in~\(\Sigma_X\) is equivalent to the cocycle condition
\begin{equation}
  \label{eq:2-cocycle_varphi}
  \varphi(g_1,g_2\cdot g_3,x)\cdot \varphi(g_2,g_3,x)
  = \varphi(g_1\cdot g_2, g_3,x)\cdot \varphi(g_1,g_2,\vartheta_{g_3}(x))
\end{equation}
for all \(g_1,g_2,g_3\in G\), \(x\in X^+\) for which
\(\vartheta_{g_3}(x)\), \(\vartheta_{g_2}\circ\vartheta_{g_3}(x)\), and
\(\vartheta_{g_1}\circ \vartheta_{g_2}\circ\vartheta_{g_3}(x)\) are defined.  A
different trivialisation of the circle bundle~\(\Sigma_X\prto
(G\ltimes X^+)^1\) modifies~\(\varphi\) by the coboundary
\begin{equation}
  \label{eq:2-coboundary_varphi}
  \partial \psi(g_1,g_2,x) \defeq
  \psi(g_2,x) \psi(g_1\cdot g_2, x)^{-1} \psi(g_1,\vartheta_{g_2}(x))
\end{equation}
of a continuous function \(\psi\colon (G\ltimes X^+)^1 \to \T\)
normalised by \(\psi(1,x)=0\) for all \(x\in X^+\).
Thus isomorphism classes of twists of~\(G\ltimes X^+\) are in
bijection with the groupoid cohomology \(H^2(G\ltimes X,\T)\), that
is, the quotient of the group of continuous maps \(\varphi\colon
(G\ltimes X^+)^2 \to \T\) satisfying~\eqref{eq:2-cocycle_varphi} by
the group of \(2\)\nb-coboundaries \(\partial \psi\) of continuous
\(1\)\nb-cochains \(\psi\colon (G\ltimes X^+)^1 \to \T\),
where~\(\partial\psi\) is defined in~\eqref{eq:2-coboundary_varphi}.

In the easiest case, the function~\(\varphi\) above does not depend
on~\(x\).  Then \(\varphi\colon G\times G\to \T\) is a normalised
\(2\)\nb-cocycle on~\(G\) in the usual sense.  These cocycles appear,
for instance, in the classification of projective representations of
the group~\(G\).  This is related to the twists above because the
Hilbert space representations of the twisted group algebra for a
\(2\)\nb-cocycle \(\varphi\colon G\times G\to\T\) are exactly the
projective representations \(\pi\colon G\to \textup{U}(\Hilm[H])\) with
\(\pi(g)\pi(h) = \varphi(g,h) \pi(g h)\) for all \(g,h\in G\).

The group~\(\Z\) has no nontrivial \(2\)\nb-cocycles.  They do
appear, however, for the group~\(\Z^2\).  A well known example is
the noncommutative torus.  Its usual gauge action corresponds to a
\(\Z^2\)\nb-grading, where \(U^n V^m\) for the canonical generators
\(U,V\) has degree \((n,m)\in\Z^2\).  In this case, \(A_e=\C = B_e =
B_e^+\), and~\(\hat{A}_e^+\) has only one point.  The transformation
groupoid \(G\ltimes \hat{A}_e^+\) is simply \(G=\Z^2\).  This is
discrete, so~\(\Sigma\) is always trivial as a principal circle
bundle.  Thus the only non-trivial aspect of~\(\Sigma\) is a
\(2\)\nb-cocycle \(\varphi\colon \Z\times \Z\to \T\).  The
cohomology group \(H^2(\Z^2,\T)\) is isomorphic to~\(\T\),
and the resulting twisted group algebras of~\(\Z^2\) are exactly the
noncommutative tori.

\begin{proposition}
  \label{pro:commutative_partial_action_no_twist}
  If there are subsets \(S_g \subseteq A_g\) such that~\(S_g\)
  generates~\(A_g\) as a right \(A_e\)\nb-module, \(S_g\cdot S_h
  \subseteq S_{g h}\), and \(\chi(a^* b)\ge0\) for all \(a,b\in
  S_g\), \(g\in G\), \(\chi\in j(X)\subseteq \hat{A}_e\), then the
  twist~\(\Sigma_X\) is trivial and so the \Cstar\nb-hull of~\(A\)
  is \(\Cst(G\ltimes X^+)\).
\end{proposition}

\begin{proof}
  If \(\chi\in \mathcal{D}_{g^{-1}}\), then there is \(b\in A_g\)
  with \(\chi(b^* b)\neq0\).  Since~\(S_g\) generates~\(A_g\) as a
  right \(A_e\)\nb-module, we may write \(b = \sum_{i=1}^n a_i\cdot
  c_i\) with \(a_i\in S_g\), \(c_i\in A_e\).  Then
  \[
  \chi(b^* b)
  = \sum_{i,j=1}^n \chi(c_i^*)\chi(c_j) \chi(a_i^* a_j).
  \]
  Hence there are \(i,j\) with \(\chi(a_i^* a_j)\neq0\).  Then
  \(\chi(a_i^* a_i)\neq0\) by~\eqref{eq:chi_aabb}.  This shows that
  \(\bigcup_{a\in S_g} U_a = \mathcal{D}_{g^{-1}}\).  We have
  \((g,\chi,[a])= (g,\chi,[b])\) for all \(a,b\in S_g\), \(\chi \in
  j(X)\cap U_a\cap U_b\) because \(\chi(a^* b)\ge0\) for all
  \(\chi\in j(X)\).  Hence the local sections~\(\sigma_a\)
  of~\(\Sigma_X|_{\{g\}\times\mathcal{D}_{g^{-1}}}\) for \(a\in
  S_g\) coincide on the intersections of their domains and thus
  combine to a global trivialisation.  This trivialisation is
  multiplicative as well.
\end{proof}

If~\(\hat{A}_e^+\) itself is locally compact, then we may take
\(X^+=X=\hat{A}_e^+\) with the inclusion map~\(j\).  Since
\(\hat{A}_e^+\) is closed in~\(\hat{A}_e\), this happens
if~\(\hat{A}_e\) is locally compact.

\begin{theorem}
  \label{the:Ae_locally_compact_hull}
  Assume that \(\hat{A}_e^+\) is locally compact in the
  topology~\(\tau_c\).  Call a representation of~\(A\) integrable if
  its restriction to~\(A_e\) is locally bounded.  Let~\(\Sigma\) be
  the twisted groupoid constructed above.  Then \(\Cst(G\ltimes
  (\hat{A}_e^+,\tau_c),\Sigma)\) is a \Cstar\nb-hull for the
  integrable representations of~\(A\).
\end{theorem}

\begin{proof}
  If \(X^+=(\hat{A}_e^+,\tau_c)\), then integrability is compatible
  with induction by
  Proposition~\ref{pro:commutative_integrability_induction} because
  the construction of the topology~\(\tau_c\) is natural and
  compatible with restriction to open subsets.
  Theorem~\ref{the:twisted_groupoid} shows that \(\Cst(G\ltimes
  (\hat{A}_e^+,\tau_c),\Sigma)\) is a \Cstar\nb-hull for the class
  of representations of~\(A\) whose restriction to~\(A_e\) is
  \(\hat{A}_e^+\)\nb-integrable.  The locally bounded
  representations of~\(A_e\) are equivalent to the locally bounded
  representations of the
  pro-\Cstar\nb-algebra~\(\Cont(\hat{A}_e,\tau_c)\) by Propositions
  \ref{pro:commutative_pro-completion}
  and~\ref{pro:locally_bounded_pro-Cstar}.  Restrictions of
  representations of~\(A\) to~\(A_e\) are automatically inducible by
  Lemma~\ref{lem:induced_repr_inducible}.  By a
  pro-\Cstar\nb-algebraic variant of
  Lemma~\ref{lem:inducible_Ae_commutative}, the inducible, locally
  bounded representations of~\(A_e\) are equivalent to those
  representations of~\(\Cont(\hat{A}_e,\tau_c)\) that factor through
  the quotient \(\Cont(\hat{A}_e^+,\tau_c)\).  Since~\(\hat{A}_e^+\)
  is locally compact, \(\Cont_0(\hat{A}_e^+,\tau_c)\) is dense in
  the pro-\Cstar\nb-algebra \(\Cont(\hat{A}_e^+,\tau_c)\).  Hence
  \(\Cont_0(\hat{A}_e^+,\tau_c)\) is a \Cstar\nb-hull for the
  inducible, locally bounded representations of~\(A_e\) by
  Theorem~\ref{the:locally_bounded_relative_hull}.
\end{proof}

Assume that~\(A_e\) is countably generated.  Then the usual topology
on~\(\hat{A}_e\) is metrisable and hence compactly generated, so
that~\(\tau_c\) is the standard topology on~\(\hat{A}_e\).  A
representation of~\(A_e\) is locally bounded if and only if all
symmetric elements of~\(A_e\) act by regular, self-adjoint operators
by Theorem~\ref{the:locally_bounded_versus_regular}.  Thus a
representation~\(\pi\) of~\(A\) is integrable as in
Theorem~\ref{the:Ae_locally_compact_hull} if and only if
\(\cl{\pi(a)}\) is regular and self-adjoint for all \(a\in A_e\)
with \(a=a^*\).  This class of integrable representations has the
\Cstar\nb-hull \(\Cst(G\ltimes \hat{A}_e^+,\Sigma)\)
if~\(\hat{A}_e^+\) is locally compact.

In particular, if~\(A_e\) is finitely generated, then~\(\hat{A}_e\)
is mapped homeomorphically onto a closed subset of~\(\R^n\) for some
\(n\in\N\) by evaluating characters on a finite set of symmetric
generators.  Thus~\(\hat{A}_e\) is locally compact.  The discussion
above gives:

\begin{corollary}
  \label{cor:Ae_fin_gen_hull}
  Assume that~\(A_e\) is finitely generated.  Call a representation
  of~\(A\) integrable if its restriction to~\(A_e\) is locally
  bounded.  Then~\(\hat{A}_e^+\) is locally compact and
  \(\Cst(G\ltimes \hat{A}_e^+,\Sigma)\) for the twisted
  groupoid~\(\Sigma\) constructed above is a \Cstar\nb-hull for the
  integrable representations of~\(A\).  Moreover, a
  representation~\(\pi\) of~\(A\) is integrable if and only if
  \(\cl{\pi(a)}\) is regular and self-adjoint for all \(a\in A_e\)
  with \(a=a^*\).
\end{corollary}

Corollary~\ref{cor:Ae_fin_gen_hull} covers all the examples
considered in \cites{Savchuk-Schmudgen:Unbounded_induced,
  Dowerk-Savchuk:Induced}, except for the enveloping
algebra~\(\mathcal{W}\) of the Virasoro algebra that is studied in
\cite{Savchuk-Schmudgen:Unbounded_induced}*{§9.3}.

The \Star{}algebra~\(\mathcal{W}\) is \(\Z\)\nb-graded.  Its unit
fibre~\(\mathcal{W}_0\) is noncommutative.  The first step in the
study of its representations
in~\cite{Savchuk-Schmudgen:Unbounded_induced}*{§9.3} is to
replace~\(\mathcal{W}\) by a certain \(\Z\)\nb-graded quotient
\(A\defeq \mathcal{W}/\mathcal{I}\), whose unit fibre \(A_0 =
\mathcal{W}_0/(\mathcal{I}\cap \mathcal{W}_0)\) is commutative by
construction.  The motivation is that all ``integrable''
representations of~\(\mathcal{W}\) factor through~\(A\).  The main
result in \cite{Savchuk-Schmudgen:Unbounded_induced}*{§9.3} shows
that the partial action of~\(\Z\) on~\(\hat{A}_e^+\) is free and
that the disjoint union \(Y \defeq X_1 \sqcup X_2 \sqcup X_3\) of
the three families of characters described in (61)--(63)
of~\cite{Savchuk-Schmudgen:Unbounded_induced} is a fundamental
domain, that is, it meets each orbit of the partial action exactly
once.  Each subset~\(X_i\) is closed in~\(\hat{A}_e\) and locally
compact and second countable in the subspace topology.  Hence so
is~\(Y\).  Since~\(\Z\) acts by partial homeomorphisms and~\(Y\) is
a fundamental domain, there is a continuous bijection
\[
X\defeq \bigsqcup_{n\in\Z} (\mathcal{D}_{-n}\cap Y) \to
\hat{A}_e^+,\qquad (n,y)\mapsto \vartheta_n(y).
\]
Each~\(\mathcal{D}_{-n}\cap Y\) is an open subset of~\(Y\), so
that~\(X\) is locally compact.  I have not checked whether this
continuous bijection is a homeomorphism.  If so,
then~\(\hat{A}_e^+\) would be locally compact and the results
in~\cite{Savchuk-Schmudgen:Unbounded_induced} for the Virasoro
algebra would be contained in
Theorem~\ref{the:Ae_locally_compact_hull} after passing to the
quotient~\(\mathcal{W}/\mathcal{I}\).  If not, we would use the
locally compact space~\(X\).  The partial action of~\(\Z\)
on~\(\hat{A}_e^+\) is clearly continuous on~\(X\) as well, so that
Theorem~\ref{the:twisted_groupoid} applies.

\section{Rieffel deformation}
\label{sec:Rieffel}

Let~\(G\) be a discrete group.
Given a normalised \(2\)\nb-cocycle on~\(G\), Rieffel deformation is a
deformation functor that modifies the multiplication on a
\(G\)\nb-graded \Star{}algebra by the \(2\)\nb-cocycle.  There is a
similar process for Fell bundles over~\(G\), which we may transfer to
section \Cstar\nb-algebras.  This is
how Rieffel deformation is usually considered.  The setting of
graded algebras or Fell bundles is easier.  We now define
Rieffel deformation more precisely and show that it is compatible with
the construction of \Cstar\nb-hulls in
Theorem~\ref{the:Fell_bundle_sections_hull}.  This deformation
process has also recently been treated
in~\cite{Raeburn:Deformations}.

A \emph{normalised \(2\)\nb-cocycle} on a group~\(G\) is a function
\(\Lambda\colon G\times G\to\textup{U}(1)\) with
\(\Lambda(e,g)=1=\Lambda(g,e)\) for all \(g\in G\) and
\begin{equation}
  \label{eq:2-cocycle}
  \Lambda(g,h\cdot k)\Lambda(h,k) = \Lambda(g\cdot h,k)\Lambda(g,h)
\end{equation}
for all \(g,h,k\in G\).  Let \(A= \bigoplus_{g\in G} A_g\) be a
\(G\)\nb-graded algebra.  Let~\(A^\Lambda\) be the same
\(G\)\nb-graded vector space with the deformed multiplication and
involution
\[
\sum_{g\in G} a_g * \sum_{h\in G} b_h
\defeq \sum_{g,h\in G} \Lambda(g,h) a_g b_h,\qquad
\bigl(\sum a_g\bigr)^\dagger \defeq
 \sum \conj{\Lambda(g^{-1},g)} a_g^*,
\]
where \(a_g,b_g\in A_g\) for all \(g\in G\).  We call~\(A^\Lambda\)
the \emph{Rieffel deformation} of~\(A\) with respect to~\(\Lambda\).

\begin{lemma}
  \label{lem:Rieffel_def_works}
  The deformed multiplication and involution on~\(A^\Lambda\) give a
  \(G\)\nb-graded \Star{}algebra with \(a * b = a b\) if
  \(a\in A_e\) or \(b\in A_e\), and \(a^\dagger *b = a^* b\) for all
  \(g\in G\), \(a,b\in A_g\).
\end{lemma}

\begin{proof}
  The multiplication remains associative by the \(2\)\nb-cocycle
  condition~\eqref{eq:2-cocycle}.  The normalisation of~\(\Lambda\)
  and~\eqref{eq:2-cocycle} for \(g,g^{-1},g\) give \(\Lambda(g,g^{-1})
  = \Lambda(g^{-1},g)\) for all \(g\in G\).  Thus
  \[
  (a_g^\dagger)^\dagger
  = \conj{\Lambda(g,g^{-1})} \cdot
  \conj{\conj{\Lambda(g^{-1},g)}} (a_g^*)^*
  = a_g
  \]
  for \(a_g\in A_g\).  The normalisation condition
  and~\eqref{eq:2-cocycle} for \(g,h,h^{-1}\) and \(g
  h,h^{-1},g^{-1}\) for \(g,h\in G\) give
  \begin{align*}
    \Lambda(g h, h^{-1}) \Lambda(g,h) &= \Lambda(h,h^{-1}),\\
    \Lambda(g, g^{-1}) \Lambda(g h,h^{-1})
    &= \Lambda(g h,h^{-1} g^{-1}) \Lambda(h^{-1}, g^{-1}).
  \end{align*}
  Hence \(\Lambda(g,g^{-1})\Lambda(h,h^{-1}) = \Lambda(g,h) \Lambda(g
  h, h^{-1} g^{-1}) \Lambda(h^{-1},g^{-1})\).  This implies the
  condition \((a_g * b_h)^\dagger = b_h^\dagger * a_g^\dagger\) for
  \(a_g\in A_g\), \(b_h\in A_g\):
  \begin{align*}
    (a_g * b_h)^\dagger
    &= \conj{\Lambda(g,h)} \cdot \conj{\Lambda(g h, (g h)^{-1})}
    \cdot (a_g b_h)^*
    \\&= \conj{\Lambda(g,g^{-1})\cdot \Lambda(h,h^{-1})}\cdot
    \Lambda(h^{-1},g^{-1})\cdot b_h^* a_g^*
    = b_h^\dagger * a_g^\dagger.
  \end{align*}
  Thus the deformed multiplication and involution give a
  \Star{}algebra.  The formula \(a^\dagger *b = a^* b\) for \(g\in
  G\), \(a,b\in A_g\) is trivial, and \(a * b = a b\) if \(a\in A_e\)
  or \(b\in A_e\) follows from the normalisation of~\(\Lambda\).
\end{proof}

The same formulas work if~\((B_g)_{g\in G}\) is a Fell bundle
over~\(G\).  Let~\((B^\Lambda_g)_{g\in G}\) be the same Banach space
bundle as~\(B_g\) with the multiplication and involution \(a_g * b_h
\defeq \Lambda(g,h) a_g b_h\) and \(a_g^\dagger =
\conj{\Lambda(g^{-1},g)} a_g^*\) for \(g,h\in G\), \(a_g\in B_g\),
\(b_h\in B_h\).  By Lemma~\ref{lem:Rieffel_def_works}, the deformation
does not change~\(a b\) for \(a\in B_e\) or \(b\in B_e\) and \(a^* b\)
and~\(a b^*\) for \(a,b\in B_g\).  Hence \(B_g^\Lambda = B_g\) as
Hilbert \(B_e\)\nb-bimodules, so that the positivity and completeness
conditions for a Fell bundle are not affected by the deformation.  We
call~\((B_g^\Lambda)_{g\in G}\) the \emph{Rieffel deformation} of the
Fell bundle~\((B_g)_{g\in G}\) with respect to~\(\Lambda\).

For a \Cstar\nb-algebra of the form \(B = \Cst(B_g)\) for a Fell
bundle~\((B_g)_{g\in G}\) over~\(G\), we define its Rieffel
deformation with respect to~\(\Lambda\) as \(B^\Lambda \defeq
\Cst(B_g^\Lambda)\) for the deformed Fell bundle.

If~\(G\) is an Abelian group, then \(\Cst(B_g)\) for a Fell bundle
over~\(G\) carries a canonical continuous action of~\(\hat{G}\),
called the dual action.  Conversely, any \Cstar\nb-algebra with a
continuous \(\hat{G}\)\nb-action~\(\beta\) is of the form
\(B=\Cst(B_g)\), where~\((B_g)_{g\in G}\) is the spectral
decomposition of the action,
\[
B_g = \{b\in B\mid
\beta_\chi(b) = \chi(g)\cdot b\text{ for all }\chi\in \hat{G}\}.
\]
Thus Rieffel deformation takes a \Cstar\nb-algebra with a continuous
\(\hat{G}\)\nb-action to another \Cstar\nb-algebra with a continuous
\(\hat{G}\)\nb-action.  This is how it is usually formulated.
Since~\(\hat{G}\) is compact, there are no analytic difficulties with
oscillatory integrals as
in~\cite{Rieffel:Deformation_Rd}.

\begin{theorem}
  \label{the:Rieffel_deformation}
  Let \(A = \bigoplus_{g\in G} A_g\) be a \(G\)\nb-graded
  \Star{}algebra and let~\(B_e\) be a \Cstar\nb-hull for a class of
  integrable representations of~\(A_e\).  Assume that integrability
  is compatible with induction for~\(A\).  Let~\(\Lambda\) be a
  normalised \(2\)\nb-cocycle on~\(G\).  Then integrability is also
  compatible with induction for~\(A^\Lambda\), and the
  \Cstar\nb-hull for the integrable representations of~\(A^\Lambda\)
  is the Rieffel deformation with respect to~\(\Lambda\) of the
  \Cstar\nb-hull for the integrable representations of~\(A\).
\end{theorem}

\begin{proof}
  The compatibility condition in
  Definition~\ref{def:induced_actions_integrable} is equivalent to the
  integrability of \(A_g \otimes_{A_e} (\Hilms,\pi)\) for all \(g\in
  G\), which only involves a single~\(A_g\) with its
  \(A_e\)\nb-bimodule structure and the \(A_e\)\nb-valued inner
  product \(\braket{a}{b} = a^* b\) for \(a,b\in A_g\).  This is not
  changed by Rieffel deformation by Lemma~\ref{lem:Rieffel_def_works}.
  Hence~\(A^\Lambda\) inherits the compatibility condition from~\(A\),
  and Theorem~\ref{the:Fell_bundle_sections_hull} applies to both
  \(A\) and~\(A^\Lambda\).

  The Hilbert \(B^+_e\)\nb-bimodule~\(B^+_g\) depends only on~\(A_g\) with
  the extra structure above and the universal inducible, integrable
  representation \((\Hilms[B]^+_e,\mu^+_e)\) of~\(A_e\) by
  Remark~\ref{rem:Fell_explicit}.  Since none of this is changed by
  Rieffel deformation, the Fell bundle obtained from~\(A^\Lambda\) has
  the same fibres~\((B_g^+)^\Lambda\) as~\(B_g^+\).  Rieffel deformation
  changes the multiplication maps \(A_g\times A_h \to A_{g h}\) and
  the involution \(A_g \to A_{g^{-1}}\) for fixed \(g,h\in G\) only by
  a scalar.  Inspecting the construction above, we see that the
  multiplication maps \(B_g^+\times B_h^+ \to B_{g h}^+\) and the involution
  \(B_g^+ \to B^+_{g^{-1}}\) in the Fell bundle are changed by exactly the
  same scalars.  Hence the Fell bundle for~\(A^\Lambda\)
  is~\((B_g^+)^\Lambda\).  Now the assertion follows from
  Theorem~\ref{the:Fell_bundle_sections_hull}.
\end{proof}

\section{Twisted Weyl algebras}
\label{sec:Weyl_twisted}

We illustrate our theory by studying \Cstar\nb-hulls of twisted
\(n\)\nb-dimensional Weyl algebras for \(1\le n\le\infty\).  We begin
with the case \(n=1\), where no twists occur.  Then we consider the
case of finite~\(n\) without twists and with twists.  Finally, we
consider the case \(n=\infty\) with and without twists.

The (\(1\)\nb-dimensional) \emph{Weyl algebra}~\(A\) is the universal
\Star{}algebra with one generator~\(a\) and the relation \(a a^* = a^*
a + 1\).  There is a unique \(\Z\)\nb-grading \(A= \bigoplus_{n\in\Z}
A_n\) with \(a\in A_1\).  The \Star{}subalgebra~\(A_0\) is isomorphic
to the polynomial algebra~\(\C[N]\) with \(N= a^* a\), which is
commutative.  The other subspaces \(A_k\subseteq A\) for \(k\in\N\)
are isomorphic to~\(A_0\) as left or right \(A_0\)\nb-modules because
\(A_k = A_0\cdot a^k = a^k\cdot A_0\) and \(A_{-k} = (a^*)^k\cdot A_0
= A_0\cdot (a^*)^k\) for all \(k\ge0\).  The spectrum~\(\hat{A}_0\)
of~\(A_0\) is~\(\R\), where the character \(\C[N]\to\C\) for
\(t\in\R\) evaluates a
polynomial at~\(t\).  A character is positive if and only if
it is positive on~\((a^*)^k a^k\) and~\(a^k (a^*)^k\) for all
\(k\ge1\).  This happens if and only if \(t\in\N\) by
\cite{Savchuk-Schmudgen:Unbounded_induced}*{Example 10}.

Since \(N a = a^* a a = (a a^* -1)a = a\cdot (a^* a-1) = a\cdot
(N-1)\), the partial automorphism~\(\vartheta_1\) of \(\hat{A}_0^+=\N\)
associated to the \(A_0\)\nb-bimodule~\(A_1\) acts on~\(\hat{A}_e^+\)
by the automorphism \(N\mapsto N-1\), which corresponds to translation
by~\(-1\).  By induction, we get \(N\cdot a^k = a\cdot (N-1)\cdot
a^{k-1} = \dotsb = a^k\cdot (N-k)\).  The domain of~\(\vartheta_k\) is as
big as it could possibly be, that is, it contains all \(n\in\N\) with
\(n\ge k\) by~\eqref{eq:partial_act_Aeplus_dom} (see also
\cite{Savchuk-Schmudgen:Unbounded_induced}*{Example 16}).  For any
\(k,l\in\N\) there is a unique \(n\in\Z\) with \(k-n=l\).  Thus the
transformation groupoid \(\Z\ltimes_\vartheta \N\) is simply the pair
groupoid on~\(\N\).  There can be no twist in this case.  First, the
pair groupoid simply has no non-trivial twists.  And secondly, the
generators~\(a^k,(a^*)^k\) for \(k\ge0\) satisfy the positivity
condition in
Proposition~\ref{pro:commutative_partial_action_no_twist}, which also
rules out a twist.

Since no proper non-empty subset of~\(\N\) is invariant under the
partial action~\(\vartheta\) of~\(\Z\), a commutative \Cstar\nb-hull
for~\(A_0\) for which integrability is compatible with induction
gives either \(B_0^+= \Cont_0(\N)\) or \(B_0^+=\{0\}\).  In the
second case, \(A\) has no non-zero integrable representations.  In
the first case, the \Cstar\nb-hull for the integrable
representations of~\(A\) is the groupoid \Cstar\nb-algebra
\(\Comp(\ell^2\N)\) of the pair groupoid~\(\N\times\N\).

The universal representation of~\(A\) on~\(\Comp(\ell^2\N)\) is
equivalent to a representation~\(\pi\) of~\(A\) on~\(\ell^2\N\) by
Proposition~\ref{pro:representation_Hilm_Comp}.  The domain of this
representation is the space~\(\mathcal{S}(\N)\) of rapidly
decreasing sequences, with \(\pi(a)(\delta_k) = \sqrt{k}
\delta_{k-1}\) for \(k\in\N\), so \(\pi(a^*)(\delta_k) = \sqrt{k+1}
\delta_{k+1}\), \(\pi(N)(\delta_k) = k\delta_k\).  By
Theorem~\ref{the:regular_self-adjoint_Cstar-hull}, a
representation~\(\pi\) of~\(A_0\) on a Hilbert module~\(\Hilm\) is
integrable if and only~\(\cl{\pi(N^k)}\) is regular and self-adjoint
for each \(k\in\N\) or, equivalently, \(\pi(N)\) is regular and
self-adjoint and \(\cl{\pi(N^k)} = \cl{\pi(N)}{}^k\) for all
\(k\in\N\).  By definition, a representation of~\(A\) is integrable
if and only if its restriction to~\(A_0\) is integrable.

The \(\Z\)\nb-grading on the \Cstar\nb-hull~\(\Comp(\ell^2\N)\) is
``inner'': it is induced by the \(\Z\)\nb-grading on~\(\ell^2\N\)
where~\(\delta_k\) has degree~\(k\).  Equivalently, the dual action
of~\(\T\) on~\(\Comp(\ell^2\N)\) associated to the \(\Z\)\nb-grading
is the inner action associated to the unitary representation \(U\colon
\T\to\textup{U}(\ell^2\N)\), where \(U_z(\delta_k) \defeq
z^k\delta_k\) for all \(z\in\T\), \(k\in\N\).

\medskip

Now let \(m\in\N\) and let \(\Theta = (\Theta_{jk})\) be an
antisymmetric \(m\times m\)-matrix.  Let \(\lambda_{jk} =
\exp(2\pi\ima\Theta_{jk})\).  Let~\(A^{m,\Theta}\) be the
\Star{}algebra with generators \(a_1,\dotsc,a_m\) and the
commutation relations \(a_j^{\vphantom{*}} a_j^* = a_j^*
a_j^{\vphantom{*}} + 1\) for \(1\le j\le m\) and
\begin{equation}
  \label{eq:Weyl_twisted_commutation_jk}
  a_j^{\vphantom{*}} a_k^{\vphantom{*}}
  = \lambda_{j k}^{\vphantom{*}} a_k^{\vphantom{*}} a_j^{\vphantom{*}},\qquad
  a_j^* a_k^{\vphantom{*}} = \lambda_{j k}^{-1} a_k^{\vphantom{*}} a_j^*
\end{equation}
for \(1\le j\neq k\le m\).  Since \(\lambda_{j k}^{\vphantom{-1}} =
\lambda_{k j}^{-1}\), the relations~\eqref{eq:Weyl_twisted_commutation_jk} for
\((j,k)\) and~\((k,j)\) are equivalent; so it suffices to
require~\eqref{eq:Weyl_twisted_commutation_jk} for \(1\le j<k\le m\).
The \Star{}algebra~\(A^{m,\Theta}\) is \(\Z^m\)\nb-graded by
giving~\(a_j\) degree \(e_j\in\Z^m\), where \(e_1,\dotsc,e_m\) is the
standard basis of~\(\Z^m\).

We first consider the case \(\Theta=0\) and write \(A^m\defeq
A^{m,0}\).  This is the \emph{\(m\)\nb-dimensional Weyl algebra},
which is the tensor product of \(m\)~copies of the
\(1\)\nb-dimensional Weyl algebra, with the induced
\(\Z^m\)\nb-grading.  Thus the zero fibre~\(A^m_0\) for \(0\in\Z^m\)
is isomorphic to the polynomial algebra \(\C[N_1,\dotsc,N_m]\) in the
\(m\)~generators \(N_j^{\vphantom{*}} = a_j^* a_j^{\vphantom{*}}\).
Its spectrum is~\(\R^m\).
Each~\(A^m_k\) for \(k\in\Z^m\) is isomorphic to~\(A^m_0\) both as a
left and a right \(A^m_0\)\nb-module; the generator is the product
of~\(a_j^{k_j}\) for \(k_j\ge0\) or~\((a_j^*)^{-k_j}\) for \(k_j<0\)
from \(j=1,\dotsc,m\).  Here the order of the factors does not matter
because \(\Theta=0\).  We may identify~\(A^m_k\) with the exterior
tensor product of the \(A^1\)\nb-bimodules \(A^1_{k_1} \otimes
A^1_{k_2} \otimes \dotsb \otimes A^1_{k_m}\).  Hence the space of
positive characters on~\(A^m\) is~\(\N^m\), and the partial action
of~\(\Z^m\) on~\(\N^m\) is the exterior product of the partial
actions of~\(\Z\) on~\(\N\) for the \(1\)\nb-dimensional Weyl
algebras.  That is, \(k\in\Z^m\) acts on~\(\N^m\) by translation
by~\(-k\) with the maximal possible domain.  Thus the transformation
groupoid \(\Z^m \ltimes \hat{A}_0^+\) is isomorphic to the pair
groupoid of the discrete set~\(\N^m\).

Once again, the only \(\Z^m\)\nb-invariant subsets of~\(\hat{A}_0^+\)
are the empty set and~\(\N^m\), so that the only inducible commutative
\Cstar\nb-hulls of~\(A_0\) for which integrability is compatible with
induction are~\(\{0\}\) and~\(\Cont_0(\N^m)\).  The first case is
boring, and the second case leads to the
\Cstar\nb-hull~\(\Comp(\ell^2\N^m)\) of the \(m\)\nb-dimensional Weyl
algebra.

As for \(m=1\), the universal representation of~\(A^m\) is
equivalent to a representation on~\(\ell^2\N^m\).  This has the
domain~\(\mathcal{S}(\N^m)\), and the representation is determined by
\[
\pi(a_j)(\delta_{(k_1,\dotsc,k_m)}) =
\sqrt{k_j} \delta_{(k_1,\dotsc,k_j-1,\dotsc,k_m)}
\]
for \((k_1,\dotsc,k_m)\in\N^m\) and \(j=1,\dotsc,m\).  Hence
\(\pi(N_j)(\delta_{(k_1,\dotsc,k_m)}) =
k_j\delta_{(k_1,\dotsc,k_m)}\).  A representation of~\(A\) is
integrable if and only if its restriction to~\(A_0\) is integrable in
the sense that it integrates to a representation of~\(\Cont_0(\R^m)\).
This automatically descends to a representation of~\(\Cont_0(\N^m)\)
by Lemma~\ref{lem:induced_repr_inducible}.  There are several ways to
characterise when a representation of~\(\C[N_1,\dotsc,N_m]\)
integrates to a representation of~\(\Cont_0(\R^m)\).  One is
that~\(\cl{\pi(N_j)}\) for \(j=1,\dotsc,m\) are strongly commuting,
regular, self-adjoint operators and \(\cl{\pi(N_j^k)} =
\cl{\pi(N_j)}{}^k\) for all \(j=1,\dotsc,m\), \(k\in\N\), compare
\cite{Schmudgen:Unbounded_book}*{Theorem 9.1.13}.

The groups~\(\Z^m\) for \(m\ge2\) have non-trivial \(2\)\nb-cocycles,
and~\(A^{m,\Theta}\) is, by definition, a Rieffel deformation
of~\(A^{m,0}\) for the normalised \(2\)\nb-cocycle
\begin{equation}
  \label{eq:Rieffel_Lambda}
  \Lambda\bigl((x_1,\dotsc,x_m),(y_1,\dotsc,y_m)\bigr) \defeq
  \prod_{j=1}^m \prod_{k=j+1}^m \lambda_{jk}^{x_k y_j}.
\end{equation}
We could also use the cohomologous antisymmetric \(2\)\nb-cocycle
\[
\prod_{j=1}^m \prod_{k=j+1}^m
\sqrt{\lambda_{jk}}^{x_k y_j - x_j y_k}
= \prod_{j,k=1}^m \exp(-\pi\ima \Theta_{jk} x_j y_k).
\]
Theorem~\ref{the:Rieffel_deformation} says that the
\Cstar\nb-hull~\(B^{m,\Theta}\) of~\(A^{m,\Theta}\) is the Rieffel
deformation of the \Cstar\nb-hull~\(B^{m,0}\) of~\(A^{m,0}\) for the
same \(2\)\nb-cocycle.

In the classification of Fell bundles with commutative unit fibre,
the important cohomology is that of the transformation groupoid
\(G\ltimes X^+\), not of~\(G\) itself.  Here \(G\ltimes X^+\) is the
pair groupoid of~\(\N^m\).

\begin{lemma}[compare \cite{Mantoiu-Purice-Richard:Magnetic}*{Lemma 2.9}]
  \label{lem:cohomology_pair_groupoid}
  The cohomology of the pair groupoid of a discrete set~\(X\) with
  coefficients in an Abelian group~\(H\) vanishes in all positive
  degrees.
\end{lemma}

\begin{proof}
  The set of composable \(n\)\nb-tuples in the pair groupoid
  of~\(X\) is~\(X^{n+1}\).  The groupoid cohomology with
  coefficients~\(H\) is the cohomology of the chain complex with
  cochains \(C_n \defeq H^{X^{n+1}}\), the space of all maps
  \(X^{n+1}\to H\), and with the boundary map \(\partial\colon
  C_n\to C_{n+1}\),
  \[
  \partial \varphi(x_0,\dotsc,x_n) \defeq \sum_{i=0}^n (-1)^i
  \varphi(x_0,\dotsc, \widehat{x_i},\dotsc, x_n);
  \]
  here the hat means that the entry~\(x_i\) is deleted.  Pick some
  point \(x_0\in X\) and let \(h\varphi(x_1,\dotsc,x_n) \defeq
  \varphi(x_0,x_1,\dotsc,x_n)\) for all \(n\in\N\),
  \(x_1,\dotsc,x_n\in X\), \(\varphi\in C_{n+1}\).  Then
  \(\partial\circ h(\varphi) + h \circ \partial(\varphi)=\varphi\)
  for all \(\varphi\in C_n\), \(n\ge1\).  Thus the cohomology
  vanishes in positive degrees.
\end{proof}

Any twist of the pair groupoid on~\(\N^m\) is trivial by
Lemma~\ref{lem:cohomology_pair_groupoid}.  Therefore, the
\Cstar\nb-hull~\(B^{m,\Theta}\) is isomorphic
to~\(\Comp(\ell^2\N^m)\), the untwisted groupoid \Cstar\nb-algebra
of the pair groupoid.  The proof of
Lemma~\ref{lem:cohomology_pair_groupoid} is explicit and so allows
to construct this isomorphism.
We explain another way to construct it, using
properties of Rieffel deformation.  Since
the \(\Z\)\nb-grading on the \Cstar\nb-hull \(\Comp(\ell^2\N)\) is
inner or, equivalently, the corresponding action of~\(\T\) is
inner, the same holds for the \(\Z^m\)\nb-grading on the
\Cstar\nb-hull~\(\Comp(\ell^2\N^m)\) and the corresponding
\(\T^m\)\nb-action on~\(\Comp(\ell^2\N^m)\).  Explicitly, the
\(\T^m\)\nb-action is induced by the unitary representation
of~\(\T^m\) on~\(\ell^2\N^m\) defined by
\[
U_{(z_1,\dotsc,z_m)}\delta_{(k_1,\dotsc,k_m)} \defeq
z_1^{-k_1} \dotsm z_m^{-k_m} \delta_{(k_1,\dotsc,k_m)}.
\]
Rieffel deformation of \Cstar\nb-algebras for \emph{inner} actions
does not change the isomorphism type of the \Cstar\nb-algebra.
Hence the \Cstar\nb-hull for the integrable representations
of~\(A^{m,\Theta}\) is also isomorphic to~\(\Comp(\ell^2\N^m)\).

We make this more explicit in our Fell bundle language.  Let
\(U\colon \T^m\to\mathcal{UM}(B)\) be a strictly continuous
homomorphism to the group of unitary multipliers of a
\Cstar\nb-algebra~\(B\) and let \(\alpha_z(b) \defeq U_z b U_z^*\)
for \(z\in\T^m\), \(b\in B\) be the resulting inner action.  Let
\((B_k)_{k\in\Z^m}\) be the spectral decomposition of this action,
that is, \(b\in B_k\) if and only if \(\alpha_z(b) = z^k\cdot b\)
for all \(z\in\T^m\).  In particular, \(U\in\mathcal{UM}(B_0)\)
because~\(\T^m\) is commutative.

Assume for simplicity that the \(2\)\nb-cocycle~\(\Lambda\) is a
bicharacter as above.  For fixed \(k\in\Z^m\), we view
\(\Lambda(k,\blank)\colon \Z^m\to\T\) as an
element~\(\tilde\Lambda(k)\) of the dual group~\(\T^m\).  The map
\(\tilde\Lambda\colon \Z^m \to \T^m\) is a group homomorphism.
The maps \(\psi_k\colon B_k \to B_k\), \(b\mapsto
U_{\tilde\Lambda(k)}^*\cdot b\), for \(k\in\N\) are Banach space
isomorphisms that modify the multiplication as follows:
\begin{align*}
  \psi_k^{\vphantom{*}}(b_1^{\vphantom{*}}) \psi_l^{\vphantom{*}}(b_2^{\vphantom{*}})
  &= U^*_{\tilde\Lambda(k)} b_1^{\vphantom{*}} U^*_{\tilde\Lambda(l)} b_2^{\vphantom{*}}
  = U_{\tilde\Lambda(k+l)}^* \alpha_{\tilde\Lambda(l)}(b_1^{\vphantom{*}}) b_2^{\vphantom{*}}
  = \psi_{k+l}^{\vphantom{*}}(\Lambda(k,l)\cdot b_1^{\vphantom{*}} b_2^{\vphantom{*}}).
\end{align*}
They keep the involution unchanged.  This is exactly what Rieffel
deformation does.  Hence the maps~\(\psi_k\) form an isomorphism
between the undeformed and Rieffel deformed Fell bundles.  This
finishes the proof that the Rieffel deformed algebra for an inner
action is canonically isomorphic to the original algebra.

The universal representation of~\(A^{m,\Theta}\)
on~\(\Comp(\ell^2\N^m)\) again corresponds to a representation
of~\(A^{m,\Theta}\) on~\(\ell^2(\N^m)\).  We may construct it by
carrying over the isomorphism \(B^{m,\Theta} \cong B^{m,0}\) between
the \Cstar\nb-hulls.  This is the inverse of the isomorphism above, so
it multiplies on the left by the unitary~\(U_{\tilde\Lambda(k)}\) of
degree~\(0\) on elements of degree~\(k\).  We do exactly the same on
elements of~\(A^{m,\Theta}\) and so let \(x\in A^{m,\Theta}_k\) for
\(k\in\Z^m\) act on~\(\ell^2\N^m\) by the operator
\(U_{\tilde\Lambda(k)} \pi^{m,0}(x)\), where~\(\pi^{m,0}\) is the
universal representation of the untwisted Weyl algebra~\(A^{m,0}\)
on~\(\ell^2\N^m\) described above.  The same computation as above
shows that this defines a \Star{}representation of~\(A^{m,\Theta}\).
We compute it explicitly.

First, the action of elements of~\(A^{m,\Theta}_0\) on~\(\ell^2\N^m\)
is not changed.  The domain of a representation of~\(A^{m,\Theta}\) is
equal to the domain of its restriction to~\(A^{m,\Theta}_0\).  Hence
the domain of our representation is the Schwartz
space~\(\mathcal{S}(\N^m)\), as in the untwisted case.  Let \(1\le
j\le m\).  The generator~\(a_j\) has degree \(e_j\in\Z^m\), and
\[
\tilde\Lambda(e_j) =
(\lambda_{1,j},\dotsc,\lambda_{j-1,j},1,\dotsc,1)\in\T^m
\]
for our first definition of~\(\Lambda\) in~\eqref{eq:Rieffel_Lambda}.
Thus
\[
\pi^{m,\Theta}(a_j) \delta_{(k_1,\dotsc,k_m)}
= U_{\tilde\Lambda(e_j)} \pi^{m,0}(a_j) \delta_{(k_1,\dotsc,k_m)}
= \biggl(\prod_{l=1}^{j-1} \lambda_{j,l}^{k_l} \biggr) \sqrt{k_j}
\delta_{(k_1,\dotsc,k_j-1,\dotsc,k_m)}.
\]
These operators on~\(\mathcal{S}(\N^m)\) satisfy the defining
relations of~\(A^{m,\Theta}\).

\medskip

The \emph{infinite-dimensional Weyl algebra~\(A^\infty\)} is the
universal \Star{}algebra with generators~\(a_j\) for \(j\in\N\) and
relations \(a_j^* a_j^{\vphantom{*}} = a_j^{\vphantom{*}} a_j^* +
1\), \(a_j a_k = a_k a_j\), and \(a_j^* a_k^{\vphantom{*}} =
a_k^{\vphantom{*}} a_j^*\) for \(0\le j<k\).  Let~\(\Z[\N]\) be the
free Abelian group on countably many generators.  The Weyl
algebra~\(A^\infty\) is \(\Z[\N]\)\nb-graded, where~\(a_j\) has
degree \(e_j\in\Z[\N]\), the \(j\)th generator of~\(\Z[\N]\).

The \Star{}algebra~\(A^\infty\) is a tensor product of infinitely
many \(1\)\nb-dimensional Weyl algebras.  The zero
fibre~\(A^\infty_0\) is the polynomial algebra in the generators
\(N_j^{\vphantom{*}} = a_j^* a_j^{\vphantom{*}}\) for \(j\in\N\).
Hence its spectrum is the
infinite product \(\hat{A}^\infty_0 = \R^\N\), which is not locally
compact.  The tensor product structure of~\(A^\infty\) shows that a
character is positive if and only if each component is.  That is,
\((A^\infty)_0^+ \cong \N^\N\) is a product of countably many copies
of the discrete space~\(\N\).  Since~\(\N\) is not compact, this is
not locally compact either.  Hence to build a commutative
\Cstar\nb-hull for~\(A^\infty_0\), we must choose some locally
compact space~\(X\) with a continuous, injective map \(f\colon X\to
\N^\N\).  Here we have simplified notation by assuming that already
\(f(X)\subseteq \N^\N\); otherwise, the first step in our
construction would replace~\(X\) by \(X^+ \defeq f^{-1}(\N^\N)\).
For \(X\)\nb-integrability to be compatible with induction, we also
need \(f(X)\) to be invariant under the partial action
of~\(\Z[\N]\), and we need the restricted partial action to lift to
a continuous partial action on~\(X\).

The partial action of the group~\(\Z[\N]\) on~\(\N^\N\) is the
obvious one by translations.  It is free, and two
points \((n_k)\) and~\((n'_k)\) in~\(\N^\N\) belong to the same
orbit if and only if there is~\(k_0\) such that \(n_k = n'_k\) for
all \(k\ge k_0\).  Briefly, such points are called \emph{tail
  equivalent}.  This partial action is minimal, that is, an open,
\(\Z[\N]\)-invariant subset is either empty or the whole space.
Hence~\(\N^\N\) has no \(\Z[\N]\)\nb-invariant, locally closed
subsets.  Thus~\(\N^\N\) has no \(\Z[\N]\)\nb-invariant, locally
compact subspaces.

Let~\(K\) be any compact subset of~\(\N^\N\).  Then the projection
\(p_j\colon \N^\N\to\N\) to the \(j\)th coordinate must map~\(K\) to
a compact subset of~\(\N\).  So there is an upper bound~\(M_j\) with
\(p_j(K)\subseteq [0,M_j]_\N \defeq [0,M_j]\cap \N\).  Then
\(K\subseteq \prod_{j\in\N} [0,M_j]_\N\), and the right hand side is
compact.  The closure of \(\prod_{j\in\N} [0,M_j]_\N\) under tail
equivalence is
\[
X_{(M_k)} \defeq \bigcup_{j\in\N}
\biggl(\N^j \times \prod_{k>j} [0,M_k]_\N\biggr),
\]
the restricted product of copies of~\(\N\) with respect to the
compact-open subsets \([0,M_j]_\N\).  There is a unique topology
on~\(X_{(M_k)}\) where each subset \(\N^j \times \prod_{k>j}
[0,M_k]_\N\) is open and carries the product topology.  This
topology is locally compact, and the partial action of~\(\Z[\N]\)
on~\(X_{(M_k)}\) by translation is continuous.

\begin{lemma}
  The Local--Global Principle fails for the \(X_{(M_k)}\)\nb-integrable
  representations of~\(A^\infty\).
\end{lemma}

\begin{proof}
  Since the map \(X\to \hat{A}^\infty_0\) is not a homeomorphism
  onto its image and~\(\hat{A}^\infty_0\) is metrisable, the
  Local--Global Principle fails for the \(X\)\nb-integrable
  representations of~\(A^\infty_0\) by
  Theorem~\ref{the:commutative_local-global}.  Applying induction
  from~\(A^\infty_0\) to~\(A^\infty\) to a counterexample for the
  Local--Global Principle for~\(A^\infty_0\) produces such a
  counterexample also for~\(A^\infty\).  Explicitly, choose a
  sequence \((n_k)\) such that \(n_k>M_k\) for infinitely
  many~\(k\).  Let \(x_k\defeq n_k\delta_k \in\N^\N\), that is,
  \(x_k\in\N^\N\) has only one non-zero entry, which is~\(n_k\) in
  the \(k\)th place.  This sequence belongs to~\(X_{(M_k)}\) and
  converges to~\(0\) in the product topology
  on~\(\hat{A}^\infty_0\), but not in the topology of~\(X_{(M_k)}\).
  The resulting representation of~\(A^\infty_0\)
  on~\(\Cont(\bar{\N})\) is not \(X_{(M_k)}\)\nb-integrable, but it
  becomes integrable when we tensor with any Hilbert space
  representation of~\(\Cont(\bar{\N})\), see the proof of
  Theorem~\ref{the:commutative_local-global}.  Now induce this
  (inducible) representation of~\(A^\infty_0\) to a representation
  of~\(A^\infty\) on~\(\Cont(\bar{\N})\).  This gives a
  counterexample for the Local--Global Principle for~\(A^\infty\).
\end{proof}

I do not know a class of integrable representations of~\(A^\infty\)
with a \Cstar\nb-hull for which the Local--Global Principle holds.

Let~\(S\) be the set of all words in the letters
\(a_j^{\vphantom{*}},a_j^*\).  If
\(\chi\in\N^\N\) is a positive character and \(x,y\in S\cap
A^\infty_k\) for some \(k\in\Z[\N]\), then \(\chi(x^* y)\ge0\).  Hence
Proposition~\ref{pro:commutative_partial_action_no_twist} shows that
there is no twist in our case, that is, the \Cstar\nb-hull of the
\(X_{(M_k)}\)\nb-integrable representations of~\(A^\infty\) is the
groupoid \Cstar\nb-algebra of the transformation groupoid
\(\Z[\N]\ltimes X_{(M_k)}\).  This \Cstar\nb-hull is canonically
isomorphic to one of the host algebras for~\(A^\infty\) constructed
in~\cite{Grundling-Neeb:Full_regularity}, namely, to the one that is
denoted \(\mathcal{L}[\mathbf{n}]\)
in~\cite{Grundling-Neeb:Full_regularity} with \(\mathbf{n}_k = M_k+1\)
for all \(k\in\N\).  We remark in passing that the construction of a
\emph{full} host algebra from these host algebras
in~\cite{Grundling-Neeb:Full_regularity} is wrong: the resulting
\(\Cst\)\nb-algebra has too many Hilbert space representations, so
it is not a host algebra any more.  An erratum
to~\cite{Grundling-Neeb:Full_regularity} is currently being written.

The compact subset \(T \defeq \prod_{k\in\N} [0,M_k]_\N\) that we
started with is a complete transversal in \(\Z[\N]\ltimes
X_{(M_k)}\), that is, the range map in \(\Z[\N]\ltimes X_{(M_k)}\)
restricted to~\(s^{-1}(T)\) is an open surjection
onto~\(X_{(M_k)}\).  Hence the groupoid \(\Z[\N]\ltimes X_{(M_k)}\)
is Morita equivalent to its restriction to the compact subset~\(T\).
This restriction is the tail equivalence relation on~\(T\).  Its
groupoid \Cstar\nb-algebra is well known: it is the UHF-algebra for
\(\prod_{k\in\N} (M_k+1)\), that is, the infinite tensor product of
the matrix algebras \(\bigotimes_{k\in\N} \Mat_{M_k+1}\).  The
\Cstar\nb-algebra of \(\Z[\N]\ltimes X_{(M_k)}\) itself is the
\Cstar\nb-stabilisation of this UHF-algebra.

Thus the \(X_{(M_k)}\)\nb-integrable representations of~\(A^\infty\)
are equivalent to the representations of the (stabilisation of the)
UHF-algebra of type \(\prod_{k\in\N} (M_k+1)\).  This depends very
subtly on the choice of the sequence~\((M_k)\).  There is no canonical
\Star{}homomorphism between these UHF-algebras if we
increase~\((M_k)\): for some \((M_k)\le (M_k')\), there is not even a
non-zero map between their K\nb-theory groups.  Instead, there are
canonical morphisms, that is, there is a canonical nondegenerate
\Star{}homomorphism \(\Cst(\Z[N]\ltimes X_{(M_k')}) \to
\Mult(\Cst(\Z[N]\ltimes X_{(M_k)}))\) if \((M_k)\le (M_k')\).  They are
constructed as follows.  The inclusion map \(X_{(M_k)} \injto
X_{(M_k')}\) is continuous with dense range, but not proper.  Thus it
induces an injective, nondegenerate \Star{}homomorphism
\(\Cont_0(X_{(M_k')}) \to \Contb(X_{(M_k)})\).  Therefore, if a
representation of~\(A_0^\infty\) is \(X_{(M_k)}\)\nb-integrable, then
it is also \(X_{(M_k')}\)\nb-integrable.  If a representation
of~\(A^\infty\) has \(X_{(M_k)}\)\nb-integrable restriction
to~\(A^\infty_0\), then its restriction to~\(A^\infty_0\) is also
\(X_{(M_k')}\)\nb-integrable.  When we apply this to the universal
representation of~\(A^\infty\) on the \(\Cst\)\nb-hull
\(\Cst(\Z[N]\ltimes X_{(M_k)})\), this gives the desired canonical
morphism \(\Cst(\Z[N]\ltimes X_{(M_k')}) \to \Cst(\Z[N]\ltimes
X_{(M_k)})\) if \((M_k)\le (M_k')\).  It is injective, say, because
\(\Cst(\Z[N]\ltimes X_{(M_k')})\) is simple.

Now let \(\Theta = (\Theta_{jk})_{j,k\in\N}\) be a skew-symmetric
matrix.  It corresponds first to a matrix \(\lambda_{jk} \defeq
\exp(2\pi\ima\Theta_{jk})\) and then to a
\(2\)\nb-cocycle~\(\Lambda\) on~\(\Z[\N]\) as
in~\eqref{eq:Rieffel_Lambda}.  The Rieffel deformation
of~\(A^\infty\) by~\(\Theta\) is the universal
\Star{}algebra~\(A^{\infty,\Theta}\) with the same
generators~\((a_j)_{j\in\N}\) and the relations \(a_j^{\vphantom{*}}
a_j^* = a_j^* a_j^{\vphantom{*}} + 1\), \(a_j a_k = \lambda_{j k}
a_k a_j\), and \(a_j^* a_k^{\vphantom{*}} = \lambda_{j k}^{-1}
a_k^{\vphantom{*}} a_j^*\) for \(0\le j<k\).  We
define~\(X_{(M_k)}\) for a sequence~\((M_k)\) and the
\(X_{(M_k)}\)\nb-integrable representations of~\(A^{\infty,\Theta}\)
as above.  By Theorem~\ref{the:Rieffel_deformation}, this has a
\Cstar\nb-hull, namely, the Rieffel deformation of the
\Cstar\nb-hull for the \(X_{(M_k)}\)\nb-integrable representations
of the undeformed Weyl algebra.  The Rieffel deformation gives a
twist of the groupoid \(\Z[\N]\ltimes X_{(M_k)}\), and the
\Cstar\nb-hull is the twisted groupoid \Cstar\nb-algebra of
\(\Z[\N]\ltimes X_{(M_k)}\) for this twist.

\begin{proposition}
  Let \((M_k)\in\N^\N\).  The \Cstar\nb-hulls for the
  \(X_{(M_k)}\)-integrable representations of the twisted Weyl
  algebras~\(A^{\infty,\Theta}\) are isomorphic for all~\(\Theta\).
\end{proposition}

\begin{proof}
  The \Cstar\nb-hull of~\(A^{\infty,\Theta}\) is a twisted groupoid
  algebra of the transformation groupoid \(\Z[\N]\ltimes
  X_{(M_k)}\), which is isomorphic to the tail equivalence
  relation~\(R\) on~\(X_{(M_k)}\).  We are going to prove that any
  twist \(X_{(M_k)}\times\T\into \Sigma\prto R\) is trivial.  Hence
  the \Cstar\nb-hull of~\(A^{\infty,\Theta}\) is isomorphic to the
  untwisted groupoid \Cstar\nb-algebra of~\(R\) for all~\(\Theta\).

  The arrow space of~\(R\) is totally disconnected
  because~\(X_{(M_k)}\) is totally disconnected and~\(R\) is étale.
  Hence any locally trivial principal bundle over~\(R\) is trivial.
  Thus \(\Sigma \cong R\times\T\) as a topological space, and the
  twist is described by a continuous \(2\)\nb-cocycle
  \(\varphi\colon R^{(2)}\defeq R\times_{s,X,r} R \to \T\).  We must
  show that~\(\varphi\) is a coboundary.

  Let~\(R_d\) for \(d\in\N\) be the equivalence relation
  on~\(X_{(M_k)}\) defined by \((n_k) \mathrel{R_d} (n_k')\) if and
  only if \(n_k = n_k'\) for all \(k\ge d\).  This is an increasing
  sequence of open subsets \(R_d\subseteq R\) with \(R=\bigcup
  R_d\), and each~\(R_d\) is also an equivalence relation.  The
  equivalence relation~\(R_d\) is isomorphic to the product of the
  pair groupoid on~\(\N^d\) and the space~\(X_{(M_{k+d})}\) for the
  shifted sequence \((M_{k+d})_{k\in\N}\).  Thus the cohomology
  of~\(R_d\) with coefficients in~\(\T\) is isomorphic to the
  cohomology of the pair groupoid on~\(\N^d\) with values in the
  Abelian group of continuous map \(X_{(M_{k+d})}\to\T\).
  This cohomology vanishes in
  positive degrees by Lemma~\ref{lem:cohomology_pair_groupoid}.
  Therefore, for each \(d\in\N\) there is \(\psi_d\colon R_d\to\T\)
  such that \(\varphi|_{R_d}\colon R_d\times_{s,r} R_d\to \T\) is
  the coboundary~\(\partial\psi_d\).  The restriction
  of~\(\psi_{d+1}\) to~\(R_d\) and~\(\psi_d\)
  both have coboundary~\(\varphi|_{R_d}\).  Hence
  \(\psi_{d+1}^{-1}|_{R_d}\cdot \psi_d\) is a \(1\)\nb-cocycle
  on~\(R_d\).  Again by Lemma~\ref{lem:cohomology_pair_groupoid},
  there is \(\chi\colon X\to\T\) with \(\psi_{d+1}^{-1}|_{R_d} \cdot
  \psi_d = \partial_{R_d}\chi\).  We replace~\(\psi_{d+1}\) by
  \(\psi_{d+1}' \defeq \psi_{d+1} \cdot \partial_{R_{d+1}}\chi\),
  where~\(\partial_{R_{d+1}}\chi\) means the coboundary of~\(\chi\)
  on the groupoid~\(R_{d+1}\).  This still satisfies \(\partial
  \psi_{d+1}' = \partial\psi_{d+1} = \varphi|_{R_{d+1}}\), and
  \(\psi'_{d+1}|_{R_d} = \psi_d\).  Proceeding like this, we get
  continuous maps \(\psi'_d\colon R_d \to\T\) for all \(d\in\N\)
  that satisfy \(\psi'_{d+1}|_{R_d} = \psi'_d\) and \(\partial
  \psi_d' = \varphi|_{R_d}\) for all \(d\in\N\).  These combine to a
  continuous map \(\psi'\colon R\to \T\) with \(\partial \psi' =
  \varphi\).
\end{proof}

\end{document}